%% file: kdsstab-rr.tex
\documentclass[reqno,10pt]{amsart}

\usepackage{amsmath,amssymb,amsthm,mathrsfs}
\usepackage{slashed}
\usepackage{xifthen}
\usepackage{graphicx}
\usepackage{longtable}

\usepackage{xcolor}
\definecolor{winered}{rgb}{0.7,0,0}
\definecolor{lessblue}{rgb}{0,0,0.7}
\usepackage[pdftex,colorlinks=true,linkcolor=winered,citecolor=lessblue,breaklinks=true,bookmarksopen=true]{hyperref}

\usepackage{titletoc}

\makeatletter

\def\@tocline#1#2#3#4#5#6#7{
\begingroup
  \par
    \parindent\z@ \leftskip#3 \relax \advance\leftskip\@tempdima\relax
                  \rightskip\@pnumwidth plus 4em \parfillskip-\@pnumwidth
    \ifcase #1 
       \vskip 0.6em \hskip 0em 
       \or
       \or \hskip 0em 
       \or \hskip 1em 
    \fi%
    #6
    \nobreak\relax{\leavevmode\leaders\hbox{\,.}\hfill}
    \hbox to\@pnumwidth {\@tocpagenum{#7}}
  \par
\endgroup
}

 \def\l@section{\@tocline{0}{0pt}{0pc}{}{}}

\renewcommand{\tocsection}[3]{%
  \indentlabel{\@ifnotempty{#2}{ 
    \ignorespaces\bfseries{#2. #3}}}
  \indentlabel{\@ifempty{#2}{\ignorespaces\bfseries{#3}}{}} 
    \vspace{1.5pt}}

\renewcommand{\tocsubsection}[3]{%
  \indentlabel{\@ifnotempty{#2}{
    \ignorespaces#2. #3}}
  \indentlabel{\@ifempty{#2}{\ignorespaces #3}{}}
    \vspace{1.5pt}}

\renewcommand{\tocsubsubsection}[3]{%
  \indentlabel{\@ifnotempty{#2}{
    \ignorespaces#2. #3}}
  \indentlabel{\@ifempty{#2}{\ignorespaces #3}{}}
    \vspace{1.5pt}}

\makeatother

\makeatletter
\def\@nomenstarted{0}
\newlength{\@nomenoldtabcolsep}

\newcommand{\nomenstart}
  {%
    \def\@nomenstarted{1}%
    \setlength{\@nomenoldtabcolsep}{\tabcolsep}%
    \setlength{\tabcolsep}{3.5pt}%
    \begin{longtable}{p{0.15\textwidth} p{0.81\textwidth}}
  }

\newcommand{\nomenitem}[2]{%
    \ifcase\@nomenstarted%
      \or 
      \or \\ 
    \fi%
    #1\,{\leavevmode\leaders\hbox{\,.}\hfill} & #2%
    \def\@nomenstarted{2}%
  }%
\newcommand{\nomenend}
  {\\%
      \end{longtable}%
      \setlength{\tabcolsep}{\@nomenoldtabcolsep}%
      \def\@nomenstarted{0}%
  }
\makeatother

\allowdisplaybreaks

\numberwithin{equation}{section}
\newtheorem{thm}{Theorem}[section]
\newtheorem{prop}[thm]{Proposition}
\newtheorem{lemma}[thm]{Lemma}
\newtheorem{cor}[thm]{Corollary}
\newtheorem*{thm*}{Theorem}
\newtheorem*{prop*}{Proposition}
\newtheorem*{cor*}{Corollary}
\newtheorem*{conj*}{Conjecture}

\theoremstyle{definition}
\newtheorem{definition}[thm]{Definition}

\theoremstyle{remark}
\newtheorem{rmk}[thm]{Remark}

\usepackage{chngcntr}
\counterwithin{figure}{section}

\newcommand{\mc}{\mathcal}

\newcommand{\cC}{\mc C}

\newcommand{\cH}{\mc H}
\newcommand{\cI}{\mc I}

\newcommand{\cL}{\mc L}
\newcommand{\cM}{\mc M}

\newcommand{\cO}{\mc O}
\newcommand{\cP}{\mc P}
\newcommand{\cQ}{\mc Q}
\newcommand{\cR}{\mc R}
\newcommand{\cS}{\mc S}

\newcommand{\cU}{\mc U}
\newcommand{\cV}{\mc V}

\newcommand{\cX}{\mc X}
\newcommand{\cY}{\mc Y}
\newcommand{\cZ}{\mc Z}

\newcommand{\ms}{\mathscr}

\newcommand{\sC}{\ms C}
\newcommand{\sD}{\ms D}
\newcommand{\sDsupp}{\dot{\ms D}}

\newcommand{\sR}{\ms R}

\newcommand{\C}{\mathbb{C}}
\newcommand{\N}{\mathbb{N}}
\newcommand{\R}{\mathbb{R}}
\newcommand{\Z}{\mathbb{Z}}
\newcommand{\Sph}{\mathbb{S}}
\newcommand{\scal}{\mathbf{S}}
\newcommand{\vect}{\mathbf{V}}

\newcommand{\TT}{\mathbb{T}}

\newcommand{\sfm}{\mathsf{m}}

\newcommand{\sfG}{\mathsf{G}}

\newcommand{\sfM}{\mathsf{M}}
\newcommand{\sfX}{\mathsf{X}}

\newcommand{\bfzero}{\mathbf{0}}
\newcommand{\bfa}{\mathbf{a}}
\newcommand{\bfc}{\mathbf{c}}
\newcommand{\bfA}{\mathbf{A}}
\newcommand{\bfB}{\mathbf{B}}
\newcommand{\bfX}{\mathbf{X}}

\newcommand{\frakt}{\mathfrak{t}}

\newcommand{\sld}{\slashed{d}{}}
\newcommand{\slg}{\slashed{g}{}}
\newcommand{\sln}{\slashed{n}{}}
\newcommand{\slG}{\slashed{G}{}}

\newcommand{\sldelta}{\slashed{\delta}{}}
\newcommand{\slDelta}{\slashed{\Delta}{}}
\newcommand{\slnabla}{\slashed{\nabla}{}}

\newcommand{\sltr}{\slashed{\tr}}

\newcommand{\ran}{\operatorname{ran}}

\newcommand{\End}{\operatorname{End}}

\newcommand{\Hom}{\operatorname{Hom}}
\renewcommand{\Re}{\operatorname{Re}}
\renewcommand{\Im}{\operatorname{Im}}
\newcommand{\Id}{\operatorname{Id}}

\newcommand{\mathspan}{\operatorname{span}}
\newcommand{\supp}{\operatorname{supp}}

\newcommand{\eucl}{\mathrm{eucl}}

\newcommand{\tr}{\operatorname{tr}}

\newcommand{\weak}{\mathrm{weak}}
\newcommand{\op}{\mathrm{op}}

\newcommand{\ord}{\operatorname{ord}}
\newcommand{\rank}{\operatorname{rank}}

\newcommand{\diag}{\operatorname{diag}}
\newcommand{\Res}{\operatorname{Res}}
\newcommand{\res}{\operatorname{res}}

\newcommand{\la}{\langle}
\newcommand{\ra}{\rangle}
\newcommand{\pa}{\partial}
\newcommand{\tn}{\textnormal}
\newcommand{\ff}{\tn{ff}}
\newcommand{\eps}{\epsilon}
\newcommand{\Ups}{\Upsilon}

\newcommand{\xra}{\xrightarrow}
\newcommand{\wt}{\widetilde}
\newcommand{\wh}{\widehat}
\newcommand{\ol}{\overline}
\newcommand{\ul}{\underline}

\newcommand{\ftrans}{\;\!\wh{\ }\;\!}

\newcommand{\IVP}{\mathrm{IVP}}
\newcommand{\AS}{\mathrm{AS}}
\newcommand{\hra}{\hookrightarrow}
\newcommand{\spec}{\operatorname{spec}}

\newcommand{\bop}{{\mathrm{b}}}
\newcommand{\bl}{{\mathrm{b}}}

\newcommand{\cp}{{\mathrm{c}}}

\newcommand{\cl}{{\mathrm{cl}}}

\newcommand{\semi}{\hbar}

\newcommand{\Diff}{\mathrm{Diff}}

\newcommand{\Vf}{\mathcal V}

\DeclareMathOperator{\Op}{Op}
\newcommand{\Oph}{\Op_\semi}
\newcommand{\Vb}{\Vf_\bop}
\newcommand{\Diffb}{\Diff_\bop}

\newcommand{\Psib}{\Psi_\bop}

\newcommand{\Diffbh}{\Diff_{\bop,\semi}}
\newcommand{\Psibh}{\Psi_{\bop,\semi}}

\newcommand{\Psih}{\Psi_\semi}
\newcommand{\Psihcl}{\Psi_{\semi,\mathrm{cl}}}
\newcommand{\Diffh}{\Diff_\semi}

\newcommand{\WF}{\mathrm{WF}}
\newcommand{\Ell}{\mathrm{Ell}}
\newcommand{\Char}{\mathrm{Char}}

\newcommand{\Shom}{S_{\mathrm{hom}}}

\newcommand{\WFb}{\WF_{\bop}}

\newcommand{\Omegab}{{}^{\bop}\Omega}
\newcommand{\Lambdab}{{}^{\bop}\Lambda}
\newcommand{\Tb}{{}^{\bop}T}
\newcommand{\Tzero}{{}^0T}
\newcommand{\Tzerodual}{{}^0T^*}

\newcommand{\Sb}{{}^{\bop}S}
\newcommand{\Nb}{{}^{\bop}N}

\newcommand{\bdiff}{{}^{\bop}d}

\newcommand{\WFbh}{\WF_{\bop,\semi}}
\newcommand{\Ellbh}{\Ell_{\bop,\semi}}

\newcommand{\half}{\frac{1}{2}}

\newcommand{\sigmabh}{\sigma_{\bop,\semi}}
\newcommand{\sigmabhfull}{\sigma_{\bop,\semi}^{\mathrm{full}}}

\newcommand{\ham}{H}
\newcommand{\rham}{{\mathsf{H}}}
\newcommand{\sub}{{\mathrm{sub}}}
\newcommand{\numin}{\nu_{\mathrm{min}}}

\newcommand{\bhm}[1][]{\ensuremath{M_{\bullet\ifthenelse{\isempty{#1}}{}{,#1}}}}

\newcommand{\loc}{{\mathrm{loc}}}
\newcommand{\CI}{\cC^\infty}
\newcommand{\CIdot}{\dot\cC^\infty}
\newcommand{\CIdotc}{\CIdot_\cp}
\newcommand{\CIc}{\cC^\infty_\cp}
\newcommand{\CmI}{\cC^{-\infty}}

\newcommand{\Lb}{L_{\bop}}
\newcommand{\Hb}{H_{\bop}}

\newcommand{\Hbext}{\bar H_{\bop}}
\newcommand{\Hbsupp}{\dot H_{\bop}}
\newcommand{\Hext}{\bar H}
\newcommand{\Hextloc}{\bar H_\loc}
\newcommand{\Hsupp}{\dot H}
\newcommand{\Hsupploc}{\dot H_\loc}
\newcommand{\Hbh}{H_{\bop,\semi}}

\newcommand{\DT}{{\mathrm{DT}}}

\newcommand{\Rhalf}[1]{\R^{#1}_+}
\newcommand{\Rhalfc}[1]{\overline{\R^{#1}_+}}
\newcommand{\fw}{{\mathrm{fw}}}
\newcommand{\bw}{{\mathrm{bw}}}

\newcommand{\tdel}{\wt{\delta}}

\newcommand{\Ric}{\mathrm{Ric}}
\newcommand{\Ein}{\mathrm{Ein}}

\newcommand{\openbigpmatrix}[1]
  {%
    \def\@bigpmatrixsize{#1}%
    \addtolength{\arraycolsep}{-#1}%
    \begin{pmatrix}%
  }
\newcommand{\closebigpmatrix}
  {%
    \end{pmatrix}%
    \addtolength{\arraycolsep}{\@bigpmatrixsize}%
  }

\newcommand{\BoxCP}{\Box^{\mathrm{CP}}}
\newcommand{\wtBoxCP}{\wt\Box^{\mathrm{CP}}}

\newcommand{\BoxGauge}{\Box^\Ups}

\DeclareGraphicsExtensions{.mps}
\newcommand{\inclfig}[1]{\includegraphics{#1}}

\begin{document}

\title[Non-linear stability of Kerr--de~Sitter]{The global non-linear stability of the Kerr--de~Sitter family of black holes}

\author{Peter Hintz}
\address{Department of Mathematics, University of California, Berkeley, CA 94720-3840, USA}
\email{phintz@berkeley.edu}

\author{Andr\'as Vasy}
\address{Department of Mathematics, Stanford University, CA 94305-2125, USA}
\email{andras@math.stanford.edu}

\date{June 14, 2016. Final revision: February 11, 2018.}

\subjclass[2010]{Primary 83C57, Secondary 83C05, 35B40, 58J47, 83C35}
\keywords{Einstein's equation, black hole stability, constraint damping, global iteration, gauge modification, Nash--Moser iteration, microlocal analysis}

\begin{abstract}
  We establish the full global non-linear stability of the Kerr--de~Sitter family of black holes, as solutions of the initial value problem for the Einstein vacuum equations with positive cosmological constant, for small angular momenta, and without any symmetry assumptions on the initial data. We achieve this by extending the linear and non-linear analysis on black hole spacetimes described in a sequence of earlier papers by the authors: we develop a general framework which enables us to deal systematically with the diffeomorphism invariance of Einstein's equations. In particular, the iteration scheme used to solve Einstein's equations automatically finds the parameters of the Kerr--de~Sitter black hole that the solution is asymptotic to, the exponentially decaying tail of the solution, and the gauge in which we are able to find the solution; the gauge here is a wave map/DeTurck type gauge, modified by source terms which are treated as unknowns, lying in a suitable finite-dimensional space.
\end{abstract}

\maketitle

\tableofcontents

\section{Introduction}
\label{SecIntro}

According to Einstein's theory of General Relativity, a vacuum spacetime with cosmological constant $\Lambda\in\R$ is a $(3+1)$-dimensional manifold $M$ equipped with a Lorentzian metric $g$ satisfying the Einstein vacuum equation $\Ein(g)=\Lambda g$, where $\Ein(g)=\Ric(g)-\half R_g g$ is the Einstein tensor. An equivalent formulation of this equation is
\begin{equation}
\label{EqIntroEinstein}
  \Ric(g) + \Lambda g = 0.
\end{equation}
A Kerr--de~Sitter spacetime, discovered by Kerr \cite{KerrKerr} and Carter \cite{CarterHamiltonJacobiEinstein}, models a stationary, rotating black hole within a universe with $\Lambda>0$: far from the black hole, the spacetime behaves like de~Sitter space with cosmological constant $\Lambda$, and close to the event horizon of the black hole like a Kerr black hole. Fixing $\Lambda>0$, a $(3+1)$-dimensional Kerr--de~Sitter spacetime $(M^\circ,g_b)$ depends, up to diffeomorphism equivalence, on two real parameters, namely the mass $\bhm>0$ of the black hole and its angular momentum $a$. For our purposes it is in fact better to consider the angular momentum as a vector $\bfa\in\R^3$. The Kerr--de~Sitter family of black holes is then a smooth family $g_b$ of stationary Lorentzian metrics, parameterized by $b=(\bhm,\bfa)$, on a fixed 4-dimensional manifold $M^\circ\cong\R_{t_*}\times(0,\infty)_r\times\Sph^2$ solving the equation~\eqref{EqIntroEinstein}. The Schwarzschild--de~Sitter family is the subfamily $(M^\circ,g_b)$, $b=(\bhm,\bfzero)$, of the Kerr--de~Sitter family; a Schwarzschild--de~Sitter black hole describes a static, non-rotating black hole. We point out that according to the currently accepted $\Lambda$CDM model, the cosmological constant is indeed positive in our universe \cite{RiessEtAlLambda,PerlmutterEtAlLambda}.

The equation \eqref{EqIntroEinstein} is a non-linear second order partial differential equation (PDE) for the metric tensor $g$. Due to the diffeomorphism invariance of this equation, the formulation of a well-posed initial value problem is more subtle than for (non-linear) wave equations. This was first accomplished by Choquet-Bruhat \cite{ChoquetBruhatLocalEinstein}, who with Geroch \cite{ChoquetBruhatGerochMGHD} proved the existence of maximal globally hyperbolic developments for sufficiently smooth initial data. We will discuss such formulations in detail later in this introduction as well as in \S\ref{SecHyp}. The correct notion of initial data is a triple $(\Sigma_0,h,k)$, consisting of
\begin{itemize}
\item a $3$-manifold $\Sigma_0$,
\item a Riemannian metric $h$ on $\Sigma_0$,
\item a symmetric 2-tensor $k$ on $\Sigma_0$,
\end{itemize}
subject to the \emph{constraint equations}, which are the Gauss--Codazzi equations on $\Sigma_0$ implied by \eqref{EqIntroEinstein}. Fixing $\Sigma_0$ as a submanifold of $M^\circ$, a metric $g$ satisfying \eqref{EqIntroEinstein} is then said to solve the initial value problem with data $(\Sigma_0,h,k)$ if
\begin{itemize}
\item $\Sigma_0$ is spacelike with respect to $g$;
\item $h$ is the Riemannian metric on $\Sigma_0$ induced by $g$;
\item $k$ is the second fundamental form of $\Sigma_0$ within $M^\circ$.
\end{itemize}
Our main result concerns the \emph{global non-linear asymptotic stability of the Kerr--de~Sitter family as solutions of the initial value problem for \eqref{EqIntroEinstein}}; we prove this for slowly rotating black holes, i.e.\ near $a=0$. To state the result in the simplest form, let us fix a Schwarzschild--de~Sitter spacetime $(M^\circ,g_{b_0})$, and within it a compact spacelike hypersurface $\Sigma_0\subset\{t_*=0\}\subset M^\circ$ extending slightly beyond the event horizon $r=r_-$ and the cosmological horizon $r=r_+$; let $(h_{b_0},k_{b_0})$ be the initial data on $\Sigma_0$ induced by $g_{b_0}$. Denote by $\Sigma_{t_*}$ the translates of $\Sigma_0$ along the flow of $\pa_{t_*}$, and let $\Omega^\circ=\bigcup_{t_*\geq 0}\Sigma_{t_*} \subset M^\circ$ be the spacetime region swept out by these; see Figure~\ref{FigIntroBaby}. Note that since we only consider slow rotation speeds, it suffices to consider perturbations of Schwarzschild--de~Sitter initial data, which in particular includes slowly rotating Kerr--de~Sitter black holes initial data (and their perturbations).

\begin{thm}[Stability of the Kerr--de~Sitter family for small $a$; informal version]
\label{ThmIntroBaby}
  Suppose $(h,k)$ are smooth initial data on $\Sigma_0$, satisfying the constraint equations, which are close to the data $(h_{b_0},k_{b_0})$ of a Schwarzschild--de~Sitter spacetime in a high regularity norm. Then there exist a solution $g$ of \eqref{EqIntroEinstein} attaining these initial data at $\Sigma_0$, and black hole parameters $b$ which are close to $b_0$, so that
  \[
    g - g_b = \cO(e^{-\alpha t_*})
  \]
  for a constant $\alpha>0$ independent of the initial data; that is, $g$ decays exponentially fast to the Kerr--de~Sitter metric $g_b$. Moreover, $g$ and $b$ are quantitatively controlled by $(h,k)$.
\end{thm}

\begin{figure}[!ht]
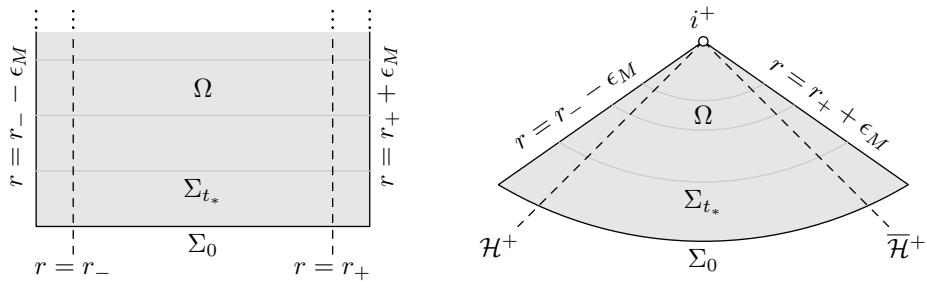

  \centering
  \inclfig{IntroBaby}
  \caption{Setup for the initial value problem for perturbations of a Schwarzschild--de~Sitter spacetime $(M^\circ,g_{b_0})$, showing the Cauchy surface $\Sigma_0$ of $\Omega$ and a few translates $\Sigma_{t_*}$; here $\eps_M>0$ is small. \emph{Left:} Product-type picture, illustrating the stationary nature of $g_{b_0}$. \emph{Right:} Penrose diagram of the same setup. The event horizon is $\cH^+=\{r=r_-\}$, the cosmological horizon is $\ol\cH{}^+=\{r=r_+\}$, and the (idealized) future timelike infinity is $i^+$.}
  \label{FigIntroBaby}
\end{figure}

In particular, \emph{we do not require any symmetry assumptions on the initial data}. We refer to Theorem~\ref{ThmIntroPrecise} for a more precise version of the theorem. Above, we measure the pointwise size of tensors on $\Sigma_0$ by means of the Riemannian metric $h_{b_0}$, and the pointwise size of tensors on the spacetime $M^\circ$ by means of a fixed smooth stationary Riemannian metric $g_R$ on $M^\circ$. The norms we use for $(h-h_{b_0},k-k_{b_0})$ on $\Sigma_0$ and of $g-g_b$ on $\Sigma_{t_*}$ are then high regularity Sobolev norms; any two choices of $g_R$ yield equivalent norms. If $(h,k)$ are smooth and sufficiently close to $(h_{b_0},k_{b_0})$ in a fixed high regularity norm, the solution $g$ we obtain is smooth as well, and in a suitable Fr\'echet space of smooth symmetric 2-tensors on $M^\circ$ depends smoothly on $(h,k)$, as does $b$.

In terms of the maximal globally hyperbolic development (MGHD) of the initial data $(h,k)$, Theorem~\ref{ThmIntroBaby} states that the MGHD contains a subset isometric to $\Omega^\circ$ on which the metric decays at an exponential rate to $g_b$.

We stress that a \emph{single} member of the Kerr--de~Sitter family is \emph{not} stable: small perturbations of the initial data of, say, a Schwarzschild--de~Sitter black hole, will in general result in a solution which decays to a Kerr--de~Sitter metric with slightly different mass and non-zero angular momentum. We are not aware of any way by which one can determine the final black hole parameters $b$, but which does not require finding the global solution of the initial value problem. Investigating any such method, potentially via making a connection to different notions of mass on asymptotically hyperbolic manifolds \cite{WangAHMMass,ZhangAHMMass,ChruscielHerzlichAHMMass,ChenWangYauConserved}, would be a very interesting problem.

Earlier global non-linear stability results for the Einstein equation include Fried\-rich's work \cite{FriedrichStability} on the stability of $(3+1)$-dimensional de~Sitter space, the monumental proof by Christodoulou and Klainerman \cite{ChristodoulouKlainermanStability} of the stability of $(3+1)$-dimensional Minkowski space. Partial simplifications and extensions of these results were proved by Anderson \cite{AndersonStabilityEvenDS} on higher-dimensional de~Sitter spacetimes, Lindblad and Rodnianski \cite{LindbladRodnianskiGlobalExistence,LindbladRodnianskiGlobalStability} and Bieri--Zipser \cite{BieriZipserStability} on Minkowski space, further Ringstr\"om \cite{RingstromEinsteinScalarStability} for a general Einstein--scalar field system, as well as by Rodnianski and Speck \cite{RodnianskiSpeckEulerEinsteinDS} (and the related \cite{SpeckFluidExpanding}) on Friedmann--Lema\^{i}tre--Robertson--Walker spacetimes; see \S\ref{SubsecIntroPrev} for further references. Theorem~\ref{ThmIntroBaby} is the first result for the Einstein equation proving an \emph{orbital stability} statement (i.e.\ decay to a member of a family of spacetimes, rather than decay to the spacetime one is perturbing), and the flexibility of the techniques we use should allow for investigations of many further orbital stability questions. Natural examples are the non-linear stability of the Kerr--Newman--de~Sitter family of rotating and charged black holes as solutions of the coupled Einstein--Maxwell system,\footnote{Since the first version of this paper, this has been accomplished by the first author \cite{HintzKNdSStability}.} and the stability of higher-dimensional black holes.

The proof of Theorem~\ref{ThmIntroBaby} will be given in \S\ref{SecKdSStab}. It uses a generalized wave coordinate gauge adjusted `dynamically' (from infinity) by finite-dimensional gauge modifications. The key tool is the precise analysis of the linearized problem around a Schwarzschild--de~Sitter metric. We develop a robust framework that has powerful stability properties with respect to perturbations; we will describe its main ingredients, in particular the manner in which we adapt our choice of gauge, in \S\ref{SubsecIntroIdeas}. (As a by-product of our analysis, we obtain a very general finite-codimension solvability result for quasilinear wave equations on Kerr--de~Sitter spaces, see Appendix~\ref{SecQ}.) The restriction to small angular momenta in Theorem~\ref{ThmIntroBaby} is then due to the fact that the required algebra is straightforward for linear equations on a Schwarzschild--de~Sitter background, but gets rather complicated for non-zero angular momenta; we explain the main calculations one would have to check to extend our result to large angular momenta in Remark~\ref{RmkIntroNonObvious}. Our framework builds on a number of recent advances in the global geometric microlocal analysis of black hole spacetimes which we recall in \S\ref{SubsecIntroIdeas}; `traditional' energy estimates play a very minor role, and are essentially only used to deal with the Cauchy surface $\Sigma_0$ and the artificial boundaries at $r=r_\pm\pm\eps_M$ in Figure~\ref{FigIntroBaby}. For solving the non-linear problem, we use a Nash--Moser iteration scheme, which proceeds by solving a \emph{linear} equation \emph{globally} at each step and is thus rather different in character from bootstrap arguments. (See also the introduction of \cite{HintzVasySemilinear}.)

Our main theorem and the arguments involved in its proof allow for further conclusions regarding the phenomenon of \emph{ringdown}, the problem of \emph{black hole uniqueness}, and suggest a future path to a definitive resolution of Penrose's \emph{Strong Cosmic Censorship} conjecture for cosmological spacetimes; see \S\ref{SubsecIntroCsq} for more on this.

Using our methods, we give a direct proof of the linear stability of slowly rotating Kerr--de~Sitter spacetimes in \S\ref{SecKdSLStab} as a `warm-up,' illustrating the techniques developed in the preceding sections.

\begin{thm}
\label{ThmIntroBabyL}
  Fix a slowly rotating Kerr--de~Sitter spacetime $(M^\circ,g_b)$, and $\Sigma_0$ as above. Suppose $(h',k')$ is a pair of symmetric 2-tensors, with high regularity, solving the linearized constraint equations around $(h_b,k_b)$. Then there exist a solution $r$ of the linearized Einstein vacuum equation
  \[
    D_{g_b}(\Ric+\Lambda)(r)=0,
  \]
  attaining these initial data at $\Sigma_0$, and $b'\in\R^4$ such that
  \[
    r - \frac{d}{ds}g_{b+s b'}|_{s=0} = \cO(e^{-\alpha t_*})
  \]
  for a constant $\alpha>0$ independent of the initial data; that is, the gravitational perturbation $r$ decays exponentially fast to a linearized Kerr--de~Sitter metric.
\end{thm}

We stress that the non-linear stability is only slightly more complicated to prove than the linear stability, given the robust framework we set up in this paper; we explain this in the discussion leading up to the statement of Theorem~\ref{ThmIntroPrecise}.

Theorem~\ref{ThmIntroBabyL} is the analogue of the recent result of Dafermos, Holzegel and Rodnianski \cite{DafermosHolzegelRodnianskiSchwarzschildStability} on the linear stability of the Schwarzschild spacetime (i.e.\ with $\Lambda=0$ and $a=0$). We will discuss the differences and similarities of their paper (and related works) with the present paper in some detail below.

We point out that the Ricci-flat analogues of Theorems~\ref{ThmIntroBaby} and \ref{ThmIntroBabyL}, i.e.\ with cosmological constant $\Lambda=0$, remain very interesting and challenging problems to study: the limit $\Lambda\to 0+$ is rather degenerate in that it replaces an asymptotically hyperbolic problem (far away from the black hole) by an asymptotically Euclidean one, which in particular drastically affects the low frequency behavior of the problem and thus the expected decay rates (polynomial rather than exponential). Furthermore, for $\Lambda=0$, one needs to use an additional `null-structure' of the non-linearity to analyze non-linear interactions near the light cone at infinity (`null infinity'), while this is not needed for $\Lambda>0$. See \S\ref{SubsecIntroPrev} for references and further discussion.

We take this opportunity to comment on the role of the \emph{small} positive cosmological constant and black holes in an astronomical context. While on the spatial scales relevant for the study of isolated gravitational systems the cosmological constant is negligible, no matter what its value is (or even what sign it has), the positivity of $\Lambda$ is important on the \emph{infinite} time scale on which one must work when studying stability questions. Put differently, $\Lambda$ is negligible for the large, but \emph{finite} time scale corresponding to the spatial scale on which it is negligible; thus, for large times (which in particular covers computations in numerical general relativity), $\Lambda$ can be ignored. The idealized case $\Lambda=0$ is nonetheless very interesting not only from a mathematical but also from a practical point of view, as it allows for the clean definition of quantities of physical (and experimental) interest such as null infinity, gravitational wave energy, the Christodoulou memory effect, etc.

\subsection{Main ideas of the proof}
\label{SubsecIntroIdeas}

For the reader unfamiliar with Einstein's equations, we begin by describing some of the fundamental difficulties one faces when studying the equation \eqref{EqIntroEinstein}. Typically, PDEs have many solutions; one can specify additional data. For instance, for hyperbolic second order PDEs such as the wave equation, say on a closed manifold cross time, one can specify Cauchy data, i.e.\ a pair of data corresponding to the initial amplitude and momentum of the wave. The question of stability for solutions of such a PDE is then whether for small perturbations of the additional data the solution still exists and is close to the original solution. Of course, this depends on the region on which we intend to solve the PDE, and more precisely on the function spaces we solve the PDE in. Typically, for an evolution equation like the wave equation, one has short time solvability and stability, sometimes global in time existence, and then stability is understood as a statement that globally the solution is close to the unperturbed solution, and indeed sometimes in the stronger sense of being asymptotic to it as time tends to infinity.

The Einstein equation is closest in nature to hyperbolic equations. Thus, the stability question for the Einstein equation (with fixed $\Lambda$) is whether when one perturbs `Cauchy data,' the solutions are globally `close,' and possibly even asymptotic to the unperturbed solution. However, since the equations are not hyperbolic, one needs to be careful by what one means by `the' solution, `closeness' and `Cauchy data.' Concretely, the root of the lack of hyperbolicity of \eqref{EqIntroEinstein} is the diffeomorphism invariance: if $\phi$ is a diffeomorphism that is the identity map near an initial hypersurface, and if $g$ solves Einstein's equations, then so does $\phi^*g$, with the same initial data. Thus, one cannot expect uniqueness of solutions without fixing this diffeomorphism invariance; by duality one cannot expect solvability for arbitrary Cauchy data either.

It turns out that there are hyperbolic formulations of Einstein's equations; these formulations break the diffeomorphism invariance by requiring more than merely solving Einstein's equations. A way to achieve this is to require that one works in coordinates which themselves solve wave equations, as in the pioneering work \cite{ChoquetBruhatLocalEinstein} and also used in \cite{LindbladRodnianskiGlobalExistence,LindbladRodnianskiGlobalStability}; very general hyperbolic formulations of Einstein's equations, where the wave equations may in particular have (fixed) source terms, were worked out in \cite{FriedrichHyperbolicityEinstein}. More sophisticated than the source-free wave coordinate gauge, and more geometric in nature, is DeTurck's method \cite{DeTurckPrescribedRicci}, see also the paper by Graham--Lee \cite{GrahamLeeConformalEinstein}, which fixes a background metric $g^0$ and requires that the identity map be a harmonic map from $(M^\circ,g)$ to $(M^\circ,g^0)$, where $g$ is the solution we are seeking. This can be achieved by considering a PDE that differs from the Einstein equation due to the presence of an extra \emph{gauge-fixing} term:
\begin{equation}
\label{EqIntroDeTurck}
  \Ric(g)+\Lambda g-\Phi(g,g^0)=0.
\end{equation}
We call this the \emph{gauged Einstein equation} (sometimes also called \emph{reduced Einstein equation}). For suitable $\Phi$, discussed below, this equation is actually hyperbolic. One then shows that one can construct Cauchy data for this equation from the geometric initial data $(\Sigma_0,h,k)$ (with $\Sigma_0\hra M^\circ$) so that $\Phi$ vanishes at first on $\Sigma_0$, and then identically on the domain of dependence of $\Sigma_0$. Thus, one has a solution of Einstein's equations as well, in the gauge $\Phi(g,g^0)=0$.

Concretely, fixing a metric $g^0$ on $M^\circ$, the DeTurck gauge (or wave map gauge) takes the form
\begin{equation}
\label{EqIntroDeTurckExtra}
  \Phi(g,g^0)=\delta_g^*\Ups(g),\quad \Ups(g)=g(g^0)^{-1}\delta_g \sfG_g g^0,
\end{equation}
where $\delta_g^*$ is the symmetric gradient relative to $g$, $\delta_g$ is its adjoint (divergence), and $\sfG_g r:=r-\frac{1}{2}(\tr_g r)g$ is the trace reversal operator. Here $\Ups(g)$ is the \emph{gauge one form}. One typically chooses $g^0$ to be a metric near which one wishes to show stability, and in the setting of Theorem~\ref{ThmIntroBaby} we will in fact take $g^0=g_{b_0}$; thus $\Ups(g_{b_0})$ vanishes. Given initial data satisfying the constraint equations, one then constructs Cauchy data for $g$ in \eqref{EqIntroDeTurck} giving rise to the given initial data and moreover solving $\Ups(g)=0$ (note that $\Ups(g)$ is a \emph{first order} non-linear differential operator) at $\Sigma_0$. Solving the gauged Einstein equation \eqref{EqIntroDeTurck} and then using the constraint equations, the normal derivative of $\Ups(g)$ at $\Sigma_0$ also vanishes. Then, applying $\delta_g \sfG_g$ to \eqref{EqIntroDeTurck} shows in view of the second Bianchi identity that $\delta_g \sfG_g\delta_g^*\Ups(g)=0$. Since
\[
  \BoxCP_g:=2\delta_g \sfG_g\delta_g^*
\]
is a wave operator, this shows that $\Ups(g)$ vanishes identically. (See \S\ref{SecHyp} for more details.)

The specific choice of gauge, i.e.\ in this case the choice of background metric $g^0$, is irrelevant for the purpose of establishing the short-time existence of solutions of the Einstein equation. Indeed, if one fixes an initial data set but chooses two different background metrics $g^0$, then solving the resulting two versions of \eqref{EqIntroDeTurck} will produce two generally different symmetric 2-tensors $g$ on $M^\circ$ which attain the given data and solve the Einstein equation; however, they are always related by a diffeomorphism, i.e.\ one is the pullback of the other by a diffeomorphism. On the other hand, \emph{global} (or even just long time) \emph{existence} of solutions of the Einstein equation by means of hyperbolic formulations like \eqref{EqIntroDeTurck} depends very sensitively on the choice of gauge, as does the asymptotic behavior of global solutions. (Global existence and large time asymptotics are strongly coupled, as one can usually only obtain the former when one has a precise understanding of the latter.) Thus, finding a suitable gauge is the fundamental problem in the study of \eqref{EqIntroEinstein}, which we overcome in the present paper in the setting of Theorem~\ref{ThmIntroBaby}.

Let us now proceed to discuss \eqref{EqIntroDeTurck} from the perspective most useful for the present paper. (For a general overview of the large body of work on Kerr and Kerr--de~Sitter spacetimes in the last decades, we refer the reader to~\S\ref{SubsecIntroPrev}.) The key advances in understanding hyperbolic equations globally on a background like Kerr--de~Sitter space were the paper \cite{VasyMicroKerrdS} by the second author, where microlocal tools were introduced and used to provide a Fredholm framework for global non-elliptic analysis; for waves on Kerr--de~Sitter spacetimes, this uses the microlocal analysis at the trapped set of Wunsch--Zworski \cite{WunschZworskiNormHypResolvent} and Dyatlov \cite{DyatlovSpectralGaps} as an external input. In the paper \cite{HintzVasySemilinear}, the techniques of \cite{VasyMicroKerrdS} were extended to non-stationary settings, using the framework of Melrose's b-pseudodifferential operators, and shown to apply to semi-linear equations; in \cite{HintzQuasilinearDS}, the techniques were extended to quasilinear equations on de~Sitter-like spaces by introducing operators with non-smooth (high regularity b-Sobolev) coefficients. In \cite{HintzVasyQuasilinearKdS}, the additional difficulty of trapping in Kerr--de~Sitter space was overcome by Nash--Moser iteration based techniques, using the simple formulation of the Nash--Moser iteration scheme given by Saint-Raymond \cite{SaintRaymondNashMoser}.

The papers \cite{VasyMicroKerrdS,HintzVasySemilinear,HintzQuasilinearDS,HintzVasyQuasilinearKdS} show global existence and asymptotic stability of solutions to (systems of) second order quasilinear PDEs under the following conditions:
\begin{enumerate}
\item the second order terms in the PDE are given by the wave equation  with respect to a metric $g$ which depends on the unknown (vector-valued) function $u$ as well as on the gradient $\nabla u$;
\item for $u\equiv 0$, the metric $g$ is a Kerr--de~Sitter metric (one can also handle small stationary perturbations of the Kerr--de~Sitter metric, though this is not used in the present paper);
\item\label{ItIntroCondTrapped} the linearized operator at $u=0$, which we denote by $L$, satisfies a subprincipal sign condition on the trapped set;
\item\label{ItIntroCondRes} the equation $L v=0$ has no \emph{resonances} $\sigma\in\C$ in the closed upper half plane, that is, there do not exist \emph{mode solutions} of the form $v(t_*,x)=e^{-i\sigma t_*}w(x)$ with $\Im\sigma\geq 0$.
\end{enumerate}

Conditions \eqref{ItIntroCondTrapped} and \eqref{ItIntroCondRes} rule out ways in which the \emph{linearized} equation could fail to be asymptotically stable. More precisely, \eqref{ItIntroCondTrapped} rules out growing high frequency solutions localized near the trapped set, and \eqref{ItIntroCondRes} rules out a bounded frequency solution which is either oscillating or exponentially growing in time.  The result of the above series of papers is to show that these are the only two ways in which stability of the \emph{nonlinear} equation could fail. The absence of resonances in the closed upper half plane is clearly a desirable condition when solving non-linear equations. That it is sufficient, given the other assumptions, relies on the fact that solutions of $L v=0$ (with $L$ not necessarily satisfying assumption~\eqref{ItIntroCondRes}) with smooth initial data and (compactly supported) forcing admit partial asymptotic expansions into a finite sum of mode solutions with frequencies $\sigma_j\in\C$, plus a remainder term with decay $e^{-\alpha t_*}$, $\alpha>-\Im\sigma_j$, \emph{provided} the operator $L$ satisfies suitable estimates. These are so-called \emph{high energy estimates} for the Mellin-transformed normal operator family $\wh L(\sigma)w(x):=e^{i\sigma t_*}L(e^{-i\sigma t_*}w)|_{t_*=0}$, namely, polynomial bounds on the operator norm of $\wh L(\sigma)^{-1}$ as $|\Re\sigma|\to\infty$, $\Im\sigma\geq-\alpha$, acting between suitable Sobolev spaces. The main result of \cite{VasyMicroKerrdS} is that $\wh L(\sigma)^{-1}$ is a meromorphic family of operators satisfying such estimates; the poles of this family are precisely the resonances. (This result strongly uses the asymptotically \emph{hyperbolic} nature of the de~Sitter end of Kerr--de~Sitter spacetimes.) Provided that high energy estimates hold for $\alpha>0$ (which requires assumption~\eqref{ItIntroCondTrapped}), this implies that solutions of $L v=0$ are a finite sum of mode solutions and an \emph{exponentially decaying} remainder term; the asymptotic expansion here is proved by means of a contour shifting argument on the Fourier-transform side; see \S\ref{SubsubsecAsyExpReg}.

Now, for the DeTurck-gauged Einstein equation \eqref{EqIntroDeTurck}, where the unknown function is the metric $g$ itself, the condition~\eqref{ItIntroCondTrapped} at the trapped set holds, but the condition~\eqref{ItIntroCondRes} on resonances \emph{does not}. That is, non-decaying ($\Im\sigma\geq 0$) and even exponentially growing modes ($\Im\sigma>0$) exist. In fact, even on de~Sitter space, the linearized DeTurck-gauged Einstein equation, with $g^0$ equal to the de~Sitter metric, has \emph{exponentially growing} modes, cf.\ the indicial root computation of \cite{GrahamLeeConformalEinstein} on hyperbolic space---the same computation also works under the metric signature change in de~Sitter space, see Appendix~\ref{SecdS}. Thus, the key achievement of this paper is to provide a precise understanding of the nature of the resonances in the upper half plane, and how to overcome their presence. We remark that the first condition, at the trapping, ensures that there is at most a \emph{finite-dimensional} space of non-decaying mode solutions.

We point out that all conceptual difficulties in the study of \eqref{EqIntroDeTurck} (beyond the difficulties overcome in the papers mentioned above) are already present in the simpler case of the static model of de~Sitter space, with the exception of the presence of a non-trivial \emph{family} of stationary solutions in the black hole case; in fact, what happens on de~Sitter space served as a very useful guide to understanding the equations on Kerr--de~Sitter space. Thus, in Appendix~\ref{SecdS}, we illustrate our approach to the resolution of the black hole stability problem by re-proving the non-linear stability of the static model of de~Sitter space.

In general, if one has a given non-linear hyperbolic equation whose linearization around a fixed solution has growing modes, there is not much one can do: at the linearized level, one then gets growing solutions; substituting such solutions into the non-linearity gives even more growth, resulting in the breakdown of the local non-linear solution. A typical example is the ODE $u'=u^2$, with initial condition at $0$: the function $u\equiv 0$ solves this, and given any interval $[0,T]$ one has stability (changing the initial data slightly), but for any non-zero positive initial condition, regardless how small, the solution blows up at finite time, thus there is no stability on $[0,\infty)$. Here the linearized operator is just the derivative $v\mapsto v'$, which has a non-decaying mode $1$ (with frequency $\sigma=0$). This illustrates that even non-decaying modes, not only growing ones, are dangerous for stability; thus they should be considered borderline unstable for non-linear analysis, rather than borderline stable, unless the operator has a special structure. Now, the linearization of the Kerr--de~Sitter family around a fixed member of this family gives rise (up to infinitesimal diffeomorphisms, i.e.\ Lie derivatives) to zero resonant states of the linearization of \eqref{EqIntroDeTurck} around this member; but these will of course correspond to the (non-linear) Kerr--de~Sitter solution when solving the quasilinear equation \eqref{EqIntroDeTurck}, and we describe this further below.

The primary reason one can overcome the presence of growing modes for the gauged Einstein equation \eqref{EqIntroDeTurck} is of course that the equation is not fixed: one can choose from a family of potential gauges; any gauge satisfying the above principally wave, asymptotically Kerr--de~Sitter, condition is a candidate, and one needs to check whether the two conditions stated above are satisfied. However, even with this gauge freedom we are unable to eliminate the resonances in the closed upper half plane, even ignoring the $0$ resonance which is unavoidable as discussed above. Even for Einstein's equations near de~Sitter space---where one knows that stability holds by \cite{FriedrichStability}---the best we can arrange in the context of modifications of DeTurck gauges is the absence of all non-decaying modes apart from a resonance at $0$, but it is quite delicate to see that this can in fact be arranged. Indeed, the arguments rely on the special asymptotic structure of de~Sitter space, which reduces the computation of resonances to finding indicial roots of regular singular ODEs, much like in the Riemannian work of Graham--Lee \cite{GrahamLeeConformalEinstein}; see Remark~\ref{RmkdSBAlwaysZero}.

While we do not have a modified DeTurck gauge in the Kerr--de~Sitter setting which satisfies all of our requirements, we can come part of the way in a crucial manner. Namely, if $L$ is the linearization of \eqref{EqIntroDeTurck}, with $g^0=g_{b_0}$, around $g=g_{b_0}$, then
\begin{equation}
\label{EqIntroEinsteinLin}
  L r = D_g(\Ric+\Lambda)(r) + \delta_g^*(\delta_g \sfG_g r);
\end{equation}
the second term here breaks the infinitesimal diffeomorphism invariance---if $r$ solves the linearized Einstein equation $D_g(\Ric+\Lambda)(r)=0$, then so does $r+\delta_g^*\omega$ for any 1-form $\omega$. Now suppose $\phi$ is a growing mode solution of $L\phi=0$. Without the gauge term present, $\phi$ would be a mode solution of the linearized Einstein equation, i.e.\ a growing gravitational wave; if non-linear stability is to have a chance of being true, such $\phi$ must be `unphysical,' that is, equal to $0$ up to gauge changes, i.e.\ $\phi=\delta_g^*\omega$. This statement is commonly called \emph{mode stability}; we introduce and prove the slightly stronger notion of `ungauged Einstein mode stability' (UEMS), including a precise description of the zero mode, in \S\ref{SecUEMS}. One may hope that even with the gauge term present, all growing mode solutions such as $\phi$ are pure gauge modes in this sense. To see what this affords us, consider a cutoff $\chi(t_*)$, identically $1$ for large times but $0$ near $\Sigma_0$. Then
\[
  L\bigl(\delta_g^*(\chi\omega)\bigr) = \delta_g^*\theta,\quad \theta=\delta_g \sfG_g\delta_g^*(\chi\omega);
\]
that is, we can generate the asymptotic behavior of $\phi$ by adding a source term $\delta_g^*\theta$---\emph{which is a pure gauge term}---to the right hand side; looking at this the other way around, we can eliminate the asymptotic behavior $\phi$ from any solution of $L r=0$ by adding a suitable multiple of $\delta_g^*\theta$ to the right hand side. In the non-linear equation \eqref{EqIntroDeTurck} then, using the form \eqref{EqIntroDeTurckExtra} of the gauge-fixing term, this suggests that we try to solve
\[
  \Ric(g)+\Lambda g - \delta_g^*(\Ups(g)-\theta) = 0,
\]
where $\theta$ lies in a fixed finite-dimensional space of compactly supported 1-forms corresponding to the growing pure gauge modes; that is, we solve the initial value problem for this equation, regarding the pair $(\theta,g)$ as our unknown. Solving this equation for fixed $\theta$, which one can do at least for short times, produces a solution of Einstein's equations in the gauge $\Ups(g)=\theta$, which in the language of \cite{FriedrichHyperbolicityEinstein} amounts to using non-trivial gauge source functions induced by 1-forms $\theta$ as above; in contrast to \cite{FriedrichHyperbolicityEinstein} however, \emph{we regard the gauge source functions as unknowns} (albeit in a merely finite-dimensional space) which we need to solve for.

Going even one step further, one can hope (or try to arrange) for all Kerr--de~Sitter metrics $g_b$ to satisfy the gauge condition $\Ups(g_b)=0$ (which of course depends on the concrete presentation of the metrics): then, we could incorporate the Kerr--de~Sitter metric that our solution $g$ decays to by adding another parameter $b$; that is, we could solve
\begin{equation}
\label{EqIntroIdeal}
  (\Ric+\Lambda)(g_b+\wt g) - \delta_{g_b+\wt g}^*(\Ups(g_b+\wt g)-\theta) = 0
\end{equation}
for the triple $(b,\theta,\wt g)$, with $\wt g$ now in a \emph{decaying} function space; a key fact here is that even though we constructed the 1-forms $\theta$ from studying the linearized equation around $g_{b_0}$, adding $\theta$ to the equation as done here also ensures that \emph{for nearby linearizations}, we can eliminate the growing asymptotic behavior corresponding to pure gauge resonances; a general version of this perturbation-type statement is the main result of \S\ref{SecAsy}. If both our hopes (regarding growing modes and the interaction of the Kerr--de~Sitter family with the gauge condition) proved to be well-founded, we could indeed solve \eqref{EqIntroIdeal} by appealing to a general quasilinear existence result, based on a Nash--Moser iteration scheme; this is an extension of the main result of \cite{HintzVasyQuasilinearKdS}, accommodating for the presence of the finite-dimensional variables $b$ and $\theta$. (We shall only state a simple version, ignoring the presence of the non-trivial stationary family encoded by the parameter $b$, of such a general result in Appendix~\ref{SecB}.)

This illustrates a central feature of our non-linear framework: \emph{The non-linear iteration scheme finds not only the suitable Kerr--de~Sitter metric the solution of \eqref{EqIntroEinstein} should converge to, but also the correct gauge modification $\theta$!}\footnote{For the initial value problem in the ODE example $u'=u^2$, one can eliminate the constant asymptotic behavior of the linearized equation by adding a suitable forcing term, or equivalently by modifying the initial data; thus, the non-linear framework would show the solvability of $u'=u^2$, $u(0)=u_0$, with $u_0$ small, up to modifying $u_0$ by a small quantity---and this is of course trivial, if one modifies it by $-u_0$! A more interesting example would be an ODE of the form $(\pa_x+1)\pa_x u=u^2$, solving near $u=0$; the decaying mode $e^{-x}$ of the linearized equation causes no problems, and the zero mode $1$ can be eliminated by modifying the initial data---which now lie in a $2$-dimensional space---by elements in a fixed $1$-dimensional space.}

Unfortunately, neither of these two hopes proves to be true for the stated hyperbolic version of the Einstein equation.

\emph{First}, consider the mode stability statement: we expect the presence of the gauge term in \eqref{EqIntroEinsteinLin} to cause growing modes which are not pure gauge modes (as can again easily be seen for the DeTurck gauge on de~Sitter space); in view of UEMS, they cannot be solutions of the linearized Einstein equation. When studying the problem of linear stability, such growing modes therefore cannot appear as the asymptotic behavior of a gravitational wave; in fact, one can argue, as we shall do in \S\ref{SecKdSLStab}, that the linearized constraint equations restrict the space of allowed asymptotics, ruling out growing modes which are not pure gauge. While such an argument is adequate for the linear stability problem, it is not clear how to extend it to the non-linear problem, since it is not at all robust; for example, it breaks down immediately if the initial data satisfy the \emph{non-linear} constraint equations, as is of course the case for the non-linear stability problem.

It turns out that the properties of $\BoxCP_g=2\delta_g \sfG_g\delta_g^*$, or rather a suitable replacement $\wtBoxCP_g$, are crucial for constructing an appropriate modification of the hyperbolic equation \eqref{EqIntroDeTurck}. Recall that $\BoxCP_g$ is the operator governing the propagation of the gauge condition $\Ups(g)-\theta=0$, or equivalently the propagation of the constraints. The key insight, which has been exploited before in the numerics literature \cite{GundlachCalabreseHinderMartinConstraintDamping,PretoriusBinaryBlackHole}, is that one can modify the gauged Einstein equation by additional terms, preserving its hyperbolic nature, to arrange for \emph{constraint damping}, which says that solutions of the correspondingly modified constraint propagation operator $\wtBoxCP_g$ \emph{decay exponentially}. Concretely, note that in $\BoxCP_g$, the part $\delta_g \sfG_g$ is firmly fixed since we need to use the Bianchi identity for this to play any role. However, we have flexibility regarding $\delta_g^*$ as long as we change it in a way that does not destroy at least the properties of our gauged Einstein equation that we already have, in particular the principal symbol. Now, the principal symbol of the linearization of $\Phi(g,g^0)$ depends on $\delta_g^*$ only via its principal symbol, which is independent of $g$, so we can replace $\delta_g^*$ by any, even $g$-independent, differential operator with the same principal symbol, for instance by considering
\[
  \tdel^*\omega=\delta_{g_0}^*\omega+\gamma_1\,dt_*\otimes_s\omega - \gamma_2 g_0 \tr_{g_0}(dt_*\otimes_s\omega),
\]
where $\gamma_1,\gamma_2$ are fixed real numbers. What we show in \S\ref{SecSCP} is that for $g_0$ being a Schwarzschild-de~Sitter metric ($a=0$), we can {\em choose} $\gamma_1,\gamma_2\gg 0$ so that for $g=g_0$, the operator $\wtBoxCP_g=2\delta_g \sfG_g \tdel^*$ has {\em no resonances} in the closed upper half plane, i.e.\ only has decaying modes. We call this property {\em stable constraint propagation} or SCP. Note that by a general feature of our analysis, this implies the analogous stability statement when $g$ is merely suitably close to $g_0$, in particular when it is asymptotic to a Kerr--de~Sitter metric with small $a$. Dropping the modifications by $\theta$ and $b$ considered above for brevity, the hyperbolic operator we will study is then $\Ric(g)+\Lambda g - \tdel^*\Ups(g)$.

The role of SCP is that it ensures that the resonances of the linearized gauged Einstein equation in the closed upper half plane (corresponding to non-decaying modes) are {\em either} resonances (modes) of the linearized ungauged Einstein operator $D(\Ric+\Lambda)$ {\em or} pure gauge modes, i.e.\ of the form $\delta_g^*\theta$ for some one-form $\theta$; indeed, granted UEMS, this is a simple consequence of the linearized second Bianchi identity applied to \eqref{EqIntroEinsteinLin} (with $\delta_g^*$ there replaced by $\tdel$).

\emph{Second}, we discuss the (in)compatibility issue of the Kerr--de~Sitter family with the wave map gauge when the background metric is fixed, say $g_{b_0}$. Putting the Kerr--de~Sitter metric $g_b$ into this gauge would require solving the wave map equation $\Box_{g_b,g_{b_0}}\phi=0$ globally, and then replacing $g_b$ by $\phi_*(g_b)$ (see Remark~\ref{RmkHypDTPutIntoGauge} for details); this can be rewritten as a semi-linear wave equation with stationary, non-decaying forcing term (essentially $\Ups(g_b)$), whose linearization around the identity map for $b=b_0$ has resonances at $0$ and, at least on de~Sitter space where this is easy to check, also in the upper half plane. While the growing modes can be eliminated by modifying the initial data of the wave map within a finite-dimensional space, the zero mode, corresponding to Killing vector fields of $g_{b_0}$, cannot be eliminated; in the above ODE example, one cannot solve $u'=u^2+1$, with $1$ being the stationary forcing term, globally if the only freedom one has is perturbing the initial data.

We remark that this difficulty does not appear in the double null gauge used e.g.\ in \cite{DafermosHolzegelRodnianskiSchwarzschildStability}; however, the double null gauge formulation of Einstein's equations does not fit into our general non-linear framework.

The simple way out is that one relaxes the gauge condition further: rather than demanding that $\Ups(g)-\theta=0$, we demand that $\Ups(g)-\Ups(g_b)-\theta=0$ near infinity if $g_b$ is the Kerr--de~Sitter metric that $g$ is decaying towards; recall here again that our non-linear iteration scheme finds $b$ (and $\theta$) automatically. Near $\Sigma_0$, one would like to use a fixed gauge condition, since otherwise one would need to use different Cauchy data, constructed from the same geometric initial data, at each step of the iteration, depending on the gauge at $\Sigma_0$. With a cutoff $\chi$ as above, we thus consider grafted metrics
\[
  g_{b_0,b} := (1-\chi)g_{b_0} + \chi g_b
\]
which interpolate between $g_{b_0}$ near $\Sigma_0$ and $g_b$ near future infinity.

\begin{rmk}
  We again stress that the two issues discussed above, SCP and the change of the asymptotic gauge condition, only arise in the non-linear problem. However, by the perturbative statement following \eqref{EqIntroIdeal}, SCP also allows us to deduce the linear stability of slowly rotating Kerr--de~Sitter spacetimes \emph{directly, by a simple perturbation argument}, from the linear stability of Schwarzschild--de~Sitter space; this is in contrast to the techniques used in \cite{DafermosHolzegelRodnianskiSchwarzschildStability} in the setting of $\Lambda=0$, which do not allow for such perturbation arguments off $a=0$.

  The linear stability of \emph{Schwarzschild}--de~Sitter spacetimes in turn is a direct consequence of the results of \cite{VasyMicroKerrdS} together with UEMS, proved in \S\ref{SecUEMS}, and the symbolic analysis at the trapped set of \S\ref{SubsecESGTrap} (which relies on \cite{HintzPsdoInner}); see Theorem~\ref{ThmKdSLStabNaive}. The rest of the bulk of the paper, including SCP, is needed to build the robust perturbation framework required for the proof of non-linear stability (and the linear stability of slowly rotating Kerr--de~Sitter spacetimes).
\end{rmk}

We can now state the precise version of Theorem~\ref{ThmIntroBaby} which we will prove in this paper:

\begin{thm}[Stability of the Kerr--de~Sitter family for small $a$; precise version]
\label{ThmIntroPrecise}
  Let $h,k\in\CI(\Sigma_0;S^2T^*\Sigma_0)$ be initial data satisfying the constraint equations, and suppose $h$ and $k$ are close to the Schwarzschild--de~Sitter initial data $(h_{b_0},k_{b_0})$ in the topology of $H^{21}(\Sigma_0;S^2T^*\Sigma_0)\oplus H^{20}(\Sigma_0;S^2T^*\Sigma_0)$. Then there exist Kerr--de~Sitter black hole parameters $b$ close to $b_0$, a compactly supported gauge modification $\theta\in\CIc(\Omega^\circ;T^*\Omega^\circ)$ (lying in a fixed finite-dimensional space $\ol\Theta$) and a symmetric 2-tensor $\wt g\in\CI(\Omega^\circ;S^2T^*\Omega^\circ)$, with $\wt g=\cO(e^{-\alpha t_*})$ together with all its stationary derivatives (here $\alpha>0$ independent of the initial data), such that the metric
  \[
    g = g_b + \wt g
  \]
  solves the Einstein equation
  \begin{equation}
  \label{EqIntroPreciseEinstein}
    \Ric(g)+\Lambda g=0
  \end{equation}
  in the gauge
  \begin{equation}
  \label{EqIntroPreciseGauge}
    \Ups(g) - \Ups(g_{b_0,b}) - \theta = 0,
  \end{equation}
  where we define $\Ups(g):=g g_{b_0}^{-1}\delta_g \sfG_g g_{b_0}$ (which is \eqref{EqIntroDeTurckExtra} with $g^0=g_{b_0}$), and with $g$ attaining the data $(h,k)$ at $\Sigma_0$.
\end{thm}

See Theorem~\ref{ThmKdSStab} for a slightly more natural description (in terms of function spaces) of $\wt g$. In order to minimize the necessary bookkeeping, we are very crude in describing the regularity of the coefficients, as well as the mapping properties, of various operators; thus, the number of derivatives used in this theorem is far from optimal. (With a bit more care, as in \cite{HintzVasyQuasilinearKdS}, it should be possible to show that $12$ derivatives are enough, and even this is still rather crude.)

As explained above, the finite-dimensional space $\ol\Theta$ of compactly supported gauge modifications appearing in the statement of Theorem~\ref{ThmIntroPrecise}, as well as its dimension, can be computed in principle: it would suffice to compute the non-decaying resonant states of the linearized, modified Einstein operator $D_{g_{b_0}}(\Ric+\Lambda)-\tdel^*D_{g_{b_0}}\Ups$. While we do not do this here, this can easily be done for the static de~Sitter metric, see Appendix~\ref{SecdS}.

The reader will have noticed the absence of $\tdel^*$ (or $\delta_g^*$) in the formulation of Theorem~\ref{ThmIntroPrecise}, and in fact at first sight $\tdel^*$ may seem to play no role: indeed, while the non-linear equation we solve takes the form
\begin{equation}
\label{EqIntroWhyTdel}
  \Ric(g)+\Lambda g - \tdel^*(\Ups(g)-\Ups(g_{b_0,b})-\theta)=0,
\end{equation}
the non-linear solution $g$ satisfies both the Einstein equation \eqref{EqIntroPreciseEinstein} and the gauge condition \eqref{EqIntroPreciseGauge}; therefore, the \emph{same} $g$ (not merely up to a diffeomorphism) also solves the same equation with $\delta_g^*$ in place of $\tdel^*$! But note that this is only true provided the initial data satisfy the constraint equations. We will show however that one can solve \eqref{EqIntroWhyTdel}, given \emph{any} Cauchy data, for $(b,\theta,g)$; and we only use the constraint equations for the initial data \emph{at the very end}, after having solved \eqref{EqIntroWhyTdel}, to conclude that we do have a solution of \eqref{EqIntroPreciseEinstein}. On the other hand, it is \emph{not} possible to solve \eqref{EqIntroWhyTdel} globally for arbitrary Cauchy data if one used $\delta_g^*$ instead of $\tdel^*$, since modifying the parameters $b$ and $\theta$ is then no longer sufficient to eliminate all non-decaying resonant states, the problematic ones of course being those which are not pure gauge modes. From this perspective, the introduction of $\tdel^*$ has the effect of making it unnecessary to worry about the constraint equations---which are rather delicate---being satisfied when solving the equation \eqref{EqIntroWhyTdel}, and thus paves the way for the application of the robust perturbative techniques developed in \S\ref{SecAsy}.

\begin{rmk}
\label{RmkIntroNonObvious}
  There are only three places where the result of the paper depends on a computation whose result is not a priori `obvious.' The first is UEMS itself in \S\ref{SecUEMS}; this is, on the one hand, well established in the physics literature, and on the other hand its failure would certainly doom the stability of the Kerr--de~Sitter family for small $a$. The second is the subprincipal symbol computation at the trapped set, in the settings of SCP in \S\ref{SubsecSCPHi} and for the linearized gauged Einstein equation in \S\ref{SubsecESGTrap}, which involves large (but finite!) dimensional linear algebra; its failure would break our analysis in the DeTurck-type gauge we are using, but would not exclude the possibility of proving Kerr--de~Sitter stability in another gauge. (By contrast, the failure of the radial point subprincipal symbol computation would at worst affect the threshold regularity $1/2$ in Theorem~\ref{ThmKeyESG}, and thus merely necessitate using slightly higher regularity than we currently use.) The third significant computation finally is that of the semiclassical subprincipal symbol at the zero section for SCP, in the form of Lemma~\ref{LemmaSCPLoRadialSubpr}, whose effect is similar to the subprincipal computation at the trapped set.

  These are also exactly the `non-obvious' computations to check if one wanted to extend Theorem~\ref{ThmIntroPrecise} to a larger range of angular momenta, i.e.\ allowing the initial data $h$ and $k$ to be close to the initial data of a Kerr--de~Sitter spacetime with angular momentum in a larger range (rather than merely in a neighborhood of $0$). The rest of our analysis does not change for large angular momenta, provided the Kerr--de~Sitter black hole one is perturbing is non-degenerate (in particular subextremal) in a suitable sense; see specifically the discussions in \cite[\S6.1]{VasyMicroKerrdS} and around \cite[Equation~(6.13)]{VasyMicroKerrdS}.

  Likewise, these are the computations to check for the stability analysis of higher-dimensional black holes with $\Lambda>0$; in this case, one in addition needs to extend the construction of the smooth family of metrics in \S\ref{SecKdS} to the higher-dimensional case.
\end{rmk}

\subsection{Further consequences}
\label{SubsecIntroCsq}

As an immediate consequence of our main theorem, we find that Kerr--de~Sitter spacetimes are the \emph{unique} stationary solutions of Einstein's field equations with positive cosmological constant in a neighborhood of Schwarzschild--de~Sitter spacetimes, as measured by the Sobolev norms on their initial data in Theorem~\ref{ThmIntroPrecise}. This gives a dynamical proof of a corresponding theorem for $\Lambda=0$ by Alexakis--Ionescu--Klainerman \cite{AlexakisIonescuKlainermanUniqueness} who prove the uniqueness of Kerr black holes in the vicinity of a member of the Kerr family.

Moreover, our black hole stability result is a crucial step towards a definitive resolution of Penrose's Strong Cosmic Censorship Conjecture for positive cosmological constants. (We refer the reader to the introduction of \cite{LukOhReissnerNordstrom} for an overview of this conjecture.) In fact, we expect that ongoing work by Dafermos--Luk \cite{DafermosLukKerrCauchyHorI} on the $\cC^0$ stability of the Cauchy horizon of Kerr spacetimes should combine with our main theorem to give, \emph{unconditionally}, the $\cC^0$ stability of the Cauchy horizon of Kerr--de~Sitter spacetimes. The decay assumptions along the black hole event horizon which are the starting point of the analysis of \cite{DafermosLukKerrCauchyHorI} are merely polynomial, corresponding to the expected decay of solutions to Einstein's equations in the asymptotically flat setting; however, as shown in \cite{HintzVasyCauchyHorizon} for linear wave equations, the exponential decay rate exhibited for $\Lambda>0$ should allow for a stronger conclusion; a natural conjecture, following \cite[Theorem~1.1]{HintzVasyCauchyHorizon}, would be that the metric has $H^{1/2+\alpha/\kappa}$ regularity at the Cauchy horizon, where $\kappa>0$ is the surface gravity of the Cauchy horizon. (Indeed, the linear analysis in the present paper can be shown to imply this for solutions of \emph{linearized} gravity.) We refer to the work of Costa, Gir\~ao, Nat\'ario and Silva \cite{CostaGiraoNatarioSilvaCauchy1,CostaGiraoNatarioSilvaCauchy2,CostaGiraoNatarioSilvaCauchy3} on the non-linear Einstein--Maxwell--scalar field system under the assumption of spherical symmetry for results of a similar flavor.

Lastly, we can make the asymptotic analysis of solutions to the (linearized) Einstein equation more precise and thus study the phenomenon of \emph{ringdown}. Concretely, for the linear problem, one can in principle obtain a (partial) asymptotic expansion of the gravitational wave beyond the leading order, linearized Kerr--de~Sitter, term; one may even hope for a \emph{complete asymptotic expansion} akin to the one established in \cite{BonyHaefnerDecay,DyatlovAsymptoticDistribution} for the scalar wave equation. For the non-linear problem, this implies that one can `see' shallow quasinormal modes for timescales which are logarithmic in the size of the initial data. See Remark~\ref{RmkKdSStabRingdown} for further details. Very recently, the ringdown from a binary black hole merger has been measured for the first time \cite{LIGOBlackHoleMerger}.

\subsection{Previous and related work}
\label{SubsecIntroPrev}

The aforementioned papers \cite{VasyMicroKerrdS,HintzVasySemilinear,HintzVasyQuasilinearKdS}---on which the analysis of the present paper directly builds---and our general philosophy to the study of waves on black hole spacetimes, mostly with $\Lambda>0$, build on a host of previous works.

On Schwarzschild--de~Sitter space, Bachelot \cite{BachelotSchwarzschildScattering} set up the functional analytic scattering theory, and S\'a Barreto--Zworski \cite{SaBarretoZworskiResonances} and Bony--H\"afner \cite{BonyHaefnerDecay} studied resonances and exponential wave decay away from the event horizon; Melrose, S\'a Barreto and Vasy \cite{MelroseSaBarretoVasySdS} proved exponential decay to constants across the horizons.

The precise study of waves on rotating Kerr--de~Sitter spacetimes requires an analysis at normally hyperbolically trapped sets, which was first accomplished in the breakthrough work of Wunsch--Zworski \cite{WunschZworskiNormHypResolvent}. This was later extended and simplified by Nonnenmacher--Zworski \cite{NonnenmacherZworskiQuantumDecay} and Dyatlov \cite{DyatlovSpectralGaps}; see also \cite{HintzVasyNormHyp}. This enabled Dyatlov to obtain full asymptotic expansions for linear waves on exact, slowly rotating Kerr--de~Sitter spaces into quasinormal modes (resonances) \cite{DyatlovAsymptoticDistribution}, following his earlier work on exponential energy decay \cite{DyatlovQNM,DyatlovQNMExtended}; see also the more recent \cite{DyatlovWaveAsymptotics}.

Using rather different, physical space, techniques, Dafermos--Rodnianski \cite{DafermosRodnianskiSdS} proved super-polynomial energy decay on Schwarzschild--de~Sitter spacetimes. Such techniques were also used by Schlue in his analysis of linear waves in the \emph{cosmological} part of Kerr--de~Sitter spacetimes \cite{SchlueCosmological}, and by Keller for the Maxwell equation \cite{KellerMaxwellSdS}. We furthermore mention Warnick's physical space approach to the study of resonances \cite{WarnickQNMs}.

Regarding work on spacetimes \emph{without} black holes, but in the microlocal spirit, we mention specifically the works \cite{BaskinParamKGondS,BaskinStrichartzDeSitterLong,BaskinVasyWunschRadMink,BaskinVasyWunschRadMink2}.

The general microlocal analytic and geometric framework underlying our global study of asymptotically Kerr--de~Sitter type spaces by compactifying them to manifolds with boundary, which are then naturally equipped with b-metrics, is Melrose's b-analysis \cite{MelroseAPS}. The considerable flexibility and power of a microlocal point of view is exploited throughout the present paper, especially in \S\ref{SecAsy}, \S\ref{SecSCP} and \S\ref{SecESG}. We specifically mention the ease with which bundle-valued equations can be treated, as first noted in \cite{VasyMicroKerrdS}, and shown concretely in \cite{HintzPsdoInner,HintzVasyKdsFormResonances} where the authors prove decay to stationary states for Maxwell's (and more general) equations. We also point out that a stronger notion of normal hyperbolicity, called $r$-normal hyperbolicity---which is stable under perturbations \cite{HirschShubPughInvariantManifolds}---was proved for Kerr and Kerr--de~Sitter spacetimes in \cite{WunschZworskiNormHypResolvent,DyatlovWaveAsymptotics}, and allows for global results for (non-)linear waves under very general assumptions \cite{VasyMicroKerrdS,HintzVasyQuasilinearKdS}. Since, as we show, solutions to Einstein's equations near Kerr--de~Sitter always decay to an \emph{exact} Kerr--de~Sitter solution (up to exponentially decaying tails), the flexibility afforded by $r$-normal hyperbolicity is not used here.

Linear and non-linear wave equations on black hole spacetimes with $\Lambda=0$, specifically Kerr and Schwarzschild, have received more attention. They do not directly fit into the general frameworks mentioned above; a fundamental difference is that waves decay at most at a fixed polynomial rate on general asymptotically flat ($\Lambda=0$) spacetimes, which is in stark contrast to the exponential decay rate on spacetimes with asymptotically hyperbolic ($\Lambda>0$) ends. Directly related to the topic of the present paper is the recent proof of the linear stability of the Schwarzschild spacetime under gravitational perturbations without symmetry assumptions on the data \cite{DafermosHolzegelRodnianskiSchwarzschildStability}, which we already discussed above; a less quantitative version of this was obtained by simpler means by Hung, Keller, and Wang \cite{HungKellerWangSchwarzschild}. After pioneering work by Wald \cite{WaldSchwarzschild} and Kay--Wald \cite{KayWaldSchwarzschild}, Dafermos, Rodnianski and Shlapentokh-Rothman \cite{DafermosRodnianskiKerrDecaySmall,DafermosRodnianskiShlapentokhRothmanDecay} recently proved polynomial decay for the scalar wave equation on all (exact) subextremal Kerr spacetimes; Tataru and Tohaneanu \cite{TataruDecayAsympFlat,TataruTohaneanuKerrLocalEnergy} proved Price's law, i.e.\ precise polynomial decay rates, for slowly rotating Kerr spacetimes, and Marzuola, Metcalfe, Tataru and Tohaneanu obtained Strichartz estimates \cite{MarzuolaMetcalfeTataruTohaneanuStrichartz,TohaneanuKerrStrichartz}. There is also work by Donninger, Schlag and Soffer \cite{DonningerSchlagSofferPrice} on $L^\infty$ estimates on Schwarzschild black holes, following $L^\infty$ estimates of Dafermos and Rodnianski \cite{DafermosRodnianskiRedShift}, and of Blue and Soffer \cite{BlueSofferPhaseSpace} on non-rotating charged black holes giving $L^6$ estimates. Apart from \cite{DafermosHolzegelRodnianskiSchwarzschildStability}, bundle-valued (or coupled systems of) equations were studied in particular in the contexts of Maxwell's equations by Andersson and Blue \cite{BlueMaxwellSchwarzschild,AnderssonBlueHiddenKerr,AnderssonBlueMaxwellKerr} and Sterbenz--Tataru \cite{SterbenzTataruMaxwellSchwarzschild}, see also \cite{IngleseNicoloMaxwellSchwarzschild,DonningerSchlagSofferSchwarzschild}, and for Dirac equations by Finster, Kamran, Smoller and Yau \cite{FinsterKamranSmollerYauDiracKerrNewman}. Non-linear problems on exterior $\Lambda=0$ black hole spacetimes were studied by Dafermos, Holzegel and Rodnianski \cite{DafermosHolzegelRodnianskiKerrBw} who constructed \emph{backward} solutions of the Einstein vacuum equations settling down to Kerr \emph{exponentially fast} (regarding this point, see also \cite{DafermosShlapentokhRothmanBlueshift}); for forward problems, Dafermos \cite{DafermosEinsteinMaxwellScalarStability,DafermosBlackHoleNoSingularities} studied the non-linear Einstein--Maxwell--scalar field system under the assumption of spherical symmetry. We also mention Luk's work \cite{LukKerrNonlinear} on semi-linear equations on Kerr, as well as the steps towards understanding a model problem related to Kerr stability under the assumption of axial symmetry \cite{IonescuKlainermanWavemap}.

A fundamental driving force behind a large number of these works is Klainerman's vector field method \cite{KlainermanUniformDecay}; subsequent works by Klainerman and Chris\-to\-dou\-lou \cite{KlainermanNullCondition,ChristodoulouGlobalSolutionsSmallData} introduce the `null condition' which plays a major role in the analysis of non-linear interactions near the light cone in particular in $(3+1)$-dimensional asymptotically flat spacetimes---in the asymptotically hyperbolic case which we study here, there is no analogue of this condition.

Using these and related techniques, a number of works prove the global non-linear stability of Minkowski space as a solution to Einstein's field equations coupled to various matter models; we mention the works by Speck \cite{SpeckEinsteinMaxwell} for the Einstein--Maxwell system, LeFloch--Ma \cite{LeFlochMaEinsteinMassive} for the Einstein equation coupled to a massive scalar field, Taylor \cite{TaylorEinsteinVlasov} for Einstein--Vlasov, and references therein. There is also a large amount of literature studying stability questions under symmetry assumptions on the spacetime: we only mention the work by Choquet-Bruhat and Moncrief \cite{ChoquetBruhatMoncriefStabilityU1} in which they in particular solve for a (time-dependent) finite-dimensional (Teichm\"uller) parameter, but we point out that this is unrelated to the finite-dimensional gauge issues discussed in \S\ref{SubsecIntroIdeas}.

There is also ongoing work by Dafermos--Luk \cite{DafermosLukKerrCauchyHorI} on the stability of the interior (`Cauchy') horizon of Kerr black holes; note that the black hole interior is largely unaffected by the presence of a cosmological constant, but the a priori decay assumptions along the event horizon, which determine regularity properties at the Cauchy horizon, are vastly different: the merely polynomial decay rates on asymptotically flat ($\Lambda=0$) spacetimes, as compared to the exponential decay rate on asymptotically hyperbolic ($\Lambda>0$) spacetimes, is a low frequency effect, related to the very delicate behavior of the resolvent near zero energy on asymptotically flat spaces. A precise study in the spirit of \cite{VasyMicroKerrdS} and \cite{DatchevVasyGluing} is currently in progress \cite{HaefnerVasyKerr}.

In the physics community, black hole perturbation theory, i.e.\ the study of \emph{linearized} perturbations of black hole spacetimes, has a long history. For us, the most convenient formulation, which we use heavily in \S\ref{SecUEMS}, is due to Ishibashi, Kodama and Seto \cite{KodamaIshibashiSetoBranes,KodamaIshibashiMaster,IshibashiKodamaHigherDim}, building on earlier work by Kodama--Sasaki \cite{KodamaSasakiPerturbation}. The study was initiated in the seminal paper by Regge--Wheeler \cite{ReggeWheelerSchwarzschild}, with extensions by Vishveshwara \cite{VishveshwaraSchwarzschild} and Zerilli \cite{ZerilliPotential}, analyzing metric perturbations of the Schwarzschild spacetime; a gauge-invariant formalism was introduced by Moncrief \cite{MoncriefGravitation}, later extended to allow for coupling with matter models by Gerlach--Sengupta \cite{GerlachSenguptaSpherical} and Martel--Poisson \cite{MartelPoissonFormalism}. A different approach to the study of gravitational perturbations, relying on the Newman--Penrose formalism \cite{NewmanPenroseSpin}, was pursued by Bardeen--Press \cite{BardeenPressSchwarzschild} and Teukolsky \cite{TeukolskyKerr}, who discovered that certain curvature components satisfy decoupled wave equations; their mode stability was proved by Whiting \cite{WhitingKerrModeStability}. We refer to Chandrasekhar's monograph \cite{ChandrasekharBlackHoles} for a more detailed account.

For surveys of numerical investigations of quasinormal modes, often with the goal of quantifying the phenomenon of ringdown discussed in \S\ref{SubsecIntroCsq}, we refer the reader to the articles \cite{KokkotasSchmidtQNM,BertiCardosoStarinetsQNM} and the references therein. We also mention the paper by Dyatlov--Zworski \cite{DyatlovZworskiTrapping} connecting recent mathematical advances in particular related to quasinormal modes with the physics literature.

\subsection{Outline of the paper}
\label{SubsecIntroOutline}

We only give a broad outline and suggest ways to read the paper; we refer to the introductions of the individual sections for further details.

\begin{itemize}
\item In \S\ref{SecHyp} we discuss in detail the constraint equations and hyperbolic formulations of Einstein's equations;
\item in \S\ref{SecKdS} we give a precise description of the Kerr--de~Sitter family and its geometry as needed for the study of initial value problems for wave equations;
\item in \S\ref{SecKey} we describe the key ingredients of the proof in detail, namely UEMS, SCP and ESG, `essential spectral gap;' the latter is the statement that solutions of the linearized gauged Einstein equation, with our modification that gives SCP, have finite asymptotic expansions up to exponentially decaying remainders. As mentioned before, the key element here is that the subprincipal symbol of our linearized modified gauged Einstein equation has the correct behavior at the trapped set;
\item in \S\ref{SecAsy} we recall the linear global microlocal analysis results both in the smooth and in the non-smooth (Sobolev coefficients) settings, slightly extending these to explicitly accommodate initial value problems with non-vanishing initial data. (Our earlier works considered only inhomogeneous PDEs with vanishing initial conditions.) We also show how to modify the PDE in a finite rank manner in order to ensure solvability on spaces of decaying functions in spite of the presence of non-decaying modes (resonances);
\item in \S\ref{SecOp} we do some explicit computations for the Schwarzschild--de~Sitter metric that will be useful in the remaining sections;
\item in \S\ref{SecUEMS} we show the mode stability for the (ungauged) linearized Einstein equation (UEMS), in \S\ref{SecSCP} we show the stable gauge propagation (SCP), while in \S\ref{SecESG} we show the final key ingredient, the essential spectral gap (ESG) for the linearized gauged Einstein equation;
\item in \S\ref{SecKdSLStab} we put these ingredients together to show the linear stability of slowly rotating Kerr--de~Sitter black holes;
\item in \S\ref{SecKdSStab} we show their non-linear stability. (We also construct initial data sets in \S\ref{SubsecKdSStabIn}.)
\end{itemize}

In order for the reader to see that the linear stability result is extremely simple given the three key ingredients and the results of \S\ref{SubsecAsySm}, we suggest reading \S\ref{SecHyp}--\S\ref{SecKey} and taking the results in \S\ref{SecKey} for granted, looking up the two important results in \S\ref{SubsecAsySm} (Corollaries~\ref{CorAsySmResNoAsy} and \ref{CorAsySmPertNoAsy}), and then reading \S\ref{SecKdSLStab}. Following this, one may read \S\ref{SubsecKdSStab} for the proof of non-linear stability, which again uses the results of \S\ref{SecKey} as black boxes together with the perturbative analysis of \S\ref{SubsecAsyExp}, in particular Theorem~\ref{ThmAsyExpMain}. Only the reader interested in the (very instructive!) proofs of the key ingredients needs to consult \S\ref{SecOp}--\S\ref{SecESG}.

Appendix~\ref{SecB} recalls basic notions of Melrose's b-analysis. In Appendix~\ref{SecQ}, we state and prove a very general finite-codimensional solvability theorem for quasilinear wave equations on (Kerr--)de~Sitter-like spaces. In Appendix~\ref{SecdS} finally, we illustrate some of the key ideas of this paper by proving the non-linear stability of the static model of de~Sitter space; we recommend reading this section early on, since many of the obstacles we need to overcome in the black hole setting are exhibited very clearly in this simpler setting.

\subsection{Notation}
\label{SubsecIntroSummary}

For the convenience of the reader, we list some of the notation used throughout the paper, and give references to their first definition. (Some quantities and sets will be shrunk later in the paper as necessary, but we only give the first reference.)

\nomenstart
  \nomenitem{$\BoxCP_g$}{constraint propagation operator, see \eqref{EqHypDTCP}}
  \nomenitem{$\wtBoxCP_g$}{modified constraint propagation operator, see \eqref{EqKeySCPBoxCPMod}}
  \nomenitem{$\BoxGauge_g$}{wave operator for arranging linearized gauge conditions, see \eqref{EqHypLinBoxGauge}}
  \nomenitem{$\alpha$}{exponential decay rate, see \S\ref{SubsecKeySCP}}
  \nomenitem{$B$}{space of black hole parameters ($\subset\R^4$), see \eqref{EqKdSB}}
  \nomenitem{$b_0$}{fixed Schwarzschild--de~Sitter parameters, $b_0\in B$, see \eqref{EqSdSParam}}
  \nomenitem{$\Gamma$}{trapped set on Schwarzschild--de~Sitter space, see \eqref{EqKdSGeoTrapped}}
  \nomenitem{$\delta_g$}{divergence, $(\delta_g u)_{i_1\ldots i_n}=-u_{i_1\ldots i_n j;}{}^j$}
  \nomenitem{$\delta_g^*$}{symmetric gradient, $(\delta_g u)_{ij}=\frac{1}{2}(u_{i;j}+u_{j;i})$}
  \nomenitem{$\tdel$}{modified symmetric gradient, see \eqref{EqKeySCPDeltaTilde}}
  \nomenitem{$\dot\sD'$}{distributions, on a domain with corners, with supported character at the boundary, see \cite[Appendix~B]{HormanderAnalysisPDE3}}
  \nomenitem{$D^{s,\alpha}$}{space of data for initial value problems for wave equations, see Definition~\ref{DefAsySmDataSpace}}
  \nomenitem{$g_b$}{Kerr--de~Sitter metric with parameters $b\in B$, see \S\ref{SubsecKdSSlow}}
  \nomenitem{$G_b$}{dual metric of $g_b$}
  \nomenitem{$g'_b(b')$}{linearized (around $g_b$) Kerr--de~Sitter metric, with linearized parameters $b'\in T_b B$, see Definition~\ref{DefKdSSlowLinMet}}
  \nomenitem{$g^{\prime\Ups}_b(b')$}{linearized Kerr--de~Sitter metric put into the linearized wave map gauge, see Proposition~\ref{PropKdSLStabComplement}}
  \nomenitem{$\sfG_g$}{trace reversal operator, see \eqref{EqHypDTEinsteinify}}
  \nomenitem{$\ham_p$}{Hamilton vector field of the function $p$ on phase space, see Appendix~\ref{SubsecB}}
  \nomenitem{$\Hext^s$}{Sobolev space of extendible distributions on a domain with boundary or corners, see \cite[Appendix~B]{HormanderAnalysisPDE3}}
  \nomenitem{$\Hb^{s,\alpha}$}{weighted b-Sobolev space, see \eqref{EqBSobolevWeighted}}
  \nomenitem{$\Hb^{s,\alpha}(\Omega)^{\bullet,-}$}{space of restrictions of elements of $\Hb^{s,\alpha}(M)$ vanishing in the past of $\Sigma_0$ to the interior of $\Omega$, see \S\ref{SubsecBPsdo}}
  \nomenitem{$\Hbext^{s,\alpha}$}{weighted b-Sobolev space of extendible distributions on a domain with corners, see~\eqref{EqBSuppExt}}
  \nomenitem{$\Hbh^{s,\alpha}$}{semiclassical weighted b-Sobolev space, see Appendix~\ref{SubsecBSemi}}
  \nomenitem{$\cL_b$}{b-conormal bundle of the horizons of $(M,g_b)$, see \eqref{EqKdSGeoHorizonBConormal}}
  \nomenitem{$M$}{compactification of $M^\circ$ at future infinity, see \eqref{EqKdSSlowCompProd}}
  \nomenitem{$\cM$}{static coordinate chart $\subset M^\circ$ of a fixed Schwarzschild--de~Sitter spacetime, see \eqref{EqSdSMfdStatic}}
  \nomenitem{$M^\circ$}{open $4$-manifold on which the metrics $g_b$ are defined, see \eqref{EqSdSMfd}}
  \nomenitem{$\bhm$}{black hole mass, see \eqref{EqKdSB}}
  \nomenitem{$\bhm[0]$}{mass of a fixed Schwarzschild--de~Sitter black hole, see \eqref{EqSdSParam}}
  \nomenitem{$\Psib$}{algebra of b-pseudodifferential operators, see Appendix~\ref{SubsecBPsdo}}
  \nomenitem{$\Psibh$}{algebra of semiclassical b-pseudodifferential operators, see Appendix~\ref{SubsecBSemi}}
  \nomenitem{$\cR_b$}{(generalized) radial set of $(M,g_b)$ at the horizons, see \eqref{EqKdSGeoRadSet}}
  \nomenitem{$\sR_g$}{curvature term appearing in the linearization of $\Ric$, see \eqref{EqHypDTCurvatureTerm}}
  \nomenitem{$\Res(L)$}{set of resonances of the operator $L$, see \eqref{EqAsySmResSet}}
  \nomenitem{$\Res(L,\sigma)$}{linear space of resonant states of $L$ at $\sigma$, see \eqref{EqAsySmResStates}}
  \nomenitem{$\Res^*(L,\sigma)$}{linear space of dual resonant states of $L$ at $\sigma$, see \eqref{EqAsySmResDualStates}}
  \nomenitem{$\Sb^*$}{b-cosphere bundle, see Appendix~\ref{SubsecB}}
  \nomenitem{$\Sigma_0$}{Cauchy surface of the domain $\Omega$, see \eqref{EqKdSWaveDomain}}
  \nomenitem{$\Sigma_b$}{characteristic set of $G_b$, see \eqref{EqKdSGeoChar}}
  \nomenitem{$t$}{static time coordinate, see \eqref{EqSdSMetric}, or Boyer--Lindquist coordinate, see \eqref{EqKdSSlowStationary}}
  \nomenitem{$t_*$}{timelike function, smooth across the horizons, see \eqref{EqSdSTStar}}
  \nomenitem{$\Tb^*$}{b-cotangent bundle, see Appendix~\ref{SubsecB}}
  \nomenitem{$\ol{\Tb^*}$}{radially compactified b-cotangent bundle, see \eqref{EqBCompact}}
  \nomenitem{$\cU_B$}{small neighborhood of $b_0$ (parameters of slowly rotating Kerr--de~Sitter black holes), see Lemma~\ref{LemmaKdSSlowHorizon}}
  \nomenitem{$\Vb$}{space of b-vector fields, see Appendix~\ref{SubsecB}}
  \nomenitem{$X$}{boundary of $M$ at future infinity, see \eqref{EqKdSSlowCompProd}}
  \nomenitem{$\cX$}{spatial slice of the static chart $\cM$, see \eqref{EqSdSMfdStatic}}
  \nomenitem{$Y$}{boundary of $\Omega$ at future infinity, see \S\ref{SubsecKdSWave}}
  \nomenitem{$\Ups$}{gauge 1-form, see \eqref{EqKdSInGauge}}
  \nomenitem{$\Omega$}{domain with corners $\subset M$ on which we solve wave equations, see \eqref{EqKdSWaveDomain}}
  \nomenitem{$\omega^\Ups_b(b')$}{1-form used to put $g'_b(b')$ into the correct linearized gauge, see Proposition~\ref{PropKdSLStabComplement}.}
\nomenend

\newlength{\oldparindent}
\setlength{\oldparindent}{\parindent}
\setlength{\parindent}{0pt}
Furthermore, we repeatedly use the following acronyms:

\nomenstart
  \nomenitem{ESG}{`essential spectral gap,' see \S\ref{SubsecKeyESG}}
  \nomenitem{SCP}{`stable constraint propagation,' see \S\ref{SubsecKeySCP}}
  \nomenitem{UEMS}{`ungauged Einstein mode stability,' see \S\ref{SubsecKeyUEMS}}
\nomenend

\setlength{\parindent}{\oldparindent}

\subsection*{Acknowledgments}

We are very grateful to Richard Melrose and Maciej Zworski for discussions over the years that eventually inspired the completion of this work. We are also thankful to Rafe Mazzeo, Gunther Uhlmann, Jared Wunsch, Richard Bamler, Mihalis Dafermos, Semyon Dyatlov, Jeffrey Galkowski, Jesse Gell-Redman, Robin Graham, Dietrich H\"afner, Andrew Hassell, Gustav Holz\-egel, Sergiu Klainerman, Jason Metcalfe, Sung-Jin Oh, Michael Singer, Michael Taylor and Micha\l{} Wrochna for discussions, comments and their interest in this project. We are grateful to Igor Khavkine for pointing out references in the physics literature. We would also like to express our sincere gratitude to an anonymous and very thorough referee whose detailed comments led to substantial improvements in exposition and content throughout the entire paper.

The authors gratefully acknowledge partial support from the NSF under grant numbers DMS-1068742 and  DMS-1361432. At the time of writing of the first version of this paper, P.\ H.\ was a Miller Research Fellow, and he would like to thank the Miller Institute at the University of California, Berkeley, for support. The final revisions were made during the time P.\ H.\ served as a Clay Research Fellow.

\section{Hyperbolic formulations of the Einstein vacuum equations}
\label{SecHyp}

\subsection{Initial value problems; DeTurck's method}
\label{SubsecHypDT}

Einstein's field equations with a cosmological constant $\Lambda$ for a Lorentzian metric $g$ of signature $(1,3)$ on a smooth manifold $M$ take the form
\begin{equation}
\label{EqHypDTEinstein}
  \Ric(g) + \Lambda g = 0.
\end{equation}
The correct generalization to the case of $(n+1)$ dimensions is $\Ein(g)=\Lambda g$, which is equivalent to $(\Ric+\frac{2\Lambda}{n-1})g=0$; by a slight abuse of terminology and for the sake of brevity, we will however refer to \eqref{EqHypDTEinstein} as the Einstein vacuum equations also in the general case. Given a globally hyperbolic solution $(M,g)$ and a spacelike hypersurface $\Sigma_0\subset M$, the negative definite Riemannian metric $h$ on $\Sigma_0$ induced by $g$ and the second fundamental form $k(X,Y)=\la\nabla_X Y,N\ra$, $X,Y\in T\Sigma_0$, of $\Sigma_0$ satisfy the \emph{constraint equations}
\begin{equation}
\label{EqHypDTConstraints}
\begin{gathered}
  R_h + (\tr_h k)^2 - |k|_h^2 = (1-n)\Lambda, \\
  \delta_h k + d\tr_h k = 0,
\end{gathered}
\end{equation}
where $R_h$ is the scalar curvature of $h$, and $(\delta_h r)_\mu=-r_{\mu\nu;}{}^\nu$ is the divergence of the symmetric 2-tensor $r$. We recall that, given a unit normal vector field $N$ on $\Sigma_0$, the constraint equations are equivalent to the equations
\begin{equation}
\label{EqHypDTConstraints2}
  \Ein_g(N,N) = \frac{(n-1)\Lambda}{2}, \qquad \Ein_g(N,X)=0,\quad X\in T\Sigma_0,
\end{equation}
for the Einstein tensor $\Ein_g=\sfG_g\Ric(g)$, where
\begin{equation}
\label{EqHypDTEinsteinify}
  \sfG_g r=r-\frac{1}{2}(\tr_g r)g.
\end{equation}

Conversely, given an initial data set $(\Sigma_0,h,k)$ with $\Sigma_0$ a smooth $3$-manifold, $h$ a negative definite Riemannian metric on $\Sigma_0$ and $k$ a symmetric 2-tensor on $\Sigma_0$, one can consider the non-characteristic initial value problem for the Einstein equation \eqref{EqHypDTEinstein}, which asks for a Lorentzian $4$-manifold $(M,g)$ and an embedding $\Sigma_0\hra M$ such that $h$ and $k$ are, respectively, the induced metric and second fundamental form of $\Sigma_0$ in $M$. We refer to the survey of Bartnik--Isenberg \cite{BartnikIsenbergConstraints} for a detailed discussion of the constraint equations; see also \S\ref{SubsecKdSStabIn}.

As explained in the introduction, solving the initial value problem is non-trivial because of the lack of hyperbolicity of Einstein's equations due to their diffeomorphism invariance. However, as first shown by Choquet-Bruhat \cite{ChoquetBruhatLocalEinstein}, the initial value problem admits a local solution, provided $(h,k)$ are sufficiently regular, and the solution is unique up to diffeomorphisms in this sense; Choquet-Bruhat--Geroch \cite{ChoquetBruhatGerochMGHD} then proved the existence of a \emph{maximal} globally hyperbolic development of the initial data. (Sbierski \cite{SbierskiMGHD} recently gave a proof of this fact which avoids the use of Zorn's Lemma.)

We now explain the method of DeTurck \cite{DeTurckPrescribedRicci} for solving the initial value problem in some detail; we follow the presentation of Graham and Lee \cite{GrahamLeeConformalEinstein}. Given the initial data set $(\Sigma_0,h,k)$, we define $M=\R_{x^0}\times\Sigma_0$ and embed $\Sigma_0\hra\{x^0=0\}\subset M$; the task is to find a Lorentzian metric $g$ on $M$ near $\Sigma_0$ solving the Einstein equation and inducing the initial data $(h,k)$ on $\Sigma_0$. Choose a smooth non-degenerate background metric $g^0$, which can have arbitrary signature. We then define the gauge 1-form
\begin{equation}
\label{EqHypDTGaugeTerm}
  \Ups(g) := g (g^0)^{-1}\delta_g \sfG_g g^0 \in \CI(M,T^*M),
\end{equation}
viewing $g (g^0)^{-1}$ as a bundle automorphism of $T^*M$. As a non-linear differential operator acting on $g\in\CI(M,S^2T^*M)$, the operator $\Ups(g)$ is of first order.

\begin{rmk}
\label{RmkHypDTPutIntoGauge}
  A simple calculation in local coordinates gives
  \begin{equation}
  \label{EqHypDTPutIntoGaugeUpsLoc}
    \Ups(g)_\mu = g_{\mu\kappa} g^{\nu\lambda}(\Gamma(g)_{\nu\lambda}^\kappa-\Gamma(g^0)_{\nu\lambda}^\kappa).
  \end{equation}
  Thus, $\Ups(g)=0$ if and only if the pointwise identity map $\Id\colon(M,g)\to(M,g^0)$ is a wave map. Now, given any local solution $g$ of the initial value problem for Einstein's equations, we can solve the wave map equation $\Box_{g,g^0}\phi=0$ for $\phi\colon (M,g)\to(M,g^0)$ with initial data $\phi|_{\Sigma_0}=\Id_{\Sigma_0}$ and $D\phi|_{\Sigma_0}=\Id_{T\Sigma_0}$. Indeed, recalling that
  \[
    (\Box_{g,g^0}\phi)^k = g^{\mu\nu}\bigl(\pa_\mu\pa_\nu \phi^k - \Gamma(g)^\lambda_{\mu\nu}\pa_\lambda\phi^k + \Gamma(g^0)^k_{ij}\pa_\mu\phi^i\pa_\nu\phi^j\bigr),\quad  \phi(x)=(\phi^k(x^\mu)),
  \]
  we see that $\Box_{g,g^0}\phi=0$ is a semi-linear wave equation, hence a solution $\phi$ is guaranteed to exist locally near $\Sigma_0$, and $\phi$ is a diffeomorphism of a small neighborhood $U$ of $\Sigma_0$ onto $\phi(U)$; let us restrict the domain of $\phi$ to such a neighborhood $U$. Then $g^\Ups:=\phi_*g$ is well-defined on $\phi(U)$, and $\phi\colon(U,g)\to(\phi(U),g^\Ups)$ is an isometry; hence we conclude that $\Id\colon(\phi(U),g^\Ups)\to(\phi(U),g^0)$ is a wave map, so $\Ups(g^\Ups)=0$. Moreover, by our choice of initial conditions for $\phi$, the metric $g^\Ups$ induces the given initial data $(h,k)$ pointwise on $\Sigma_0$.
\end{rmk}

With $(\delta_g^*u)_{\mu\nu}=\frac{1}{2}(u_{\mu;\nu}+u_{\nu;\mu})$ denoting the symmetric gradient of a 1-form $u$, the non-linear differential operator
\begin{equation}
\label{EqHypDTOp}
  P_\DT(g) := \Ric(g) + \Lambda g - \delta_g^*\Ups(g)
\end{equation}
is hyperbolic. Indeed, following \cite[\S3]{GrahamLeeConformalEinstein}, the linearizations of the various terms are given by
\begin{equation}
\label{EqHypDTDRic}
  D_g\Ric(r) = \frac{1}{2}\Box_g r - \delta_g^*\delta_g \sfG_g r + \sR_g(r),
\end{equation}
where $(\Box_g r)_{\mu\nu} = -r_{\mu\nu;\kappa}{}^\kappa$, and
\begin{equation}
\label{EqHypDTCurvatureTerm}
  \sR_g(r)_{\mu\nu} = r^{\kappa\lambda}(R_g)_{\kappa\mu\nu\lambda} + \frac{1}{2}\bigl(\Ric(g)_\mu{}^\lambda r_{\lambda\nu} + \Ric(g)_\nu{}^\lambda r_{\lambda\mu}\bigr),
\end{equation}
and
\begin{equation}
\label{EqHypDTDUps}
  D_g\Ups(r) = -\delta_g \sfG_g r + \sC(r)-\sD(r),
\end{equation}
where
\begin{gather*}
  C_{\mu\nu}^\kappa = \frac{1}{2}((g^0)^{-1})^{\kappa\lambda}(g^0_{\mu\lambda;\nu}+g^0_{\nu\lambda;\mu}-g^0_{\mu\nu;\lambda}), \quad D^\kappa=g^{\mu\nu}C_{\mu\nu}^\kappa, \\
  \sC(r)_\kappa = g_{\kappa\lambda}C^\lambda_{\mu\nu}r^{\mu\nu}, \quad \sD(r)_\kappa=D^\lambda r_{\kappa\lambda},
\end{gather*}
so $D_g P_\DT(r)$ is equal to the principally scalar wave operator $\frac{1}{2}\Box_g r$ plus lower order terms. Therefore, one can solve the Cauchy problem for the quasilinear hyperbolic system $P_\DT(g)=0$, where the Cauchy data
\[
  \gamma_0(g):=(g|_{\Sigma_0},\cL_{\pa_{x^0}}g|_{\Sigma_0})=(g_0,g_1),
\]
$g_0,g_1\in\CI(\Sigma_0,S^2T^*_{\Sigma_0}M)$, are arbitrary. Moreover, we saw in Remark~\ref{RmkHypDTPutIntoGauge} that every solution of the Einstein vacuum equations can be realized as a solution of $P_\DT(g)=0$ by putting the solution into the wave map gauge $\Ups(g)=0$. (On the other hand, a solution of $P_\DT(g)=0$ with \emph{general} Cauchy data $(g_0,g_1)$ will have no relationship with the Einstein equation!)

We can now explain how to solve the initial value problem for \eqref{EqHypDTEinstein}. Given an initial data set $(\Sigma_0,h,k)$, so $h,k\in\CI(\Sigma_0,S^2 T^*\Sigma_0)$ (in local coordinates: $3\times 3$ matrices), one first constructs $g_0,g_1\in\CI(\Sigma_0,S^2T^*_{\Sigma_0}M)$ (in local coordinates: $4\times 4$ matrices) with the following properties:
\begin{itemize}
\item $g_0$ is of Lorentzian signature;
\item the data on $\Sigma_0$ induced by a metric $\ol g$ with $\gamma_0(\ol g)=(g_0,g_1)$ are equal to $(h,k)$;
\item $\Ups(\ol g)|_{\Sigma_0}=0$ as an element of $\CI(\Sigma_0,T^*_{\Sigma_0}M)$.
\end{itemize}
Note here that the metric induced on $\Sigma_0$ (which we want to be equal to $h$) only depends on $g_0$, while the second fundamental form and gauge 1-form (which we want to be $k$ and $\Ups(\ol g)$, respectively) only depend on $(g_0,g_1)$; in other words, any metric $g$ with the same Cauchy data $(g_0,g_1)$ induces the given initial data on $\Sigma_0$ and satisfies $\Ups(g)=0$ at $\Sigma_0$. We refer the reader to \cite[\S18.9]{TaylorPDE} for the construction of the Cauchy data, and also to \S\ref{SubsecKdSIn} for a detailed discussion in the context of the black hole stability problem.

Next, one solves the \emph{gauged Einstein equation}
\begin{equation}
\label{EqHypDTGaugedEinstein}
  P_\DT(g)=0,\quad \gamma_0(g)=(g_0,g_1)
\end{equation}
locally near $\Sigma_0$. Applying $\sfG_g$ to this equation and using the constraint equations \eqref{EqHypDTConstraints2}, we conclude that $(\sfG_g\delta_g^*\Ups(g))(N,N)=0$, where $N$ is a unit normal to $\Sigma_0\subset M$ with respect to the solution metric $g$, and $(\sfG_g\delta_g^*\Ups(g))(N,X)=0$ for all $X\in T\Sigma_0$. It is easy to see \cite[\S18.8]{TaylorPDE} that $\Ups(g)|_{\Sigma_0}=0$ and these equations together imply $\cL_{\pa_{x^0}}\Ups(g)|_{\Sigma_0}=0$. The final insight is that the second Bianchi identity,\footnote{The second Bianchi identity is in fact in a consequence of the diffeomorphism invariance $\Ric(\phi^*g)=\phi^*\Ric(g)$ of the Ricci tensor, see \cite[\S3.2]{ChowKnopfRicci}.} written as $\delta_g \sfG_g\Ric(g)=0$ for \emph{any} metric $g$, implies a hyperbolic evolution equation for $\Ups(g)$, which we call the \emph{(unmodified) constraint propagation equation}, to wit
\begin{equation}
\label{EqHypDTCPEq}
  \BoxCP_g\Ups(g) = 0,
\end{equation}
where $\BoxCP_g$ is the \emph{(unmodified) constraint propagation operator} defined as
\begin{equation}
\label{EqHypDTCP}
  \BoxCP_g := 2\delta_g \sfG_g\delta_g^*.
\end{equation}
The notation is justified: one easily verifies $\BoxCP_g u=\Box_g u-\Ric(g)(u,\cdot)$, with $\Box_g$ the tensor Laplacian on 1-forms. The terminology is motivated by the following fact: given a solution $g$ of \eqref{EqHypDTGaugedEinstein} with arbitrary initial data, and a spacelike surface $\Sigma_1$ (with unit normal $N$) at which $\Ups(g)|_{\Sigma_1}=0$, the constraint equations \eqref{EqHypDTConstraints2} are equivalent to $(\Ups(g)_{\Sigma_1},\cL_N\Ups(g)|_{\Sigma_1})=0$; it is in this sense that \eqref{EqHypDTCPEq} governs the propagation of the constraints.

Now, the uniqueness of solutions of the Cauchy problem for the equation \eqref{EqHypDTCPEq} implies $\Ups(g)\equiv 0$, and thus $g$ indeed satisfies the Einstein equation $\Ric(g)+\Lambda g=0$ in the wave map gauge $\Ups(g)=0$, and $g$ induces the given initial data on $\Sigma_0$. This justifies the terminology `gauged Einstein equation' for the equation \eqref{EqHypDTGaugedEinstein}, since its solution solves the Einstein equation in the chosen gauge.

\subsection{Initial value problems for linearized gravity}
\label{SubsecHypLin}

Suppose now we have a smooth family $g_s$, $s\in(-\eps,\eps)$, of Lorentzian metrics solving the Einstein equation $\Ric(g_s)+\Lambda g_s=0$ on a fixed $(n+1)$-dimensional manifold $M$, and a hypersurface $\Sigma_0\subset M$ which is spacelike for $g=g_0$. Let
\[
  r=\frac{d}{ds}g_s|_{s=0}.
\]
Differentiating the equation at $s=0$ gives the \emph{linearized (ungauged) Einstein equation}
\begin{equation}
\label{EqHypLinEinstein}
  D_g(\Ric+\Lambda)(r) = 0.
\end{equation}
The linearized constraints can be derived as the linearization of \eqref{EqHypDTConstraints} around the initial data induced by $g_0$, hence they are equations for the linearized metric $h'$ and the linearized second fundamental form $k'$, with $h',k'\in\CI(\Sigma_0,S^2T^*\Sigma_0)$; alternatively, we can use \eqref{EqHypDTConstraints2}: if $N_s$ is a unit normal field to $\Sigma_0$ with respect to the metric $g_s$, so $N_s$ also depends smoothly on $s$, then differentiating $(\Ein_{g_s}-\frac{(n-1)\Lambda}{2} g_s)(N_s,N_s)=0$ at $s=0$ gives
\begin{equation}
\label{EqHypLinConstraints1}
  \Bigl(D_g\Ein(r)-\frac{(n-1)\Lambda}{2}r\Bigr)(N,N) = 0,
\end{equation}
where we used \eqref{EqHypDTConstraints2} to see that the terms coming from differentiating either of the $N_s$ gives $0$; on the other hand, using the Einstein equation, the derivative of $\Ein_{g_s}(N_s,X)=0$ at $s=0$ takes the form
\[
  D_g\Ein(r)(N,X) + \frac{(n-1)\Lambda}{2}g(N',X) = 0, \quad N'=\frac{d}{ds}N_s|_{s=0};
\]
now $g_s(N_s,X)=0$ yields $g(N',X)=-r(N,X)$ upon differentiation, hence we arrive at
\begin{equation}
\label{EqHypLinConstraints2}
  \Bigl(D_g\Ein(r)-\frac{(n-1)\Lambda}{2}r\Bigr)(N,X) = 0,\quad X\in T\Sigma_0.
\end{equation}

If moreover each of the metrics $g_s$ is in the wave map gauge $\Ups(g_s)=0$, with $\Ups$ given in \eqref{EqHypDTGaugeTerm} for a fixed background metric $g^0$, then we also get
\[
  D_g\Ups(r) = 0;
\]
therefore, in this case, $r$ solves the \emph{linearized gauged Einstein equation}
\begin{equation}
\label{EqHypLinGaugedEinstein}
  D_g(\Ric+\Lambda)(r) - \delta_g^*D_g\Ups(r) = 0.
\end{equation}
As in the non-linear setting, one can use \eqref{EqHypLinGaugedEinstein}, which is a principally scalar wave equation as discussed after \eqref{EqHypDTOp}, to prove the well-posedness of the initial value problem for the linearized Einstein equation: given an initial data set $(h',k')$ of symmetric 2-tensors on $\Sigma_0$ satisfying the linearized constraint equations, one constructs Cauchy data $(r_0,r_1)$ for the gauged equation \eqref{EqHypLinGaugedEinstein} satisfying the linearized gauge condition at $\Sigma_0$; solving the Cauchy problem for \eqref{EqHypLinGaugedEinstein} yields a symmetric 2-tensor $r$. The linearized constraints in the form \eqref{EqHypLinConstraints1}--\eqref{EqHypLinConstraints2} imply that $\sfG_g\delta_g^* D_g\Ups(r)=0$ at $\Sigma_0$, which as before implies $\cL_{\pa_{x^0}}D_g\Ups(r)=0$ at $\Sigma_0$. Finally, since $\Ric(g)+\Lambda g=0$, linearizing the second Bianchi identity in $g$ gives
\[
  \delta_g \sfG_g D_g(\Ric+\Lambda)(r) = 0,
\]
and thus \eqref{EqHypLinGaugedEinstein} implies the evolution equation $\BoxCP_g(D_g\Ups(r))=0$, hence $D_g\Ups(r)\equiv 0$, and we therefore obtain a solution $r$ of \eqref{EqHypLinEinstein}.

Analogously to the discussion in Remark~\ref{RmkHypDTPutIntoGauge}, one can put a given solution $r$ of the linearized Einstein equation \eqref{EqHypLinEinstein}, with $g$ solving $\Ric(g)+\Lambda g=0$, into the linearized gauge $D_g\Ups(r)=0$ by solving a linearized wave map equation. Concretely, the diffeomorphism invariance of the non-linear equation \eqref{EqHypDTEinstein} implies that $D_g(\Ric+\Lambda)(\delta_g^*\omega')=0$ for all $\omega'\in\CI(M,T^*M)$; indeed, $\delta_g^*\omega'=\frac{1}{2}\cL_{(\omega')^\sharp}g$ is a Lie derivative. Thus, putting $r$ into the gauge $D_g\Ups(r)=0$ amounts to finding $\omega'$ such that
\begin{equation}
\label{EqHypLinGauge}
  D_g\Ups(r+\delta_g^*\omega')=0
\end{equation}
holds, or equivalently
\begin{equation}
\label{EqHypLinAdjustGauge}
  \BoxGauge_g\omega' = 2 D_g\Ups(r),
\end{equation}
where we define
\begin{equation}
\label{EqHypLinBoxGauge}
  \BoxGauge_g = -2 D_g\Ups\circ\delta_g^*,
\end{equation}
which agrees to leading order with the wave operator $\Box_g$ on 1-forms due to \eqref{EqHypDTDUps}. Thus, one can solve \eqref{EqHypLinAdjustGauge}, with any prescribed initial data $\gamma_0(\omega')$, and then \eqref{EqHypLinGauge} holds. Taking $\gamma_0(\omega')=0$ ensures that the linearized initial data, i.e.\ the induced linearized metric and linearized second fundamental form on $\Sigma_0$, of $r$ and $r+\delta_g^*\omega'$ coincide.

We remark that if one chooses the background metric $g^0$ in \eqref{EqHypDTGaugeTerm} to be equal to the metric $g$ around which we linearize, then $D_g\Ups(r)=-\delta_g \sfG_g r$ by \eqref{EqHypDTDUps} and thus $\BoxGauge_g=\Box_g+\Lambda$, see also \eqref{EqHypDTCP}.

\section{The Kerr--de~Sitter family of black hole spacetimes}
\label{SecKdS}

Let us fix the cosmological constant $\Lambda>0$. The Kerr--de~Sitter family of black holes, which we will recall momentarily, is parameterized by the mass $\bhm>0$ and the angular momentum $a$ of the black hole. We shall only consider black holes which are not `too large,' $9\Lambda\bhm^2<1$, which for Schwarzschild--de~Sitter black holes ensures that the cosmological horizon is outside of the event horizon; and the angular momentum will be small, $|a|\ll 1$, i.e.\ we only study \emph{slowly rotating} Kerr--de~Sitter black holes.

Since the $SO(3)$-action on Kerr--de~Sitter metrics (via pullback) degenerates at $a=0$, it will be useful to in fact use the larger, and hence redundant, parameter space
\begin{equation}
\label{EqKdSB}
  B = \{ (\bhm,\bfa)\colon \bhm>0, \ \bfa\in\R^3 \} \subset \R^4.
\end{equation}
This allows us to keep track of the rotation axis $\bfa/|\bfa|$ of the black hole for $a=|\bfa|\neq 0$. (Here, $|\cdot|$ denotes the Euclidean norm in $\R^3$.) The subfamily of Schwarzschild--de~Sitter black holes is then parameterized by elements $(\bhm,\bfzero)\subset B$, $\bhm>0$. As explained in the introduction, the relevance of Kerr--de~Sitter spacetimes in general relativity is that they are solutions of the Einstein vacuum equations with a cosmological constant:
\[
  \Ric(g_b) + \Lambda g_b = 0.
\]

In this section, we will define a manifold $M^\circ$ and the Kerr--de~Sitter family $g_b$ of smooth, stationary Lorentzian metrics on $M^\circ$; we proceed in two steps, first defining Schwarzschild--de  Sitter metrics in \S\ref{SubsecSdS}, and then Kerr--de~Sitter metrics \S\ref{SubsecKdSSlow}, in particular proving the smoothness of the family. A key aspect of our approach to the black hole stability problem is that we work on a compactification of Kerr--de~Sitter space at future infinity; we discuss this in \S\ref{SubsecKdSCpt}. In \S\ref{SubsecKdSGeo} and \S\ref{SubsecKdSWave}, we describe the geometric structure of these spacetimes in detail and explain how to set up initial value problems for wave equations. In \S\ref{SubsecKdSIn} finally, we construct Cauchy data for hyperbolic formulations of the Einstein equation, in suitable wave map gauges, out of geometric initial data.

\subsection{Schwarzschild--de~Sitter black holes}
\label{SubsecSdS}

We fix a black hole mass $\bhm[0]>0$ such that
\begin{equation}
\label{EqSdSNonDeg}
  9\Lambda\bhm[0]^2\in(0,1),
\end{equation}
and let
\begin{equation}
\label{EqSdSParam}
  b_0 = (\bhm[0],\bfzero) \in B
\end{equation}
be the parameters for a non-rotating Schwarzschild--de~Sitter black hole. In the \emph{static coordinate patch} $\cM=\R_t\times \cI_r\times\Sph^2$, which covers the \emph{exterior region} (also known as the \emph{domain of outer communications}) of the black hole, with the interval $\cI$ defined below, the metric is defined by
\begin{equation}
\label{EqSdSMetric}
  g_{b_0} = \mu_{b_0}\,dt^2 - \mu_{b_0}^{-1}\,dr^2 - r^2\,\slg,\quad \mu_{b_0}(r) = 1-\frac{2\bhm[0]}{r} - \frac{\Lambda r^2}{3},
\end{equation}
where $\slg$ is the round metric on $\Sph^2$. The non-degeneracy condition \eqref{EqSdSNonDeg} ensures that $\mu_{b_0}(r)$ has exactly two positive simple roots $0<r_{b_0,-}<r_{b_0,+}<\infty$, and then the given form of the metric $g_{b_0}$ is valid for
\begin{equation}
\label{EqSdSStaticRad}
  r\in\cI:=(r_{b_0,-},r_{b_0,+}).
\end{equation}
See Figure~\ref{FigSdSPenrose1}.

\begin{figure}[!ht]
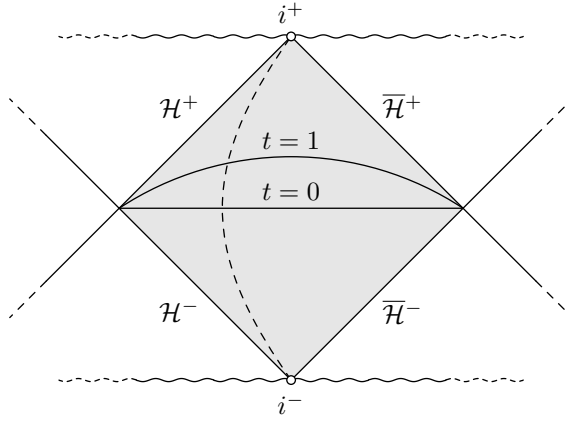

  \centering
  \inclfig{SdSPenrose1}
  \caption{Penrose diagram of Schwarzschild--de~Sitter space. The form \eqref{EqSdSMetric} of the metric is valid in the shaded region. Here, $\cH^\pm$ denotes the future/past event horizon, $\ol\cH{}^\pm$ the future/past cosmological horizon, and $i^\pm$ future/past timelike infinity. Also indicated are two level sets of the static time coordinate $t$, and a level set of $r$ (dashed).}
  \label{FigSdSPenrose1}
\end{figure}

The singularity of the expression \eqref{EqSdSMetric} at $r=r_{b_0,\pm}$ is resolved by a change of coordinates: we let
\begin{equation}
\label{EqSdSTStar}
  t_* = t - F_{b_0}(r),\quad F_{b_0}'(r)=\pm(\mu_{b_0}(r)^{-1}+c_{b_0,\pm}(r))\ \tn{near}\ r=r_{b_0,\pm},
\end{equation}
where $c_{b_0,\pm}(r)$ is smooth up to $r=r_{b_0,\pm}$; then
\begin{equation}
\label{EqSdSExt}
  g_{b_0} = \mu_{b_0}\,dt_*^2 \pm 2(1+\mu_{b_0} c_{b_0,\pm})\,dt_*\,dr + (2 c_{b_0,\pm} + \mu_{b_0} c_{b_0,\pm}^2)\,dr^2 - r^2\slg.
\end{equation}
It is easy to see that one can choose $c_{b_0,\pm}$ such that $dt_*$ is timelike up to $r=r_{b_0,\pm}$. In fact, there is a natural choice of $c_{b_0,\pm}$ making $|dt_*|^2_{G_{b_0}}$ constant, which will be convenient for computations later on:

\begin{lemma}
\label{LemmaSdSExt}
  Denote by $r_c:=\sqrt[3]{3\bhm[0]/\Lambda}$ the unique critical point of $\mu_{b_0}$ in $(r_{b_0,-},r_{b_0,+})$, and let $c_{t_*}=\mu_{b_0}(r_c)^{-1/2}=(1-\sqrt[3]{9\Lambda\bhm[0]^2})^{-1/2}$. Then
  \begin{equation}
  \label{EqSdSExtFPrime}
    F_{b_0}'(r)
      = \begin{cases}
          -\mu_{b_0}(r)^{-1}\sqrt{1-c_{t_*}^2\mu_{b_0}(r)}, & r<r_c, \\
          \mu_{b_0}(r)^{-1}\sqrt{1-c_{t_*}^2\mu_{b_0}(r)},  & r>r_c
        \end{cases}
  \end{equation}
  defines a smooth function $F_{b_0}(r)$ on $\cI$ up to an additive constant. Let $t_*=t-F_{b_0}(r)$. Then the metric $g_{b_0}$ and the dual metric $G_{b_0}$ are given by
  \begin{gather*}
    g_{b_0} = \mu_{b_0}\,dt_*^2 \pm 2\sqrt{1-c_{t_*}^2\mu_{b_0}}\,dt_*\,dr - c_{t_*}^2\,dr^2 - r^2\slg, \\
    G_{b_0} = c_{t_*}^2\,\pa_{t_*}^2 \pm 2\sqrt{1-c_{t_*}^2\mu_{b_0}}\,\pa_{t_*}\pa_r - \mu_{b_0}\,\pa_r^2 - r^{-2}\slG,
  \end{gather*}
  for $\pm(r-r_c)\geq 0$, where $\slG$ is the dual metric on $\Sph^2$. In particular, $|dt_*|^2_G\equiv c_{t_*}^2$.
\end{lemma}
\begin{proof}
  For $F_{b_0}$ as in \eqref{EqSdSTStar}, we have $|dt_*|_{G_{b_0}}^2 = -(2 c_{b_0,\pm}+\mu_{b_0} c_{b_0,\pm}^2)$. For this to be constant, $|dt_*|_{G_{b_0}}^2 = c_{t_*}^2$, with $c_{t_*}$ determined in the course of the calculation, the smoothness of the functions $c_{b_0,\pm}$ at $r_{b_0,\pm}$ forces
  \[
    c_{b_0,\pm} = \mu_{b_0}^{-1}(-1+\sqrt{1-c_{t_*}^2 \mu_{b_0}}).
  \]
  In order for the two functions $\pm(\mu_{b_0}^{-1}+c_{b_0,\pm})$ to be real-valued and to match up, together with all derivatives, at a point $r_c\in(r_{b_0,-},r_{b_0,+})$, we thus need to arrange that the function defined as $\pm\sqrt{1-c_{t_*}^2 \mu_{b_0}}$ in $\pm(r-r_c)>0$ is smooth and real-valued. This forces $r_c$ to be the unique critical point of $\mu_{b_0}$ on $(r_{b_0,-},r_{b_0,+})$, which is a non-degenerate maximum, and $c_{t_*}^2=\mu_{b_0}(r_c)^{-1}$; this in turn is also sufficient for the smoothness of \eqref{EqSdSExtFPrime}, and the lemma is proved.
\end{proof}

\begin{rmk}
\label{RmkSdSCpmOtherSide}
  Once we have chosen a function $F_{b_0}$, or rather its derivative $F_{b_0}'$, for instance as in the above lemma, then one can define $c_{b_0,-}\in\CI([r_{b_0,-},r_{b_0,+}))$, i.e.\ up to but excluding $r_{b_0,+}$, and likewise $c_{b_0,+}\in\CI((r_{b_0,-},r_{b_0,+}])$, by \eqref{EqSdSTStar}; that is, $c_{b_0,\pm}=\pm F_{b_0}'-\mu_{b_0}^{-1}$.
\end{rmk}

We can extend $c_{b_0,\pm}$ in an arbitrary manner smoothly (with the choice given in this lemma even uniquely by analyticity, but this is irrelevant) beyond $r=r_{b_0,\pm}$. We can now define the smooth manifold
\begin{equation}
\label{EqSdSMfd}
  M^\circ = \R_{t_*} \times X, \quad
  X=I_r \times \Sph^2,\quad I_r=(r_{I,-},r_{I,+}):=(r_{b_0,-}-3\eps_M,r_{b_0,+}+3\eps_M),
\end{equation}
for $\eps_M>0$ small, and $g_{b_0}$, defined by \eqref{EqSdSExt} up to and beyond $r=r_{b_0,\pm}$, is a smooth Lorentzian metric on $M^\circ$ satisfying Einstein's equations. See Figure~\ref{FigSdSPenrose2} (and also Figure~\ref{FigIntroBaby}). At the end of \S\ref{SubsecKdSSlow}, we will compactify $M^\circ$ at future infinity, obtaining a manifold $M$ with boundary.

\begin{figure}[!ht]
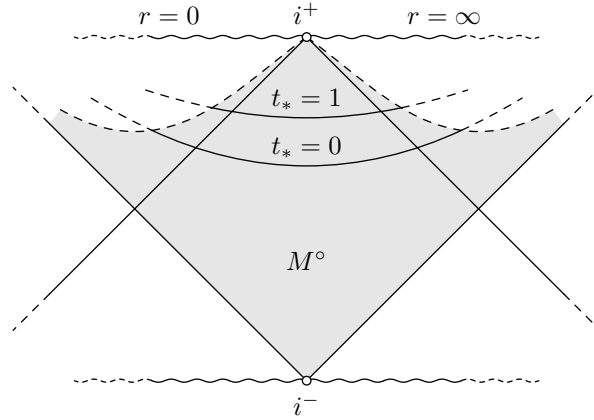

  \centering
  \inclfig{SdSPenrose2}
  \caption{The smooth manifold $M^\circ$ (shaded) within Schwarz\-schild--de~Sitter space, and two exemplary level sets of the timelike function $t_*$. The form \eqref{EqSdSExt} of the metric in fact extends beyond the dashed boundary of $M^\circ$, $r=r_\pm\pm 3\eps_M$, all the way up to (but excluding) the black hole singularity $r=0$ and the conformal boundary $r=\infty$ of the cosmological region.}
  \label{FigSdSPenrose2}
\end{figure}

Using the polar coordinate map, we can also view the spatial slice $X$ as
\begin{equation}
\label{EqSdSSpatialR3}
  X = \{p\in\R^3\colon r_{I,-}<|p|<r_{I,+}\} \subset \R^3.
\end{equation}
The static coordinate chart on $M^\circ$ (i.e.\ the dashed region in Figure~\ref{FigSdSPenrose1}) is the region
\begin{equation}
\label{EqSdSMfdStatic}
  \cM = \R_t \times \cX \subset M^\circ,\quad \cX = \cI \times \Sph^2,
\end{equation}
with $\cI$ defined in \eqref{EqSdSStaticRad}.

\subsection{Slowly rotating Kerr--de~Sitter black holes}
\label{SubsecKdSSlow}

Given Schwarzschild--de~Sitter parameters $b_0=(\bhm[0],\bfzero)\in B$ and the smooth manifold \eqref{EqSdSMfd} with time function $t_*\in\CI(M^\circ)$, we now proceed to define the Kerr--de~Sitter family of metrics, depending on the parameters $b\in B$, as a smooth family $g_b$ of stationary Lorentzian metrics on $M^\circ$ for $b$ close to $b_0$.

Given the angular momentum $a=|\bfa|$ of a black hole of mass $\bhm$ (spinning around the axis $\bfa/|\bfa|\in\R^3$ for $a\neq 0$), the Kerr--de~Sitter metric with parameters $b=(\bhm,\bfa)$ in \emph{Boyer--Lindquist coordinates} $(t,r,\phi,\theta)\in\R\times \cI_b\times\Sph^1_\phi\times(0,\pi)$, with $\cI_b\subset\R$ an interval defined below, takes the form
\begin{equation}
\label{EqKdSSlowStationary}
\begin{split}
  g_b &= -\rho_b^2\Bigl(\frac{dr^2}{\wt\mu_b}+\frac{d\theta^2}{\kappa_b}\Bigr) - \frac{\kappa_b\sin^2\theta}{(1+\lambda_b)^2\rho_b^2}(a\,dt-(r^2+a^2)\,d\phi)^2 \\
    &\qquad + \frac{\wt\mu_b}{(1+\lambda_b)^2\rho_b^2}(dt-a\sin^2\theta\,d\phi)^2,
\end{split}
\end{equation}
where
\begin{equation}
\label{EqKdSSlowFunctions}
\begin{gathered}
  \wt\mu_b(r) = (r^2+a^2)\Bigl(1-\frac{\Lambda r^2}{3}\Bigr) - 2\bhm r, \\
  \rho_b^2 = r^2+a^2\cos^2\theta, \quad \lambda_b = \frac{\Lambda a^2}{3}, \quad \kappa_b = 1+\lambda_b\cos^2\theta.
\end{gathered}
\end{equation}
For $a=0$, we have $\wt\mu_b(r)=r^2\mu_b(r)$, with $\mu_b$ defined in \eqref{EqSdSMetric}. For $a\neq 0$, the spherical coordinates $(\phi,\theta)$ are chosen such that the vector $\bfa/|\bfa|\in\Sph^2$ is defined by $\theta=0$, and the vector field $\pa_\phi$ generates counterclockwise rotation around $\bfa/|\bfa|$, with $\R^3$ carrying the standard orientation. Thus, for $\bfa=(0,0,a)$, $a>0$, the coordinates $(\phi,\theta)$ are the standard spherical coordinates of $\Sph^2\hra\R^3$. We note that, using these standard spherical coordinates, the expression for $g_{(\bhm,(0,0,-a))}$ is given by \eqref{EqKdSSlowStationary} with $a$ replaced by $-a$: this simply means that reflecting the angular momentum vector across the origin is equivalent to reversing the direction of rotation.

\begin{lemma}
\label{LemmaKdSSlowHorizon}
  Let $r_{b_0,-}<r_{b_0,+}$ denote the unique positive roots of $\wt\mu_{b_0}$. Then for $b$ in an open neighborhood $b_0\in\cU_B\subset B$, the largest two positive roots
  \[
    r_{b,-}<r_{b,+}
  \]
  of $\wt\mu_b$ depend smoothly on $b\in\cU_B$. In particular, for $b$ near $b_0$, we have
  \[
    |r_{b,\pm}-r_{b_0,\pm}| < \eps_M,
  \]
  and so $r_{b,\pm}\in I$, with $I$ defined in \eqref{EqSdSMfd}.
\end{lemma}
\begin{proof}
  This follows from the simplicity of the roots $r_{b_0,\pm}$ of $\wt\mu_{b_0}$ and the implicit function theorem.
\end{proof}

The interval $\cI_b$ in which the radial variable $r$ of the Boyer--Lindquist coordinate system takes values is then
\[
  \cI_b = (r_{b,-},r_{b,+}).
\]

The coordinate singularity of \eqref{EqKdSSlowStationary} is removed by a change of variables
\begin{equation}
\label{EqKdSSlowCoordChange}
  t_* = t - F_b(r),\quad \phi_* = \phi - \Phi_b(r),
\end{equation}
where $F_b,\Phi_b$ are smooth functions on $(r_{b,-},r_{b,+})$ such that
\begin{equation}
\label{EqKdSSlowCoordChange2}
  F_b'(r) = \pm\Bigl(\frac{(1+\lambda_b)(r^2+a^2)}{\wt\mu_b}+c_{b,\pm}\Bigr),
  \quad
  \Phi_b'(r) = \pm\Bigl(\frac{(1+\lambda_b)a}{\wt\mu_b} + \wt c_{b,\pm}\Bigr)
\end{equation}
near $r=r_{b,\pm}$, with $c_{b,\pm},\wt c_{b,\pm}$ smooth up to $r_{b,\pm}$. Here, for $b=b_0$ we take $c_{b_0,-}\in\CI((r_{I,-},r_{b_0,+}))$ and $c_{b_0,+}\in\CI((r_{b_0,-},r_{I,+}))$, with $r_{I,\pm}$ defined in \eqref{EqSdSMfd}, to be equal to any fixed choice for the Schwarzschild--de~Sitter space $(M,g_{b_0})$, e.g.\ the one in Lemma~\ref{LemmaSdSExt}.

\begin{lemma}
\label{LemmaKdSSlowCbeta}
  Fix radii $r_1,r_2$ with $r_{b_0,-}+\eps_M < r_1 < r_2 < r_{b_0,+}-\eps_M$. Then there exists a neighborhood $b_0\in\cU_B\subset B$ such that the following holds:
  \begin{enumerate}
    \item There exist smooth functions
      \begin{gather*}
        \cU_B\times(r_{I,-},r_2)\ni(b,r)\mapsto c_{b,-}(r), \\
        \cU_B\times(r_1,r_{I,+})\ni(b,r)\mapsto c_{b,+}(r),
      \end{gather*}
      which are equal to the given $c_{b_0,\pm}$ for $b=b_0$, such that the two functions
      \[
        \pm\Bigl(\frac{(1+\lambda_b)(r^2+a^2)}{\wt\mu_b}+c_{b,\pm}\Bigr)
      \]
      agree on $(r_1,r_2)$.
    \item There exist functions
      \begin{gather*}
        \cU_B\times(r_{I,-},r_2)\ni(b,r)\mapsto \wt c_{b,-}(r), \\
        \cU_B\times(r_1,r_{I,+})\ni(b,r)\mapsto \wt c_{b,+}(r),
      \end{gather*}
      with $a^{-1}\wt c_{b,\pm}$ smooth, and $\wt c_{b_0,\pm}\equiv 0$, such that the two functions
      \[
        \pm\Bigl(\frac{(1+\lambda_b)a}{\wt\mu_b}+\wt c_{b,\pm}\Bigr)
      \]
      agree on $(r_1,r_2)$.
  \end{enumerate}
\end{lemma}
\begin{proof}
  We can take $c_{b,-}\equiv c_{b_0,-}$ on $(r_{I,-},r_2)$; then, for a cutoff $\chi\in\CI(\R)$,  with $\chi\equiv 1$ on $[r_1,r_2]$ and $\chi\equiv 0$ on $[r_{b_0,+}-\eps_M,r_{I,+})$, we put
  \[
    c_{b,+} = -\Bigl(\frac{2(1+\lambda_b)(r^2+a^2)}{\wt\mu_b}+c_{b_0,-}\Bigr)\chi + c_{b_0,+}(1-\chi).
  \]
  A completely analogous construction works for $\wt c_{b,\pm}$: we can take $\wt c_{b,-}\equiv 0$ and put
  \[
    \wt c_{b,+}=-\frac{2(1+\lambda_b)a}{\wt\mu_b}\chi.
  \]
  Clearly, the functions $a^{-1}\wt c_{b,\pm}$ depend smoothly on $b$.
\end{proof}

This lemma ensures that the definitions on $F_b'$ and $\Phi_b'$ in the two regions in \eqref{EqKdSSlowCoordChange2} coincide, hence making $F_b$ and $\Phi_b$ well-defined up to an additive constant. Using $a F_b'-(r^2+a^2)\Phi_b' = \pm(a c_{b,\pm}-(r^2+a^2)\wt c_{b,\pm})$ and $F_b'-a\sin^2\theta\,\Phi_b'=\pm\bigl((1+\lambda_b)\rho_b^2/\wt\mu_b + c_{b,\pm}-a\sin^2\theta\,\wt c_{b,\pm}\bigr)$ for $r>r_1$ (`$+$' sign) or $r<r_2$ (`$-$' sign), one now computes
\begin{equation}
\label{EqKdSSlowGBeta}
\begin{split}
  g_b &= -\frac{\kappa_b\sin^2\theta}{(1+\lambda_b)^2\rho_b^2}\bigl(a(dt_*\pm c_{b,\pm}\,dr)-(r^2+a^2)(d\phi_*\pm\wt c_{b,\pm}dr)\bigr)^2 \\
    & \quad + \frac{\wt\mu_b}{(1+\lambda_b)\rho_b^2}\bigl(dt_*\pm c_{b,\pm}dr - a\sin^2\theta\,(d\phi_*\pm\wt c_{b,\pm}\,dr)\bigr)^2 \\
    & \quad \pm \frac{2}{1+\lambda_b}\bigl(dt_*\pm c_{b,\pm}dr - a\sin^2\theta\,(d\phi_*\pm\wt c_{b,\pm}\,dr)\bigr)\,dr - \frac{\rho_b^2}{\kappa_b}\,d\theta^2,
\end{split}
\end{equation}
which now extends smoothly to (and across) $r_{b,\pm}$. Since one can compute the volume form to be\footnote{For these calculations, a convenient frame of $TM^\circ$ is $v_1=\pa_r\mp c_{b,\pm}\pa_{t_*}\mp\wt c_{b,\pm}\pa_{\phi_*}$, $v_2=a\sin^2\theta\,\pa_{t_*}+\pa_{\phi_*}$, $v_3=\pa_{t_*}$, $v_4=\pa_\theta$.} $|dg_b|=(1+\lambda_b)^{-2}\rho_b^2\sin\theta\,dt_*\,dr\,d\phi_*\,d\theta$, the metric $g_b$ in the coordinates used in \eqref{EqKdSSlowGBeta} is a non-degenerate Lorentzian metric, apart from the singularity of the spherical coordinates at $\theta=0,\pi$, which we proceed to discuss: first, we compute the dual metric to be
\begin{equation}
\label{EqKdSSlowGBetaDual}
\begin{split}
  \rho_b^2 G_b &= -\wt\mu_b(\pa_r\mp c_{b,\pm}\pa_{t_*}\mp\wt c_{b,\pm}\pa_{\phi_*})^2 \\
   & \quad \pm 2 a(1+\lambda_b)(\pa_r\mp c_{b,\pm}\pa_{t_*}\mp\wt c_{b,\pm}\pa_{\phi_*})\pa_{\phi_*} \\
   & \quad \pm 2(1+\lambda_b)(r^2+a^2)(\pa_r\mp c_{b,\pm}\pa_{t_*}\mp\wt c_{b,\pm}\pa_{\phi_*})\pa_{t_*} \\
   & \quad -\frac{(1+\lambda_b)^2}{\kappa_b\sin^2\theta}(a\sin^2\theta\,\pa_{t_*}+\pa_{\phi_*})^2 - \kappa_b\,\pa_\theta^2.
\end{split}
\end{equation}
Smooth coordinates on $\Sph^2$ near the poles $\theta=0,\pi$ are $x=\sin\theta\cos\phi_*$ and $y=\sin\theta\sin\phi_*$, and for the change of variables $\zeta\,d\phi_*+\eta\,d\theta=\lambda\,dx+\nu\,dy$, one finds $\sin^2\theta=x^2+y^2$ and $\zeta=\nu x-\lambda y$, that is,
\[
  \pa_{\phi_*} = y\pa_x - x\pa_y,
\]
and thus the smoothness of $\rho_b^2 G_b$ near the poles follows from writing $-\kappa_b$ times the term coming from the last line of \eqref{EqKdSSlowGBetaDual} as
\[
  \frac{(1+\lambda_b)^2}{\sin^2\theta}\pa_{\phi_*}^2 + \kappa_b^2\pa_\theta^2 = (1+\lambda_b)^2(\sin^{-2}\theta\,\pa_{\phi_*}^2 + \pa_\theta^2) + (\kappa_b^2-(1+\lambda_b)^2)\pa_\theta^2.
\]
Indeed the first summand is smooth at the poles, since $\sin^{-2}\theta\,\pa_{\phi_*}^2+\pa_\theta^2=\slG$ is the dual metric of the round metric on $\Sph^2$ in spherical coordinates; and we can rewrite the second summand as
\begin{equation}
\label{EqKdSSlowGBetaDualSphPart}
  -(2+\lambda_b(1+\cos^2\theta))\lambda_b\sin^2\theta\,\pa_\theta^2
\end{equation}
and observe that $\sin^2\theta\,\pa_\theta^2=(1-x^2-y^2)(x\pa_x+y\pa_y)^2$ is smooth at $(x,y)=0$ as well. Since the volume form is given by
\[
  |dg_b|=(1+\lambda_b)^{-2}(r^2+a^2(1-x^2-y^2))(1-x^2-y^2)^{-1/2}\,dt_*\,dr\,dx\,dy
\]
and thus smooth at the poles, we conclude that $g_b$ indeed extends smoothly and non-degenerately to the poles.

Using the map
\[
  (t_*,r,\phi_*,\theta) \mapsto (t_*, r\sin\theta\cos\phi_*, r\sin\theta\sin\phi_*, r\cos\theta) \in \R_{t_*}\times X \subset M^\circ,
\]
with $X\subset\R^3$ as in \eqref{EqSdSSpatialR3}, we can thus push the metric $g_b$ forward to a smooth, stationary, non-degenerate Lorentzian metric, which we continue to denote by $g_b$, on $M^\circ$, and $g_{b_0}$ is equal (pointwise!) to the extended Schwarzschild--de~Sitter metric defined in \S\ref{SubsecSdS}. The Boyer--Lindquist coordinate patch of the Kerr--de~Sitter black hole with parameters $b\in B$ is the subset $\{r_{b,-}<r<r_{b,+}\}\subset M^\circ$.

Since the choice of spherical coordinates does not depend smoothly on $\bfa$ near $\bfa=0$, the smooth dependence of $g_b$, as a family of metrics on $M^\circ$, on $b$ is not automatic; we thus prove:

\begin{prop}
\label{PropKdSSlowSmooth}
  Let the neighborhood $\cU_B\subset B$ of $b_0$ be as in Lemma~\ref{LemmaKdSSlowHorizon}. Then the smooth Lorentzian metric $g_b$ on $M^\circ$ depends smoothly on $b\in B$.
\end{prop}

Here, the smoothness of the family $g_b$ is equivalent to the statement that the map
\[
  B \times \R_{t_*}\times X \ni (b,t_*,p) \mapsto g_b(t_*,p),
\]
with $g_b(t_*,p)$ on the right the matrix of $g_b$ at the given point in the global coordinate system $(t_*,p)\in M^\circ$, is a smooth map $\cU_B\times\R_{t_*}\times X\to\R^{4\times 4}$.

\begin{proof}[Proof of Proposition~\ref{PropKdSSlowSmooth}]
  Given $(\bhm,\bfa)\in B$ with $a=|\bfa|\neq 0$, let us denote the spherical coordinate system on $X$ with north pole $\theta=0$ at $\bfa/|\bfa|$ by $\phi_{b,*},\theta_b$, so the pushforwards of the functions in \eqref{EqKdSSlowFunctions} to $M^\circ$ are simply obtained by replacing $\theta$ by $\theta_b$. Then, if $(r,\phi_{b,*},\theta_b)$ are the polar coordinates of a point $p\in X$, we have
  \[
    r = |p|,\quad a\cos\theta_b=\Big\la\bfa,\frac{p}{|p|}\Big\ra,\quad a^2\sin^2\theta_b = |\bfa|^2-a^2\cos\theta_b^2,
  \]
  where $|\cdot|$ and $\la\cdot,\cdot\ra$ denote the Euclidean norm and inner product on $X\subset\R^3$, respectively. Since $0\notin X$, this shows that $r$, $a\cos\theta_b$ and $a^2\sin^2\theta_b$, and hence the pushforwards of $\wt\mu_b$, $\rho_b$, $\lambda_b$, and $\kappa_b$, are smooth (in $b$) families of smooth functions on $M^\circ$, as are $c_{b,\pm}$ and $a^{-1}\wt c_{b,\pm}$ (which only depend on $r$), on their respective domains of definition, by Lemma~\ref{LemmaKdSSlowCbeta}.

  We can now prove the smooth dependence of the dual metric $G_b$ on $b$: in light of the expression \eqref{EqKdSSlowGBetaDual} and the discussion around \eqref{EqKdSSlowGBetaDualSphPart}, all we need to show is that the vector fields $\pa_{t_*}$, $\pa_r$, $a\,\pa_{\phi_{b,*}}$ and $a\sin\theta_b\,\pa_{\theta_b}$ depend smoothly on $\bfa$, in particular near $\bfa$ where they are defined to be identically $0$. (Note that the 2-tensor in \eqref{EqKdSSlowGBetaDualSphPart} is a smooth multiple of $a^2\sin^2\theta_b\,\pa_{\theta_b}^2$.) Indeed, this proves that $G_b$ is a smooth family of smooth sections of $S^2 T M^\circ$, and we already checked the non-degeneracy of $G_b$ at the poles, where the spherical coordinates are singular.

  For $\pa_{t_*}$ and for the radial vector field $\pa_r=|p|^{-1}p\pa_p$, which do not depend on $b$, the smoothness is clear. Further, we have
  \[
    a\,\pa_{\phi_{b,*}} = \nabla_{\bfa\times p}\quad\tn{at }p\in X,
  \]
  i.e.\ differentiation in the direction of the vector $\bfa\times p$. Indeed, if $\bfa=a\vec e_3:=(0,0,a)$, both sides equal $a(x\pa_y-y\pa_x)$ on $\R^3_{x,y,z}$, and if $\bfa\in\R^3$ is any given vector and $R\in SO(3)$ is a rotation with $R\bfa=a\vec e_3$, $a=|\bfa|$, then $(R_*(a\pa_{\phi_{b,*}}))|_p=(a\pa_{\phi_{a\vec e_3,*}})|_{R(p)}$, which we just observed to be equal to $\nabla_{a\vec e_3\times R(p)}=R_*\nabla_{\bfa\times p}$.

  In a similar vein, one sees that
  \[
    a\sin\theta_b\,\pa_{\theta_b} = |p|^{-1}\nabla_{p\times(p\times\bfa)}\quad\tn{at }p\in X,
  \]
  and the latter expression is clearly smooth in $\bfa$, finishing the proof of the proposition.
\end{proof}

\begin{rmk}
\label{RmkKdSSlowNonSdS}
  If one were interested in analyzing the non-linear stability of Kerr--de~Sitter spacetimes for general parameters---i.e.\ dropping the assumption of small angular momentum---one notes that the construction described in this section can be performed in the neighborhood of any Kerr--de~Sitter spacetime which is non-degenerate in the sense that the two largest roots of $\wt\mu_b$ are simple and positive.
\end{rmk}

Given the smooth family of Kerr--de~Sitter metrics $g_b$ on $M^\circ$ defined in the previous section, we can define its linearization around any $g_b$, $b\in\cU_B$.

\begin{definition}
\label{DefKdSSlowLinMet}
  For $b\in\cU_B$ and $b'\in T_b B$, we define the element $g'_b(b')$ of the linearized Kerr--de~Sitter family, linearized around $g_b$, by
  \[
    g'_b(b') = \frac{d}{ds} g_{b+s b'}|_{s=0}.
  \]
\end{definition}

Notice here that $\cU_B\subset B\subset\R^4$ is an open subset of $\R^4$, and thus we can identify $T_b B=\R^4$. The linearization of the Kerr--de~Sitter family around $g_b$ is the 4-dimensional vector space $g'_b(T_b B)\equiv \{g'_b(b')\colon b'\in T_b B\}$.

\begin{rmk}
\label{RmkKdSSlowLinMetDegFreedom}
  Define $d_b$ to be the number of parameters needed to describe a linearized Kerr--de~Sitter metric modulo Lie derivatives (i.e.\ the number of `physical degrees of freedom'); that is,
  \[
    \Gamma_b = \frac{g'_b(T_b B)}{\ran\delta_{g_b}^* \cap g'_b(T_b B)}, \quad d_b = \dim\Gamma_b.
  \]
  Then one can show that $d_b=4$ if $\bfa=\bfzero$, and $d_b=2$ if $\bfa\neq\bfzero$. See also the related discussion at the end of \cite[\S6.2.2]{DafermosHolzegelRodnianskiSchwarzschildStability}. The reason for $d_b=4$ for Schwarzschild--de~Sitter parameters $b$ is that slowly rotating Kerr--de~Sitter metrics with rotation axes which are far apart are related by a rotation by a large angle. This is one of the reasons to use the redundant (for non-zero angular momenta) parameterization \eqref{EqKdSB} of the Kerr--de~Sitter family.
\end{rmk}

\subsection{Compactification}
\label{SubsecKdSCpt}

In order to make full use of the asymptotic structure of Kerr--de~Sitter spaces, it is convenient to compactify the spacetime $M^\circ$, defined in \eqref{EqSdSMfd}, at future infinity: we define
\[
  \tau := e^{-t_*}
\]
and put
\[
  M = \bigl(M^\circ \sqcup ([0,\infty)_\tau\times X)\bigr) / \sim, \quad (t_*,x)\sim (\tau=e^{-t_*},x),
\]
which is a smooth manifold with boundary, where the smooth structure is defined such that $\tau$ is a boundary defining function, i.e.\ $\tau\in\CI(M)$ vanishes simply at $\tau=0$. Thus,
\begin{equation}
\label{EqKdSSlowCompProd}
  M\cong [0,\infty)_\tau\times X
\end{equation}
as manifolds with boundary. We often regard $X$ as the boundary at infinity of $M$. On $M$, we can now use the natural bundles $\Tb M$ (the \emph{b-tangent bundle}), $\Tb^*M$ (the \emph{b-cotangent bundle}), and their tensor powers. We refer the reader to Appendix~\ref{SecB} for precise definitions of these objects. The bundles $\Tb M$ and $\Tb^*M$ are naturally isomorphic to the usual translation-invariant (in $t_*$) tangent/cotangent bundle on $M^\circ$. Crucially, they extend smoothly to $X$; analysis on $M$ near $X$, or b-microlocal analysis in $\Tb^*M$ near $\Tb_X^*M$, is thus \emph{automatically} (and necessarily) uniform, i.e.\ gives uniform control on $M^\circ$ as $t_*\to\infty$. Structures central to understanding waves uniformly as $t_*\to\infty$, such as horizons or trapping, arise naturally as submanifolds of $\Tb_X^*M$ at which the null-geodesic flow (lifted to the cotangent bundle) has a special structure (invariant manifolds, saddle points of the flow, etc).

The use of the compactification, a common technique in geometric analysis, means that we do not need to repeatedly refer to the asymptotic structure when making statements about $t_*\to\infty$, though one could equivalently do that without introducing the compactification. Thus, to some extent, compactifying is a matter of taste, but it provides a very convenient language.

Note that smooth functions on $M$ are smooth functions of $(\tau,x)=(e^{-t_*},x)$ down to $\tau=0$, hence they have Taylor expansions at $\tau=0$ into powers of $\tau=e^{-t_*}$; in particular, they are invariant under translations in $t_*$ up to a remainder which decays exponentially fast as $t_*\to\infty$. Since $\pa_{t_*}=-\tau\pa_\tau$ and $dt_*=-\frac{d\tau}{\tau}$, and since functions on $M^\circ$ which are constant in $t_*$ extend to smooth functions on $M$, we can thus rephrase Proposition~\ref{PropKdSSlowSmooth} as follows:

\begin{prop}
\label{PropKdSSlowCompSmooth}
  With $\cU_B\subset B$ a neighborhood of $b_0$ as in Lemma~\ref{LemmaKdSSlowCbeta}, the family $g_b$ of Kerr--de~Sitter metrics with $b\in B$ is a smooth family of non-degenerate signature $(1,3)$ sections of $\CI(M;S^2\,\Tb^*M)$.
\end{prop}

The stationary nature of $g_b$ can be recast in this setting as the invariance of $g_b$ with respect to dilations in $\tau$ in the product decomposition \eqref{EqKdSSlowCompProd} of $M$.

\subsection{Geometric and dynamical aspects of Kerr--de~Sitter spacetimes}
\label{SubsecKdSGeo}

The dual metric function $G_b\in\CI(\Tb^*M)$ is defined by $G_b(z,\zeta)=|\zeta|_{(G_b)_z}^2$ for $z\in M$, $\zeta\in\Tb^*_z M$, and the characteristic set is
\begin{equation}
\label{EqKdSGeoChar}
  \Sigma_b = G_b^{-1}(0) \subset \Tb^*M\setminus o,
\end{equation}
which is conic in the fibers of $\Tb^*M$. We occasionally identify $\Sigma_b$ with its closure in the radially compactified (in the fibers) b-cotangent bundle, minus the zero section $o$, so $\Sigma_b\subset\ol{\Tb^*}M\setminus o$. See~\eqref{EqBCompact} for the definition of $\ol{\Tb^*}M$. Sometimes, we also identify $\Sigma_b$ with its boundary at fiber infinity $\pa\Sigma_b\subset\Sb^*M\subset\ol{\Tb^*}M$. Since by construction $dt_*=-d\tau/\tau$ is timelike everywhere on $M$, we can split the characteristic set into its two connected components
\[
  \Sigma_b = \Sigma_b^+ \sqcup \Sigma_b^-,
\]
the future, resp.\ backward, light cone $\Sigma_b^+$, resp.\ $\Sigma_b^-$, where
\begin{equation}
\label{EqKdSGeoCharPM}
  \Sigma_b^\pm = \Bigl\{ \zeta\in\Sigma_b \colon \pm\Big\la \zeta,-\frac{d\tau}{\tau}\Big\ra_{G_b} > 0 \Bigr\}
\end{equation}
\emph{The sign in the superscript will always indicate the future/past component of the characteristic set.}

We now discuss two main features of the null-geodesic flow on $(M,g_b)$: the saddle point structure (corresponding to the red-shift effect) at the event and cosmological horizons of Kerr--de~Sitter spacetimes, and the trapped set for Schwarzschild--de~Sitter, i.e.\ the photon sphere. The global dynamics of the null-geodesic flow of slowly rotating Kerr--de~Sitter spacetimes were described in detail in \cite[\S\S6.3--6.4]{VasyMicroKerrdS}. Here, we merely recall that for a future-causal null-bicharacteristic $\gamma$, i.e.\ an integral curve of the Hamilton vector field $H_{G_b}$ within $\Sigma_b^+$, one of the following three possibilities must occur in the \emph{backward} direction along $\gamma$: either
\begin{enumerate}
\item $\gamma$ tends to the trapped set (strictly speaking, the boundary of $\Gamma$, defined in~\eqref{EqKdSGeoTrapped}, within $\Tb^*M$ at future infinity), or
\item $\gamma$ tends to one of the radial points $\cR_{b,-}$ or $\cR_{b,+}$ (the b-conormal bundle of the event or cosmological horizon at future infinity), defined in~\eqref{EqKdSGeoRadSet}, or
\item $\gamma$ crosses the initial Cauchy hypersurface $\{t_*=0\}$ in finite time.
\end{enumerate}

We first describe the null-geodesic flow on $(M,g_b)$ near the horizons: defining $t_0$ and $\phi_0$ near $r=r_{b,\pm}$ exactly like $t_*$ and $\phi_*$ in \eqref{EqKdSSlowCoordChange}, except with $c_{b,\pm}=\wt c_{b,\pm}\equiv 0$ simplifies our calculations: introducing smooth coordinates on $\Tb^*M$ by letting $\tau_0=e^{-t_0}$ and writing b-covectors over the point $(\tau_0,r,\omega)\in M$ as
\[
  \sigma\,\frac{d\tau_0}{\tau_0} + \xi\,dr + \zeta\,d\phi_0 + \eta\,d\theta
\]
the dual metric function can be read off from \eqref{EqKdSSlowGBetaDual} by taking $c_{b,\pm}=\wt c_{b,\pm}=0$, and replacing $\pa_{t_*}$, $\pa_r$, $\pa_{\phi_*}$ and $\pa_\theta$ by $-\sigma$, $\xi$, $\zeta$ and $\eta$, respectively, so
\begin{equation}
\label{EqKdSGeoGb}
\begin{split}
  \rho_b^2 G_b &= -\wt\mu_b \xi^2 \pm 2 a(1+\lambda_b)\xi\zeta \mp 2(1+\lambda_b)(r^2+a^2)\xi\sigma \\
    &\quad - \frac{(1+\lambda_b)^2}{\kappa_b\sin^2\theta}(a\sin^2\theta\,\sigma-\zeta)^2 - \kappa_b\eta^2
\end{split}
\end{equation}
Denote the b-conormal bundles of the horizons by
\begin{equation}
\label{EqKdSGeoHorizonBConormal}
  \cL_{b,\pm} := \Nb^*\{r=r_{b,\pm}\} \setminus o = \{ (\tau_0,r_{b,\pm},\phi_0,\theta; 0,\xi,0,0) \} \setminus o \subset \Sigma_b.
\end{equation}
\emph{The sign in the subscript will always indicate the horizon at which one is working, `$-$' denoting the event horizon at $r=r_{b,-}$ and `$+$' the cosmological horizon at $r=r_{b,+}$.} One easily checks that the vector field $\ham_{\rho_b^2 G_b}=\rho_b^2\ham_{G_b}$ (equality holding on $\Sigma_b$) is tangent to $\cL_{b,\pm}$, so $\cL_{b,\pm}$ is invariant under the $\ham_{G_b}$-flow. Since $\la dr,dr\ra_{G_b}=0$ and $\rho_b^2\la dt_*,dr\ra_{G_b}=\pm(1+\lambda_b)(r_{b,\pm}^2+a^2)$ at $r=r_{b,\pm}$, $dr$ is future lightlike at $\cL_{b,+}$ and past lightlike at $\cL_{b,-}$. This allows us to locate the components of $\cL_{b,\pm}$ in the two halves of the characteristic set. To wit, if
\[
  \cL_{b,\pm}^+ = \cL_{b,\pm} \cap \Sigma_b^+,\quad \cL_{b,\pm}^- = \cL_{b,\pm} \cap \Sigma_b^-,
\]
adhering to our rule that signs in superscript indicate being a subset of the future/past light cone, then
\[
  \cL_{b,-}^\pm = \{ \pm\xi < 0 \} \cap \cL_{b,-}, \quad \cL_{b,+}^\pm = \{ \pm\xi > 0 \} \cap \cL_{b,+}.
\]
Let
\[
  \cL_b^+ = \cL_{b,-}^+ \cup \cL_{b,+}^+ \subset \Sigma_b^+, \quad \cL_b^- = \cL_{b,-}^- \cup \cL_{b,+}^- \subset \Sigma_b^-
\]
be the components of $\cL=\cL_{b,+}\cup\cL_{b,-}$ within the future/past light cone $\Sigma_b^\pm$.  

Let us identify $X$ with the boundary at future infinity $\{0\}_\tau\times X\subset M$, see~\eqref{EqKdSSlowCompProd}. We then define the boundaries at future infinity of the above invariant manifolds by
\begin{equation}
\label{EqKdSGeoRadSet}
  \cR^{(\pm)}_{b,(\pm)}=\cL^{(\pm)}_{b,(\pm)} \cap \Tb^*_X M;
\end{equation}
each of them itself is invariant under the $\ham_{G_b}$-flow. Moreover, we denote the boundaries of $\cR^{(\pm)}_{b,(\pm)}$ at fiber infinity by $\pa\cR^{(\pm)}_{b,(\pm)}\subset\Sb^*_X M$, likewise for $\cL^{(\pm)}_{b,(\pm)}$. See Figure~\ref{FigKdSGeoRadial}.

\begin{figure}[!ht]
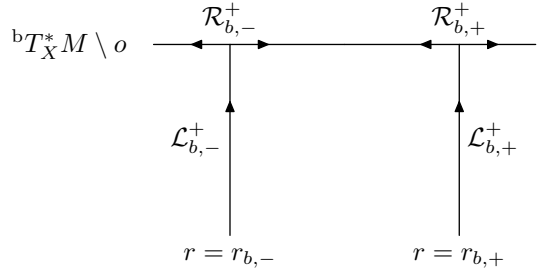

  \centering
  \inclfig{KdSGeoRadial}
  \caption{The b-conormal bundles $\cL_\pm^+$ of the horizons as well as their boundaries $\cR_\pm^+$ in $\Tb^*_X M\setminus o$. The arrows indicate the Hamilton vector field $\ham_{G_b}$. In $\Gamma^-$, the subscripts are replaced by `$-$', and the directions of the arrows are reversed.}
  \label{FigKdSGeoRadial}
\end{figure}

We claim that the \emph{(generalized) radial sets} $\pa\cR_b^\pm$ are \emph{saddle points} for (a rescaled version of the) null-geodesic flow $\ham_{G_b}$. Concretely, $\pa\cR_b^+$ has stable manifold $\pa\cL_b^+\subset\Sb^*M$ transversal to $\Sb^*_X M$, with unstable manifold $\Sigma_b^+\cap\ol{\Tb^*_X}M$ within the boundary at future infinity, i.e.\ it is a source within $\ol{\Tb^*_X}M$; on the other hand, $\pa\cR_b^-$ has unstable manifold $\pa\cL_b^-\subset\Sb^*M$ and stable manifold $\Sigma_b^-\cap\ol{\Tb^*_X}M$, i.e.\ it is a sink within $\ol{\Tb^*_X}M$. To verify this claim, let us introduce coordinates
\begin{equation}
\label{EqKdSGeoCoords}
  \wh\rho = |\xi|^{-1}, \quad \wh\sigma=\wh\rho\sigma, \quad \wh\zeta=\wh\rho\zeta, \quad \wh\eta=\wh\rho\eta
\end{equation}
of $\ol{\Tb^*}M$ near $\pa\cR_b^\pm$, and define the rescaled Hamilton vector field
\begin{equation}
\label{EqKdSGeoHamRescaled}
  \rham_{\rho_b^2 G_b} := \wh\rho\ham_{\rho_b^2 G_b},
\end{equation}
which is homogeneous of degree $0$ with respect to dilations in the fibers of $\Tb^*M$, hence extends to a smooth vector field on $\ol{\Tb^*}M$ tangent to $\Sb^*M$. Let us consider the $\rham_{G_b}$-flow near $\pa\cR_b^+$ first, where $\wh\sigma=\wh\zeta=\wh\eta=0$. There, with `$\pm$' denoting the component of the characteristic set (i.e.\ $\pm\xi>0$), one computes
\begin{gather*}
  \wh\rho^{-1}\rham_{\rho_b^2 G_b}\wh\rho = \ham_{\rho_b^2 G_b}|\xi|^{-1} = \pm\xi^{-2}\pa_r(\rho_b^2 G_b) = \mp\wt\mu_b'(r_{b,+}) = \pm|\wt\mu_b'(r_{b,+})|, \\
  \tau_0^{-1}\rham_{\rho_b^2 G_b}\tau_0 = \wh\rho \pa_\sigma(\rho_b^2 G_b) = \mp 2(1+\lambda_b)(r_{b,+}^2+a^2).
\end{gather*}
The calculation at $\cR_b^-$ is completely analogous. Defining
\begin{equation}
\label{EqKdSGeoRadPointQuant}
  \beta_{b,\pm,0} := \frac{|\wt\mu_b'(r_{b,\pm})|}{\rho_b^2}, \quad \beta_{b,\pm} := \frac{2(1+\lambda_b)(r_{b,\pm}^2+a^2)}{|\wt\mu_b'(r_{b,\pm})|},
\end{equation}
and using that $\log\tau_0-\log\tau$ is a function of $r$ only, while $\ham_{G_b}r=0$ at $\cL_{b,\pm}$, we thus find
\begin{equation}
\label{EqKdSGeoRadPoint}
\begin{gathered}
  \wh\rho^{-1}\ham_{G_b}\wh\rho = \beta_{b,\pm,0}, \quad -\tau^{-1}\ham_{G_b}\tau = \beta_{b,\pm,0}\beta_{b,\pm}\quad \tn{at }\cR_{b,\pm}^+, \\
  \wh\rho^{-1}\ham_{G_b}\wh\rho = -\beta_{b,\pm,0}, \quad -\tau^{-1}\ham_{G_b}\tau = -\beta_{b,\pm,0}\beta_{b,\pm}\quad \tn{at }\cR_{b,\pm}^-.
\end{gathered}
\end{equation}
We note that the functions $\beta_{b,\pm,0}$ are functions on the 2-spheres $\tau=0,r=r_{b,\pm}$; if $b\in B$ are Schwarzschild--de~Sitter parameters, so $a=0$, then they are in fact constants.

In order to finish the proof of the saddle point structure of the flow in the normal directions at $\pa\cR_b$, it suffices to note that, using \eqref{EqKdSGeoGb}, the function
\[
  \rho_0 := \wh\rho^2\Bigl(\frac{(1+\lambda_b)^2}{\kappa_b\sin^2\theta}(a\sin^2\theta\,\sigma-\zeta)^2 + \kappa_b\eta^2 + \sigma^2\Bigr),
\]
which is smooth on $\ol{\Tb^*}M$, satisfies $\ham_{G_b}\rho_0=\pm 2\beta_{b,\pm,0}\rho_0$ at $\cR^\pm_b$ by \eqref{EqKdSGeoRadPoint}, i.e.\ with the same sign as $\wh\rho\ham_{G_b}\wh\rho$; and $\rho_0$ is a quadratic defining function of $\pa\cR_b$ within $\pa\Sigma_b\cap\Sb^*_X M$. (See \cite[\S3.2]{HintzVasyCauchyHorizon} for further details.)

We next discuss the trapped set in the exterior region of a Schwarzschild--de~Sitter spacetime with parameters $b_0=(\bhm[0],\bfzero)$. Writing covectors in static coordinates as
\begin{equation}
\label{EqKdSGeoTrapCovec}
  -\sigma\,dt + \xi\,dr + \eta,\quad \eta\in T^*\Sph^2,
\end{equation}
the dual metric function is given by
\[
  G_{b_0} = \mu_{b_0}^{-1}\sigma^2 - \mu_{b_0}\xi^2 - r^{-2}|\eta|^2,
\]
so the Hamilton vector field equals
\[
  \ham_{G_{b_0}} = -2\mu_{b_0}^{-1}\sigma\pa_t - 2\mu_{b_0}\xi\pa_r - r^{-2}\ham_{|\eta|^2} + (\mu_{b_0}^{-2}\mu_{b_0}'\sigma^2+\mu_{b_0}'\xi^2-2r^{-3}|\eta|^2)\pa_\xi.
\]
Now, in $\mu_{b_0}>0$ and within the characteristic set $\Sigma_{b_0}$, we have $\ham_{G_{b_0}} r=-2\mu_{b_0}\xi=0$ iff $\xi=0$; for $\xi=0$, hence $\mu_{b_0}^{-1}\sigma^2=r^{-2}|\eta|^2$, we then have $\ham_{G_{b_0}}^2 r=-2\mu_{b_0}\ham_{G_{b_0}}\xi=0$ iff $\mu_{b_0}^{-2}\mu_{b_0}'\sigma^2=2r^{-3}|\eta|^2$, which is equivalent to
\[
  0=(\mu_{b_0} r^{-2})' = 2r^{-4}(3\bhm[0]-r),
\]
hence the radius of the photon sphere is $r_P=3\bhm[0]$, and the trapped set in phase space $T^*M^\circ\setminus o$ is
\begin{equation}
\label{EqKdSGeoTrapped}
  \Gamma = \{ (t,r_P,\omega; \sigma,0,\eta) \colon \sigma^2=\mu_{b_0} r^{-2}|\eta|^2 \},
\end{equation}
The trapped set has two components,
\begin{equation}
\label{EqKdSGeoTrappedPM}
  \Gamma = \Gamma^-\cup\Gamma^+,\quad \Gamma^\pm = \Gamma\cap\Sigma_{b_0}^\pm.
\end{equation}
At $\Gamma$, we have
\begin{equation}
\label{EqKdSGeoTrappedHam}
  \ham_{G_{b_0}}=-2\mu_{b_0}^{-1}\sigma\pa_t-r^{-2}\ham_{|\eta|^2}.
\end{equation}

\subsection{Wave equations on Kerr--de~Sitter spacetimes}
\label{SubsecKdSWave}

Within $M$, we next single out a domain $\Omega\subset M$ on which we will solve various wave equations, in particular Einstein's field equations, in the course of our arguments. Concretely, let
\[
  Y = [r_{b_0,-}-\eps_M,r_{b_0,+}+\eps_M]_r \times \Sph^2 \subset X
\]
be a smoothly bounded compact domain in the spatial slice $X$, and let
\begin{equation}
\label{EqKdSWaveDomain}
  \Omega = [0,1]_\tau\times Y \subset M, \quad \Sigma_0 = \{\tau=1\}\cap\Omega.
\end{equation}
Then $\Omega$ is a submanifold with corners of $M$. We will often identify $Y$, which is isometric to each spatial slice $\{\tau=c\}\cap\Omega$, $c\in[0,1]$, with the boundary of $\Omega$ at future infinity, so $Y=\{\tau=0\}\cap\Omega$. See Figure~\ref{FigKdSWave}. We discuss function spaces on domains such as $\Omega$ at the end of \S\ref{SubsecBPsdo}.

\begin{figure}[!ht]
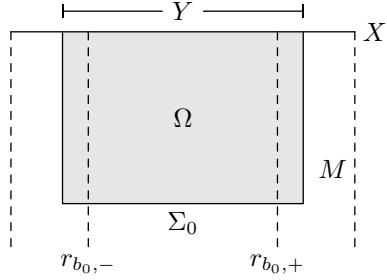

  \centering
  \inclfig{KdSWave}
  \caption{The domain $\Omega$ (shaded), with its boundary $Y=\Omega\cap X$ at future infinity, as a smooth domain with corners within $M$. The Cauchy surface is $\Sigma_0$, extending a bit beyond the horizons of slowly rotating Kerr--de~Sitter spacetimes.}
  \label{FigKdSWave}
\end{figure}

By a slight abuse of notation, we define the `finite part' of $\Omega$ as
\[
  \Omega^\circ = [0,\infty)_{t_*}\times Y,
\]
so $\Omega^\circ$ still contains the initial hypersurface $\Sigma_0=\{t_*=0\}$.

By construction of the function $t_*$, the surface $\Sigma_0$ is spacelike with respect to all metrics $g_b$, $b\in\cU_B$. Furthermore, we claim that the lateral boundary $\R_{t_*}\times\pa Y$, which has two components (one beyond the event and one beyond the cosmological horizon), is spacelike: indeed, this follows from \eqref{EqKdSSlowGBetaDual} and
\[
  \rho_b^2 G_b(dr,dr)=-\wt\mu_b > 0
\]
there; more precisely, by \eqref{EqKdSGeoCharPM}, the outward pointing conormal $\pm dr$ is future timelike at $r=r_{b,\pm}\pm\eps_M$ for sufficiently small $\eps_M$.

The Cauchy data of a function $u\in\CI(\Omega^\circ;E)$, which is a section of some tensor bundle $E=\bigotimes^k T^*\Omega^\circ$, are defined by
\begin{equation}
\label{EqKdSWaveInitialData}
  \gamma_0(u) := (u|_{\Sigma_0},(\cL_{\pa_{t_*}}u)|_{\Sigma_0}) \in \CI(\Sigma_0;E_{\Sigma_0}) \oplus \CI(\Sigma_0;E_{\Sigma_0}).
\end{equation}

Given a linear operator $L\in\Diff^2(M^\circ;E)$, with smooth coefficients, whose principal symbol is scalar and equal to $G_b\otimes\Id$ for some parameters $b\in\cU_B$, one can then study initial value problems for $L$. Using energy estimates, see e.g.\ \cite[\S6.5]{TaylorPDE} or \cite[\S23]{HormanderAnalysisPDE3}, and also \S\ref{SubsubsecAsyExpLocal}, one can show that
\[
  \begin{cases}
    L u = f & \tn{in }\Omega^\circ, \\
    \gamma_0(u) = (u_0,u_1) & \tn{on }\Sigma_0,
  \end{cases}
\]
with forcing $f\in\CI(\Omega^\circ;E)$ and initial data $u_0,u_1\in\CI(\Sigma_0;E_{\Sigma_0})$, has a unique solution $u\in\CI(\Omega^\circ;E)$. The future timelike nature of the outward pointing conormal at the lateral boundary of $\Omega^\circ$ ensures that no boundary data need to be specified there.

\subsection{Kerr--de~Sitter type wave map gauges}
\label{SubsecKdSIn}

Let us fix the background metric for the wave map gauge 1-form \eqref{EqHypDTGaugeTerm} to be the fixed Schwarzschild--de~Sitter metric $g_{b_0}$, so
\begin{equation}
\label{EqKdSInGauge}
  \Ups(g) = g g_{b_0}^{-1} \delta_g \sfG_g g_{b_0}.
\end{equation}
Now $g_{b_0}$ satisfies the gauge condition $\Ups(g_{b_0})=0$, but the particular form of the metrics $g_b$, $b\neq b_0$, we constructed is rather arbitrary, so in general we expect $\Ups(g_b)\neq 0$ for $b\neq b_0$. Therefore, we need to study more flexible gauges; a natural candidate is $\Ups(g)-\Ups(g_b)=0$, but for the formulation of the non-linear stability problem for the Einstein equation \eqref{EqHypDTEinstein}, one would like to have a fixed gauge near $\Sigma_0$ (we will choose the wave map gauge relative to $g_{b_0}$), yet a gauge relative to the metric $g_b$ of the final state. To implement such gauges, fix a cutoff function
\begin{equation}
\label{EqKdSInCutoff}
  \chi\in\CI(\R), \quad \chi(t_*)\equiv 0\tn{ for }t_*\leq 1,\quad \chi(t_*)\equiv 1\tn{ for }t_*\geq 2,
\end{equation}
and define
\begin{equation}
\label{EqKdSInPatchedMetric}
  g_{b_1,b_2} := (1-\chi)g_{b_1} + \chi g_{b_2},
\end{equation}
which is a Lorentzian metric for $b_1,b_2\in\cU_B$ smoothly interpolating between $g_{b_1}$ and $g_{b_2}$. We will then consider the gauge condition
\[
  \Ups(g) - \Ups(g_{b_1,b_2}) = 0.
\]
For $b_1=b_2=b$, this becomes $\Ups(g)-\Ups(g_b)=0$.

Denote the Kerr--de~Sitter initial data by
\begin{equation}
\label{EqKdSInKdSData}
  (h_b,k_b) \in \CI(\Sigma_0;S^2 T^*\Sigma_0) \oplus \CI(\Sigma_0;S^2 T^*\Sigma_0),
\end{equation}
that is, $h_b$ is the pullback of $g_b$ to $\Sigma_0$, and $k_b$ is the second fundamental form of $\Sigma_0$ with respect to the ambient metric $g_b$. In the remainder of this section, we merely study the initial gauge. In Proposition~\ref{PropKdSIn} below, we construct a map $i_b$, taking initial data on $\Sigma_0$ into Cauchy data for the gauged Einstein equation with gauge condition $\Ups(g)-\Ups(g_b)=0$, so that it maps the Kerr--de~Sitter initial data to
\[
  i_b(h_b,k_b) = \gamma_0(g_b) = (g_b|_{\Sigma_0},\cL_{\pa_{t_*}}g_b|_{\Sigma_0})=(g_b|_{\Sigma_0},0).
\]
The point is that this guarantees that the metric $g_b$ itself, rather than a pullback of it by some diffeomorphism, is the solution of the gauged Einstein equation with gauge $\Ups(g)-\Ups(g_b)=0$ and initial data $i_b(h_b,k_b)$.

The flexibility in choosing the gauge in this manner is very useful in the proof of linear stability of $g_b$ given in \S\ref{SecKdSLStab}, which is naturally done with the global choice of gauge $\Ups(g)-\Ups(g_b)=0$. For the full non-linear result, as indicated above, we will use only a single gauge $\Ups(g)-\Ups(g_{b_0})=0$ near $\Sigma_0$, which reads $\Ups(g)=0$ there; thus, for non-linear stability, we shall only use Proposition~\ref{PropKdSIn} for $b=b_0$.

\begin{prop}
\label{PropKdSIn}
  There exist neighborhoods $H\subset\cC^1(\Sigma_0;S^2T^*\Sigma_0)$ and $K\subset\cC^0(\Sigma_0;S^2T^*\Sigma_0)$ of $h_{b_0}$ and $k_{b_0}$, respectively, so that, firstly, $h_b\in H$ and $k_b\in K$ for all $b\in\cU_B$ (shrinking $\cU_B$, if necessary); and secondly, for each $b\in\cU_B$, there exists a map
  \begin{align*}
    i_b \colon (H\cap\cC^m(\Sigma_0;S^2&T^*\Sigma_0)) \times (K\cap\cC^{m-1}(\Sigma_0;S^2T^*\Sigma_0)) \\
      & \to \cC^m(\Sigma_0;S^2T^*_{\Sigma_0}M^\circ) \times \cC^{m-1}(\Sigma_0;S^2T^*_{\Sigma_0}M^\circ),
  \end{align*}
  smooth for all $m\geq 1$ (and smoothly depending on $b\in\cU_B$), with the following properties:
  \begin{enumerate}
    \item If $i_b(h,k)=(g_0,g_1)$, and $g\in\cC^1(M^\circ;S^2T^*M^\circ)$ is any symmetric 2-tensor with $\gamma_0(g)=(g_0,g_1)$, then $h$ and $k$ are, respectively, the metric and the second fundamental form of $\Sigma_0$ induced by $g$, and $\Ups(g)-\Ups(g_b)=0$ at $\Sigma_0$;
    \item for Kerr--de~Sitter initial data \eqref{EqKdSInKdSData}, one has $i_b(h_b,k_b)=\gamma_0(g_b)$.
  \end{enumerate}
\end{prop}

The constraint equations play no role in the construction of the map $i_b$, as expected following the discussion of the initial value problem in \S\ref{SubsecHypDT}, and hence we do not need to restrict the spaces $H$ and $K$ further here.

\begin{proof}[Proof of Proposition~\ref{PropKdSIn}]
  If we define $\phi_b\in\CI(\Sigma_0)$ and $\omega_b\in\CI(\Sigma_0,T^*\Sigma_0)$ by writing
  \[
    g_b = \phi_b\,dt_*^2 + 2\,dt_*\cdot\omega_b + h_b,
  \]
  then we can define $g_0$ in $(g_0,g_1)=i_b(h,k)$ simply by
  \[
    g_0 = \phi_b\,dt_*^2 + 2\,dt_*\cdot\omega_b + h;
  \]
  this will be a non-degenerate Lorentzian signature section of $S^2T^*_{\Sigma_0}M^\circ$ if $h$ is sufficiently close in $\cC^0$ to $h_{b_0}$ (and hence to $h_b$). This choice of $g_0$ fixes a future timelike unit vector field $N\perp T\Sigma_0$, and we now need to choose $g_1$ so that
  \[
    g_0(\nabla_X^{g_0+t_*g_1}Y,N) = k(X,Y)\,\quad X,Y\in T\Sigma_0,
  \]
  at $\Sigma_0=\{t_*=0\}$, and such that $g_1=0$ if $(h,k)=(h_b,k_b)$; here the superscript denotes the metric with respect to which $\nabla$ is the Levi-Civita connection. That is, we want
  \begin{equation}
  \label{EqKdSInG1}
    g_0\bigl((\nabla_X^{g_0+t_*g_1}-\nabla_X^{g_0})Y,N) = k(X,Y) - g_0(\nabla^{g_0}_X Y,N);
  \end{equation}
  by definition of $g_0$, the right hand is identically zero for $(h,k)=(h_b,k_b)$. The difference of the two covariant derivatives on the left hand side is tensorial in $X$ and $Y$, and in a local coordinate system, one computes
  \[
    2(\nabla^{g_0+t_*g_1}_\mu\pa_\nu - \nabla^{g_0}_\mu\pa_\nu) = g_0^{\kappa\lambda}\bigl((\pa_\mu t_*)(g_1)_{\nu\lambda} + (\pa_\nu t_*)(g_1)_{\mu\lambda} - (\pa_\lambda t_*)(g_1)_{\mu\nu})\pa_\kappa
  \]
  at $t_*=0$, hence, using $X t_*=0=Y t_*$, \eqref{EqKdSInG1} is equivalent to
  \begin{equation}
  \label{EqKdSInSecondFundCond}
    -(N t_*)g_1(X,Y) = 2\bigl(k(X,Y)-g_0(\nabla^{g_0}_X Y,N)\bigr),
  \end{equation}
  which determines $g_1$ on $T\Sigma_0\times T\Sigma_0$, since $N t_*>0$ by the future timelike nature of $dt_*$ and $N$. Hence, the symmetric 2-tensor $g_1\in\cC^{m-1}(\Sigma_0;S^2T^*_{\Sigma_0}M^\circ)$ is determined uniquely once we fix the values of $g_1(N,N)$ and $g_1(X,N)$ for $X\in T\Sigma_0$.

  These values in turn are determined by the gauge condition, which reads $\Ups(g_0+t_* g_1)-\Ups(g_b)=0$ at $t_*=0$. Using \eqref{EqHypDTPutIntoGaugeUpsLoc} gives
  \[
    \Ups(g_0+t_*g_1)_\mu = \Ups(g_0)_\mu + \frac{1}{2}(g_0)^{\nu\lambda}\bigl((\pa_\nu t_*)(g_1)_{\lambda\mu} + (\pa_\lambda t_*)(g_1)_{\nu\mu} - (\pa_\mu t_*)(g_1)_{\nu\lambda}\bigr)
  \]
  at $t_*=0$, so the gauge condition is equivalent to
  \begin{equation}
  \label{EqKdSInGaugeCond}
  \begin{split}
    \bigl(\Ups(g_b)-\Ups(g_0)\bigr)(V) &= \bigl(\Ups(g_0+t_*g_1)-\Ups(g_0)\bigr)(V) \\
      &=g_1(\nabla^{g_0}t_*,V) - \frac{1}{2}(V t_*) \tr_{g_0} g_1 = (\sfG_{g_0}g_1)(\nabla^{g_0}t_*,V)
  \end{split}
  \end{equation}
  for all $V\in T_{\Sigma_0}M^\circ$. Now $\nabla^{g_0}t_*\perp T\Sigma_0$ is a non-zero scalar multiple of the unit normal vector $N$, hence \eqref{EqKdSInGaugeCond} determines $(\sfG_{g_0}g_1)(N,X)=g_1(N,X)$ for $X\in T\Sigma_0$. Since \eqref{EqKdSInSecondFundCond} determines $g_1(X,Y)$ for $X,Y\in T\Sigma_0$, we have sufficient information to calculate the value of $\tr_{g_0}g_1-g_1(N,N)$; and then we can solve $g_1(N,N)=2(\sfG_{g_0}g_1)(N,N)+(\tr_{g_0}g_1-g_1(N,N))$ for $g_1(N,N)$. Note that if $(h,k)=(h_b,k_b)$, this construction gives $g_1\equiv 0$.

  This finishes the construction of $g_1$ and hence of $i_b$; the smoothness and mapping properties follow from an inspection of the proof.
\end{proof}

As a consequence, we show that the linearization of $i_b$ yields correctly gauged Cauchy data for the linearized gauged Einstein equation; we phrase this for smooth data for simplicity:

\begin{cor}
\label{CorKdSInLin}
  Let $b\in\cU_B$. Suppose $(h,k)$ are smooth initial data on $\Sigma_0$ satisfying the constraint equations \eqref{EqHypDTConstraints}, let $(g_0,g_1)=i_b(h,k)$, and let $g$ be a symmetric 2-tensor with $\gamma_0(g)=(g_0,g_1)$. Moreover, suppose $(h',k')$ are smooth solutions of the linearized constraint equations around $(h,k)$, let $(r_0,r_1)=D_{(h,k)}i_b(h',k')$, and let $r$ be a smooth symmetric 2-tensor with $\gamma_0(r)=(r_0,r_1)$.
  
  Then $r$ induces the data $(h',k')$ on $\Sigma_0$, and $D_g\Ups(r)=0$ at $\Sigma_0$.
\end{cor}

The case of main interest will be $g=g_b$.

\begin{proof}[Proof of Corollary~\ref{CorKdSInLin}]
  The first statement follows immediately from the definition of $i_b$. Since $\Ups(i_b(h+sh',k+sk'))=\Ups(g_b)$ at $\Sigma_0$ for all small $s\in\R$, the linearized gauge condition $D_g\Ups(r)=0$ at $\Sigma_0$ follows by differentiation and evaluation at $s=0$.
\end{proof}

Since $D_{(h,k)}i_b$ is a linear operator with smooth coefficients, the conclusion continues to hold even for distributional $h',k'$.

\section{Key ingredients of the proof}
\label{SecKey}

Throughout this section, we continue to use the notation of \S\ref{SecKdS}, so let
\begin{equation}
\label{EqKeyParam}
  b_0=(\bhm[0],\bfzero)\in B,\quad \bhm[0]>0,
\end{equation}
be parameters for a Schwarzschild--de~Sitter spacetime. We describe the three main ingredients using which we will prove the non-linear stability of slowly rotating Kerr--de~Sitter black holes: first, UEMS---the mode stability for the ungauged Einstein equation, $D_{g_{b_0}}(\Ric+\Lambda)(r)=0$---see \S\ref{SubsecKeyUEMS}; second, SCP---the existence of a hyperbolic formulation of the linearized Einstein equation for which the constraint propagation equation exhibits mode stability, in a form which is stable under perturbations---see \S\ref{SubsecKeySCP}; and third, ESG---quantitative high energy bounds for the linearized Einstein equations which are stable under perturbations---see \S\ref{SubsecKeyESG}. We recall that the key for the non-linear problem is to understand the linear stability problem in a robust perturbative framework, and the ingredients SCP and ESG will allow us to set up such a framework, which then also makes UEMS a stable property.

Motivated by the discussion in \S\ref{SubsecIntroIdeas} of our approach to the non-linear stability problem, the linearized gauged Einstein equation which we will study to establish the linear stability of slowly rotating Kerr--de~Sitter spacetimes is
\begin{equation}
\label{EqKeyLinEin}
  L_b r = \bigl(D_{g_b}(\Ric+\Lambda) - \tdel^*D_{g_b}\Ups\bigr)r = 0,
\end{equation}
where $\Ups(g)$ is the gauge 1-form defined in \eqref{EqKdSInGauge}, and where $\tdel^*$ is a modification of $\delta_{g_b}^*$ which we define in \eqref{EqKeySCPDeltaTilde} below. For proving the linear stability of the metric $g_b$, using the linearization of the gauge condition $\Ups(g)-\Ups(g_b)=0$ around $g=g_b$ leads to the condition $D_{g_b}\Ups(r)=0$, hence the form of $L_b$ given here.

\subsection{Mode stability for the ungauged Einstein equation}
\label{SubsecKeyUEMS}

\begin{thm}[UEMS---Ungauged Einstein equation: Mode Stability]
\label{ThmKeyUEMS}
  Let $b_0$ be the parameters of a Schwarzschild--de~Sitter black hole as in \eqref{EqKeyParam}.
  \begin{enumerate}
  \item \label{ItKeyUEMSNonzero} Let $\sigma\in\C$, $\Im\sigma\geq 0$, $\sigma\neq 0$, and suppose that $h(t_*,x)=e^{-i\sigma t_*}h_0(x)$, with $h_0\in\CI(Y,S^2T_Y^*\Omega^\circ)$, is a mode solution of the linearized Einstein equation
    \begin{equation}
    \label{EqKeyUEMSLinEin}
      D_{g_{b_0}}(\Ric+\Lambda)(h) = 0.
    \end{equation}
    Then there exists a 1-form $\omega(t_*,x)=e^{-i\sigma t_*}\omega_0(x)$, with $\omega_0\in\CI(Y,T_Y^*\Omega^\circ)$, such that
    \[
      h = \delta_{g_{b_0}}^*\omega.
    \]
  \item \label{ItKeyUEMSZero} For all $k\in\N_0$, and all generalized mode solutions
    \[
      h(t_*,x) = \sum_{j=0}^k t_*^j h_j(x),\quad h_j\in\CI(Y,S^2T_Y^*\Omega^\circ),\ j=0,\ldots,k,
    \]
    of the linearized Einstein equation \eqref{EqKeyUEMSLinEin}, there exist $b'\in T_{b_0}B$ and $\omega\in\CI(\Omega^\circ,S^2T^*\Omega^\circ)$ such that
    \[
      h = g_{b_0}'(b') + \delta_{g_{b_0}}^*\omega.
    \]
  \end{enumerate}
\end{thm}

Thus, \eqref{ItKeyUEMSNonzero} asserts that any mode solution $h$ of the linearized Einstein equations which is exponentially growing, or non-decaying and oscillating, is a pure gauge solution $h=\delta_{g_{b_0}}^*\omega=\frac{1}{2}\cL_{\omega^\sharp}g_{b_0}$ coming from an infinitesimal diffeomorphism (Lie derivative), i.e.\ does not constitute a physical degree of freedom for the linearized Einstein equation. On the other hand, \eqref{ItKeyUEMSZero} asserts that mode solutions with frequency $\sigma=0$, and possibly containing polynomially growing terms, are linearized Kerr--de~Sitter metrics, up to a Lie derivative. Thus, the only non-decaying mode solutions of the linearized Einstein equation \eqref{EqKeyUEMSLinEin}, up to Lie derivatives, are the linearized Kerr--de~Sitter metrics $g'_{b_0}(b')$, linearized around the Schwarzschild--de~Sitter metric $g_{b_0}$.

Observe here that a small displacement of the center of a black hole with parameters $b=(\bhm,\bfa)$, preserving its mass and its angular momentum vector, corresponds to pulling back the metric $g_b$ on $M^\circ$ by a translation; the restriction of the pullback metric to $\Omega^\circ$ is well-defined. Therefore, on the linearized level, the change in the metric due to an infinitesimal displacement of the black hole is given by the Lie derivative of $g_b$ along a translation vector field. This justifies our restriction of the space $B$ of black hole parameters to only include mass and angular momentum, but not the center of mass.

The particular form of the Kerr--de~Sitter metrics $g_b$ and its linearizations used here is rather arbitrary: if $\phi_b\colon M^\circ\to M^\circ$ is a smooth family of diffeomorphisms that commute with translations in $t_*$, then one obtains another smooth family of Kerr--de~Sitter metrics on the extended spacetime $M^\circ$ by setting $\wt g_b:=\phi_b^*g_b$. Given a solution $h$ of \eqref{EqKeyUEMSLinEin}, we have $D_{\wt g_{b_0}}(\Ric+\Lambda)(\phi_{b_0}^*h)=0$, and the conclusion $h=\cL_X g_{b_0}$ in \eqref{ItKeyUEMSNonzero} for $X=\omega^\sharp$ implies $\phi_{b_0}^*h=\frac{1}{2}\cL_{(\phi_{b_0}^{-1})_*X}\wt g_{b_0}$. In order to see the invariance of \eqref{ItKeyUEMSZero} under $\phi_b$, we in addition compute
\[
  \wt g_{b_0}'(b') = \frac{d}{ds}\wt g_{b_0+s b'}|_{s=0} = \phi_{b_0}^*g_{b_0}'(b') + \cL_X\wt g_{b_0},\quad X=\frac{d}{ds}\phi_{b_0+s b'}|_{s=0},
\]
hence the linearized metrics indeed agree under the diffeomorphism $\phi_b$ up to a Lie derivative, as desired. In particular, UEMS is independent of the specific choice of the functions $c_{b,\pm}$ and $\wt c_{b,\pm}$ in \eqref{EqKdSSlowCoordChange2}.

We stress that Theorem~\ref{ThmKeyUEMS} \emph{only} concerns the mode stability of a (single) \emph{Schwarz\-schild}--de~Sitter black hole; the assumption does \emph{not} concern the mode stability of nearby slowly rotating Kerr--de~Sitter spacetimes. We give the proof of the theorem in \S\ref{SecUEMS}.

\subsection{Mode stability for the constraint propagation equation}
\label{SubsecKeySCP}

Recall from  \eqref{EqHypDTCPEq}, or rather its modification, taking the modification of $\delta_g^*$ into account, that for a solutions $r$ of the linearized gauged Einstein equation \eqref{EqKeyLinEin} (with \emph{arbitrary} Cauchy data), the linearized gauge 1-form $\rho=D_{g_b}\Ups(r)$ solves the equation
\begin{equation}
\label{EqKeySCPBoxCPMod}
  \wtBoxCP_{g_b}\rho \equiv 2\delta_{g_b}\sfG_{g_b}\tdel^* \rho = 0.
\end{equation}
The asymptotic behavior of general solutions $\rho\in\CI(M^\circ,T^*M^\circ)$ of this equation depends on the specific choice of $\tdel^*$. We will show:

\begin{thm}[SCP---Stable Constraint Propagation]
\label{ThmKeySCP}
  One can choose $\tdel^*$, equal to $\delta_{g_{b_0}}^*$ up to $0$-th order terms, such that there exists $\alpha>0$ with the property that smooth solutions $\rho\in\CI(M^\circ,T^*M^\circ)$ of the equation $\wtBoxCP_{g_{b_0}}\rho=0$ decay exponentially in $t_*$ with rate $\alpha$, that is, $\rho=\cO(e^{-\alpha t_*})$.
\end{thm}

In fact, we will give a very concrete definition of $\tdel^*$. Namely, we will show that
\begin{equation}
\label{EqKeySCPDeltaTilde}
  \tdel^*u:=\delta_{g_{b_0}}^*u + \gamma\,dt_*\cdot u - \frac{1}{2}e\gamma \la u,dt_*\ra_{G_{b_0}} g_{b_0}
\end{equation}
works, for $e<1$ close to $1$ and for sufficiently large $\gamma>0$, for all $b$ near $b_0$. We give the proof of Theorem~\ref{ThmKeySCP} in \S\ref{SecSCP}.

We recall from \S\ref{SubsecIntroIdeas} that the usefulness of SCP stems from the observation that for a non-decaying smooth mode solution $r$ of the equation $L_{b_0} r=0$, the constraint propagation equation $\wtBoxCP_{g_{b_0}}D_{g_{b_0}}\Ups(r)=0$ implies $D_{g_{b_0}}\Ups(r)=0$, and hence $r$ in fact solves the linearized \emph{ungauged} Einstein equation $D_{g_{b_0}}(\Ric+\Lambda)r=0$. We can then appeal to UEMS to obtain very precise information about $r$. Thus, all non-decaying resonant states of $L_{b_0}$ are pure gauge solutions, plus linearized Kerr--de~Sitter metrics $g'_{b_0}(b')$; and furthermore all non-decaying resonant states $r$ satisfy the linearized gauge condition $D_{g_{b_0}}\Ups(r)=0$.

In fact, we will show in \S\ref{SecKdSLStab}, using the additional ingredient ESG below, that SCP is sufficient to deduce the mode stability, and in fact linear stability, of slowly rotating Kerr--de~Sitter black holes as well.

Note here that, as in \S\ref{SubsecKeyUEMS}, the above theorem \emph{only} concerns a (fixed) \emph{Schwarz\-schild}--de~Sitter metric.

\subsection{High energy estimates for the linearized gauged Einstein equation}
\label{SubsecKeyESG}

The final key ingredient ensures that solutions of the linear equations which appear in the non-linear iteration scheme have asymptotic expansions up to exponentially decaying remainder terms. Specifically, we need this to be satisfied for solutions of the linearized gauged Einstein equation \eqref{EqKeyLinEin} for $b=b_0$, and for perturbations of this equation. Since the linear equations we will need to solve are always invariant under translations in $t_*$ up to operators with exponentially decaying coefficients, the following theorem suffices for this purpose:

\begin{thm}[ESG---Essential Spectral Gap for the linearized gauged Einstein equation; informal version]
\label{ThmKeyESGInformal}
  For the choice of $\tdel^*$ in SCP, the linearized Einstein operator $L_{b_0}$ defined in \eqref{EqKeyLinEin} has a positive essential spectral gap; that is, there exists $\alpha>0$ and integers $N_L\geq 0$, $d_j\geq 1$ for $1\leq j\leq N_L$, and $n_{j\ell}\geq 0$ for $1\leq j\leq N_L$, $1\leq\ell\leq d_j$, as well as smooth sections $a_{j\ell k}\in\CI(Y,S^2 T^*_Y\Omega^\circ)$, such that solutions $r$ of the equation $L_{b_0}r=0$ with smooth initial data on $\Sigma_0$ have an asymptotic expansion
  \begin{equation}
  \label{EqKeyESGExpansion}
    r = \sum_{j=1}^{N_L}\sum_{\ell=1}^{d_j} r_{j\ell} \biggr(\sum_{k=0}^{n_{j\ell}} e^{-i\sigma_j t_*} t_*^k a_{j\ell k}(x)\biggr) + r',
  \end{equation}
  in $\Omega^\circ$, with $r_{j\ell}\in\C$ and $r'=\cO(e^{-\alpha t_*})$.

  The same holds true, with possibly different $N_L,d_j,n_{j\ell},a_{j\ell k}$, but the same constant $\alpha>0$, for the operator $L_b$, with $b\in\cU_B$.
\end{thm}

In order to give a precise and quantitative statement, which in addition will be straightforward to check, we recall from \cite{VasyMicroKerrdS} and, directly related to the present context, from \cite{HintzPsdoInner} that the only difficulty in obtaining an asymptotic expansion \eqref{EqKeyESGExpansion} is the precise understanding of the operator $L_{b_0}$ at the trapped set $\Gamma\subset T^*M^\circ\setminus o$ defined in \eqref{EqKdSGeoTrapped}; we discuss this in detail in \S\ref{SecAsy}. Concretely, one is set if, for a fixed $t_*$-independent positive definite inner product on the bundle $S^2T^*M^\circ$, one can find a stationary, elliptic ps.d.o.\ $Q\in\Psi^0(M^\circ;\End(S^2T^*M^\circ))$, defined microlocally near $\Gamma$, with microlocal inverse $Q^-$, such that in the coordinates \eqref{EqKdSGeoTrapCovec}, we have
\begin{equation}
\label{EqKeyESGTrappingBound}
  \pm|\sigma|^{-1}\sigma_1\Bigl(\frac{1}{2i}\bigl(Q L_{b_0}Q^- -(Q L_{b_0} Q^-)^*\bigr)\Bigr) < \alpha_\Gamma\Id
\end{equation}
in $\Gamma^\pm=\Gamma\cap\{\pm\sigma<0\}$ (recall \eqref{EqKdSGeoTrappedPM}), where $\alpha_\Gamma$ is a positive constant satisfying $\alpha_\Gamma<\numin/2$, with $\numin$ the minimal expansion rate of the $\ham_{G_{b_0}}$-flow in the normal directions at $\Gamma$, defined in \cite[Equation~(5.1)]{DyatlovResonanceProjectors} and explicitly computed for Schwarzschild--de~Sitter spacetimes in \cite[\S3]{DyatlovWaveAsymptotics} and \cite[\S2]{HintzPsdoInner}. If one arranges the estimate \eqref{EqKeyESGTrappingBound}, it also implies the same estimate at the trapped set for perturbations of $L_{b_0}$ within any fixed finite-dimensional family of stationary, second order, principally scalar differential operators; in particular, it holds for $L_b$ with $b\in\cU_B$, where we possibly need to shrink the neighborhood $\cU_B$ of $b_0$. (In the latter case, one in fact does not need to use the structural stability of the trapped set, since one can check it directly for Kerr--de~Sitter spacetimes, see \cite[\S3]{DyatlovWaveAsymptotics}, building on \cite[\S2]{WunschZworskiNormHypResolvent} and \cite[\S6]{VasyMicroKerrdS}. See the discussion at the end of \S\ref{SubsecESGTrap} for further details.)

Observe that the condition \eqref{EqKeyESGTrappingBound} only concerns principal symbols; thus, checking it amounts to an algebraic computation, which is most easily done using the framework of pseudodifferential inner products, see \cite{HintzPsdoInner} and Remark~\ref{RmkAsySmPsdoInner}.

\begin{thm}[ESG---Essential Spectral Gap for the linearized gauged Einstein equation; precise version]
\label{ThmKeyESG}
  Fix $\alpha_\Gamma>0$. Then there exist a neighborhood $\cU_B$ of $b_0$ and a stationary ps.d.o.\ $Q$ such that \eqref{EqKeyESGTrappingBound} holds for $L_b$ with $b\in\cU_B$. Moreover, for any fixed $s>1/2$, there exist constants $\alpha>0$, $C_L<\infty$ and $C>0$ such that for all $b\in\cU_B$, the estimate
  \begin{equation}
  \label{EqKeyESGHighEnergyEst}
    \| u \|_{H^s_{\la\sigma\ra^{-1}}} \leq C\la\sigma\ra^2 \| \wh{L_b}(\sigma)u \|_{H^{s-1}_{\la\sigma\ra^{-1}}}
  \end{equation}
  holds for all $u$ for which the norms on both sides are finite, provided $\Im\sigma>-\alpha$ and $|\Re\sigma|\geq C_L$, as well as for $\Im\sigma=-\alpha$. Here, $H^s_\semi\equiv\bar H^s_\semi(Y,S^2T^*_Y\Omega)$ is the semiclassical Sobolev space of distributions which are extendible at $\pa Y$ (see Appendix~\ref{SubsecBSemi}), defined using the volume density induced by the metric $g_{b_0}$, and using a fixed stationary positive definite inner product on $S^2T^*M^\circ$.

  Moreover, by reducing $\alpha>0$ if necessary, one can arrange that all resonances $\sigma$ of $L_{b_0}$, i.e.\ poles of the meromorphic family $\wh{L_{b_0}}(\sigma)^{-1}$, which satisfy $\Im\sigma>-\alpha$, in fact satisfy $\Im\sigma\geq 0$.
\end{thm}

\begin{rmk}
\label{RmkKeyESGDecay}
  The dependence of the decay rate $\alpha$ on the regularity $s$ is very mild: the inequality that needs to be satisfied is $s>1/2+\beta\alpha$, with $\beta$ the larger value of $\beta_{b_0,\pm}$ defined in \eqref{EqKdSGeoRadPointQuant}. In particular, if this holds for some $s$ and $\alpha$ as in Theorem~\ref{ThmKeyESG}, then it continuous to hold for all larger values of $s$ as well. The size of $\alpha>0$ is thus really only restricted by the location of resonances in the lower half plane.
\end{rmk}

It is crucial here that there exists a choice of $\tdel^*$ which makes SCP hold and for which at the same time ESG is valid as well; this turns out to be very easy to arrange. We will prove this theorem in \S\ref{SecESG}.

That Theorem~\ref{ThmKeyESGInformal} is a consequence of Theorem~\ref{ThmKeyESG} follows from a representation of solutions of the linear equation $L_{b_0}r=0$ in the $t_*$-Fourier domain, and shifting the contour in the inverse Fourier transform from a line $\Im\sigma\gg 1$ to $\Im\sigma=-\alpha$, which is justified by the estimate~\eqref{EqKeyESGHighEnergyEst}. The asymptotic expansion is then a consequence of the residue theorem: the exponents $\sigma_j\in\C$ are the poles of $\wh{L_{b_0}}(\sigma)^{-1}$, $n_{j\ell}$ is one more than the order of the pole, and $d_j$ is the rank of the resonance; see~\S\ref{SubsubsecAsySmRes} for definitions, and the beginning of~\S\ref{SubsubsecAsyExpReg} for a sketch of the contour shifting argument.

Under the assumptions of this theorem, the operators $L_b$ have only finitely many resonances $\sigma$, in the half space $\Im\sigma\geq-\alpha$, and no resonances $\sigma$ have $\Im\sigma=-\alpha$. By general perturbation arguments discussed in \S\ref{SecAsy}, the total rank of the resonances of $L_b$ in $\Im\sigma>-\alpha$ thus remains constant, see Proposition~\ref{PropAsySmPert}. The essential spectral gap of $L_b$ is, by definition, the supremum over all $\wt\alpha>0$ such that $L_b$ has only finitely many resonances in the half space $\Im\sigma>-\wt\alpha$.

The regularity assumption in ESG, $s>1/2$, which we will justify in \S\ref{SubsecESGRadial}, is due to radial point estimates at the event and cosmological horizons intersected with future infinity $X$ (microlocally: near $\pa\cR_b$), as we explain in \S\ref{SubsecAsySm}. The power $\la\sigma\ra^2$ on the other hand is due to the loss of one power of the semiclassical parameter in the estimate at the semiclassical trapped set, which for $L_{b_0}$ can be identified with $\Gamma\cap\{\sigma=\mp 1\}$.

\section{Asymptotic analysis of linear waves}
\label{SecAsy}

In this section, we discuss the global regularity and asymptotic analysis for solutions to initial value problems for linear wave equations as needed for our proofs of linear and non-linear stability: thus, we show how to quantitatively control solutions (their asymptotic behavior, the unstable, i.e.\ exponentially growing, modes and exponentially decaying tails) using methods which are stable under perturbations. A crucial feature of the linear analysis is that we allow a modification of the initial data and forcing terms by elements of a (fixed) finite-dimensional space $\cZ$, as motivated in \S\ref{SubsecIntroIdeas}, and we show that for suitable choices of the space $\cZ$, one can solve any given linear initial value problem \emph{for perturbations of a fixed operator} on decaying function spaces if one modifies the initial data and the forcing by a suitable element $z\in\cZ$.

The analysis necessary for showing linear stability is presented in \S\ref{SubsecAsySm}. Due to the stationary nature of the Kerr--de~Sitter metrics $g_b$ defined in \S\ref{SubsecKdSSlow}, the linear wave equations we need to study have smooth coefficients and are stationary, i.e.\ they commute with $t_*$-translations; thus, we are in the setting of \cite{VasyMicroKerrdS} (addressing initial value problems explicitly however), and one can perform the analysis on the Mellin-transformed (in $\tau=e^{-t_*}$; equivalently: Fourier-transformed in $-t_*$) side. Perturbative arguments then yield the continuous dependence of the global linear solution operators on the wave operator one is inverting; these arguments will turn out to be rather straightforward in the settings of interest, since one only needs to consider a finite-dimensional family of operators, parameterized by $b\in\cU_B$.

The proof of non-linear stability requires additional work. In \S\ref{SubsecAsyExp}, we thus study the linear wave-type operators which appear naturally in the iteration scheme we will employ in \S\ref{SecKdSStab} to solve the gauged Einstein equation: these operators have a stationary part which depends only on the finite-dimensional parameter $b\in\cU_B$, but they are perturbed by a second order principally scalar operator which has small and exponentially decaying but non-smooth (yet high regularity) coefficients. Such operators were discussed in great detail in \cite{HintzQuasilinearDS,HintzVasyQuasilinearKdS}; we need to extend these works slightly to incorporate initial data problems, as well as the finite-dimensional modification space $\cZ$, which may also contain non-smooth elements. The forcing terms we consider are of the same type as the non-smooth perturbations, thus in \S\ref{SubsecAsyExp} we find exponentially decaying solutions to initial value problems with exponentially decaying forcing modified by elements of the space $\cZ$. The motivation is that these are precisely the linear equations we need to analyze in order to solve non-linear equations in \S\ref{SecKdSStab}---that is, solving a non-linear equation for a quantity which is exponentially decaying requires a study of stationary linear equations perturbed by operators which have coefficients and forcing terms of the same form.

We reiterate that \emph{the proof of linear stability of slowly rotating Kerr--de~Sitter spacetimes only relies on the material in} \S\ref{SubsecAsySm}. Thus, the reader interested in the latter can skip \S\ref{SubsecAsyExp}, which is \emph{only} needed for the non-linear analysis. Moreover, we point out that the technical details in the proof of non-linear stability in \S\ref{SecKdSStab}, to the extent that they are related to the material of the present section, on a conceptual level only rely on the ideas presented in \S\ref{SubsecAsySm} for the smooth, stationary linear theory; the non-smooth extension of this theory, while necessary to justify our non-linear iteration scheme, is primarily of technical nature, but does not introduce substantially new ideas.

\subsection{Smooth stationary wave equations}
\label{SubsecAsySm}

Let $E_X\to X$ be a complex vector bundle of finite rank, and define the vector bundle $E=\pi_X^*E_X\to M$, where $\pi_X\colon M\to X$ is the projection $\pi_X(\tau,x)=x$ onto the boundary at future infinity. Then dilations in $\tau$ by $c\in\R^\times$ induce isomorphisms $E_{(\tau,x)}\cong E_{(c\tau,x)}$. Equivalently, one can consider the restriction $E^\circ\to M^\circ$ of the bundle $E$ to $M^\circ$, and then translations in $t_*$ induce bundle isomorphisms of $E^\circ$. In particular, there is a well-defined notion of stationary sections of $E\to M$, which can also be thought of as pullbacks of sections of $E_X\to X$ along $\pi_X$. We call $E\to M$ a \emph{stationary vector bundle}.

The easiest example is the trivial bundle $E_X=X\times\C$, suited for the study of scalar waves. The main examples of interest for the purposes of the present paper are the cotangent bundle $E^\circ=T^*M^\circ$ and the bundle of symmetric 2-tensors $E^\circ=S^2T^*M^\circ$; in these cases, $E_X=T^*_X M^\circ$ (using $X\cong\{t_*=0\}$, for example) and $E_X=S^2T^*_X M^\circ$ (or rather $E_X=\Tb^*_X M$ and $E_X=S^2\,\Tb^*_X M$), and $E=\Tb^*M$ and $E=S^2\,\Tb^*M$, respectively. However, notice the natural inner products on the bundles in these examples are not (positive or negative) definite and hence largely irrelevant for the purposes of quantitative analysis; in fact one of the central guiding principles of the linear analysis throughout this paper is that positive definite fiber inner products only need to be chosen carefully at certain places---at radial points at the horizons and at the trapping, discussed below ---, and there the correct choices can be read off directly from properties of the linear operator, specifically the behavior of its subprincipal part.

Returning to the general setup, let $L\in\Diffb^2(M;E)$ be a stationary, principally scalar operator acting on sections of $E$; that is, $L$ commutes with dilations in $\tau$. (Equivalently, but less naturally from the point of view of non-stationary problems discussed in \S\ref{SubsecAsyExp}, one can view $L\in\Diff^2(M^\circ;E^\circ)$, and $L$ commutes with translations in $t_*$.) We start by assuming:
\begin{enumerate}
\item \label{ItAsySm1Sym} the principal symbol of $L$ is
  \begin{equation}
  \label{EqAsySm1Sym}
    \sigma_{\bop,2}(L)(\zeta) = |\zeta|_{G_b}^2 \otimes \Id, \quad \zeta\in\Tb^*M,
  \end{equation}
  for some Kerr--de~Sitter parameters $b\in\cU_B$; here $\Id$ is the identity operator on $\pi^*E\to\Tb^*M\setminus o$, with $\pi\colon\Tb^*M\setminus o\to M$ the projection. (This suffices for our applications, however one can allow more general metrics; see in particular \cite[\S1]{HintzVasyQuasilinearKdS}.)
\end{enumerate}

Denote by $\wh\rho$ the defining function \eqref{EqKdSGeoCoords} of fiber infinity in $\ol{\Tb^*}M$ near the (generalized) radial set $\pa\cR_{b,\pm}$ at the horizon $r=r_{b,\pm}$, defined in \eqref{EqKdSGeoRadSet}, and recall the definition \eqref{EqKdSGeoRadPoint} of the dynamical quantities $\beta_{b,\pm}$ and $\beta_{b,\pm,0}$. We define the positive function
\begin{equation}
\label{EqAsySm2RadialWeight}
  \beta\in\CI(\pa\cR_b),\quad \beta|_{\pa\cR_{b,\pm}}=\beta_{b,\pm}>0.
\end{equation}
Moreover, fix a positive definite inner product on $E$ near $r=r_{b,\pm}$, and write the subprincipal symbol of $L$ as
\begin{equation}
\label{EqAsySm2RadialSubpr}
  s\wh\rho\sigma_{\bop,1}\Bigl(\frac{1}{2i}(L-L^*)\Bigr) = -\beta_{b,\pm,0}\wh\beta_\pm
\end{equation}
at $\pa\cR_{b,\pm}^s$, with $s=\pm$ specifying the future (`$+$') and past (`$-$') half of the radial set. This defines $\wh\beta_\pm\in\CI(\pa\cR_b;\End(\pi^*E))$, with $\pi\colon\ol{\Tb^*}M\to M$; and $\wh\beta_\pm$ is pointwise self-adjoint with respect to the chosen inner product. We define
\begin{equation}
\label{EqAsySm2RadialSubprInf}
  \wh\beta := \inf_{\pa\cR_b}\wh\beta_\pm \in \R
\end{equation}
to be the smallest eigenvalue of all $\wh\beta_\pm(p)$, $p\in\pa\cR_b$. We further assume:

\begin{enumerate}
\setcounter{enumi}{1}
\item \label{ItAsySm3Trapped} denote by $\sigma$ the dual variable of $\tau$ at $X$, so $\sigma=\sigma_{\bop,1}(\tau D_\tau)$ (this is independent of choices at $X$). Then for all $\alpha_\Gamma>0$, there exists a ps.d.o.\ $Q\in\Psib^0(M;E)$, defined microlocally near the trapped set $\Gamma$ (given explicitly for Schwarzschild--de~Sitter metrics in \eqref{EqKdSGeoTrapped}), and elliptic near $\Gamma$ with microlocal parametrix $Q^-$, such that
  \begin{equation}
  \label{EqAsySm3TrappedSubpr}
    \pm|\sigma|^{-1} \sigma_{\bop,1}\Bigl(\frac{1}{2i}\bigl(Q L Q^- - (Q L Q^-)^*\bigr)\Bigr) < \alpha_\Gamma\Id
  \end{equation}
  on $\Gamma^\pm=\Gamma\cap\{\pm\sigma<0\}$.
\end{enumerate}

The scalar wave operator $\Box_{g_b}$ is the simplest example; one has $\wh\beta_\pm=0$, and the left hand side of \eqref{EqAsySm3TrappedSubpr} is equal to $0$, so indeed any $\alpha_\Gamma>0$ works. For the tensor wave operator on 1-forms and symmetric 2-tensors, one can still arrange \eqref{EqAsySm3TrappedSubpr} for any $\alpha_\Gamma>0$ by \cite[Theorem~4.8]{HintzPsdoInner}, and the calculations in \S\ref{SubsecESGRadial} will imply that one can choose a fiber inner product such that $\wh\beta=0$.

\begin{rmk}
\label{RmkAsySm3SubprAtTrapping}
  All of our arguments in \S\ref{SecAsy} go through with only minor modifications if we merely assume that \eqref{EqAsySm3TrappedSubpr} holds for some \emph{fixed} constant
  \begin{equation}
  \label{EqAsySm3TrappedBound}
    \alpha_\Gamma<\frac{\numin}{2},
  \end{equation}
  with $\numin$ the minimal expansion rate of the Hamilton flow of $G_b$ at $\Gamma$, see the discussion following \eqref{EqKeyESGTrappingBound}. When we turn to non-smooth perturbations of $L$ in \S\ref{SubsecAsyExp}, assuming that $\alpha_\Gamma>0$ is any fixed small number will simplify the bookkeeping of regularity, which is the main reason for us to make this stronger assumption; in our applications, it is always satisfied.
\end{rmk}

\begin{rmk}
\label{RmkAsySmPsdoInner}
  The quantities in \eqref{EqAsySm2RadialSubpr}--\eqref{EqAsySm3TrappedSubpr} are symbolic. A convenient way of calculating them involves \emph{pseudodifferential inner products}, introduced in \cite{HintzPsdoInner} and partially extending the scope of the partial connection for real principal type systems defined by Dencker \cite{DenckerPolarization}: first, one defines the \emph{subprincipal operator} in a local trivialization of $E$ and using local coordinates on $M$ by
  \[
    S_\sub(L) = -i\ham_{G_b}\otimes\Id + \sigma_\sub(L).
  \]
  (Here, we trivialize the half-density bundle $\Omegab^\half(M)$ using $|dg_b|^\half$ in order to define the subprincipal symbol.) This in fact gives a well-defined operator
  \[
    S_\sub(L) \in \Diffb^1(\Tb^*M\setminus o;\pi^*E),
  \]
  which is homogeneous of degree $1$ with respect to dilations in the fibers of $\Tb^*M\setminus o$. A pseudodifferential inner product (or $\Psi$-inner product) is then defined in terms of a positive definite inner product $k$ on the vector bundle $\pi^*E\to\Tb^*M\setminus o$ which is homogeneous of degree $0$ in the fibers of $\Tb^*M\setminus o$; equivalently, an inner product on $\pi_S^*E\to\Sb^*M$, where $\pi_S\colon\Sb^*M\to M$. Thus, $k$ can vary from point to point \emph{in phase space}. Now, fixing a positive definite inner product $k_0$ on $E\to M$, there exists an elliptic symbol $q\in\Shom^0(\Tb^*M\setminus o,\pi^*\End(E))$ intertwining the two inner products, i.e.\ $k(e,f)=k_0(q e,q f)$ for $e,f\in\pi^*\End(E)$, and for the quantization $Q\in\Psib^0(M;E)$ with parametrix $Q^-$, we have
  \begin{align*}
    \sigma_{\bop,1}\Bigl(\frac{1}{2i}\bigl(Q L Q^- - (Q L Q^-)^{*k_0}\bigr)\Bigr) = &q\frac{1}{2i}\bigl(S_\sub(L) - S_\sub(L)^{*k}\bigr)q^{-1} \\
     & \quad\in \Shom^1(\Tb^*M\setminus o,\pi^*\End(E)),
  \end{align*}
  where the superscripts denote the inner product used to define the adjoints, and we in addition use the symplectic volume density to define the adjoint of $S_\sub(L)$; see \cite[Proposition~3.12]{HintzPsdoInner}. Both sides are self-adjoint with respect to $k_0$. Therefore, the problem of obtaining \eqref{EqAsySm3TrappedSubpr} is reduced to finding an inner product $k$ such that the eigenvalues of $\pm|\sigma|^{-1}\frac{1}{2i}(S_\sub(L)-S_\sub(L)^{*k})$ are bounded from above by $\alpha_\Gamma$. (In our applications, the choice of $k$ will be clear from an inspection of the form of $S_\sub(L)$.)

  An a priori different, dynamically more natural, statement is that the exponential map $\pi^*E_{(x,\xi)}\to\pi^*E_{\exp(t_* H_{G_b})(x,\xi)}$ defined naturally from $i S_\sub(L)$ has norm growing slower than any exponential as $t_*\to\infty$. This statement is indeed an immediate consequence of our bound. In the context of scalar equations, the converse follows by averaging along the flow; in the present context it appears to be a bit more subtle.
\end{rmk}

\begin{rmk}
\label{RmkAsySmPsdoInnerRad}
  One could define $\wh\beta_\pm$ in \eqref{EqAsySm2RadialSubpr} for a conjugated version  of $L$ as in \eqref{ItAsySm3Trapped}, thereby possibly increasing the value of $\wh\beta$, but this increase in generality is unnecessary for our applications: the optimal choice for the pseudodifferential inner product near the radial sets will turn out to be constant in the fibers of $\Tb^*M\setminus o$.
\end{rmk}

Since $L$ is stationary, we can analyze it by considering its Mellin-transformed normal operator family $\wh L(\sigma)$, which is an entire family (depending on $\sigma\in\C$) of operators in $\Diff^2(X;E_X)$, defined by
\[
  \wh L(\sigma)u = \tau^{-i\sigma} L \tau^{i\sigma} u,\quad u\in\CIc(X;E_X).
\]
The natural function spaces on which $\wh L(\sigma)$ acts are semiclassical Sobolev spaces, see \S\ref{SubsecBPsdo}. In order to measure the size of sections of $E$, we equip $E$ with any stationary positive definite inner product; all such choices lead to equivalent norms since $X$ is compact. We state the estimates for $\wh L(\sigma)$ using spaces of extendible (and later supported) distributions; we refer the reader to \cite[Appendix~B]{HormanderAnalysisPDE3} and the end of Appendix~\ref{SubsecBPsdo} for definitions. We then recall:

\begin{thm}
\label{ThmAsySmMero}
  Let $C\in\R$, and fix $s>s_0=1/2+\sup(\beta C)-\wh\beta$. Let $\cY^{s-1}:=\Hext^{s-1}(Y;E_Y)$ and $\cX^s:=\{u\in\Hext^s(Y;E_Y)\colon \wh L(\sigma)u\in\cY^{s-1}\}$. (Note that the principal symbol of $\wh L(\sigma)$ is independent of $\sigma$.) Then
  \[
    \wh L(\sigma)\colon\cX^s\to\cY^{s-1}, \quad \Im\sigma>-C,
  \]
  is a holomorphic family of Fredholm operators. Furthermore, there exist constants $C,C_1,\alpha>0$ such that the high energy estimate
  \begin{equation}
  \label{EqAsySmMeroEst}
    \|u\|_{\Hext_{\la\sigma\ra^{-1}}^s(Y;E_Y)} \leq C\|\wh L(\sigma)u\|_{\Hext_{\la\sigma\ra^{-1}}^{s-1}(Y;E_Y)}, \quad \Im\sigma>-\alpha,\ |\Re\sigma|>C_1,
  \end{equation}
  holds for any fixed $s>1/2+\alpha\sup\beta-\wh\beta$. In particular, $\wh L(\sigma)$ is invertible there. Moreover, changing $\alpha>0$ by an arbitrarily small amount if necessary, one can arrange that \eqref{EqAsySmMeroEst} holds for \emph{all} $\sigma$ with $\Im\sigma=-\alpha$ as well.

  Lastly, for any fixed $C_2>0$, the estimate
  \begin{equation}
  \label{EqAsySmMeroEstPosIm}
    \|u\|_{\Hext_{\la\sigma\ra^{-1}}^s(Y;E_Y)} \leq C\la\sigma\ra^{-1}\|\wh L(\sigma)u\|_{\Hext_{\la\sigma\ra^{-1}}^{s-1}(Y;E_Y)}, \quad \Im\sigma>C_2,\ |\Re\sigma|>C_1,
  \end{equation}
  holds.
\end{thm}
\begin{proof}
  The proof is contained in the references \cite[\S2]{HintzVasySemilinear} (which uses Cauchy hypersurfaces beyond the horizons instead of the complex absorption used in \cite{VasyMicroKerrdS}), \cite[\S4.4]{HintzVasyQuasilinearKdS} (for the trapping analysis for bundle-valued operators under condition~\eqref{EqAsySm3TrappedSubpr}) and \cite[\S2]{HintzPsdoInner} (for the verification of that condition for tensor-valued equations on slowly rotating Kerr--de~Sitter spacetimes).
  
  In brief, the theorem combines the statements of \cite[Theorem~2.17]{VasyMicroKerrdS} and the extension of \cite[Theorem~7.5]{VasyMicroKerrdS} to the case of semiclassical mild local trapping, see \cite[Definition~2.16]{VasyMicroKerrdS}. In fact, the first statement does not rely on the structure of the trapped set. The quantitative high energy bounds on the other hand do use the normally hyperbolic nature of the trapped set: condition \eqref{EqAsySm3TrappedSubpr} gives a bound on the skew-adjoint part of $L$, or rather its conjugated version $Q L Q^-$, at the trapped set, and likewise for the Mellin-transformed problem, and then Dyatlov's result \cite{DyatlovSpectralGaps}, see also \cite[\S4.4]{HintzVasyQuasilinearKdS} and the remark after \cite[Theorem~1]{DyatlovSpectralGaps}, proves the semiclassical mild trapping, giving \eqref{EqAsySmMeroEst} in a half plane extending to a strip beyond the real line, as well as \eqref{EqAsySmMeroEstPosIm} in the upper half plane.
\end{proof}

\begin{rmk}
\label{RmkAsySmMeroEstLoss}
  A slightly more natural way of writing the estimates in this theorem is by means of the semiclassical rescaling $L_{\semi,z}:=\semi^2 L(\semi^{-1}z)$, where $\semi=\la\sigma\ra^{-1}$ and $z=\semi\sigma$. Replacing $\wh L(\sigma)$ by $L_{\semi,z}$ necessitates a factor of $\semi^{-2}$ on the right hand side of \eqref{EqAsySmMeroEst}, which ultimately comes from  the main result of \cite{DyatlovSpectralGaps}, and a factor of $\semi^{-1}$ on the right hand side of \eqref{EqAsySmMeroEstPosIm}. Relating the estimate \eqref{EqAsySmMeroEst}, resp.\ \eqref{EqAsySmMeroEstPosIm}, to weighted b-Sobolev spaces via the Mellin transform, see \eqref{EqBSobolevIso}, shows that the power of $\la\sigma\ra^0$, resp.\ $\la\sigma\ra^1$, on the right hand side corresponds to a loss of $2$, resp.\ $1$, derivatives for the operator $L^{-1}$ relative to elliptic estimates which would gain full $2$ derivatives.
\end{rmk}

Recall from Appendix~\ref{SubsecBPsdo} the definition of b-Sobolev spaces $\Hbext^{s,\alpha}(\Omega;E)$ of distributions which are extendible at the boundary hypersurfaces of $\Omega$. Elements of $\Hbext^{s,\alpha}(\Omega;E)$ for $s\in\N_0$ are precisely those elements of $e^{-\alpha t_*}L^2_{t_*,x}$ (consisting of distributions which are equal to $e^{-\alpha t_*}$ times an $L^2$ function) which remain in this space upon applying up to $s$ derivatives in the $(t_*,x)$ variables. Let us then define the space of data (Cauchy data and forcing) for initial value problems:

\begin{definition}
\label{DefAsySmDataSpace}
  For $s,\alpha\in\R$, the \emph{space of data with regularity $s$ and decay rate $\alpha$} is
  \[
    D^{s,\alpha}(\Omega;E) := \Hbext^{s,\alpha}(\Omega;E) \oplus \Hext^{s+1}(\Sigma_0;E_{\Sigma_0}) \oplus \Hext^s(\Sigma_0;E_{\Sigma_0}),
  \]
  with the norm $\|(f,u_0,u_1)\|_{D^{s,\alpha}} := \|f\|_{\Hbext^{s,\alpha}} + \|u_0\|_{\Hext^{s+1}} + \|u_1\|_{\Hext^s}$.
\end{definition}

We can now discuss the global behavior of the solution of the initial value problem
\begin{equation}
\label{EqAsySmIVP}
  (L,\gamma_0)u = (f,u_0,u_1) \in D^{s,\alpha}(\Omega;E),
\end{equation}
where $s>s_0=1/2+\alpha\sup\beta-\wh\beta$ so that we have the Fredholm property of $\wh L(\sigma)$ as well as the high energy estimate \eqref{EqAsySmMeroEst}. Written in a more conventional manner, this is equivalent to
\[
  \begin{cases}
    L u = f &\text{in}\ \Omega, \\
    (u|_{\Sigma_0},\,\pa_{t_*}u|_{\Sigma_0}) = (u_0,u_1) &\text{at}\ \Sigma_0.
  \end{cases}
\]
It is convenient to rephrase this as a forcing problem: we can solve \eqref{EqAsySmIVP} until $t_*=3$ using the standard local well-posedness theory, see also \S\ref{SubsubsecAsyExpLocal}; denote the local solution by $u'\in\Hext^{s+1}(\{0<t_*<3\};E)$. Then, with $\chi\in\CI(\R)$ a cutoff, $\chi\equiv 0$ for $t_*\leq 1$ and $\chi\equiv 1$ for $t_*\geq 2$, we write $u=(1-\chi)u'+v$; solving \eqref{EqAsySmIVP} is then equivalent to solving the forward problem
\begin{equation}
\label{EqAsySmIVPForcing}
  L v = \phi := \chi f + [L,\chi]u' \in \Hb^{s,\alpha}(\Omega;E)^{\bullet,-}
\end{equation}
for $v$, with $v\equiv 0$ near $\Sigma_0$. Here, `$\bullet$' indicates vanishing in $t_*\leq 0$ (i.e.\ supported distributions at $\Sigma_0$, in particular meaning vanishing Cauchy data), while `$-$' indicates extendibility beyond the artificial spacelike hypersurfaces $\pa Y\times[0,\infty)_\tau$ beyond the horizons; see also Appendix~\ref{SubsecBPsdo}. By global energy estimates, see \cite[Lemma~2.4]{HintzVasySemilinear}, the problem \eqref{EqAsySmIVPForcing} has a solution $v\in\Hb^{s+1,r_0}(\Omega;E)^{\bullet,-}$ for some large negative $r_0$. In order to analyze it (see also \cite[Lemma~3.5]{VasyMicroKerrdS}), we Mellin transform (see~\eqref{EqBMellinTrafo} for the definition) equation~\eqref{EqAsySmIVPForcing} in $\tau=e^{-t_*}$, obtaining
\[
  \wh v(\sigma)=\wh L(\sigma)^{-1}\wh\phi(\sigma)
\]
in $\Im\sigma>-r_0$. The right hand side is finite-meromorphic in the half space $\Im\sigma>-\alpha$. Defining $\Xi$ to be the finite set of poles (resonances) of $\wh L(\sigma)^{-1}$ in this half space. A contour shifting argument as in the proof of \cite[Theorem~2.21]{HintzVasySemilinear}, which generalizes \cite[Proposition~3.5]{VasyMicroKerrdS}, then implies
\begin{equation}
\label{EqAsySmIVPForcingExpansion}
  v(\tau) = \sum_{\sigma\in\Xi} i\chi \res_{\zeta=\sigma}\bigl(\tau^{i\zeta}\wh L(\zeta)^{-1}\wh\phi(\zeta)\bigr) + v', \quad v'\in\Hb^{s,\alpha}(\Omega;E)^{\bullet,-},
\end{equation}
where the regularity of $v'$ is guaranteed by \eqref{EqAsySmMeroEst}, see also \S\ref{SubsecBSemi}. The terms in this finite asymptotic expansion (often also called \emph{resonance expansion}) are \emph{(generalized) modes}; they have the form
\[
  \sum_{k=0}^d e^{-i t_*\sigma}t_*^k a_k(x)
\]
for some $a_k\in\CI(Y;E_Y)$. We will discuss these in more detail in \S\ref{SubsubsecAsySmRes}.

One can easily read off the asymptotic behavior of the solution of the original problem \eqref{EqAsySmIVP} directly. Namely, let us view
\begin{equation}
\label{EqAsySmViewAsSupp1}
  f\in\Hbext^{s,\alpha}(\Omega;E) \subset \Hbsupp^{0,\alpha}\bigl([0,1]_\tau;\Hext^s(Y;E_Y)\bigr)
\end{equation}
as a forcing term which is a supported distribution at $\Sigma_0$; that is, we regard $f$ as a distribution with merely $L^2$ regularity in $t_*$, which can then be viewed as a \emph{supported} distribution at $\Sigma_0$. Similarly, denoting by $u^1\in\Hextloc^{s+1}(\Omega^\circ;E^\circ)$ the solution of $L u^1=0$, $\gamma_0(u^1)=(u_0,u_1)$, we can view $u^1\in \Hsupploc^0\bigl((0,1]_\tau;\Hext^{s+1}(Y;E_Y)\bigr)$ as the solution of the forward problem $L u^1=[L,H]u^1=:f^1$, with $H=H(t_*)$ the Heaviside function; now $[L,H]$ is a first order differential operator whose coefficients are at most once differentiated $\delta$-distributions supported at $\Sigma_0$, so $f^1$ only depends on the Cauchy data $(u_0,u_1)$, and
\begin{equation}
\label{EqAsySmViewAsSupp2}
  f^1 \in \Hsupp^{-3/2-0}\bigl((0,1]_\tau;\Hext^s(Y;E_Y)\bigr).
\end{equation}
By a slight abuse of notation, we define the operator $[L,H]$ acting on Cauchy data by
\begin{equation}
\label{EqAsySmCommInitial}
  [L,H](u_0,u_1) := f^1.
\end{equation}
Thus, if $u$ solves \eqref{EqAsySmIVP}, then $u$ (interpreted as the supported distribution $H u$) solves the forcing problem $L u=f+f^1$ (with $u$ vanishing identically in $t_*<0$). Solving this using the Mellin transform as above yields the asymptotic expansion for $u$, given by the same sum as in \eqref{EqAsySmIVPForcingExpansion} with $f+f^1$ in place of $\phi$; however, a priori the regularity of the remainder term $u'$ is only $\Hb^{-3/2-,\alpha}$. Since no non-zero linear combination of terms in the asymptotic expansion lies in this space (due to the weight $\alpha$), we conclude that the asymptotic expansions of $u$ thus obtained and of $v$ in \eqref{EqAsySmIVPForcingExpansion} in fact agree, and we thus again obtain the $\Hb^{s,\alpha}$ regularity of $u'$ near $\tau=0$. In summary then, the solution $u$ of \eqref{EqAsySmIVP} has the form
\begin{equation}
\label{EqAsySmResExp}
  u(\tau) = \sum_{\sigma\in\Xi} i\chi\res_{\zeta=\sigma}\bigl(\tau^{i\zeta}\wh L(\zeta)^{-1}(\wh f(\zeta)+\wh{f^1}(\zeta))\bigr) + u', \quad u'\in\Hbext^{s,\alpha}(\Omega;E).
\end{equation}
(Since $f$ is a supported distribution in $t_*$ by~\eqref{EqAsySmViewAsSupp1}, the Mellin transform of $f$ here is the same as the Fourier transform in $t_*$ of $f(t_*,x)$, extended by $0$ to $t_*<0$.) This representation suggests the strategy of how to modify $f$ or the initial data (which determine $f^1$) in order to ensure that the solution $u$ of \eqref{EqAsySmIVP} is exponentially decaying: one needs to ensure that $\wh L(\zeta)^{-1}$ applied to the Mellin transform of $f+f^1$ plus the modification is regular at $\zeta\in\Xi$. We describe this in detail in the next section.

\subsubsection{Spaces of resonant states and dual states}
\label{SubsubsecAsySmRes}

As a consequence of Theorem~\ref{ThmAsySmMero}, the family $\wh L(\sigma)^{-1}$ is a meromorphic family of operators on $\CI(Y;E_Y)$ for $\sigma\in\C$. We denote the set of resonances of $L$ by
\begin{equation}
\label{EqAsySmResSet}
  \Res(L) = \bigl\{ \sigma \in\C \colon \wh L(\zeta)^{-1}\tn{ has a pole at }\zeta=\sigma \bigr\}.
\end{equation}
For $\sigma\in\Res(L)$, let us denote by
\begin{equation}
\label{EqAsySmResStates}
  \Res(L,\sigma) = \biggl\{ r\colon \exists\,n\in\N_0,\ r=\sum_{k=0}^n e^{-i\sigma t_*}t_*^k r_k(x),\ L r=0, \ r_k\in\CI(Y;E_Y) \biggr\}
\end{equation}
the space of resonant states at $\sigma$, i.e.\ the space of all (generalized) mode solutions of $L r=0$ with frequency $\sigma$ in $t_*$. For a set $\Xi\subset\C$ containing finitely many resonances of $L$, we put
\[
  \Res(L,\Xi) := \bigoplus_{\sigma\in\Xi}\Res(L,\sigma).
\]

Given $\sigma_0\in\Res(L)$ and $r\in\Res(L,\sigma_0)$, define $f:=L(\chi r)$, where $\chi\in\CI(\Omega^\circ)$ is a cutoff, $\chi=\chi(t_*)$, $\chi\equiv 0$ near $\Sigma_0$ and $\chi\equiv 1$ for large $t_*$; then $f\in\CIc(\Omega^\circ;E^\circ)$, and the forward solution of $L u = f$ is of course $u=\chi r$. Hence, every element of $\Res(L,\sigma_0)$ is realized as the (one term) asymptotic expansion of a forward solution of $L$ with compactly supported smooth forcing. On the Mellin-transformed side, we have $\wh u(\sigma)=\wh L(\sigma)^{-1}\wh f(\sigma)$, and the asymptotic part of $u$ with frequency $\sigma_0$, as a function of $(t_*,x)\in\Omega^\circ$, is equal to $i\res_{\sigma=\sigma_0}\bigl(e^{-i t_*\sigma}\wh L(\sigma)^{-1}\wh f(\sigma)\bigr)$. Since all poles of $\wh L(\sigma)^{-1}$ have finite order, this shows that we can equivalently define
\begin{equation}
\label{EqAsySmResStatesAlt}
\begin{split}
  \Res(L,\sigma) = \biggl\{ \res_{\zeta=\sigma}\bigl(e^{-i t_*\zeta}&\wh L(\zeta)^{-1}p(\zeta)\bigr) \colon p(\zeta)\tn{ polynomial in }\zeta \\
    & \tn{ with values in }\CI(Y;E_Y) \biggr\}.
\end{split}
\end{equation}
Taking the mapping properties of $\wh L(\sigma)$ on Sobolev spaces into account, we can more generally allow $p$ here to take values in $\Hext^{s-1}(Y;E_Y)$ for any fixed $s>1/2-\inf(\beta\Im\sigma)-\wh\beta$.

Fixing a stationary inner product on the bundle $E$---which for the present purpose only needs to be non-degenerate, but not necessarily positive definite---we can define the adjoints $L^*$ and $\wh L(\sigma)^*=\wh{L^*}(\ol\sigma)$. For $\sigma\in\Res(L)$, the space of dual resonant states is then
\begin{equation}
\label{EqAsySmResDualStates}
  \Res^*(L,\sigma) = \biggl\{ r=\sum_{k=0}^n e^{-i\ol\sigma t_*}t_*^k r_k(x) \colon L^* r=0, \ r_k\in\sDsupp'(Y;E_Y)\biggr\}
\end{equation}
For a finite set $\Xi\subset\C$, we put $\Res^*(L,\Xi) = \bigoplus_{\sigma\in\Xi} \Res^*(L,\sigma)$. Analogously to \eqref{EqAsySmResStatesAlt}, we also have
\begin{equation}
\label{EqAsySmResDualStatesAlt}
\begin{split}
  \Res^*(L,\sigma) = \biggl\{ \res_{\zeta=\ol\sigma}\bigl(e^{-i t_*\zeta}&\wh{L^*}(\zeta)^{-1}p(\zeta)\bigr) \colon p(\zeta)\tn{ polynomial in }\zeta \\
    &\hspace{7em} \tn{ with values in }\sDsupp'(Y;E_Y) \biggr\}
\end{split}
\end{equation}
By the below threshold regularity radial point estimate \cite[Proposition~2.4]{VasyMicroKerrdS}, and taking the contribution from the skew-adjoint part of $L$---defined using the inner product used in \eqref{EqAsySm2RadialSubpr}---into account as in \cite[Footnote~5]{HintzVasySemilinear}, one finds that the restriction of a dual resonant state $\psi\in\Res^*(L,\sigma)$ to any $t_*=const.$ slice has regularity $\Hsupp^{1/2+\inf(\beta\Im\sigma)+\wh\beta-0}(Y;E_Y)$; the norm of $\psi(t_*)$ in this space is bounded by $e^{(-\Im\sigma+0)t_*}$ as $t_*\to\infty$.

We now prove a criterion which gives a necessary and sufficient condition on a forcing term $f$ for the solution $u$ of $L u=f$ to \emph{not} have contributions in its (formal) asymptotic expansion coming from a given finite set of resonances:

\begin{prop}
\label{PropAsySmResNoAsy}
  Let $\Xi=\{\sigma_1,\ldots,\sigma_N\}\subset\Res(L)$ be a finite set of resonances, let $R^*=\Res^*(L,\Xi)$, and fix $r>\max\{-\Im\sigma_j\colon 1\leq j\leq N\}$. Suppose $s>1/2+\sup(\beta r)-\wh\beta$. Define the continuous linear map
  \[
    \lambda \colon \Hbsupp^{-\infty,r}\bigl([0,1]_\tau;\Hext^{s-1}(Y;E_Y)\bigr) \ni f \mapsto \la f,\cdot \ra \in \cL(R^*,\ol\C)
  \]
  mapping $f$ to a $\C$-antilinear functional on $R^*$. Then $\lambda(f)=0$ if and only if $\wh L(\sigma)^{-1}\wh f(\sigma)$ is holomorphic in a neighborhood of $\Xi$.
\end{prop}

In particular, this gives a criterion for $f\in\CIc(\Omega^\circ;E^\circ)$; we allow more general $f$ to include the types of terms that arose in \eqref{EqAsySmViewAsSupp1}--\eqref{EqAsySmViewAsSupp2}.

\begin{proof}[Proof of Proposition~\ref{PropAsySmResNoAsy}]
  Given $f$, we note that $\wh L(\sigma)^{-1}\wh f(\sigma)$ is holomorphic near every $\sigma_j$ iff this holds for the pairing $\la \wh f(\sigma),(\wh L(\sigma)^*)^{-1} v\ra$ for every $v\in\Hsupp^{-s}(Y;E_Y)$. Using the definition of $\wh f(\sigma)=\int_0^\infty e^{i\sigma t_*}f(t_*)\,dt_*$, this is equivalent to the holomorphicity of
  \begin{equation}
  \label{EqAsySmResNoAsyCond}
    \la f, e^{-i\ol\sigma t_*}\wh{L^*}(\ol\sigma)^{-1}v\ra
  \end{equation}
  near $\sigma_j$ (where the pairing is the $L^2$ pairing on the spacetime region $\Omega^\circ$), or in other words the vanishing of the principal part of \eqref{EqAsySmResNoAsyCond} at $\sigma_j$. Now, for a fixed $v\in\Hsupp^{-s}(Y;E_Y)$, we have for $\sigma$ near $\sigma_j$
  \begin{equation}
  \label{EqAsySmResNoAsyPole}
    e^{-i\ol\sigma t_*}\wh{L^*}(\ol\sigma)^{-1}v = \sum_{k=1}^{n_j} (\ol\sigma-\ol\sigma_j)^{-k}u_{jk} + v'(\ol\sigma),
  \end{equation}
  with $u_{jk}\in e^{-(\Im\sigma_j)t_*}\Hsupp^{1-s}(Y;E_Y)$ (using $1-s<1/2+\inf(\beta\Im\sigma_j)+\wh\beta$), and $v'$ holomorphic near $\sigma_j$ with values in $\Hbext^{\infty,-r}([0,1]_\tau;\Hbsupp^{1-s}(Y;E_Y))$; here, $n_j\geq 0$, and $u_{j n_j}\neq 0$ in case $n_j\geq 1$. Therefore, the holomorphicity of \eqref{EqAsySmResNoAsyCond} at $\sigma_j$ is equivalent to the holomorphicity of
  \[
    \sum_{k=1}^{n_j} (\sigma-\sigma_j)^{-k} \la f,u_{jk}\ra,
  \]
  at $\sigma_j$, which is equivalent to the condition $\la f,u_{jk}\ra=0$ for all $j=1,\ldots,N$, $k=1,\ldots,n_j$. This latter condition, for fixed $j$, is in turn equivalent to
  \[
    \Big\la f, c.c. \res_{\sigma=\sigma_j} (\sigma-\sigma_j)^{k-1}\ol{e^{-i\ol\sigma t_*}\wh{L^*}(\ol\sigma)^{-1}v} \Big\ra = 0,
  \]
  with $c.c.$ denoting complex conjugation, for all $k=1,\ldots,n_j$. In view of \eqref{EqAsySmResNoAsyPole}, this equality automatically holds for $k>n_j$; hence we arrive at the equivalent condition
  \[
    \Big\la f, c.c. \res_{\sigma=\sigma_j} \ol{e^{-i\ol\sigma t_*}\wh{L^*}(\ol\sigma)^{-1} v(\ol\sigma)} \Big\ra = 0
  \]
  for all polynomials $v(\zeta)$ in $\zeta$ with values in $\Hsupp^{-s}(Y;E_Y)$. Comparing this with \eqref{EqAsySmResDualStatesAlt} shows that the latter condition is equivalent to $\lambda(f)=0$, as claimed.
\end{proof}

Following the discussion around \eqref{EqAsySmViewAsSupp1}, under the hypotheses of this proposition, we can define more generally the map
\begin{equation}
\label{EqAsySmResIVPDualStatePairing}
  \lambda_\IVP \colon D^{s-1,\alpha}(\Omega;E) \to \cL(R^*,\ol\C)
\end{equation}
by
\[
  \lambda_\IVP(f,u_0,u_1) := \lambda\bigl(H f + [L,H](u_0,u_1)\bigr)
\]
with $H=H(t_*)$; recall here the notation $[L,H]$ from \eqref{EqAsySmCommInitial}. Note here that we can define $\lambda_\IVP$ on the larger space $D^{s-1,\alpha}$ due to the regularity properties of dual resonant states; the only reason for assuming $D^{s,\alpha}$ for the data in the global analysis of the initial value problem is the loss in the high energy estimates \eqref{EqAsySmMeroEst}, which are of course irrelevant when one is merely studying resonances in relatively compact subsets of $\C$. 

\begin{cor}
\label{CorAsySmResNoAsy}
  Suppose $\alpha>0$ is as in Theorem~\ref{ThmAsySmMero}, and let $s>1/2+\alpha\sup\beta-\wh\beta$. Let $\Xi$ be the set of all resonances of $L$ with $\Im\sigma>-\alpha$, and let $R^*=\Res^*(L,\Xi)$. Suppose $\cZ\subset D^{s,\alpha}(\Omega;E)$ is a finite-dimensional linear subspace. In terms of the map \eqref{EqAsySmResIVPDualStatePairing}, define the map $\lambda_\cZ\colon\cZ\to\cL(R^*,\ol\C)$ by restriction, $\lambda_\cZ:=\lambda_\IVP|_\cZ$.

  If $\lambda_\cZ$ is surjective, then for all $(f,u_0,u_1)\in D^{s,\alpha}(\Omega;E)$, there exists an element $z\in\cZ$ such that the solution of the initial value problem
  \[
    (L,\gamma_0)u = (f,u_0,u_1) + z
  \]
  has an exponentially decaying solution $u\in\Hbext^{s,\alpha}(\Omega)$.

  If moreover $\lambda_\cZ$ is bijective, this $z$ is unique, and the linear map $(f,u_0,u_1)\mapsto z$ is continuous.
\end{cor}
\begin{proof}
  In view of Proposition~\ref{PropAsySmResNoAsy} and formula~\eqref{EqAsySmResExp}, the task is to find $z\in\cZ$ such that $\lambda_\cZ(z)=-\lambda_\IVP(f,u_0,u_1)\in\cL(R^*,\ol\C)$. If $\lambda_\cZ$ is surjective, this is certainly possible, and in the case $\lambda_\cZ$ is bijective, the map $(f,u_0,u_1)\mapsto z$ is given by the composition $\lambda_\cZ^{-1}\circ\lambda_\IVP$ of continuous linear maps, hence itself linear and continuous.
\end{proof}

In other words, we can solve initial value problems for $L$ in exponentially decaying spaces if we are allowed to modify the forcing or the initial data by elements of a fixed finite-dimensional space.

\begin{rmk}
\label{RmkAsySmResGrushin}
  We can rephrase this also as follows: define the spaces $\cY=D^{s,\alpha}(\Omega;E)$ and $\cX=\{u\in\Hbext^{s,\alpha}(\Omega;E)\colon (L,\gamma_0)u\in\cY\}$, then $(L,\gamma_0)\colon\cX\to\cY$ is injective and has closed range with codimension equal to $\dim R^*$. From this perspective, the space $\cZ$ in Corollary~\ref{CorAsySmResNoAsy} merely provides a complement of $(L,\gamma_0)(\cX)$ within $\cY$ if $\lambda_\cZ$ is bijective, and in the more general case of surjectivity $(L,\gamma_0)(\cX)+\cZ=\cY$. Thus, adding the space $\cZ$ is akin to setting up a Grushin problem for $(L,\gamma_0)$, see \cite[Appendix~D]{ZworskiSemiclassical}.
\end{rmk}

We end this section by recalling the definition of further quantities associated with resonances:

\begin{definition}
\label{DefAsySmResOrderRank}
  The \emph{order} of a resonance $\sigma\in\Res(L)$ is defined to be the order of the pole of the meromorphic function $\wh L(\zeta)^{-1}$ at $\zeta=\sigma$:
  \[
    \ord_{\zeta=\sigma}\wh L(\zeta)^{-1} = \min\bigl\{ \ell\in\N_0 \colon (\zeta-\sigma)^\ell\wh L(\zeta)\tn{ is holomorphic near }\sigma \bigr\}.
  \]
  We define the \emph{rank} of $\sigma\in\Res(L)$ as
  \[
    \rank_{\zeta=\sigma}\wh L(\zeta)^{-1} = \dim \Res(L,\sigma).
  \]
  By convention, $\ord_\sigma\wh L(\zeta)^{-1}=0$ and $\rank_\sigma\wh L(\zeta)^{-1}=0$ for $\sigma\notin\Res(L)$.
\end{definition}

Expanding the exponential in \eqref{EqAsySmResStatesAlt} into its Taylor series around $\zeta=\sigma$, one easily sees using Smith factorization, see \cite{GohbergSigalOperatorResidue}, \cite[\S4.3]{GohbergLeitererHolomorphic} and also \cite[Appendix~A]{HintzThesis}, that one can equivalently define the rank as
\begin{align}
  \rank_{\zeta=\sigma}\wh L(\zeta)^{-1} &= \dim \biggl\{ q(\zeta)=\sum_{j=1}^{\ord_\sigma\wh L(\zeta)^{-1}} q_j(\zeta-\sigma)^{-j} \colon \nonumber\\
    &\hspace{20ex} \wh L(\zeta)q(\zeta) \tn{ is holomorphic near }\sigma \biggr\} \nonumber\\
\label{EqAsySmResRankTrace}&= \frac{1}{2\pi i}\tr\oint_\sigma \wh L(\zeta)^{-1}\pa_\zeta\wh L(\zeta)\,d\zeta,
\end{align}
where $\oint_\sigma$ denotes the contour integral over a small circle around $\sigma$, oriented counter-clockwise.

\subsubsection{Perturbation theory}
\label{SubsubsecAsySmPert}

We now make the linear operator depend on a finite-dimensional parameter $w\in W\subset\R^{N_W}$, with $W$ an open neighborhood of some fixed $w_0\in\R^{N_W}$. We then assume that for every $w\in W$, we are given a stationary, principally scalar operator $L_w\in\Diffb^2(M;E)$ depending continuously on $w$, and $L_{w_0}=L$ satisfies the assumptions \eqref{ItAsySm1Sym}--\eqref{ItAsySm3Trapped} of \S\ref{SubsecAsySm}. For simplicity, we make the additional assumption that the principal symbol is
\begin{equation}
\label{EqAsySmPertSymbol}
  \sigma_{\bop,2}(L_w)(\zeta)=|\zeta|_{G_{b(w)}}^2\otimes\Id,
\end{equation}
where $b(w)\in\cU_B$ depends continuously on $w\in W$; this will be satisfied in our applications.

Simple examples are $W=\cU_B$, $w_0=b_0$, and $L_w=\Box_{g_w}$ for $w\in\cU_B$ is the tensor wave operator on any tensor bundle over $M$. Linearizations of the gauged Einstein equation will be the objects of interest for the linear stability problem, see \S\ref{SecKey} and \S\ref{SecKdSLStab}.

Returning to the general setup, note that the supremum and infimum of the radial point quantity $\beta$ as well as the real number $\wt\beta$, defined in \eqref{EqAsySm2RadialWeight}--\eqref{EqAsySm2RadialSubprInf}, depend continuously on $w$; furthermore, if \eqref{EqAsySm3TrappedSubpr} (and \eqref{EqAsySm3TrappedBound}, if one were working in the more general setting) are satisfied for $L=L_{w_0}$ for some fixed elliptic operator $Q$ and positive definite inner product on $E$, then it holds for $L=L_w$ as well (with $\numin$ denoting the minimal expansion rate for the Hamilton flow of $L_w$ at the trapped set of $L_w$), provided $w$ is near $w_0$; in the more general setting, this follows from the structural stability of the trapped set, see \cite{HirschShubPughInvariantManifolds} and \cite[\S5.2]{DyatlovResonanceProjectors}, but under the assumption \eqref{EqAsySmPertSymbol}, this follows more simply by a direct computation of the location of the trapped set and the minimal expansion rate, see \cite[\S3.2]{DyatlovWaveAsymptotics}. Thus, the assumptions in \S\ref{SubsecAsySm} hold uniformly for $L_w$ with $w$ near $w_0$, and consequently Theorem~\ref{ThmAsySmMero} holds as well, with the constants $C,C_1,C_2,\alpha$ \emph{uniform} in $w\in W$, shrinking $W$ if necessary.

The estimate on $\wh{L_w}(\sigma)^{-1}$ for $\Im\sigma>C_2$, with $C_2>0$ sufficiently large, mentioned in the proof of Theorem~\ref{ThmAsySmMero} also holds uniformly in $w$; alternatively, the global energy estimates for $L_w$ on growing function spaces $\Hb^{s,-C_2}(\Omega;E)$ (increasing $C_2$ slightly if necessary) hold uniformly in $w$. Either way, it remains to study the dependence of $\wh{L_w}(\sigma)$ on $w\in W$ for $\sigma$ in the precompact set $\{-\alpha<\Im\sigma<C_2+1,\ |\Re\sigma|<C_1+1\}\subset\C$, where $\wh{L_w}(\sigma)^{-1}$ has only finitely many poles for any fixed $w\in W$ by Theorem~\ref{ThmAsySmMero}. This was first discussed in \cite[\S2.7]{VasyMicroKerrdS} and \cite[Appendix~A]{HintzThesis}; we prove a slight extension:

\begin{prop}
\label{PropAsySmPert}
  Let $V\subset\C$ be a non-empty precompact open set such that $\Res(L_{w_0})\cap\pa V=\emptyset$, and fix $s_0>1/2-\inf_{\sigma\in V}(\beta\Im\sigma)-\wh\beta$. Then:
  \begin{enumerate}
    \item \label{ItAsySmPert1Open} The set $I:=\{(w,\sigma)\in W\times V\colon \wh{L_w}(\sigma)^{-1}\tn{ exists}\}$ is open.
    \item \label{ItAsySmPert2Cont} The map $I\ni(w,\sigma)\mapsto\wh{L_w}(\sigma)^{-1}\in\cL_\weak(\Hext^{s-1},\Hext^s)$ (the space of bounded linear operators, equipped with the weak operator topology) is continuous for all $s>s_0$, and also as a map into $\cL_\op(\Hext^{s-1+\eps},\Hext^{s-\eps})$ (i.e.\ equipped with the norm topology) for $s>s_0$ and all $\eps>0$. Here, $\Hext^\rho\equiv\Hext^\rho(Y;E_Y)$.
    \item \label{ItAsySmPert3TotalRank} The set $\Res(L_w)\cap V$ depends continuously on $w$ in the Hausdorff distance sense, and the total rank
      \[
        D := \sum_{\sigma\in\Res(L_w)\cap V} \rank_{\zeta=\sigma} \wh{L_w}(\zeta)^{-1}
      \]
      is constant for $w$ near $w_0$.
    \item \label{ItAsySmPert4ResCont} The total space of resonant states $\Res(L_w,V)\subset\CI(\Omega^\circ;E^\circ)$ depends continuously on $w$ in the sense that there exists a continuous map $W\times\C^D \to \CI(\Omega^\circ;E^\circ)$ such that $\Res(L_w,V)$ is the image of $\{w\}\times\C^D$.
    \item \label{ItAsySmPert5DualCont} Likewise, fixing a smooth inner product on $E_Y$, the total space of dual states $\Res^*(L_w,V)\subset\Hsupploc^{1-s_0}(\Omega^\circ;E^\circ)$ depends continuously on $w$.
  \end{enumerate}
\end{prop}
\begin{proof}
  The main input of the proof is the fact that we have the estimates
  \begin{equation}
  \label{EqAsySmPertEst}
  \begin{gathered}
    \|u\|_{\Hext^s} \leq C\bigl(\|\wh{L_w}(\sigma)u\|_{\Hext^{s-1}} + \|u\|_{\Hext^{s_0}}\bigr), \\
    \|v\|_{\Hext^{1-s}} \leq C\bigl(\|\wh{L_w}(\sigma)^* v\|_{\Hext^{-s}} + \|v\|_{\Hext^{-N}}\bigr),
  \end{gathered}
  \end{equation}
  with $s>s_0$ and $N>s-1$ arbitrary but fixed, \emph{uniformly} for $w\in W$ and $\sigma\in V$, and the invertibility of $\wh{L_{w_0}}(\sigma)$ at some point $\sigma\in V$. The statements \eqref{ItAsySmPert1Open} and \eqref{ItAsySmPert2Cont} then follow from a simple functional analysis argument, see \cite[\S2.7]{VasyMicroKerrdS}. We remark that the first part of statement~\eqref{ItAsySmPert2Cont} is false if one uses instead the operator norm topology since the characteristic set of $\wh{L_w}(\sigma)$ varies with $w$, see also \cite[Footnote~33]{HintzQuasilinearDS}.

  In order to prove the remaining statements \eqref{ItAsySmPert3TotalRank}--\eqref{ItAsySmPert5DualCont}, it suffices to consider the case that $V$ is a small disc $V=\{|\sigma-\sigma_0|<\eps\}$ around a resonance $\sigma_0\in\Res(L_{w_0})$, with $\Res(L_{w_0})\cap\pa V=\emptyset$ and $\Res(L_{w_0})\cap V=\{\sigma_0\}$. Now $\wh{L_{w_0}}(\sigma_0)$ has index $0$ as an operator $\wh{L_{w_0}}(\sigma_0)\colon\cX^s_{w_0}\to\cY^{s-1}_{w_0}$, where $\cX_w^s=\{u\in \Hext^s\colon \wh{L_w}(\sigma_0)u\in \Hext^{s-1}\}$ and $\cY_w^{s-1}=\Hext^{s-1}$; its kernel is a subspace of $\CI(Y;E_Y)$, say with basis $\{u_1,\ldots,u_n\}$, and the range
  \[
    \cY_1 = \ran_{\cX_0^s} \wh{L_{w_0}}(\sigma_0) \subset \Hext^{s-1},
  \]
  being closed and of codimension $n$, has a complement $\cY_2\subset\CI(Y;E_Y)$ with basis $\{f_1,\ldots,f_n\}$. Now, let
  \[
    R u := \sum_{k=1}^n \la u,u_k\ra f_k,
  \]
  where the pairing is the $L^2$-pairing on $Y$, using any fixed positive definite inner product on $E_Y$. Then the family of operators
  \[
    P_w(\sigma) := \wh{L_w}(\sigma)+R \colon \cX^s_w \to \cY^{s-1}_w
  \]
  is invertible at $w=w_0$, $\sigma=\sigma_0$; moreover $P_w(\sigma)$ also satisfies the estimates \eqref{EqAsySmPertEst} (with a different constant $C$), since the contribution of $R\colon \Hext^{-\infty}\to \Hext^\infty$ can be absorbed into the error term. Therefore, parts \eqref{ItAsySmPert1Open} and \eqref{ItAsySmPert2Cont} apply to the family $P_w(\sigma)$ as well; shrinking $\eps>0$ and the parameter space $W$ if necessary, we may thus assume that $P_w(\sigma)\colon\cX^s_w\to\cY^{s-1}_w$ is invertible on $W\times V$ (note that for $u\in\cX^s_w$, we have $P_w(\sigma)u\in\cY^{s-1}_w$ indeed), with continuous inverse family in the sense of \eqref{ItAsySmPert2Cont}. Writing
  \[
    \wh{L_w}(\sigma) = Q_w(\sigma)P_w(\sigma),\quad Q_w(\sigma)=\Id-R P_w(\sigma)^{-1}\colon\Hext^{s-1}\to\Hext^{s-1},
  \]
  the invertibility of $\wh{L_w}(\sigma)$ is thus equivalent to that of $Q_w(\sigma)$; but in the decomposition $\Hext^{s-1}=\cY_1\oplus\cY_2$, we have
  \[
    Q_w(\sigma) = \begin{pmatrix} \Id & Q_{w,1}(\sigma) \\ 0 & Q_{w,2}(\sigma) \end{pmatrix},
  \]
  where
  \[
    Q_{w,1}(\sigma)=-R P_w(\sigma)^{-1}|_{\cY_2}\in\cL(\cY_2,\cY_1), \ Q_{w,2}(\sigma)=\Id-R P_w(\sigma)^{-1}|_{\cY_2}\in\cL(\cY_2).
  \]
  Therefore, the invertibility of $Q_w(\sigma)$ is in turn equivalent to that of the operator $Q_{w,2}(\sigma)$---which depends continuously on $w\in W$ and holomorphically on $\sigma\in V$---acting on the \emph{fixed, finite-dimensional space} $\cY_2$. For fixed $w$, the contour integral expression \eqref{EqAsySmResRankTrace} yields $\rank_\sigma\wh{L_w}(\zeta)^{-1}=\rank_\sigma Q_{w,2}(\zeta)^{-1}$. Therefore, claim~\eqref{ItAsySmPert3TotalRank} follows from
  \[
    \sum_{\sigma\in\Res(L_w)\cap V}\rank_{\zeta=\sigma}\wh{L_w}(\zeta)^{-1} = \frac{1}{2\pi i}\oint_{\pa V} Q_{w,2}(\zeta)^{-1}\pa_\zeta Q_{w,2}(\zeta)\,d\zeta,
  \]
  the latter expression being integer-valued and continuous in $w$.

  For establishing claim \eqref{ItAsySmPert4ResCont}, pick polynomials $p_1(\zeta),\ldots,p_D(\zeta)$ with values in $\CI(Y;E_Y)$ such that the sections
  \[
    \phi_j(0) := \oint_{\pa V} \wh{L_{w_0}}(\zeta)^{-1}p_j(\zeta)\,d\zeta \in \CI(Y;E_Y)
  \]
  span $\Res(L_{w_0},\sigma_0)$. For sufficiently small $w\in W$, $\wh{L_w}(\zeta)^{-1}$ exists for $\zeta\in\pa V$, hence the contour integral
  \[
    \phi_j(w) := \oint_{\pa V} \wh{L_w}(\zeta)^{-1}p_j(\zeta)\,d\zeta \in \CI(Y;E_Y)
  \]
  is well-defined. By \eqref{ItAsySmPert2Cont}, $\phi_j(w)$ depends continuously on $w$ in the topology of $\CI(Y;E_Y)$; therefore $\{\phi_1(w),\ldots,\phi_D(w)\}\subset\CI(Y;E_Y)$ is $D$-dimensional for small $w$. On the other hand, we have $\phi_j(w)\in\Res(L_w,V)$ for all $j$, and $\Res(L_w,V)$ is $D$-dimensional as well, so we conclude that $\Res(L_w,V)=\mathspan\{\phi_1(w),\ldots,\phi_D(w)\}$. The statement now follows for the map $W\times\C^D\ni(w,(c_1,\ldots,c_D))\mapsto\sum c_j\phi_j(w)$.

  The proof of the corresponding statement \eqref{ItAsySmPert5DualCont} for dual states proceeds in the same manner.
\end{proof}

Part~\eqref{ItAsySmPert5DualCont} of this proposition implies that Corollary~\ref{CorAsySmResNoAsy} holds uniformly for all $L_w$ with $\cZ$ fixed; we state a more general version, allowing the space $\cZ$ of modifications to depend on $w\in W$ as well. In our applications, the space $\cZ$ will naturally be a sum of finite-dimensional spaces $\cZ_j$ each parameterized by vectors in some $\C^{N_j}$; however, the sum of the $\cZ_j$ may not be direct, and its dimension may be different for different values of $w\in W$. A robust description of $\cZ$ therefore rather amounts to parameterizing its elements by $\bigoplus\C^{N_j}$, with the parameterization possibly not being one to one. This motivates the assumption in the following result:

\begin{cor}
\label{CorAsySmPertNoAsy}
  Under the assumptions stated at the beginning of this section, and under the assumptions and using the notation of Corollary~\ref{CorAsySmResNoAsy}, let $V$ be a small open neighborhood of the set $\Xi$ of resonances $\sigma$ with $\Im\sigma>-\alpha$. Suppose ${N_\cZ}\in\N_0$, and suppose we are given a continuous map
  \[
    z \colon W \times \C^{N_\cZ} \to D^{s,\alpha}(\Omega;E)
  \]
  which is linear in the second argument; we will write $z_w^{\bfc}\equiv z(w,\bfc)$. Define the map\footnote{Thus, $\lambda_{w_0}(\bfc)=\lambda_\cZ(z_{w_0}^{\bfc})$ in the notation of Corollary~\ref{CorAsySmResNoAsy} if $\cZ=z(w_0,\C^{N_\cZ})$.}
  \begin{align*}
    \lambda_w\colon\C^{N_\cZ} & \to\cL(\Res^*(L_w,V),\ol\C), \\
    \bfc & \mapsto \lambda_\IVP(z_w^{\bfc}).
  \end{align*}
  Then the bijectivity, resp.\ surjectivity, of $\lambda_{w_0}$ implies the bijectivity, resp.\ surjectivity, of $\lambda_w$ for $w\in W$ near $w_0$.

  Furthermore, assuming $\lambda_{w_0}$ is surjective, there exists a continuous map
  \begin{equation}
  \label{EqAsySmPertNoAsyMap}
  \begin{split}
    W \times D^{s,\alpha}(\Omega;E) \ni (w,(f,u_0,u_1)) \mapsto \bfc\in\C^{N_\cZ},
  \end{split}
  \end{equation}
  linear in the second argument, such that the initial value problem
  \[
    (L_w,\gamma_0)u = (f,u_0,u_1) + z_w^{\bfc}
  \]
  has an exponentially decaying solution $u\in\Hbext^{s,\alpha}(\Omega)$, and this solution $u$ depends continuously on $(w,f,u_0,u_1)$ as well. If $\lambda_{w_0}$ is bijective, then $\bfc$ and $u$ are unique given $w\in W$ and the data $(f,u_0,u_1)$.
\end{cor}
\begin{proof}
  We use Proposition~\ref{PropAsySmPert} \eqref{ItAsySmPert5DualCont} and the parameterization of the family of spaces $\Res^*(L_w,V)$ by means of a continuous map with domain $W\times\C^D$, linear in the second argument; then $\lambda_w$ is represented by a complex $D\times N_\cZ$ matrix which depends continuously on $w\in W$. The lower semicontinuity of the rank proves the first part.

  For the second part, we pick a $D$-dimensional subspace $C\subset\C^{N_\cZ}$ such that $\lambda_{w_0}|_C$ (and hence $\lambda_w|_C$) is bijective. Then we can define a map \eqref{EqAsySmPertNoAsyMap} with the stated properties by $(\lambda_w|_C)^{-1}\circ\lambda_\IVP(f,u_0,u_1)$, composed with the inclusion $C\hra\C^{N_\cZ}$.
\end{proof}

\subsection{Non-smooth exponentially decaying perturbations}
\label{SubsecAsyExp}

We now turn to the linear analysis of wave-type operators with non-smooth coefficients; thus, we will allow more general perturbations than those considered in \S\ref{SubsubsecAsySmPert}. We aim to prove an analogue of Corollary~\ref{CorAsySmPertNoAsy} together with tame estimates for the map taking the Cauchy data and forcing term into the solution and the finite-dimensional modification (encoded in the map $z$ above); these will enable us to appeal to the Nash--Moser iteration scheme when solving the Einstein vacuum equations in \S\ref{SecKdSStab}. Recall here that a \emph{tame estimate} is schematically of the form
\[
  \|u\|_s \leq C_s\bigl(\|f\|_{s+d} + \|\ell\|_{s+d}\|f\|_{s_0}\bigr),
\]
where $\|\cdot\|_s$ are e.g.\ Sobolev norms, $\ell$ denotes the coefficients of the linear operator $L$, and $L u=f$, and $\|\ell\|_{s_0}$ is assumed to have an a priori bound, while the constant $C_s$ is \emph{uniform} for such bounded $\ell$ (thus, the estimate holds uniformly for suitable perturbations of $L$); further, $d\in\R$ is the loss of derivatives, $s\geq s_0$, and $s_0$ is some a priori regularity. The point is that the right hand side contains high regularity norms $\|\cdot\|_{s+d}$ only in the first power. A very simple example is if $L$ is division by a non-smooth function $\ell>0$ on a closed $n$-dimensional manifold; in this case, the tame estimate is a version of Moser estimates for products of $H^s$-functions, and one can take $d=0$ and $s_0>n/2$; see e.g.\ \cite[\S13.3]{TaylorPDE} and Lemma~\ref{LemmaAsyExpLocalHs} below. In fact, since all our estimates are proved by (microlocal) energy methods, it is rather clear that the tame bounds all come from such Moser estimates; therefore, if one is content with having a tame estimate without explicit control on the loss $d$ and the minimal regularity $s_0$, one can skip a number of arguments below. (See also the discussion in \cite[\S4.1]{HintzVasyQuasilinearKdS}.) For the sake of completeness, and in order to show the sufficiency of the regularity of the initial data assumed for Einstein's field equations in Theorem~\ref{ThmIntroPrecise}, we do prove explicit bounds, however we will be rather generous with the number of derivatives in order to minimize the amount of bookkeeping required.

Concretely then, extending the scope of the smooth coefficient perturbation theory, we now consider a continuous family of second order differential operators
\begin{equation}
\label{EqAsyExpOpFull}
  L_{w,\wt w} = L_w + \wt L_{w,\wt w}.
\end{equation}
Here, $L_w$ is as in \S\ref{SubsubsecAsySmPert}, i.e.\ $L_w$ has scalar principal symbol $G_{b(w)}$, depending on the parameter
\begin{equation}
\label{EqAsyExpParamSpace}
  w\in W=\{p\in\R^{N_W}\colon |p-w_0|<\eps\},
\end{equation}
and $L=L_{w_0}$ satisfies the conditions of \S\ref{SubsecAsySm}. On the other hand, the parameter $\wt w$ lies in a neighborhood
\begin{equation}
\label{EqAsyExpParamSpaceTilde}
  \wt w \in \wt W^s = \{ \wt u\in\Hbext^{s,\alpha}(\Omega;E) \colon \|\wt u\|_{\Hbext^{14,\alpha}} < \eps \} \subset \Hbext^{s,\alpha}(\Omega;E)
\end{equation}
of $0$, where $s\geq 14$. For $\wt w\in\wt W^s$, we assume that
\begin{equation}
\label{EqAsyExpOpDec}
  \wt L_{w,\wt w} \in \Hbext^{s,\alpha}(\Omega)\Diffb^2(\Omega;E)
\end{equation}
is an operator with real-valued scalar principal symbol which is hyperbolic with respect to the level sets of $t_*$; we also assume that the coefficients of $\wt L_{w,\wt w}$ are exponentially decaying with the rate $\alpha>0$ defined in Theorem~\ref{ThmAsySmMero}. (Indeed, in local coordinates $(t_*,x)$ and a local trivialization of $E_Y$, the condition \eqref{EqAsyExpOpDec} precisely means that $\wt L_{w,\wt w}$ has coefficients lying in the space $\Hbext^{s,\alpha}$.) We assume that $\wt L_{w_0,0}=0$, and we require that $\wt L_{w,\wt w}$ depends continuously on $(w,\wt w)\in W\times\wt W^s$ in a tame fashion, that is, for $s\geq 14$,
\begin{equation}
\label{EqAsyExpDecTame}
  \|\wt L_{w_1,\wt w_1}-\wt L_{w_2,\wt w_2}\|_{\Hbext^{s,\alpha}\Diffb^2} \leq C_s\bigl(|w_1-w_2|+\|\wt w_1-\wt w_2\|_{\Hbext^{s,\alpha}}\bigr)
\end{equation}
for a constant $C_s<\infty$ depending only on $s$, where the norm on the left is the sum of the $\Hbext^{s,\alpha}(\Omega;E)$ norms of the coefficients of $\wt L_{w,\wt w}$ in a fixed finite covering of $\Omega$ by stationary local coordinate charts. For $s=14$, $\wt L_{w,\wt w}$ is thus a bounded family in $\Hbext^{14,\alpha}\Diffb^2(\Omega;E)$.

\begin{rmk}
\label{RmkAsyExpDataNumerology}
  Our assumptions are motivated by our application to the linear operators one needs to invert in order to solve the Einstein vacuum equations in \S\ref{SecKdSStab}; in this case, the parameter space $W$ is a neighborhood $\cU_B\subset B$ of Schwarzschild--de~Sitter parameters $b_0$ as before, together with additional finite-dimensional parameters related to modifications of the gauge, and $\wt W^s$ consists of the non-\-sta\-tion\-ary, exponentially decaying part of the Lorentzian metric at some finite step in the non-linear iteration scheme, while $L_{w,\wt w}$ is the wave operator associated to the metric $g_w+\wt w$, which therefore has coefficients with regularity $s-2$; thus, when we appeal to results of the present section in the application in \S\ref{SecKdSStab}, there will a shift of $2$ in the norm on $\wt w$.
\end{rmk}

\emph{For brevity, we omit the function spaces from the notation of norms and only keep the regularity and weight parameters; from the context it will always be clear what function space is meant.}

In this section, we shall prove:

\begin{thm}
\label{ThmAsyExpMain}
  Assume $\wh\beta\geq-1$ in \eqref{EqAsySm2RadialSubprInf}. Suppose we are given a continuous map
  \[
    z \colon W\times\wt W^s\times\C^{N_\cZ} \to D^{s,\alpha}(\Omega;E)
  \]
  which is linear in the last argument; we shall often write $z_{w,\wt w}^{\bfc}\equiv z(w,\wt w,\bfc)$. With $\Xi=\Res(L_{w_0,0})\cap\{\Im\sigma>-\alpha\}$, suppose moreover that the map
  \begin{equation}
  \label{EqAsyExpMainMod}
    \C^{N_\cZ} \ni \bfc \mapsto \lambda_\IVP(z_{w_0,0}^{\bfc})\in\cL(\Res^*(L_{w_0,0},\Xi),\ol\C)
  \end{equation}
  is surjective, with $\lambda_\IVP$ defined in \eqref{EqAsySmResIVPDualStatePairing}. Then, if $\eps>0$ in~\eqref{EqAsyExpParamSpace}--\eqref{EqAsyExpParamSpaceTilde} is small enough, there exists a continuous map
  \begin{equation}
  \label{EqAsyExpMainModMap}
  \begin{split}
    S \colon W\times\wt W^\infty\times D^{\infty,\alpha}(\Omega;E)&\ni (w,\wt w,(f,u_0,u_1)) \\
      &\qquad \mapsto (\bfc,u) \in \C^{N_\cZ} \oplus \Hbext^{\infty,\alpha}(\Omega;E),
  \end{split}
  \end{equation}
  linear in $(f,u_0,u_1)$, such that the function $u$ is a solution of
  \begin{equation}
  \label{EqAsyExpMainIVP}
    (L_{w,\wt w},\gamma_0)u = (f,u_0,u_1) + z_{w,\wt w}^{\bfc}.
  \end{equation}
  Furthermore, the map $S$ satisfies the tame estimates
  \begin{gather}
  \label{EqAsyExpMainSTameC}
    \|\bfc\| \leq C\|(f,u_0,u_1)\|_{13,\alpha}, \\
  \label{EqAsyExpMainSTameU}
    \|u\|_{s,\alpha} \leq C_s\bigl(\|(f,u_0,u_1)\|_{s+3,\alpha} + (1+\|\wt w\|_{s+4,\alpha})\|(f,u_0,u_1)\|_{13,\alpha}\bigr)
  \end{gather}
  for $s\geq 10$. In fact, the map $S$ is defined for any $\wt w\in\wt W^{14}$ and $(f,u_0,u_1)$ for which the norms on the right hand sides of \eqref{EqAsyExpMainSTameC}--\eqref{EqAsyExpMainSTameU} are finite, and produces a solution $(\bfc,u)$ of \eqref{EqAsyExpMainIVP} satisfying the tame estimates.

  If the map \eqref{EqAsyExpMainMod} is bijective, then the map $S$ with these properties is unique.
\end{thm}

\begin{rmk}
  We make the assumption on $\wh\beta$ for the sake of simplicity of presentation; it is satisfied in our application, as we show in~\S\ref{SubsecESGRadial}. For more general $\wh\beta$, put $k=\max(-1-\wh\beta,0)$; then the same results hold true if we increase the regularity parameters throughout the statement of this theorem by $k$; more precisely, the number $14$ in~\eqref{EqAsyExpParamSpaceTilde} is replaced by $14+k$, and all Sobolev regularities in the estimates~\eqref{EqAsyExpMainSTameC}--\eqref{EqAsyExpMainSTameU} are increased by $k$.
\end{rmk}

The assumption on the map \eqref{EqAsyExpMainMod} is precisely the surjectivity assumption we made in Corollary~\ref{CorAsySmPertNoAsy}, and in a less general form in Corollary~\ref{CorAsySmResNoAsy}; we stress that this is an assumption \emph{only on $L_{w_0,0}$ and $z(w_0,0,\cdot)$}, yet it guarantees the solvability of linear wave-type operators which are merely close to $L_{w_0,0}$ on decaying function spaces after finite-dimensional modifications.

The continuity of the solution map $S$ is not needed in order to prove the existence of global solutions to Einstein's field equations later on, only the uniformity of the estimates \eqref{EqAsyExpMainSTameC}--\eqref{EqAsyExpMainSTameU} matters. However, it can be used to prove the smooth dependence of these global solutions on the initial data, see Theorem~\ref{ThmKdSStab}.

\subsubsection{Local Cauchy theory}
\label{SubsubsecAsyExpLocal}

Since we explicitly consider initial value problems, with forcing terms which are extendible distributions at $\Sigma_0$, rather than only forward problems with supported forcing terms as in \cite{VasyMicroKerrdS,HintzVasySemilinear,HintzVasyQuasilinearKdS}, we need to get the regularity analysis started using energy estimates near $\Sigma_0$; once this is done, the usual microlocal propagation of singularities on the spacetime $M$ can be used to propagate this, as in the cited papers. See \S\ref{SubsubsecAsyExpReg} for details.

Let us fix $T<\infty$. We define the Banach space of data in a finite slab $\Omega_T:=\Omega\cap\{0\leq t_*\leq T\}$ by
\[
  D^s(\Omega_T;E) = \Hext^s(\Omega_T;E) \oplus \Hext^{s+1}(\Sigma_0;E_{\Sigma_0}) \oplus \Hext^s(\Sigma_0;E_{\Sigma_0}),
\]
equipped with its natural norm; since $T<\infty$, this space has no index for a weight in $t_*$. Then:
\begin{prop}
\label{PropAsyExpLocal}
  For $\wt w\in\wt W^\infty$, the initial value problem
  \begin{equation}
  \label{EqAsyExpLocalIVP}
    \begin{cases}
      L_{w,\wt w}u = f & \tn{in }\Omega_T, \\
      \gamma_0(u) = (u_0,u_1) & \tn{on }\Sigma_0
    \end{cases}
  \end{equation}
  with data $(f,u_0,u_1)\in D^\infty(\Omega_T;E)$ has a unique solution $u=S_T(w,\wt w,(f,u_0,u_1))\in\Hext^\infty(\Omega_T;E)$ satisfying the tame estimate
  \begin{equation}
  \label{EqAsyExpLocalTame}
    \|u\|_{s+1} \leq C\bigl(\|(f,u_0,u_1)\|_s + (1+\|\wt w\|_{s,\alpha})\|(f,u_0,u_1)\|_4\bigr)
  \end{equation}
  for $s\geq 4$; the constant $C$ depends only on $s$. In fact, for $\wt w\in\wt W^{14}$ and $(f,u_0,u_1)$ for which the right hand side is finite, there exists a unique solution $u\in\Hext^{s+1}(\Omega_T;E)$, and the estimate holds.
  
  Moreover, the solution map $S_T$ is continuous as a map
  \begin{equation}
  \label{EqAsyExpLocalC0}
    S_T \colon W\times\wt W^\infty\times D^\infty(\Omega_T;E) \to \Hext^\infty(\Omega_T;E).
  \end{equation}
\end{prop}

We first recall basic tame estimates on Sobolev spaces, which we state on $\R^n$ for simplicity; analogous results hold on closed manifolds and for supported or extendible Sobolev spaces on manifolds with corners.
\begin{lemma}
\label{LemmaAsyExpLocalHs}
  (See \cite[Corollary~3.2 and Proposition~3.4]{HintzVasyQuasilinearKdS}.)
  \begin{enumerate}
  \item\label{ItAsyExpLocalHsProd} Let $s_0>n/2$, $s\geq 0$. For $u,v\in\CIc(\R^n)$, we have
    \begin{gather}
    \label{EqAsyExpLocalHsProdLow}
      \|u v\|_s \leq C\|u\|_{\max(s_0,s)}\|v\|_s, \\
    \label{EqAsyExpLocalHsProd}
      \|u v\|_s \leq C(\|u\|_{s_0}\|v\|_s + \|u\|_s\|v\|_{s_0}).
    \end{gather}
  \item Let $s,s_0>n/2+1$, and let $K\Subset U\subset\R^n$, with $U$ open. Suppose that $u,w\in\CIc(\R^n)$, $\supp w\subset K$ and $|u|\geq 1$ on $U$. Then
    \begin{equation}
    \label{EqAsyExpLocalHsInv}
      \|w/u\|_s \leq C(\|u\|_{s_0})\bigl(\|w\|_s + (1+\|u\|_s)\|w\|_{s_0}\bigr).
    \end{equation}
  \item Let $\Lambda^s=\la D\ra^s$. Then for $s_0>n/2+1$, $s\geq 1$, and $u,v\in\CIc(\R^n)$,
    \begin{equation}
    \label{EqAsyExpLocalHsComm}
      \|[\Lambda^s,u]v\|_0 \leq C(\|u\|_{s_0}\|v\|_{s-1} + \|u\|_s\|v\|_{s_0-1}).
    \end{equation}
  \end{enumerate}
  Each estimate continues to hold for $u,v,w\in H^{-\infty}(\R^n)$ assuming only that the norms on its right hand side are finite, with the additional assumption that $|u|\geq 1$ on $U$ for \eqref{EqAsyExpLocalHsInv}.
\end{lemma}
\begin{proof}
  The estimates \eqref{EqAsyExpLocalHsProdLow}--\eqref{EqAsyExpLocalHsInv} are special cases of the cited statements. The commutator estimate \eqref{EqAsyExpLocalHsComm} is contained in the proof of \cite[Proposition~3.4]{HintzVasyQuasilinearKdS}.\footnote{There is a typo in the reference: the correct estimate reads $\|[\Lambda_{s'},u]v\|_0\leq C_{\mu\nu}(\|u\|_\mu\|v\|_{s'-1}+\|u\|_s\|v\|_\nu)$.}
\end{proof}

\begin{proof}[Proof of Proposition~\ref{PropAsyExpLocal}]
  The vector bundle $E$ is irrelevant for our arguments, hence we drop it from the notation.

  The last statement is an immediate consequence of the previous parts of the proposition: indeed, for $w_j\in W$, $\wt w_j\in\wt W^{s+1}$ and $(f_j,u_{j,0},u_{j,1})\in D^{s+1}(\Omega_T)$, the solutions $u_j=S_T(w_j,\wt w_j,(f_j,u_{j,0},u_{j,1}))\in\Hext^{s+2}(\Omega_T)$ satisfy the equation
  \[
    L_{w_1,\wt w_1}(u_1-u_2) = f_1-f_2 - (L_{w_1,\wt w_1}-L_{w_2,\wt w_2})u_2,
  \]
  and the estimate \eqref{EqAsyExpLocalTame} with $4$ replaced by $s$, together with the estimates \eqref{EqAsyExpDecTame} and \eqref{EqAsyExpLocalHsProdLow}, gives
  \begin{align*}
    \|u_1-u_2\|_{s+1} &\leq C(1+\|\wt w_1\|_{s,\alpha}) \\
      &\qquad\times \bigl(\|f_1-f_2\|_s + (|w_1-w_2|+\|\wt w_1-\wt w_2\|_{s,\alpha})\|u_2\|_{s+2}\bigr),
  \end{align*}
  which implies the continuity of $S_T$, as desired.

  In order to prove the existence of solutions of \eqref{EqAsyExpLocalIVP} and the tame estimate \eqref{EqAsyExpLocalTame}, we follow the arguments presented in \cite[\S16.1--16.3]{TaylorPDE} and keep track of the dependence of the estimates on the coefficients of the operator. First, we will prove the existence of a solution $u\in\cC^0([0,T],\Hext^{s+1}(Y))\cap\cC^1([0,T],\Hext^s(Y))$ for $f\in\cC^0([0,T],\Hext^s(Y))$, together with a tame estimate; it suffices to do this for small $T$ independent of the parameters and the data, since one can then iterate the solution to obtain the result for any finite $T$. We let $x^0=t_*$ and $x^1,x^2,x^3$ be coordinates on $X$. We write
  \begin{equation}
  \label{EqAsyExpLocalEvolEq}
    L_{w,\wt w} = \wt a^{00}D_0^2 - \wt a^{jk}D_jD_k - \wt b^j D_0 D_j + \wt c^\mu D_\mu + \wt d,
  \end{equation}
  where the coefficients depend on $w,\wt w$, are uniformly bounded in $\Hext^4$, and in the space $\Hext^s$, they depend continuously on $w\in W$ and $\wt w\in\wt W^s$. Moreover, $\wt a^{00}\neq 0$, since this is true for $L_{w_0,0}$ by the construction of the Kerr--de~Sitter metrics, see \S\ref{SubsecKdSWave}. (This is simply the statement that $\Sigma_0$ is non-characteristic.) When solving \eqref{EqAsyExpLocalIVP}, we can thus divide both sides by $\wt a^{00}$; by Lemma~\ref{LemmaAsyExpLocalHs}, we have
  \[
    \|f/\wt a^{00}\|_s \leq C_s(\|f\|_s + \|\wt a^{00}\|_s\|f\|_{s_0})
  \]
  for $s\geq s_0>3$. Similar estimates hold for the coefficients of $L_{w,\wt w}/\wt a^{00}$. We thus merely need to establish the solvability of the initial value problem for an operator
  \[
    \wt L := D_0^2 - a^{jk}D_jD_k - a^{0j} D_0D_j  - a^\mu D_\mu - a,
  \]
  $0\leq\mu\leq 3$, $1\leq j,k\leq 3$, $(a^{jk})$ symmetric, with coefficients in $\Hext^s$, and prove a tame estimate
  \begin{equation}
  \label{EqAsyExpLocalTamePf}
    \|u\|_{\cC^0\Hext^{s+1}\cap\cC^1\Hext^s} \leq C\bigl(\|(f,u_0,u_1)\|_s + (1+\|\bfa\|_s)\|(f,u_0,u_1)\|_4\bigr)
  \end{equation}
  for the solution of
  \begin{equation}
  \label{EqAsyExpLocalTameModIVP}
    \wt L u=f, \quad \gamma_0(u)=(u_0,u_1),
  \end{equation}
  where the constant $C=C(s)$ does not depend on the coefficients of $\wt L$, in the sense that the same constant works if one perturbs the coefficients of $\wt L$ in the (weak!) space $\Hext^4$; here, $\bfa=(a^{\mu\nu},a^\mu,a)$ is the collection of coefficients of $\wt L$. (Of course, there \emph{is} a high regularity norm $\|\bfa\|_s$ in \eqref{EqAsyExpLocalTamePf}!)

  In a neighborhood of any given point on $\Sigma_0$, we can perform a smooth coordinate change, replacing $x^j$ by $y^j(x^\mu)$ and letting $y^0=x^0$, so that $\pa_{y^0}$ is timelike. By redefining our coordinates, we may thus assume that this is already the case for the $x^\mu$ coordinates, in which case the matrix $(a^{jk})$ is positive definite. (Note that the same, fixed, coordinate change accomplishes this for perturbations, in the sense of the previous paragraph, of $\wt L$.) The equation \eqref{EqAsyExpLocalTameModIVP} is equivalent to a system of equations for $u$ and $u_\mu:=D_\mu u$,
  \[
    D_0 u=u_0,\ D_0 u_0 = a^{jk}D_j u_k + b^j D_j u_0 + a^0 u_0 + a^j u_j + a u,\ D_0 u_j=D_j u_0;
  \]
  writing this in matrix form for $(u,u_0,\cdots,u_3)$ and multiplying from the left by the symmetric, positive block matrix
  \[
    A^0 = \begin{pmatrix} 1 & 0 & 0 \\ 0 & 1 & 0 \\ 0 & 0 & (a^{kj})_{j,k=1,\ldots,3} \end{pmatrix},
  \]
  we obtain the symmetrizable hyperbolic system
  \begin{equation}
  \label{EqAsyExpLocalTameHyp}
    A^0 \pa_0 v = \sum A^j(x^\mu)\pa_j v + g,\quad v(0)=h,
  \end{equation}
  where the coefficients of the symmetric matrices $A^0,A^j$ are in $H^s$ and uniformly bounded in $H^4$, while $g\in H^s$, and $h\in H^s$ on $x^0=0$. We now solve this system and obtain a tame estimate for the solution $v$. Following \cite[\S16.1]{TaylorPDE}, we first do this assuming that $x^1,x^2,x^3$ are global coordinates on the 3-torus $\TT^3$.
  
  Defining the mollifier $J_\eps=\phi(\eps D)$ for $\phi\in\CIc(\R^3)$, identically $1$ near $0$, we consider the mollified equation
  \begin{equation}
  \label{EqAsyExpLocalTameHypMoll}
    A^0_\eps \pa_0 v_\eps = J_\eps A^j \pa_j J_\eps v_\eps + J_\eps g,\quad v_\eps(0)=J_\eps h,
  \end{equation}
  where $A^0_\eps:=J_\eps A^0$. This is an ODE for the $H^s$-valued function $v_\eps$, and we will prove uniform tame bounds for $v_\eps$ as $\eps\to 0$. To this end, define $\Lambda^s=\la D\ra^s$ for $s\in\R$; then, using the $L^2$ pairing $\la\cdot,\cdot\ra$, we write
  \begin{equation}
  \label{EqAsyExpLocalTameD0}
  \begin{split}
    \pa_0\la A^0_\eps\Lambda^s v_\eps,\Lambda^s v_\eps\ra &= \la (\pa_0 A^0_\eps)\Lambda^s v_\eps,\Lambda^s v_\eps\ra + 2\Re\la [A^0_\eps,\Lambda^s]\pa_0 v_\eps,\Lambda^s v_\eps\ra \\
      & \qquad + 2\Re\la \Lambda^s J_\eps A^j\pa_j J_\eps v_\eps, \Lambda^s v_\eps\ra + 2\Re\la\Lambda^s J_\eps g,\Lambda^s v_\eps\ra
  \end{split}
  \end{equation}
  \emph{In the bounds stated below, we write $x\lesssim y$ if $x\leq C_s y$ for a constant $C_s$ only depending on $s$ and the $H^4$ bounds on  the coefficients $A^\mu$.} The first term is then bounded by $\|(\pa_0 A^0_\eps)\Lambda^s v_\eps\|_0\|v_\eps\|_s$, which by \eqref{EqAsyExpLocalHsProdLow} is bounded by $\|v_\eps\|_s^2$, using the boundedness of $A^0_\eps$ in $H^4$. In order to estimate the second term, we first note that \eqref{EqAsyExpLocalHsInv} gives
  \[
    \|\pa_0 v_\eps\|_{s-1} \lesssim \|A^0_\eps \pa_0 v_\eps\|_{s-1} + \|A^0_\eps \pa_0 v_\eps\|_{s_0}(1+\|A^0_\eps\|_{s-1})
  \]
  for $s-1\geq s_0>5/2$, which using \eqref{EqAsyExpLocalTameHypMoll} and setting $\bfA=(A^0,\ldots,A^3)$ gives
  \begin{align*}
    \|\pa_0 v_\eps\|_{s-1} &\lesssim \sum_j \|A^j\|_{s_0}\|v_\eps\|_s  + \|A^j\|_{s-1}\|v_\eps\|_{s_0+1} +  \|g\|_{s-1} \\
      & \qquad + \bigl(\|A^j\|_{s_0}\|v_\eps\|_{s_0+1} + \|g\|_{s_0-1}\bigr)(1+\|A^0\|_{s-1}) \\
      & \lesssim \|v_\eps\|_s + \|g\|_{s-1} + \|\bfA\|_{s-1}(\|v_\eps\|_{s_0+1} + \|g\|_{s_0-1});
  \end{align*}
  this plugs into the estimate of the second term in \eqref{EqAsyExpLocalTameD0}, which by \eqref{EqAsyExpLocalHsComm} (using $s_0>3/2$, $s\geq 1$) is bounded by
  \begin{align*}
    \bigl(\|\pa_0&v_\eps\|_{s-1} + \|A_\eps^0\|_s\|\pa_0 v_\eps\|_{s_0}\bigr) \|v_\eps\|_s \\
      & \lesssim \|v_\eps\|_s^2 + \|g\|_{s-1}^2 + \|\bfA\|_s^2(\|v_\eps\|_{s_0+1}^2+\|g\|_{s_0}^2).
  \end{align*}
  The third term of \eqref{EqAsyExpLocalTameD0} is
  \begin{align*}
    2\Re\la&[\Lambda^s,J_\eps A^j J_\eps]\pa_j v_\eps,\Lambda^s v_\eps\ra - \la (\pa_j A^j)J_\eps\Lambda^s v_\eps,J_\eps\Lambda^s v_\eps\ra \\
      & \leq \bigl(\|v_\eps\|_s+\|\bfA\|_s\|v_\eps\|_{s_0+1}\bigr)\|v_\eps\|_s + \|v_\eps\|_s^2 
  \end{align*}
  for $s_0>3/2$, $s\geq 1$; and the fourth term finally is bounded by $\|g\|_s^2+\|v_\eps\|_s^2$. Combining these four estimates gives
  \[
    \pa_0\la A^0_\eps\Lambda^s v_\eps,\Lambda^s v_\eps\ra \lesssim \|v_\eps\|_s^2 + \|g\|_s^2 + \|\bfA\|_s^2\bigl(\|v_\eps\|_{s_0}^2 + \|g\|_{s_0-1}^2\bigr)
  \]
  for $s\geq s_0$, $s_0>7/2$. If we apply this for $s=s_0$, the positive definiteness of $A^0$ and Gronwall's lemma yield a uniform bound on $\|v_\eps\|_{s_0}$ by $\|g\|_{s_0}+\|h\|_{s_0}$, hence
  \[
    \pa_0\la A^0_\eps\Lambda^s v_\eps,\Lambda^s v_\eps\ra_s \lesssim \|v_\eps\|_s^2+\|g\|_s^2 + \|\bfA\|_s^2(\|g\|_{s_0}^2 + \|h\|_{s_0}^2),
  \]
  which implies that $v_\eps$ is uniformly bounded in $\cC^0([0,T];H^s)\cap\cC^1([0,T];H^{s-1})$ and satisfies the tame estimate
  \[
    \|v_\eps\|_{\cC^0 H^s\cap\cC^1 H^{s-1}} \lesssim \|g\|_{\cC^0 H^s} + \|h\|_{H^s} + \|\bfA\|_s\bigl(\|g\|_{\cC^0 H^{s_0}} + \|h\|_{H^{s_0}}\bigr).
  \]
  Simplifying the strongest assumptions on $s,s_0$ in the argument yielding this estimate, we can take $s\geq s_0:=4$ here. A standard weak limit argument then proves the existence of a solution $v$ of \eqref{EqAsyExpLocalTameHyp}, see \cite[Theorem~1.2 and Proposition~1.4]{TaylorPDE}, satisfying the same estimate. One can moreover prove a finite speed of propagation result as in \cite[\S5.6]{TaylorPDE}, and the argument there works in the present setting as well under our present regularity assumptions, since it relies only on energy estimates; thus, the local solutions of \eqref{EqAsyExpLocalTameModIVP} can be patched together to give a solution in a small slab $0\leq t_*\leq T$, and a simple iterative argument allows us to remove the smallness assumption on $T$.
  
  Returning to the wave equation \eqref{EqAsyExpLocalTameModIVP}, we have now established the existence of a solution $u$ satisfying the tame estimate \eqref{EqAsyExpLocalTamePf}. A forteriori, we have $u\in\Hext^\gamma([0,T],\Hext^{s+1-\gamma}(Y))$ for $\gamma=0,1$ together with the tame estimate \eqref{EqAsyExpLocalTamePf}, and interpolation inequalities yield this for all $\gamma\in[0,1]$. Now, writing $s=\lfloor s\rfloor+\gamma$, $k\in\Z$, $\gamma\in[0,1)$, one can use $f\in\Hext^s(\Omega_T)$ and the equation $L u=f\in\Hext^s(\Omega_T)$, written in the form \eqref{EqAsyExpLocalEvolEq}, to deduce $u\in\Hext^{k+\gamma}([0,T],\Hext^{s+1-k-\gamma}(Y))$ for $k=0,1,\ldots,\lfloor s\rfloor+1$ inductively; to see that the norm of $u$ in this space satisfies a tame estimate with the same right hand side as \eqref{EqAsyExpLocalTamePf}, one uses Lemma~\ref{LemmaAsyExpLocalHs}. We conclude that
  \[
    u\in\Hext^0([0,T],\Hext^{s+1}(Y))\cap\Hext^{s+1}([0,T],\Hext^0(Y)) = \Hext^{s+1}(\Omega_T)
  \]
  satisfies the tame estimate \eqref{EqAsyExpLocalTamePf}, finishing the proof.
\end{proof}

\subsubsection{Global regularity and asymptotic expansions}
\label{SubsubsecAsyExpReg}

We now assume that the assumptions of Theorem~\ref{ThmAsyExpMain} are satisfied. In order to explain the main difficulty in the proof of this theorem, we briefly recall from \cite[\S7.2]{HintzQuasilinearDS} and \cite[\S5.1]{HintzVasyQuasilinearKdS} (see also \cite[\S3]{VasyMicroKerrdS} for a version of this in the smooth category) the argument which establishes partial asymptotic expansions for exponentially decaying perturbations of stationary operators, like $L_{w,\wt w}$. We discuss the forward forcing problem and ignore issues of regularity for simplicity: thus, given a forcing term $f\in\Hb^{\infty,\alpha}$ vanishing near $\Sigma_0$, the forward solution $u$ satisfies $u\in\Hb^{\infty,r_0}$, with $r_0$ large, negative, and independent of $f$, and with $r_0<-\Im\sigma$ for all $\sigma\in\Res(L_w)$.\footnote{We recall the statement in Proposition~\ref{PropAsyExpRegRecall} \eqref{ItAsyExpRegRecall1Growing} below. Solvability in an exponentially weighted $L^2$-space follows by means of energy estimates; the conormality relies on microlocal propagation estimates.} Then, we rewrite the equation for $u$ as
\begin{equation}
\label{EqAsyExpRegIter}
  L_w u = f - \wt L_{w,\wt w}u \in \Hb^{\infty,r_0+\alpha}.
\end{equation}
Using the Mellin transform, we can then shift the contour in the inverse Mellin transform
\[
  u(\tau,x) = \frac{1}{2\pi} \int_{\Im\sigma=-r_0} \tau^{i\sigma} \wh{L_w}(\sigma)^{-1}(f-\wt L_{w,\wt w}u)\ftrans(\sigma)\,d\sigma
\]
from $\Im\sigma=-r_0$ to $\Im\sigma=-r_1$, with $r_1\in(r_0,r_0+\alpha]$ chosen so that no resonance of $L_w$ has imaginary part equal to $-r_1$; for instance, $r_1=r_0+(1-c)\alpha$ works for sufficiently small $c>0$. If $L_w$ has no resonances in the strip
\[
  S_1 := \{\sigma\in\C\colon -r_1<\Im\sigma<-r_0 \},
\]
i.e.\ if $\wh{L_w}(\sigma)^{-1}$ is regular in $S_1$, then we obtain $u\in\Hb^{\infty,r_1}$; see \eqref{EqBSobolevIso} for the properties of the Mellin transform on b-Sobolev spaces. This improves the weight from $r_0$ to $r_1$; we can then run the same argument again. If however $L_w$ does have resonances in $S_1$, we obtain a partial asymptotic expansion $u_0$ of $u$ corresponding to these, and a remainder term $\wt u\in\Hb^{\infty,r_1}$, so $u=u_0+\wt u$. If the partial expansion is trivial, i.e.\ $u_0=0$, then $u=\wt u\in\Hb^{\infty,r_1}$, and we can repeat the argument as before. Generically however, the partial expansion is non-trivial, i.e.\ $u_0\neq 0$; in this case, attempting to repeat this argument, equation \eqref{EqAsyExpRegIter} now reads
\[
  L_w u = f - \wt L_{w,\wt w}u_0 - \wt L_{w,\wt w}\wt u.
\]
The issue is that $\wt L_{w,\wt w}u_0$ is merely an element of $\Hb^{\infty,r_0+\eps+\alpha}$ (with $\eps>0$ depending on the imaginary part of the resonances of $L_w$ in $S_1$), but in general with no (partial) asymptotic expansion, since the coefficients of $\wt L_{w,\wt w}$, lying in $\Hb^{\infty,\alpha}$, have no asymptotic expansion either. Thus, we cannot establish an asymptotic expansion of $u$ with an exponentially decaying remainder term in $\Hb^{\infty,\alpha}$ in this case.

\begin{rmk}
  If the coefficients of $\wt L_{w,\wt w}$ did have a partial asymptotic expansion, we could subtract off the corresponding terms of $\wt L_{w,\wt w}u_0$ and continue the argument, deducing an expansion of $u$ up to an error term with weight given by the sum of the weight of $u_0$ and the weight of the error term, i.e.\ the remainder of the partial expansion, of $\wt L_{w,\wt w}$. In the present setting, the lack of an expansion for $\wt L_{w,w\wt w}$ means that we cannot proceed in this manner.
\end{rmk}

In particular, the solutions of the Cauchy problem
\begin{equation}
\label{EqAsyExpRegUz0}
  (L_{w,\wt w}, \gamma_0)u_{w,\wt w}^{\bfc}=z_{w,\wt w}^{\bfc}
\end{equation}
cannot be expected to have asymptotic expansions up to exponentially decaying remainders, and therefore we cannot cancel all growing asymptotics of $u$ at once by matching them with the asymptotics of some $u_{w,\wt w}^{\bfc}$, as we did in the proof of Proposition~\ref{PropAsySmResNoAsy}. Instead, motivated by the above argument, we proceed in steps; the point is that cancelling asymptotics within an interval of size $\leq\alpha$ can be accomplished directly using the framework of Proposition~\ref{PropAsySmResNoAsy}.

Preparing the proof of Theorem~\ref{ThmAsyExpMain}, let us now choose a weight $r_0<0$ with $-r_0>\Im\sigma$ for all $\sigma\in\Res(L_{w_0})$; with $\alpha>0$ given in Theorem~\ref{ThmAsySmMero}, fix an integer $N>(-r_0+\alpha)/\alpha+1$ and weights
\begin{equation}
\label{EqAsyExpRegWeights}
  r_N = \alpha > r_{N-1} \geq 0 > r_{N-2} > \cdots > r_0, \quad 0<r_{j+1}-r_j\leq\alpha,
\end{equation}
such that for all resonances $\sigma\in\Res(L_{w_0})$, one has $\Im\sigma\neq-r_j$, $0\leq j\leq N$. By Proposition~\ref{PropAsySmPert} and in view of the existence of a uniform essential spectral gap for the operators $L_w$, $w\in W$, as discussed at the beginning of \S\ref{SubsubsecAsySmPert}, we may assume that these conditions are also satisfied for $L_w$, $w\in W$.

We recall the regularity theory for $L_{w,\wt w}$ from \cite[\S5.1]{HintzVasyQuasilinearKdS}, which we extend to initial value problems using Proposition~\ref{PropAsyExpLocal}. For simplicity, we assume that $\alpha>0$ is so small that, say,
\begin{equation}
\label{EqAsyExpRegRadPtOK}
  1+\alpha\sup(\beta)-\wh\beta<3
\end{equation}
to simplify the regularity arithmetic in the sequel; this condition ensures that the threshold regularity at the radial sets for all operators $L_{w,\wt w}$ is bounded from above by the small absolute constant $3$. This can certainly be arranged for $\wh\beta\geq -1$; in the case that $\wh\beta<-1$, all our arguments below go through if we increase the regularity assumptions and thresholds by the amount $-1-\wh\beta$. Since in our applications the assumption $\wh\beta\geq -1$ will always be satisfied, we leave this more general case to the reader.

\begin{prop}
\label{PropAsyExpRegRecall}
  Let $s\geq 9$. Consider the initial value problem $L_{w,\wt w}u=f$, $\gamma_0(u)=(u_0,u_1)$, where $\wt w\in\wt W^s$ and $(f,u_0,u_1)\in D^{s-1,r}$, $r\in\R$.
  \begin{enumerate}
    \item \label{ItAsyExpRegRecall1Growing} (See \cite[Lemma~5.2]{HintzVasyQuasilinearKdS}.) Suppose $r\leq r_0$ with $r_0$ as above. Then there exists a unique solution $u\in\Hbext^{s,r}(\Omega;E)$ of the initial value problem, and $u$ satisfies
      \begin{equation}
      \label{EqAsyExpRegRecallGrowing}
        \|u\|_{\Hbext^{s,r}} \leq C\bigl(\|(f,u_0,u_1)\|_{D^{s-1,r}} + (1+\|\wt w\|_s)\|(f,u_0,u_1)\|_{D^{5,r}}\bigr).
      \end{equation}
    \item \label{ItAsyExpRegRecall2PropSing} (See \cite[Corollary~5.4]{HintzVasyQuasilinearKdS}.) Suppose $r\leq r_{N-2}$, and suppose we have $u\in\Hbext^{5,r}(\Omega;E)$. Then in fact $u\in\Hbext^{s,r}(\Omega;E)$, and $u$ satisfies the tame estimate
    \[
      \|u\|_{\Hbext^{s,r}} \leq C\bigl(\|u\|_{\Hbext^{5,r}}+\|(f,u_0,u_1)\|_{D^{s-1,r}} + (1+\|\wt w\|_s)\|(f,u_0,u_1)\|_{D^{5,r}}\bigr)
    \]
    \item \label{ItAsyExpRegRecall3PropSingDecay} (See \cite[Theorem~5.6]{HintzVasyQuasilinearKdS}.) Suppose $r\leq r_{N-2+j}$, $j=1,2$, and suppose we have $u\in\Hbext^{5,r}(\Omega;E)$. Then in fact $u\in\Hbext^{s-2j,r}(\Omega;E)$,  and $u$ satisfies the tame estimate
    \[
      \|u\|_{\Hbext^{s-2j,r}} \leq C\bigl(\|u\|_{\Hbext^{5,r}} + \|(f,u_0,u_1)\|_{D^{s-1,r}} + (1+\|\wt w\|_s)(\|(f,u_0,u_1)\|_{D^{5,r}}\bigr).
    \]
  \end{enumerate}
  In all three cases, $C$ is a uniform constant only depending on $s$.
\end{prop}

Note that for $\wt w\in\wt W^{\wt s}$, the coefficients of the operator $L_{w,\wt w}$ have $H^{\wt s}$ regularity, hence $L_{w,\wt w}$ satisfies the assumption \cite[Equation~(5.1)]{HintzVasyQuasilinearKdS}. The regularity assumptions here are stronger than what is actually needed, but they will simplify the arithmetic below.

\begin{proof}[Proof of Proposition~\ref{PropAsyExpRegRecall}]
  The estimates do not explicitly include a low regularity norm of the coefficients of $L_{w,\wt w}$ due to the uniform boundedness assumption \eqref{EqAsyExpParamSpaceTilde} of $\wt w$ in the norm of $\wt W^9$ (cf.\ the norm on $v$ in the line below \cite[Equation~(5.3)]{HintzVasyQuasilinearKdS}).
  
  Using Proposition~\ref{PropAsyExpLocal}, we first find a local solution $u'\in\Hext^s(\Omega_T;E)$; we can then rewrite the Cauchy problem for $u$ as a forcing problem as in \eqref{EqAsySmIVPForcing}, for which \cite[Lemma~5.2]{HintzVasyQuasilinearKdS} produces the solution, together with the stated tame estimate. The only slightly subtle point here is the radial point estimate, which however does apply at this weight and regularity level due to assumption \eqref{EqAsyExpRegRadPtOK}; see in particular the proof of \cite[Corollary~5.4]{HintzVasyQuasilinearKdS}.

  For part \eqref{ItAsyExpRegRecall2PropSing}, we proceed similarly, noting that we have $H^s$ regularity for $u$ near the Cauchy surface $\Sigma_0$ by Proposition~\ref{PropAsyExpLocal}, and b-microlocal estimates yield the global regularity as in \cite{HintzVasyQuasilinearKdS}. Part \eqref{ItAsyExpRegRecall3PropSingDecay} follows by first using part \eqref{ItAsyExpRegRecall2PropSing} with $r=r_{N-2}$, which gives $u\in\Hbext^{s,r_{N-2}}(\Omega;E)$, and then using the contour shifting argument explained above, together with Lemma~\ref{LemmaAsyExpLocalHs} \eqref{ItAsyExpLocalHsProd} for obtaining b-regularity of $\wt L_{w,\wt w}u$: since we are \emph{assuming} $u\in\Hbext^{5,r}(\Omega;E)$, the partial asymptotic expansion of $u$ corresponding to resonances $\sigma$ with $\Im\sigma\geq\max(-r,-r_{N-1})$ must be trivial; since the high energy estimate \eqref{EqAsySmMeroEst} of $L_w$ loses two powers of $\sigma$ (see Remark~\ref{RmkAsySmMeroEstLoss}), the contour shifting argument gives $u\in\Hbext^{s-2,\min(r,r_{N-1})}(\Omega;E)$. If $r>r_{N-1}$, we repeat this argument once more, losing $2$ additional derivatives.

  In fact, the reference only gives these results in case the elliptic ps.d.o.\ $Q\in\Psib^0(M;E)$ in assumption~\eqref{ItAsySm3Trapped} of \S\ref{SubsecAsySm}, concerning the subprincipal operator at the trapped set, is equal to the identity, or more generally a smooth section of $\End(E)\to M$, which is equivalent to choosing a different inner product in \eqref{EqAsySm3TrappedSubpr}. For general $Q$, we merely need to check that the proof of b-estimates at the normally hyperbolic trapping on growing function spaces presented in \cite[\S4.3]{HintzVasyQuasilinearKdS} goes through with $Q$ present. This in turn is a consequence of the observation that the form of $L_{w,\wt w}$ implies for $\wt w\in\wt W^s$:
  \[
    Q L_{w,\wt w} Q^- = L_{w,\wt w} + L',\quad L'\in\Diffb^1+\Hb^s\Diffb^2+\Hb^{s-1}\Psib^1+\Psib^{0;0}\Hb^{s-2},
  \]
  with $Q^-$ a parametrix of $Q$; see \cite[\S3]{HintzVasyQuasilinearKdS} for the definition of the space of operators $\Psib^{0;0}\Hb^{s-2}$. Indeed, this follows from the regularity of the symbols appearing in the partial expansion of the symbol of a composition of two operators, see \cite[Theorem~3.12 (2)]{HintzQuasilinearDS}; or more directly by analyzing a partial expansion of the commutator of $Q$ with a section of $E$ with (high) b-Sobolev regularity, the commutators of $Q$ with smooth differential operators being understood using the smooth b-calculus. Since the remainder term $\Psib^{0;0}\Hb^{s-2}$ is one order less regular relative to the leading order term than what is assumed in \cite[\S5.1]{HintzVasyQuasilinearKdS}, we use the regularity assumption $s\geq 9$ here. In fact, much less would suffice, but we are assuming $H^{14}$ regularity of $\wt w$ already anyway.
\end{proof}

Define the spaces of dual states in each strip
\[
  \cI_k=\{\sigma\in\C\colon -r_k<\Im\sigma<-r_{k-1}\}
\]
by
\[
  R^*_{w;k} := \Res^*(L_w,\cI_k).
\]
We have $R^*_{w;k}\subset \Hb^{1/2+\inf(\alpha\Im\sigma)+\wh\beta-1/2,-r_{k-1}}(\Omega;E)^{-,\bullet}$, the $\inf$ taken over $\sigma\in\cI_k$ and the radial set $\pa\cR$ of $L_{w_0}$ (see the discussion at the beginning of \S\ref{SubsubsecAsySmPert}); this makes use of the discussion following \eqref{EqAsySmResDualStatesAlt}. We reduce however the regularity by $1/2$ to give ourselves some room to ensure the validity of the inclusion for small $w\in W$. A forteriori, in view of \eqref{EqAsyExpRegRadPtOK}, we have
\[
  R^*_{w;k} \subset \Hb^{-2,-r_{k-1}}(\Omega;E)^{-,\bullet}.
\]
Set
\begin{equation}
\label{EqAsyExpRegDims}
  D'_k := \dim R^*_{w;k}, \quad D_k:=\sum_{j=1}^k D'_j,\ D_0:=0, \quad D:=D_N;
\end{equation}
these are constants independent of $w$. See Figure~\ref{FigAsyExpRegWeights}. Note that
\[
  D=\dim\Res(L_w, \{\Im\sigma>-\alpha\}) \leq N_\cZ,
\]
with $N_\cZ$ the number of parameters in the modification map $z$, due to the surjectivity assumption \eqref{EqAsyExpMainMod}.

\begin{figure}[!ht]
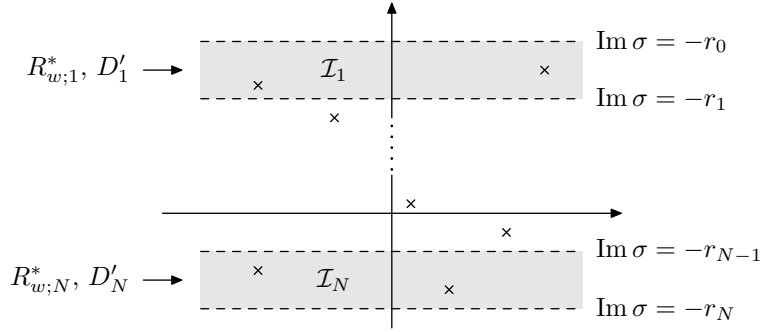

  \centering
  \inclfig{AsyExpRegWeights}
  \caption{Illustration of two strips $\cI_k\subset\C$, and the associated spaces of dual resonant states $R^*_{w;k}$ and their dimensions $D'_k$. Crosses schematically indicate resonances.}
  \label{FigAsyExpRegWeights}
\end{figure}

For $\bfc\in \C^{N_\cZ}$, we recall the notation $u_{w,\wt w}^{\bfc}$ from \eqref{EqAsyExpRegUz0}. For $w\in W$, $\wt w\in\wt W^{14}$, and $k=0,\ldots,N$, we then define
\[
  C_{w,\wt w;k} := \{ \bfc\in\C^{N_\cZ} \colon u_{w,\wt w}^{\bfc}\in \Hbext^{14-s_k,r_k}(\Omega;E) \}, 
\]
with $s_k$ given below. Thus, $C_{w,\wt w;k}$ is the space of all $\bfc\in\C^{N_\cZ}$ for which the asymptotic expansion of the solution $u_{w,\wt w}^{\bfc}$ of the corresponding initial value problem does \emph{not} contain any terms corresponding to resonances in $\Im\sigma>-r_k$. By Proposition~\ref{PropAsyExpRegRecall} \eqref{ItAsyExpRegRecall3PropSingDecay}, the regularity which we can obtain for $u_{w,\wt w}^{\bfc}$ under the assumption that it decays exponentially at rate $r_k$ is $14-s_k$ as stated, with
\begin{equation}
\label{EqAsyExpRegArithmetic}
  s_k=\begin{cases}
        0, & k\leq N-2, \\
        2, & k=N-1, \\
        4, & k=N.
      \end{cases}
\end{equation}

Clearly, $C_{w,\wt w;k}\subseteq C_{w,\wt w;\ell}$ for $k\geq\ell$. For $\bfc\in C_{w,\wt w;k-1}$, we have $\wt L_{w,\wt w}u_{w,\wt w}^{\bfc}\in\Hbext^{12-s_{k-1},r_k}(\Omega;E)\subset\Hbext^{8,r_k}(\Omega;E)$, hence the map
\begin{equation}
\label{EqAsyExpRegZNoAsy}
\begin{gathered}
  \lambda_{w,\wt w;k} \colon C_{w,\wt w;k-1} \to \cL(R_{w;k}^*,\ol\C), \\
    \bfc \mapsto \lambda_\IVP\bigl(z_{w,\wt w}^{\bfc} - (\wt L_{w,\wt w},\gamma_0)(u_{w,\wt w}^{\bfc})\bigr),
\end{gathered}
\end{equation}
is well-defined and linear; recall here the map $\lambda_\IVP$ from \eqref{EqAsySmResIVPDualStatePairing}. The argument of $\lambda_\IVP$ in this definition is the analogue of \eqref{EqAsyExpRegIter} for $f=z_{w,\wt w}^{\bfc}$ and $u=u_{w,\wt w}^{\bfc}$, now also taking the initial data into account.

\begin{lemma}
\label{LemmaAsyExpRegCk}
  The spaces $C_{w,\wt w;k}$ have the following properties:
  \begin{enumerate}
    \item \label{ItAsyExpRegCk1Alt} We have $C_{w,\wt w;0}=\C^{N_\cZ}$, and inductively
      \[
        C_{w,\wt w;k} = \{ \bfc\in C_{w,\wt w;k-1} \colon \lambda_{w,\wt w;k}(\bfc)=0 \}
      \]
      for $k\geq 1$.
    \item \label{ItAsyExpRegCk2Codim} The space $C_{w,\wt w;k}\subset C_{w,\wt w;k-1}$ has codimension $D'_k$.
    \item \label{ItAsyExpRegCk3Basis} $C_{w,\wt w;k}\subset\C^{N_\cZ}$ depends continuously on $(w,\wt w)\in W\times\wt W^{14}$ in the sense that it has a basis with each basis vector depending continuously on $(w,\wt w)\in W\times\wt W^{14}$.
    \item \label{ItAsyExpRegCk4Cont} There exist spaces
      \begin{equation}
      \label{EqAsyExpRegCkPrime}
        C'_{w,\wt w;k} := \mathspan\{ \bfc_{w,\wt w}^{D_{k-1}+1}, \ldots, \bfc_{w,\wt w}^{D_k} \} \cong \C^{D'_k},
      \end{equation}
      with each $\bfc_{w,\wt w}^\ell\in\C^{N_\cZ}$ depending continuously on $(w,\wt w)\in W\times\wt W^{14}$, so that
      \begin{equation}
      \label{EqAsyExpRegCkPrimeCompl}
        C_{w,\wt w;k-1} = C_{w,\wt w;k} \oplus C'_{w,\wt w;k}.
      \end{equation}
      By property \eqref{ItAsyExpRegCk1Alt}, $\lambda_{w,\wt w;k}$ induces a map on the quotient $C_{w,\wt w;k-1}/C_{w,\wt w;k}$ and hence induces a map
      \[
        [\lambda_{w,\wt w;k}] \colon \C^{D'_k} \cong C'_{w,\wt w;k} \to \cL(R^*_{w;k},\ol\C) \cong \C^{D'_k}
      \]
      on the fixed vector space $\C^{D'_k}$, where for the second isomorphism, we use Proposition~\ref{PropAsySmPert} to find a parametrization of $R^*_{w;k}$ which depends continuously on $w\in W$. Then the map
      \[
        W\times\wt W^{14}\ni (w,\wt w) \mapsto [\lambda_{w,\wt w;k}] \in \C^{D'_k\times D'_k}
      \]
      is continuous. See Figure~\ref{FigAsyExpRegCk}.
    \item \label{ItAsyExpRegCk5Inv} The map $[\lambda_{w,\wt w;k}]$ is invertible for all $k=1,\ldots,N$.
    \item \label{ItAsyExpRegCk6NoN} If $N_\cZ=D$, then $C_{w,\wt w;N}=0$.
  \end{enumerate}
\end{lemma}

\begin{figure}[!ht]
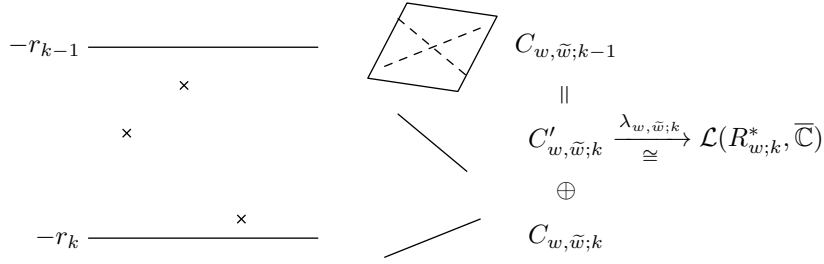

  \centering
  \inclfig{AsyExpRegCk}
  \caption{Illustration of the spaces $C_{w,\wt w;k}$ and $C'_{w,\wt w;k}$ in Lemma~\ref{LemmaAsyExpRegCk}; the crosses indicate resonances, the total rank of which in the displayed strip is equal to $D'_k$.}
  \label{FigAsyExpRegCk}
\end{figure}

\begin{proof}[Proof of Lemma~\ref{LemmaAsyExpRegCk}]
  \eqref{ItAsyExpRegCk1Alt} follows from the above discussion, and \eqref{ItAsyExpRegCk6NoN} follows from \eqref{ItAsyExpRegCk2Codim} and \eqref{EqAsyExpRegDims}.

  We first prove \eqref{ItAsyExpRegCk2Codim}--\eqref{ItAsyExpRegCk5Inv} in the special case $(w,\wt w)=(w_0,0)$, so $\wt L_{w_0,0}\equiv 0$. In this case, $\lambda_{w_0,0;k}$ is well-defined on the full space $\C^{N_\cZ}$ for all $k=1,\ldots,N$; we denote the thus extended map by $\lambda_k$. Then we have $C_{w_0,0;k}=\bigcap_{j\leq k}\ker\lambda_j$. If we let
  \[
    R^* := \bigoplus_{k=1}^N R^*_{0;k} = \Res^*(L_{w_0}, \{\Im\sigma>-\alpha\}),
  \]
  then the map $\C^{N_\cZ}\ni\bfc\mapsto\lambda_\IVP(z_{w_0,0}^\bfc)\in\cL(R^*,\ol\C)$, which we are assuming is surjective, is equal to $(\lambda_1,\ldots,\lambda_N)$; its kernel equals $C_{w_0,0;N}$ by definition. Therefore, each restriction $\lambda_k|_{C_{w_0,0;k-1}}$ is onto, hence its kernel $C_{w_0,0;k}$ has codimension $D'_k$, and $\lambda_k$ induces an isomorphism $C_{w_0,0;k-1}/C_{w_0,0;k}\to\cL(R^*_{0;k},\ol\C)$. We can then declare a set $\{\bfc_{w_0,0}^{D_{k-1}+1},\ldots,\bfc_{w_0,0}^{D_k}\}$ of vectors in $C_{w_0,0;k-1}$ whose equivalence classes in $C_{w_0,0;k-1}/C_{w_0,0;k}$ give a basis of the quotient to be a basis of $C'_{w_0,0;k}$. Lastly, we pick a basis $\{\bfc_{w_0,0}^{D+1},\ldots,\bfc_{w_0,0}^{N_\cZ}\}$ of $C_{w_0,0;N}$.

  For general $(w,\wt w)$, we can now establish \eqref{ItAsyExpRegCk2Codim}--\eqref{ItAsyExpRegCk5Inv} by induction on $k$. \emph{We will use subscripts to refer to claims \eqref{ItAsyExpRegCk2Codim}--\eqref{ItAsyExpRegCk5Inv} for any fixed $k$.} We introduce the notation
  \[
    \{\psi_{w;k}^1,\ldots,\psi_{w;k}^{D'_k}\} \subset R^*_{w;k}
  \]
  for a basis of $R^*_{w;k}$, with each $\psi_{w;k}^j$ depending continuously on $w\in W$ in the topology of $\Hbext^{-2,-r_{k-1}}(\Omega;E)^{-,\bullet}$.
  
  Let us assume now that \eqref{ItAsyExpRegCk2Codim}$_m$ holds for $m\leq k-1$ for some $1\leq k\leq N$, and that $\{\bfc_{w,\wt w}^1,\ldots,\bfc_{w,\wt w}^{N_\cZ}\}$, with each $\bfc_{w,\wt w}^\ell$ equal to the element of $\C^{N_\cZ}$ for $(w,\wt w)=(w_0,0)$ chosen in the first part of the proof, is a basis of $\C^{N_\cZ}$, with continuous dependence on $(w,\wt w)\in W\times\wt W^{14}$ as in \eqref{ItAsyExpRegCk3Basis}, so that \eqref{ItAsyExpRegCk4Cont}$_m$ and \eqref{ItAsyExpRegCk5Inv}$_m$ hold for $m\leq k-1$, and so that moreover $C_{w,\wt w;m}=\mathspan\{\bfc_{w,\wt w}^{D_m+1},\ldots,\bfc_{w,\wt w}^{N_\cZ}\}$ for $m\leq k-1$. (Note that for $k=1$, these assumptions are satisfied for the choice $\bfc_{w,\wt w}^\ell=\bfc_{w_0,0}^\ell$, $\ell=1,\ldots,N_\cZ$.) We will show how to update the $\bfc_{w,\wt w}^\ell$, $\ell=D_k+1,\ldots,N_\cZ$, so as to arrange \eqref{ItAsyExpRegCk2Codim}$_k$--\eqref{ItAsyExpRegCk5Inv}$_k$ to hold.
  
  To this end, put $C'_{w,\wt w;k}=\mathspan\{\bfc_{w,\wt w}^{D_{k-1}+1},\ldots,\bfc_{w,\wt w}^{D_k}\}$. We contend that the restriction
  \[
    \lambda_{w,\wt w;k}\colon C'_{w,\wt w;k}\to\cL(R^*_{w;k},\ol\C)
  \]
  is injective and hence (by dimension counting) an isomorphism; this implies \eqref{ItAsyExpRegCk2Codim}$_k$ and \eqref{ItAsyExpRegCk5Inv}$_k$. Injectivity holds by construction for $(w,\wt w)=(w_0,0)$; the contention therefore follows once we prove the continuity of the map
  \[
    W\times\wt W^{14} \ni (w,\wt w) \mapsto \big\la z_{w,\wt w}^{\bfc_{w,\wt w}^\ell} - (\wt L_{w,\wt w},\gamma_0)(u_{w,\wt w}^{\bfc_{w,\wt w}^\ell}),\psi_{w;k}^j\big\ra
  \]
  for $\ell=D_{k-1}+1,\ldots,D_k$ and $j=1,\ldots,D'_k$; we will in fact establish this for $\ell=D_{k-1}+1,\ldots,N_\cZ$. (Note that for such $\ell,j$, this map is indeed well-defined.) To show the latter, it suffices to prove the continuous dependence of $u_{w,\wt w}^{\bfc_{w,\wt w}^\ell}$ on $(w,\wt w)\in W\times\wt W^{14}$ in the topology of $\Hbext^{4,r_{k-1}}(\Omega;E)$. This in turn follows by an argument similar to the one used in the proof of the continuity part of Proposition~\ref{PropAsyExpLocal}: let $(w_p,\wt w_p)\in W\times\wt W^{14}$, $p=1,2$, and write $\bfc_p=\bfc_{w_p,\wt w_p}^\ell$, $z_p=z_{w_p,\wt w_p}^\ell$ and $u_p=u_{w_p,\wt w_p}^{\bfc_p}\in\Hb^{14-s_{k-1},r_{k-1}}(\Omega;E)$, then we have
  \[
    (L_{w_1,\wt w_1},\gamma_0)(u_1-u_2) = z_1-z_2 - (L_{w_1,\wt w_1}-L_{w_2,\wt w_2},0)(u_2).
  \]
  Now $z_1-z_2$ is small in $D^{14,\alpha}(\Omega;E)$ for $(w_1,\wt w_1)$ close to $(w_2,\wt w_2)$, while the last term is small in $D^{12-s_{k-1},r_{k-1}}(\Omega;E)$; by Proposition~\ref{PropAsyExpRegRecall} \eqref{ItAsyExpRegRecall2PropSing} and \eqref{ItAsyExpRegRecall3PropSingDecay}, this implies that $u_1-u_2$ is small in $\Hb^{13-2 s_{k-1},r_{k-1}}(\Omega;E)\subset\Hb^{9,r_{k-1}}(\Omega;E)$, as desired. Thus, statement \eqref{ItAsyExpRegCk4Cont}$_k$ is well-defined now, and we just proved that it holds.
  
  We also obtain \eqref{ItAsyExpRegCk3Basis}$_k$; indeed, the projection map
  \[
    \pi_{w,\wt w;k} \colon C_{w,\wt w;k-1}\ni \bfc\mapsto \bfc-(\lambda_{w,\wt w;k}|_{C'_{w,\wt w;k}})^{-1}(\lambda_{w,\wt w;k}(\bfc)) \in \ker \lambda_{w,\wt w;k} = C_{w,\wt w;k}
  \]
  is surjective, and $\pi_{w_0,0;k}$ is the identity map on $\mathspan\{\bfc_{w_0,0}^{D_k+1},\ldots,\bfc_{w_0,0}^{N_\cZ}\}$. Thus,
  \[
    \{\pi_{w,\wt w;k}(\bfc_{w,\wt w}^\ell) \colon \ell=D_k+1,\ldots,N_\cZ\}
  \]
  provides a basis of $C_{w,\wt w;k}$ which depends continuously on $(w,\wt w)\in W\times\wt W^{14}$. By an abuse of notation, we may replace $\bfc_{w,\wt w}^\ell$ by $\pi_{w,\wt w;k}(\bfc_{w,\wt w}^\ell)$ for $\ell=D_k+1,\ldots,N_\cZ$; then $C_{w,\wt w;k}=\mathspan\{\bfc_{w,\wt w}^{D_k+1},\ldots,\bfc_{w,\wt w}^D\}$. Therefore, we have arranged \eqref{ItAsyExpRegCk2Codim}$_k$--\eqref{ItAsyExpRegCk5Inv}$_k$. The proof is complete.
\end{proof}

We now have all ingredients for establishing the main result of this section:

\begin{proof}[Proof of Theorem~\ref{ThmAsyExpMain}]
  We fix the basis $\{\bfc_{w,\wt w}^\ell\colon 1\leq\ell\leq D\}$ of $\C^{N_\cZ}$ constructed in the above lemma. Suppose $(w,\wt w)\in W\times\wt W^{14}$, and suppose we are given data $(f,u_0,u_1)\in D^{13,\alpha}(\Omega;E)$. For $\bfc\in\C^{N_\cZ}$, we denote by $u^{\bfc}\in\Hbext^{14,r_0}(\Omega;E)$ the solution of the Cauchy problem \eqref{EqAsyExpMainIVP}. We first find $\bfc\in\C^{N_\cZ}$ for which $u^{\bfc}\in\Hbext^{10,\alpha}(\Omega;E)$ is exponentially decaying; by part \eqref{ItAsyExpRegCk6NoN} of the previous lemma, this $\bfc$ is unique if the map \eqref{EqAsyExpMainMod} is bijective. In order to do so, we will inductively choose $\bfc'_k\in C'_{w,\wt w;k}$ such that for $\bfc_k:=\sum_{j=1}^k \bfc'_j$, we have $u^{\bfc_k}\in\Hbext^{14-s_k,r_k}(\Omega;E)$. Suppose we have already chosen $\bfc'_m$, $m\leq k-1$ (with $1\leq k\leq N$) with this property, then writing $z_{k-1}=z_{w,\wt w}^{\bfc_{k-1}}$, the element $\bfc'_k$ is determined by the equation
  \begin{align*}
    \lambda_\IVP\bigl( (f,u_0,u_1) &+ z_{k-1} - (\wt L_{w,\wt w},\gamma_0)u^{\bfc_{k-1}} \\
      &\quad + z_{w,\wt w}^{\bfc'_k}-(\wt L_{w,\wt w},\gamma_0)u_{w,\wt w}^{\bfc'_k} \bigr) = 0 \in \cL(R^*_{w;k},\ol\C);
  \end{align*}
  the fact that $\lambda_\IVP$ does map the argument into $\cL(R^*_{w;k},\ol\C)$ uses the inductive assumption $u^{\bfc_{k-1}}\in\Hb^{14-s_{k-1},r_{k-1}}(\Omega;E)$. This equation, rewritten as
  \[
    \lambda_{w,\wt w;k}(\bfc'_k) = -\lambda_\IVP\bigl((f,u_0,u_1)+z_{k-1}-(\wt L_{w,\wt w},\gamma_0)u^{z_{k-1}}\bigr) \in \cL(R^*_{w;k},\ol\C),
  \]
  has a unique solution $\bfc'_k\in\cZ'_{w,\wt w;k}$ by part \eqref{ItAsyExpRegCk5Inv} of the previous lemma, finishing the inductive step and thus the construction of $\bfc:=\bfc_N$.
  
  Since we are assuming that $\wt w$ is uniformly bounded in $\wt W^{14}$, an inspection of the argument implies that the norm of $\bfc$ in $\C^{N_\cZ}$ is bounded by $\|(f,u_0,u_1)\|_{D^{13,\alpha}}$, proving the estimate \eqref{EqAsyExpMainSTameC}, and the solution $u=u^{\bfc}$ of the Cauchy problem \eqref{EqAsyExpMainIVP} satisfies the estimate
  \[
    \|u\|_{10,\alpha} \lesssim \|(f,u_0,u_1)\|_{13,\alpha}.
  \]
  To obtain a tame estimate for higher Sobolev norms of $u$, we use Proposition~\ref{PropAsyExpRegRecall} \eqref{ItAsyExpRegRecall3PropSingDecay}, which gives
  \[
    \|u\|_{s,\alpha} \lesssim \|(f,u_0,u_1)\|_{s+3,\alpha} + (1+\|\wt w\|_{s+4})\|(f,u_0,u_1)\|_{13,\alpha}
  \]
  for $s\geq 10$, proving \eqref{EqAsyExpMainSTameU}.
  
  Lastly, we prove the continuity of the solution map $S$, defined in \eqref{EqAsyExpMainModMap}; this follows by the usual argument, noting that for $d_j=(f_j,u_{j,0},u_{j,1})$, $(w_j,\wt w_j)\in W\times\wt W^\infty$, $(\bfc_j,u_j)=S(w_j,\wt w_j,d_j)$ and $z_j=z_{w_j,\wt w_j}^{\bfc_j}$ for $j=1,2$, we have
  \[
    u_1-u_2 = S\bigl(w_1,\wt w_1,d_1-d_2+z_1-z_2 - (L_{w_1,\wt w_1}-L_{w_2,\wt w_2},0)(u_2)\bigr)
  \]
  which implies the continuity in view of our estimates on $S(w_1,\wt w_1,\cdot)$. This finishes the proof of Theorem~\ref{ThmAsyExpMain}.
\end{proof}

\section{Computation of the explicit form of geometric operators}
\label{SecOp}

In the following four sections of the paper, we will only consider natural linear operators related to the fixed Schwarzschild--de~Sitter metric $g_{b_0}$; hence, \emph{we drop the subscript $b_0$ from now on}, so
\begin{equation}
\label{EqOpNoB0}
\begin{gathered}
  g \equiv g_{b_0},\ \bhm\equiv\bhm[0],\ r_\pm\equiv r_{b_0,\pm},\\
  \cR\equiv\cR_{b_0},\ \cL\equiv\cL_{b_0},\ \beta_{\pm,0}\equiv\beta_{b_0,\pm,0},\ \beta_\pm\equiv\beta_{b_0,\pm},
\end{gathered}
\end{equation}
where we recall the radial point quantities \eqref{EqKdSGeoRadPointQuant}.

\subsection{Warped product metrics}
\label{SecOpWarped}

We start by considering a general metric
\[
  g = q^2\,dt^2 - h,\quad q=q(x),\quad h=h(x,dx)
\]
on a manifold
\[
  \cM=\R_t\times \cX_x
\]
of arbitrary dimension $\geq 2$. It is natural to define
\[
  e_0 := q^{-1}\pa_t,\quad e^0 := q\,dt;
\]
setting $\Omega := \log q$, we compute for $v\in\CI(\cM,T\cX)$, $w\in\CI(\cM,T^*\cX)$
\begin{equation}
\label{EqOpWarpedCov}
\begin{gathered}
  \nabla_{e_0}e^0=-d\Omega,\quad \nabla_{e_0}w=e_0 w - w(\nabla^h\Omega)e^0, \\
  \nabla_v e^0=0, \quad \nabla_v w=\nabla^h_v w,
\end{gathered}
\end{equation}
where $\nabla^h$ denotes the Levi-Civita connection of the metric $h$, so in particular $\nabla^h\Omega$ is the gradient of $\Omega$ with respect to $h$. We define decompositions of the bundles $T^*\cM$ and $S^2T^*\cM$ adapted to the form of the metric, separating the normal (`$N$') and the tangential (`$T$') components: we let
\begin{equation}
\label{EqOpWarpedBundles}
  T^*\cM = W_N \oplus W_T,\quad S^2T^*\cM=V_{NN}\oplus V_{NT}\oplus V_{TT},
\end{equation}
where
\begin{gather*}
  W_N=\la e^0\ra,\quad W_T=T^*\cX, \\
  V_{NN}=\la e^0e^0\ra,\quad V_{NT}=\{ 2e^0 w\colon w\in W_T \},\quad V_{TT}=S^2T^*\cX,
\end{gather*}
where we write $\xi\eta\equiv \xi\cdot\eta = \frac{1}{2}(\xi\otimes\eta+\eta\otimes\xi)$ for the symmetrized product. We trivialize $W_N\cong \cM\times\R$ via the section $e^0$, $V_{NN}\cong \cM\times\R$ via $e^0e^0$, and we moreover identify $V_{NT}\cong W_T$ by means of $W_T\ni w\mapsto 2e^0 w\in V_{NT}$. One then easily computes the form of the operators $\delta_g$ (divergence), $\delta_g^*$ (symmetric gradient, formal adjoint of $\delta_g$), and $\sfG_g$, see \eqref{EqHypDTEinsteinify}:

\begin{lemma}
\label{LemmaOpWarped}
  In the bundle splittings \eqref{EqOpWarpedBundles}, we have
  \[
    \delta_g^*=\begin{pmatrix}
                 e_0                       & -\nabla^h\Omega \\
                 \frac{1}{2}q d_\cX q^{-1} & \frac{1}{2}e_0  \\
                 0                         & \delta_h^*
               \end{pmatrix},
    \quad
    \delta_g=\begin{pmatrix}
               -e_0          & -q^{-2}\delta_h q^2 & 0 \\
               (d_\cX\Omega) & -e_0                & -q^{-1}\delta_h q
             \end{pmatrix},
  \]
  and
  \[
    \sfG_g = \begin{pmatrix}
            \frac{1}{2}  & 0 & \frac{1}{2}\tr_h \\
            0            & 1 & 0                \\
            \frac{1}{2}h & 0 & 1-\frac{1}{2}h\tr_h
          \end{pmatrix}.
  \]
\end{lemma}

Since we need it for computations of gauge modifications, we also note
\begin{equation}
\label{EqOpWarpedDelgGg}
  2\delta_g \sfG_g=
    \begin{pmatrix}
      -e_0            & -2 q^{-2}\delta_h q^2 & -e_0\tr_h \\
      q^{-2}d_\cX q^2 & -2 e_0                & -2 q^{-1}\delta_h q-d_\cX\tr_h
    \end{pmatrix}.
\end{equation}

\subsection{Spatial warped product metrics}
\label{SubsecOpSpat}

Next, we specialize to the case $\cX=I_r\times S$, where $I\subset\R$ an open interval, $S$ is an $\sln$-dimensional Riemannian manifold, $\sln\in\N_0$, and the spatial metric $h$ is of the form
\[
  h = q^{-2}\,dr^2 + r^2\,\slg,\quad q=q(r),
\]
where $\slg=\slg(y,dy)$ is a Riemannian metric on $S$. All operators with a slash are those induced by $\slg$ and its Levi-Civita connection $\slnabla$. We write
\[
  e_1:=q\,\pa_r,\quad e^1:=q^{-1}\,dr;
\]
then for $v\in\CI(\cX,TS)$, $w\in\CI(\cX,T^*S)$, we compute
\begin{equation}
\label{EqOpSpatCov}
\begin{gathered}
  \nabla^h_{e_1}e^1=0,\quad \nabla^h_{e_1}w=e_1 w-q r^{-1}w, \\
  \nabla^h_v e^1=r q i_v\slg,\quad \nabla^h_v w=\slnabla_v w-q r^{-1} w(v) e^1,
\end{gathered}
\end{equation}
where we extend $\slg$ to a bilinear form on $T\cX$ by declaring $\slg(v,e_1)=0$ for all $v\in T\cX$. If we let
\begin{equation}
\label{EqOpSpatBundles}
  T^*\cX = W_{TN} \oplus W_{TT},\quad S^2T^*\cX=V_{TNN}\oplus V_{TNT}\oplus V_{TTT},
\end{equation}
where
\begin{gather*}
  W_{TN}=\la e^1\ra,\quad W_{TT}=T^*S, \\
  V_{TNN}=\la e^1 e^1\ra,\quad V_{TNT}=\{ 2 e^1 w\colon w\in W_{TT}\},\quad V_{TTT}=S^2T^*S,
\end{gather*}
Similar to before, we trivialize $W_{TN}$, resp.\ $V_{TTN}$, via $e^1$, resp.\ $e^1 e^1$, and identify $W_{TT}\cong V_{TNT}$ via $w\mapsto 2 e^1 w$. (If $\dim S=0$, then $W_{TT}$, $V_{TNT}$ and $V_{TTT}$ are trivial, i.e.\ have rank $0$.) Then one easily checks:

\begin{lemma}
\label{LemmaOpSpat}
  In the bundle splittings \eqref{EqOpSpatBundles}, we have
  \[
    d_\cX=\begin{pmatrix}e_1\\\sld\end{pmatrix},\quad \delta_h=\begin{pmatrix} -r^{-\sln}e_1 r^\sln & r^{-2}\sldelta \end{pmatrix},
  \]
  with $\sldelta$ here the co-differential $\CI(\cX,T^*\cX)\to\CI(\cX)$, hence
  \[
    d_\cX\Omega=\begin{pmatrix}q' \\ 0\end{pmatrix},\quad \nabla^h\Omega=\begin{pmatrix}q' & 0\end{pmatrix} \in \Hom(T^*\cX,\R);
  \]
  moreover
  \[
    h = \begin{pmatrix} 1 \\ 0 \\ r^2\slg \end{pmatrix},\quad \tr_h=\begin{pmatrix} 1 & 0 & r^{-2}\sltr \end{pmatrix} \in \Hom(S^2T^*\cX,\R),
  \]
  and
  \begin{equation}
  \label{EqOpSpatDelH}
    \delta_h^*=\begin{pmatrix}
                 e_1             & 0                         \\
                 \frac{1}{2}\sld & \frac{1}{2}r^2 e_1 r^{-2} \\
                 r q\slg         & \sldelta^*
               \end{pmatrix},
    \quad
    \delta_h=\begin{pmatrix}
               -r^{-\sln}e_1 r^\sln & r^{-2}\sldelta       & r^{-3}q \sltr \\
               0                    & -r^{-\sln}e_1 r^\sln & r^{-2}\sldelta
             \end{pmatrix}.
  \end{equation}
\end{lemma}

Putting the two decompositions \eqref{EqOpWarpedBundles} and  \eqref{EqOpSpatBundles} together, we have
\begin{equation}
\label{EqOpSpatBundlesFull}
\begin{gathered}
  T^*\cM = W_N \oplus W_{TN} \oplus W_{TT}, \\
  S^2T^*\cM = V_{NN} \oplus (V_{NTN} \oplus V_{NTT}) \oplus (V_{TNN} \oplus V_{TNT} \oplus V_{TTT}),
\end{gathered}
\end{equation}
where we wrote
\[
  V_{NT} = V_{NTN} \oplus V_{NTT},
\]
by means of the identification $V_{NT}\cong W_T=W_{TN}\oplus W_{TT}$, i.e.\ $V_{NTN}$ and $V_{NTT}$ are the preimages of $W_{TN}$ and $W_{TT}$, respectively, under this isomorphism. Thus, we trivialize $W_N$ via $e^0$, and $W_{TN}$ via $e^1$, further $V_{NN}$ via $e^0e^0$, and $V_{NTN}$ via $2e^0e^1$, as well as $V_{TNN}$ via $e^1e^1$, and identify
\[
  T^*S=V_{NTT},\ w\mapsto 2 e^0 w; \quad T^*S=V_{TNT},\ w\mapsto 2 e^1 w; \quad S^2T^*S\cong V_{TTT}.
\]

\subsection{Specialization to the Schwarzschild--de~Sitter metric}
\label{SubsecOpSdS}

Now, we specialize to the Schwarzschild--de~Sitter metric, for which $S=\Sph^2$ is equipped with the round metric, and
\[
  q^2 = 1-\frac{2\bhm}{r}-\frac{\Lambda r^2}{3} = \mu,
\]
where $\mu=\mu_{b_0}$ was defined in~\eqref{EqSdSMetric}. We compute the subprincipal operator
\[
  S_\sub(\Box_g)=-i\nabla^{\pi^*T^*\cM}_{\ham_G}
\]
at the trapped set $\Gamma$ defined in \eqref{EqKdSGeoTrapped}, where $\pi\colon T^*\cM\to \cM$ the projection and $\nabla^{\pi^*T^*\cM}$ the pullback connection induced by the Levi-Civita connection on $\cM$; see \cite[\S3.3 and \S4]{HintzPsdoInner} for details. We can compute the form of this in the (partial) trivialization \eqref{EqOpSpatBundlesFull} of the bundle $T^*\cM$ using \eqref{EqKdSGeoTrappedHam}, \eqref{EqOpWarpedCov} and \eqref{EqOpSpatCov}, with the coordinates on $T^*\cM$ given by
\begin{equation}
\label{EqOpSdSTrapCovec}
  -\sigma\,dt + \xi\,dr + \eta,\quad \eta\in T^*\Sph^2
\end{equation}
as in \eqref{EqKdSGeoTrapCovec}; so we have $\sigma_1(2 q' e_0)=-2 i r^{-1}\sigma$ at $r=r_P$, hence
\begin{equation}
\label{EqOpSdSTrapSubpr0}
\begin{split}
  S_\sub(\Box_g) &= -2\mu^{-1}\sigma D_t
   +ir^{-2}
     \begin{pmatrix}
       \ham_{|\eta|^2} & 0               & 0 \\
       0               & \ham_{|\eta|^2} & 0 \\
       0               & 0               & \nabla^{\pi_{\Sph^2}^*T^*{\Sph^2}}_{\ham_{|\eta|^2}}
     \end{pmatrix} \\
   &\qquad
   +i\begin{pmatrix}
       0              & -2r^{-1}\sigma & 0 \\
       -2r^{-1}\sigma & 0              & -2 q r^{-3}i_\eta \\
       0              & 2 q r^{-1}\eta & 0
     \end{pmatrix}
\end{split}
\end{equation}
at $\Gamma$, where $\pi_{\Sph^2}\colon\cM\to\Sph^2$ is the projection map $(t,r,\omega)\mapsto\omega$, and $\nabla^{\pi_{\Sph^2}^*T^*\Sph^2}$ is the pullback connection on  $\pi_{\Sph^2}^*T^*\Sph^2$ induced by the Levi-Civita connection on $\Sph^2$. For the details of this calculation, we refer the reader to\footnote{Our $\sigma$ differs from the $\sigma$ in \cite{HintzPsdoInner} by a sign.} \cite[Proposition~4.7]{HintzPsdoInner}.

In order to compute regularity thresholds at the radial set $\cR$, we will also need to compute the subprincipal operator of $\Box_g$, acting on symmetric 2-tensors, at $\cR$; up to a factor of $-i$, this is given by $\nabla_{\ham_G}^{\pi^*S^2T^*M^\circ}$. We first calculate the form of $\nabla_{\ham_G}^{\pi^*T^*M^\circ}$ at $\cR$. To simplify our calculations, we use the change of coordinates
\[
  t_0 := t - F,\quad F'=\pm\mu^{-1}
\]
near $r=r_\pm$, so $\pa_t=\pa_{t_0}$ and $dt_0=q^{-1}e^0\mp q^{-1}e^1$; this change of coordinates amounts to taking $c_{b_0,\pm}\equiv 0$ in \eqref{EqSdSTStar}, see also the related discussion in \S\ref{SubsecKdSGeo}. Furthermore, writing covectors as
\begin{equation}
\label{EqOpSdSSmoothCoord}
  -\sigma\,dt_0 + \xi\,dr + \eta, \quad \eta\in T^*\Sph^2,
\end{equation}
the dual metric function is $G=\mp 2\sigma\xi-\mu\xi^2-r^{-2}|\eta|^2$. Hence, at the conormal bundle of the horizon at $r=r_\pm$, given in coordinates by
\begin{equation}
\label{EqOpSdSRadLocation}
  \cL_\pm = \{ (t_0,r_\pm,\omega; 0,\xi,0) \in \Sigma \},
\end{equation}
the Hamilton vector field is $\ham_G=\pm 2\xi\pa_{t_0}+\mu'\xi^2\pa_\xi-2\mu\xi\pa_r$. We kept only those terms here which are not the product of a smooth function vanishing at $\cL_\pm$ with a vector field tangent to $\cL_\pm$. Now $\pa_{t_0}=\pa_t=q e_0$; using \eqref{EqOpWarpedCov}, \eqref{EqOpSpatCov} and Lemma~\ref{LemmaOpSpat}, we find
\begin{gather*}
  \nabla^{\pi^*T^*M^\circ}_{e_0}(u e^0) = (e_0 u)e^0 - q' u e^1, \\
  \nabla^{\pi^*T^*M^\circ}_{e_0}(u e^1) = (e_0 u)e^1 - q' u e^0, \\
  \nabla^{\pi^*T^*M^\circ}_{e_0}(w) = e_0 w,
\end{gather*}
where $w\in\CI(T^*M^\circ,\pi^*T^*\Sph^2)$ in the last line; therefore, in the static splitting \eqref{EqOpSpatBundlesFull} of $T^*M^\circ$, we have
\[
  \nabla^{\pi^*T^*M^\circ}_{\pa_{t_0}}
   = \begin{pmatrix}
       \pa_t            & -\frac{1}{2}\mu' & 0 \\
       -\frac{1}{2}\mu' & \pa_t            & 0 \\
       0                & 0                & \pa_t
     \end{pmatrix}.
\]
Since the partial frame used to define the splitting \eqref{EqOpSpatBundlesFull} ceases to be smooth at $r=r_\pm$, we use, near $r=r_\pm$, the smooth splitting
\begin{equation}
\label{EqOpSdSSmoothTM}
  u = \wt u_N\,dt_0 + \wt u_{TN}\,dr + \wt u_{TT};
\end{equation}
for smooth sections $u$ of the bundle $\pi^*T^*M^\circ\to T^*M^\circ$ near $r=r_\pm$, with $\wt u_N,\wt u_{TN}$ smooth functions on $T^*M^\circ$ near $r_\pm$, and $\wt u_{TT}$ a smooth section of $\pi^*T^*\Sph^2\to T^*M^\circ$. We have
\[
  u = u_N\,e^0 + u_{TN}\,e^1 + u_{TT}
\]
with
\[
  \begin{pmatrix} u_N \\ u_{TN} \\ u_{TT} \end{pmatrix} = \sC^{[1]}_\pm\begin{pmatrix} \wt u_N \\ \wt u_{TN} \\ \wt u_{TT} \end{pmatrix},\qquad
  \sC^{[1]}_\pm
          = \begin{pmatrix}
              q^{-1}     & 0 & 0 \\
              \mp q^{-1} & q & 0 \\
              0          & 0 & 1
            \end{pmatrix}.
\]
Therefore, the matrix of $\nabla^{\pi^*T^*M^\circ}_{\pa_{t_0}}$ in the splitting \eqref{EqOpSdSSmoothTM} at $\cL_\pm$ (where in particular $\mu=0$) is given by
\begin{equation}
\label{EqOpSdSConjForms}
  \nabla^{\pi^*T^*M^\circ}_{\pa_{t_0}}
   = (\sC^{[1]}_\pm)^{-1}
     \begin{pmatrix}
       \pa_t            & -\frac{1}{2}\mu' & 0 \\
       -\frac{1}{2}\mu' & \pa_t            & 0 \\
       0                & 0                & \pa_t
     \end{pmatrix}
     \sC^{[1]}_\pm
   = \begin{pmatrix}
       \pa_{t_0}\pm\frac{1}{2}\mu' & 0                           & 0 \\
       0                           & \pa_{t_0}\mp\frac{1}{2}\mu' & 0 \\
       0                           & 0                           & \pa_{t_0}
     \end{pmatrix}.
\end{equation}
We therefore have
\begin{equation}
\label{EqOpSdSSubprRadForms}
  \nabla^{\pi^*T^*M^\circ}_{\ham_G}
   = \pm 2\xi \pa_{t_0}
     \mp 2\kappa_\pm \xi^2\pa_\xi
     - 2\mu\xi\pa_r
     \pm 2\kappa_\pm\xi
       \begin{pmatrix}
         -1 & 0 & 0 \\
         0  & 1 & 0 \\
         0  & 0 & 0
       \end{pmatrix}
\end{equation}
at $\cL_\pm$, where we use the surface gravities
\[
  \kappa_\pm=\mp\frac{1}{2}\mu'(r_\pm)>0
\]
of the horizons, $\kappa_\pm=\beta_\pm^{-1}$ with $\beta_\pm$ as in \eqref{EqOpNoB0}. Note that $0$-th order terms of $\nabla_{\pa r}^{\pi^*T^*M^\circ}$ do not contribute at $\cL_\pm$ due to the extra factor of $\mu$; the $\pa_r$-derivative is needed to calculate the skew-adjoint part of $\nabla_{\ham_G}^{\pi^*T^*M^\circ}$ in \S\ref{SubsecESGTrap}.

The smooth splitting of sections $u$ of $\pi^* S^2T^*M^\circ\to T^*M^\circ$ induced by \eqref{EqOpSdSSmoothTM} is, analogously to \eqref{EqOpSpatBundlesFull},
\begin{equation}
\label{EqOpSdSSmoothS2TM}
  u = \wt u_{NN}\,dt_0^2 + 2\wt u_{NTN}\,dt_0\,dr + 2 dt_0\,\wt u_{NTT} + \wt u_{TNN} dr^2 + 2 dr\,\wt u_{TNT} + \wt u_{TTT},
\end{equation}
with $\wt u_{NN},\wt u_{NTN},\wt u_{TNN}$ functions on $T^*M^\circ$, $\wt u_{NTT},\wt u_{TNT}$ sections of $\pi^*T^*\Sph^2\to T^*M^\circ$, and $\wt u_{TTT}$ a section of $\pi^*S^2T^*\Sph^2\to T^*M^\circ$. The change of frame from this partial frame to the partial frame used in \eqref{EqOpSpatBundlesFull} is given by $\sC^{[1]}_\pm$ lifted to symmetric 2-tensors, i.e.\ by
\[
  \begin{pmatrix} u_{NN} \\ u_{NTN} \\ u_{NTT} \\ u_{TNN} \\ u_{TNT} \\ u_{TTT} \end{pmatrix}
     =\sC^{(2)}_\pm
       \begin{pmatrix} \wt u_{NN} \\ \wt u_{NTN} \\ \wt u_{NTT} \\ \wt u_{TNN} \\ \wt u_{TNT} \\ \wt u_{TTT} \end{pmatrix},
  \quad
  \sC^{(2)}_\pm
  = \begin{pmatrix}
      q^{-2}     & 0     & 0          & 0   & 0 & 0 \\
      \mp q^{-2} & 1     & 0          & 0   & 0 & 0 \\
      0          & 0     & q^{-1}     & 0   & 0 & 0 \\
      q^{-2}     & \mp 2 & 0          & q^2 & 0 & 0 \\
      0          & 0     & \mp q^{-1} & 0   & q & 0 \\
      0          & 0     & 0          & 0   & 0 & 1
    \end{pmatrix};
\]
changing the frame in the other direction uses the inverse matrix
\begin{equation}
\label{EqOpSdSConjTensors}
  (\sC^{(2)}_\pm)^{-1}
  = \begin{pmatrix}
      q^2    & 0            & 0          & 0      & 0      & 0 \\
      \pm 1  & 1            & 0          & 0      & 0      & 0 \\
      0      & 0            & q          & 0      & 0      & 0 \\
      q^{-2} & \pm 2 q^{-2} & 0          & q^{-2} & 0      & 0 \\
      0      & 0            & \pm q^{-1} & 0      & q^{-1} & 0 \\
      0      & 0            & 0          & 0      & 0      & 1
    \end{pmatrix}.
\end{equation}

Computing the second symmetric tensor power of the operator \eqref{EqOpSdSSubprRadForms} in the splitting \eqref{EqOpSdSSmoothS2TM}, we find
\begin{equation}
\label{EqOpSdSSubprRadTensors}
  \nabla^{\pi^*S^2T^*M^\circ}_{\ham_G}
    = \pm 2 \xi\pa_{t_0} \mp 2 \kappa_\pm\xi^2\pa_\xi - 2\mu\xi\pa_r
      \pm 2 \kappa_\pm\xi
        \begin{pmatrix}
          -2 & 0 & 0 & 0 & 0 & 0 \\
          0 & 0 & 0 & 0 & 0 & 0 \\
          0 & 0 & -1 & 0 & 0 & 0 \\
          0 & 0 & 0 & 2 & 0 & 0 \\
          0 & 0 & 0 & 0 & 1 & 0 \\
          0 & 0 & 0 & 0 & 0 & 0
        \end{pmatrix}
\end{equation}
at $\cL_\pm$, and thus at $\cR_\pm$ as a b-differential operator, writing $\pa_{t_0}=-\tau_0\pa_{\tau_0}$ for the boundary defining function $\tau_0=e^{-t_0}$ of $M$.

\section{Mode stability for the Einstein equation (UEMS)}
\label{SecUEMS}

In this section, we prove Theorem~\ref{ThmKeyUEMS}.\footnote{In the follow-up work \cite{HintzKNdSStability}, the first author gives a detailed self-contained proof of UEMS for the Einstein--Maxwell system, linearized around a spherically symmetric Reissner--Nordstr\"om--de~Sitter spacetime. A forteriori, this gives UEMS for Schwarzschild--de~Sitter spacetimes, Theorem~\ref{ThmKeyUEMS}, as the special case when both the black hole charge and the electromagnetic perturbation vanish.} We recall the mode stability results on Schwarzschild--de~Sitter space from previous physics works by Ishibashi, Kodama and Seto \cite{KodamaIshibashiSetoBranes,KodamaIshibashiMaster,IshibashiKodamaHigherDim}. These are not stated in the form we need them, but it is easy to put them into the required form.

These works rely on a decomposition of tensors into scalar and vector parts, as discussed in \cite{KodamaIshibashiSetoBranes}; we restrict our discussion here to the case of interest in the present paper, namely we work in $3$ spatial dimensions, so the spherical metric is the round metric on $\Sph^2$, corresponding to taking $n=2$ in \cite{KodamaIshibashiMaster}. (For higher dimensional spheres there would be a tensor part as well.) In brief, if $h$ is a (generalized) mode solution of the linearized Einstein equation~\eqref{EqKeyUEMSLinEin}, we shall expand $h$ into spherical harmonics; since $h$ is a symmetric 2-tensor, this can be accomplished by means of the formulas~\eqref{EqUEMSScalar} and \eqref{EqUEMSScalar0} (for scalar perturbations) and \eqref{EqUEMSVector} (for vector perturbations) below. Now, the Schwarzschild--de~Sitter metric is rotationally invariant, hence so is the linearized Einstein equation; therefore, each spherical harmonic component by itself solves the linearized Einstein equation. The calculations in the rest of the section analyze each of these components and show that each is a sum of a pure gauge term and a linearized Kerr--de~Sitter metric.

Let us describe this in more detail. First, one considers a decomposition of tensors into aspherical ($dr,dt$), mixed and spherical parts. Then scalar perturbations arise from non-constant scalar eigenfunctions of the (negative) spherical Laplacian $\slDelta$ on $\Sph^2$,
\begin{equation}
\label{EqUEMSScalarEigenfunction}
  (\slDelta+k^2)\scal=0,\quad k>0,
\end{equation}
via considering $\sld\scal$ as well as $\sldelta^*\sld\scal$ and $\scal\slg$, namely
\begin{equation}
\label{EqUEMSScalar}
  h=\wt f\scal - \frac{2 r}{k}(f\otimes_s\sld\scal) + 2r^2\Bigl[H_L\scal\slg+H_T\Bigl(\frac{1}{k^2}\sldelta^*\sld+\frac{1}{2}\slg\Bigr)\scal\Bigr],
\end{equation}
where $\wt f$ is valued in aspherical 2-tensors, $f$ in aspherical one-forms, and $H_L,H_T$ are functions, all independent of the spherical variables; see \cite[Equation~(2$\cdot$4)]{KodamaIshibashiMaster}, or \cite[\S3]{KodamaIshibashiSetoBranes} in a more general setting.

On the other hand, vector perturbations arise from eigen-1-forms of the (negative) tensor Laplacian $\slDelta$,
\begin{equation}
\label{EqUEMSVectorEigenfunction}
  (\slDelta+k^2)\vect=0,\quad \sldelta\vect=0,
\end{equation}
via considering $\sldelta^*\vect$, namely
\begin{equation}
\label{EqUEMSVector}
  h=2r(f\otimes_s\vect)-\frac{2}{k}r^2 H_T\sldelta^*\vect,
\end{equation}
where $f$ is an aspherical one-form and $H_T$ a function, all independent of the spherical variables; see \cite[\S5.2]{KodamaIshibashiMaster} or \cite[Equation~(29)]{KodamaIshibashiSetoBranes}.

Together with the rotationally invariant ($k=0$) scalar case, when there is only
\begin{equation}
\label{EqUEMSScalar0}
  h=\wt f \scal+2r^2 H_L\scal\slg,\quad \scal\equiv 1,
\end{equation}
\eqref{EqUEMSScalar} and \eqref{EqUEMSVector} give a Hilbert basis of the $L^2$ space of 2-tensors on Schwarzschild--de~Sitter space. Indeed, this decomposition is a consequence of the well-known scalar--vector--tensor decomposition on $\Sph^2$, after tensoring the latter with the $L^2$ space in the non-spherical variables $(t_*,r)$. We briefly recall the decomposition on $\Sph^2$: aspherical 2-tensors are captured by $\wt f\scal$ in \eqref{EqUEMSScalar} and \eqref{EqUEMSScalar0}; next, the Helmholtz decomposition of 1-forms on $\Sph^2$ together with the fact that $(\slDelta+k^2)$ preserves the divergence-free condition in \eqref{EqUEMSVectorEigenfunction} show that the $f\otimes_s\sld\scal$ and $f\otimes_s\vect$ terms in \eqref{EqUEMSScalar} and \eqref{EqUEMSVector} capture all symmetric products of aspherical 1-forms with spherical 1-forms. Spherical pure trace tensors are described by $H_L\scal\slg$ in \eqref{EqUEMSScalar} and \eqref{EqUEMSScalar0}. Lastly, the space of traceless tensors on $\Sph^2$ can be decomposed into an orthogonal sum of the space of traceless and divergence-free tensors, and the space of tensors of the form $(\sldelta^*+\frac{1}{2}\slg\sldelta)W$ (the operator here is the adjoint of $\sldelta$ acting on tracefree symmetric 2-tensors), with $W$ a 1-form on $\Sph^2$. Now on $\Sph^2$, there are no non-trivial traceless and divergence-free 2-tensors (see for example \cite[section~III]{HiguchiSpherical} for a proof), and the Helmholtz decomposition for $W$ then yields the last terms in \eqref{EqUEMSScalar} and \eqref{EqUEMSVector}, respectively.

One also considers possible gauge changes. In the scalar case these correspond to 1-forms
\begin{equation}
\label{EqUEMSScalarGauge}
  \xi=T\scal+rL\,\sld\scal,
\end{equation}
where $T$ is an aspherical one-form, $L$ a function, both independent of the spherical variables, and $\scal$ as in \eqref{EqUEMSScalarEigenfunction}; see \cite[Equation~(47)]{KodamaIshibashiSetoBranes}. In the vector case, they are
\[
  \xi=r L\vect,
\]
where $L$ is a function, independent of the spherical variables, with $\vect$ as in \eqref{EqUEMSVectorEigenfunction}; see \cite[Equation~(31)]{KodamaIshibashiSetoBranes}.

It is convenient to also introduce the angular momentum variable, which is related to $k$ by
\[
  k^2=l(l+1),\quad l\in\N_0,
\]
in the scalar case, and
\[
  k^2=l(l+1)-1,\quad l\in\N_+
\]
in the vector case.

\subsection{Perturbations with \texorpdfstring{$l\geq 2$}{l at least 2}}

There are special cases corresponding to $l=0,1$ in the scalar case and $l=1$ in the vector case, which we discuss below, so we first consider $l\geq 2$. Then \cite{KodamaIshibashiMaster,IshibashiKodamaHigherDim} show that certain gauge-invariant quantities constructed in \cite{KodamaIshibashiSetoBranes} necessarily vanish for any (temporal, i.e.\ with fixed frequency $\sigma$ in $t$) mode solution of the linearized ungauged Einstein equation (linearized at Schwarzschild--de~Sitter) with $\Im\sigma\geq 0$ (in our notation). In the scalar case, for $h$ as in \eqref{EqUEMSScalar}, introducing
\[
  \bfX=\frac{r}{k}\Bigl(f+\frac{r}{k}\,dH_T\Bigr),
\]
these quantities are 
\begin{equation}
\label{EqUEMSScalarGaugeInv}
  F=H_L+\frac{1}{2}H_T+\frac{1}{r}G(dr,\bfX), \quad \wt F=\wt f+2\delta_{g_{\AS}}^*\bfX,
\end{equation}
where
\[
  g_{\AS}=q^2\,dt^2-q^{-2}\,dr^2
\]
is the aspherical part of the metric; see \cite[Equations~(2$\cdot$7a)--(2$\cdot$8)]{KodamaIshibashiMaster} or \cite[Equations~(57)--(59)]{KodamaIshibashiSetoBranes}. By \eqref{EqOpSpatDelH}, we have $\delta_{g_{\AS}}^*\bfX=\delta_g^*\bfX - r G(\bfX,dr)\slg$. Now if $F=0$ and $\wt F=0$ then $h$ can be written as
\[
  h = \delta_g^*\Bigl(-2\bfX\scal+2\frac{r^2}{k^2} H_T\sld\scal\Bigr),
\]
where we use Lemma~\ref{LemmaOpSpat} to compute $\delta_g^*(r^2\sld\scal)=r^2\sldelta^*\sld\scal$; therefore, the mode $h$ is a pure gauge mode.

In the vector case, with $h$ given by \eqref{EqUEMSVector}, the gauge invariant quantity is
\begin{equation}
\label{EqUEMSVectorGaugeInv}
  F=f+\frac{r}{k}dH_T,
\end{equation}
see \cite[Equation~(5$\cdot$10)]{KodamaIshibashiMaster} or \cite[Equation~(33)]{KodamaIshibashiSetoBranes}, and if it vanishes then
\[
  h=-\frac{2}{k}\delta_g^*(r^2 H_T\vect),
\]
so again the mode is a pure gauge mode.

We remark that both in the scalar and the vector cases, the gauge 1-form is also a temporal mode 1-form, with the same frequency $\sigma$ as $h$.

Note that one in fact has a more precise result here: the gauge 1-forms are well-behaved even on the extension of the spacetime across the horizons. This is due to the fact that the vanishing of the gauge invariant quantities, constructed from a scalar or vector perturbation satisfying the linearized Einstein equation, follows by reducing to a `master equation' satisfied by a `master variable' $\Phi$. For the precise definition of $\Phi$, we refer the reader to \cite[\S3]{KodamaIshibashiMaster} (and \cite[\S4]{KodamaIshibashiMaster} for the discussion of stationary perturbations, i.e.\ $\sigma=0$); here, we merely remark that $\Phi$ is a suitable linear combination (with $r$-dependent coefficients) of components of $F$ and $\wt F$ in the scalar case, and $F$ is the vector case. This master equation is a (time-harmonic, i.e.\ stationary) Schr\"odinger equation with a positive potential: the corresponding Schr\"odinger operator is of the form
\begin{equation}
\label{EqUEMSSchrodinger}
  (\mu D_r)^2 + V_S,
\end{equation}
see \cite[Equation~(3$\cdot$5)]{KodamaIshibashiMaster}. A simple coordinate change $x=x(r)$ transforms this into the Schr\"odinger operator $D_x^2+V_S$, $x\in\R$, \emph{on the real line} with $V_S(x)$ exponentially decaying as $|x|\to\infty$;\footnote{Roughly, at $\mu=0$, the Schr\"odinger operator is essentially $(\mu D_\mu)^2+V_S$, which with $x=-\log\mu$ equals $D_x^2+V_S$, with $V_S=\cO(\mu)=\cO(e^{-x})$ exponentially decaying.} thus, this operator is essentially self-adjoint. The vanishing of the master variable for mode solutions then follows from
\begin{itemize}
  \item the spectrum of this Schr\"odinger operator being non-negative: this excludes $\Im\sigma>0$;
  \item the absence of embedded eigenvalues and, more generally, of real non-zero resonance: this is proved via a boundary pairing formula (see~\cite[\S\S2.3 and 2.5]{MelroseGeometricScattering} for the case of embedded eigenvalues) and a unique continuation at infinity for Schr\"odinger operators on the real line with exponentially decaying potentials;\footnote{See also the arguments around \cite[Equations~(3.29) and (3.32)]{HintzVasyKdsFormResonances} for a discussion in the context of asymptotically hyperbolic spaces---the present problem can be regarded as living on 1-dimensional hyperbolic space, given as the interval $(r_-,r_+)_r$ with the metric $\mu^{-2}\,d r^2$, whose Laplacian is $(\mu D_r)^2$ as in~\eqref{EqUEMSSchrodinger}}
  \item the absence of a $0$-eigenvalue and, more generally, of a $0$-resonance, due to the positivity of the potential.
\end{itemize}
Moreover, the Schr\"odinger operator~\eqref{EqUEMSSchrodinger} is well-behaved even on the extended space; the resonances of the extended problem (corresponding to these modes) correspond to resonances of the asymptotically hyperbolic problem described by this Schr\"odinger operator, see for instance \cite[Lemma~2.1]{HintzVasyKdsFormResonances} and \cite[Footnote~58]{VasyMicroKerrdS}. Thus, the vanishing of $\Phi$ at first in the exterior of the event and cosmological horizons in fact guarantees its vanishing globally (across the horizons); then the gauge invariant quantities are reconstructed from this and thus vanish; and finally the above arguments then show that the metric perturbations under consideration are all pure gauge modes. This completes the proof of UEMS for modes if $l\geq 2$; see Lemma~\ref{LemmaUEMSWithLogs} for the case of generalized modes.

\subsection{Spherically symmetric perturbations}

Next, the case $l=0$ exists only in the scalar case, and it corresponds to spherical symmetry. In this case, we do not need to assume the metric perturbation $h$ to be a generalized mode; we shall show that \emph{any} spherically symmetric perturbation arises by an infinitesimal change of the black hole mass.

The extension of Birkhoff's theorem on the classification of solutions of Einstein's equations with spherical symmetry---namely, that the only such solutions are Schwarzschild space\-times---to positive cosmological constants was done in a particularly simple manner by Schleich and Witt \cite{SchleichWittBirkhoff}. One needs to check that their arguments work already at the linear level---a priori there may be solutions of the linearized equation that do not correspond to solutions of the non-linear equation---but this is straightforward, as we show in the remainder of this section.

The proof in~\cite{SchleichWittBirkhoff} proceeds by writing the Lorentzian metric, with the negative of our sign convention, in the form
\begin{equation}
\label{EqUEMSSdS}
  g=F\,du^2+2X\,du\otimes_s dv+Y^2\,\slg,
\end{equation}
with $F,X,Y$ independent of the spherical variables, which one may always do by a diffeomorphism; on the linearized level, one can similarly bring a metric perturbation into this form by adding a $\delta_g^*$ term. Note in particular that the Schwarzschild--de~Sitter metric is locally in $r$, but globally otherwise, in this form for an appropriate choice of $t_*$ (in terms of the definition \eqref{EqSdSTStar} of $t_*$, with $c_{b_0,\pm}\equiv 0$) with $u=t_*$, $v=r$, and then $X=\pm 1$, $Y=v$, $F=\frac{\Lambda}{3} v^2+\frac{2M}{v}-1$.  Notice that by spherical symmetry a priori $g$ is of the form
\[
  g=\wt F\,d\wt u^2+2\wt X\,d\wt u\otimes_s d\wt v+\wt Z\,d\wt v^2+\wt Y^2\,\slg,
\]
with coefficients independent of the spherical variables, i.e.\ $g$ simply has an additional $\wt Z\,d\wt v^2$ term; in our near-Schwarzschild--de~Sitter regime one may even assume (for convenience only) that $\wt Z$ is small; then the coordinate change is $v=\wt v$, $u=U(\wt u,\wt v)$, and conversely $\wt v=v$, $\wt u=\wt U(u,v)$, which gives the $dv^2$ component $\wt F(\pa_{ v}\wt U)^2+2\wt X\pa_{v}\wt U+\wt Z$, which is easily solvable for $\pa_v\wt U$ whether $\wt F$ vanishes (since $\wt X$ is near $1$) or not (since $\wt Z$ is assumed small, so the discriminant of the quadratic equation is positive). Note that with $\delta \wt U$ denoting linearized change in $\wt U$ (relative to $\wt U=u$, corresponding to the trivial coordinate change $\wt v=v$, $\wt u=u$ needed in the case of the Schwarzschild--de~Sitter metric with $\wt Z=0$), $\delta\wt Z$ the change in $\wt Z$ (relative to $\wt Z=0$), we get at $\wt X=1$ the equation $2\pa_v\delta\wt U+\delta\wt Z=0$ for the linearized gauge change, which in particular preserves the $u$-modes.

The gauge term for the linearized equation around Schwarzschild--de~Sitter can be seen even more clearly by considering $\delta^*_{g_{0,\pm}}$, where $g_{0,\pm}$, is the Schwarzschild--de~Sitter metric in the above form \eqref{EqUEMSSdS}, thus with $v=r$, $u=t_*$ with an appropriate choice of $t_*$, corresponding to taking $c_{b_0,\pm}\equiv 0$ in \eqref{EqSdSTStar}. See Figure~\ref{FigUEMSNullCoord}.

\begin{figure}[!ht]
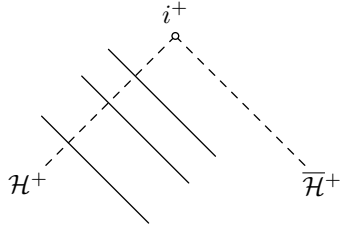

  \centering
  \inclfig{UEMSNullCoord}
  \caption{Level sets of $u$ for which $(u,v=r)$ gives coordinates near the event horizon, away from the cosmological horizon, in which the Schwarzschild--de~Sitter metric takes the form \eqref{EqUEMSSdS}. Analogous coordinates can be chosen near the cosmological horizon, away from the even horizon.}
  \label{FigUEMSNullCoord}
\end{figure}

Then, acting on aspherical 1-forms (the only ones for $l=0$, see \eqref{EqUEMSScalarGauge}), written in the basis $du$, $dr$, and with output written in terms of the basis $du\otimes du$, $2du\otimes_s dr$, $dr\otimes dr$, $\slg$, and moreover writing $\mu=\mu(r)$ for the $du^2$ component of Schwarzschild--de~Sitter, we compute
\[
  \delta^*_{g_{0,\pm}}
   =\begin{pmatrix}
      \pa_u+\frac{1}{2}\mu' & -\frac{1}{2}\mu\mu'              \\
      \frac{1}{2}\pa_r      & \frac{1}{2}\pa_u-\frac{1}{2}\mu' \\
      0                     & \pa_r                            \\
      r                     & -r \mu
    \end{pmatrix},
\]
and thus, with tangent vectors written in terms of the basis $\pa_u$, $\pa_r$,
\[
  \delta^*_{g_{0,\pm}}g_{0,\pm}
    =\begin{pmatrix}
       \mu\pa_u                             & \pa_u+\frac{1}{2}\mu' \\
       \frac{1}{2}\mu\pa_r+\frac{1}{2}\pa_u & \frac{1}{2}\pa_r      \\
       \pa_r                                & 0                     \\
       0                                    & r
     \end{pmatrix},
\]
where the second $g_{0,\pm}$ on the left is the isomorphism from the tangent to the cotangent bundle. This shows that, given an $l=0$ symmetric 2-tensor, its $dr^2$ component, say $\dot Z\,dr^2$, can be removed by subtracting $\delta^*_{g_{0,\pm}}g_{0,\pm}$ applied to an appropriate multiple $f\pa_u$ of $\pa_u$; one simply solves $\dot Z=\pa_r f$. Notice that this gauge change preserves mode expansions.

Einstein's equations\footnote{Recall that the sign conventions in \cite{SchleichWittBirkhoff}, which we are using presently, are the opposite of what we use in the rest of the paper.} $\Ric(g)-\Lambda g=0$ for the metric \eqref{EqUEMSSdS} are stated in \cite[Equations~(6)--(9)]{SchleichWittBirkhoff}; for us the important ones are
\begin{gather}
\label{EqUEMSEin01}
  -\pa_v X\pa_v Y+X\pa_v^2 Y=0, \\
\label{EqUEMSEin02}
  X^2+Y\pa_v F\pa_v Y+F(\pa_v Y)^2-2X\pa_u Y\pa_v Y-2XY\pa_u\pa_v Y=-\Lambda X^2 Y^2, \\
\begin{split}
\label{EqUEMSEin03}
  X\pa_u F\pa_v Y+2 X F\pa_u\pa_v Y-2F\pa_u X&\pa_v Y-X\pa_v F\pa_u Y \\
    & + 2X\pa_u X\pa_u Y-2X^2\pa^2_u Y=0.
\end{split}
\end{gather}
(Equation~\eqref{EqUEMSEin01} is the $dv^2$ component of the Einstein equation, equation~\eqref{EqUEMSEin02} the spherical part, simplified using \eqref{EqUEMSEin01}, and lastly\footnote{There seem to be two typos in \cite[Equation~(9)]{SchleichWittBirkhoff}: a missing factor $F$ in the second term, and the differentiations in the fifth term are with respect to $u$ rather than $v$. This does not affect the argument in \cite{SchleichWittBirkhoff} however, since this equation is only used once one has arranged $Y=v$ and $X=\pm 1$, hence $\pa_u X=\pa_v X\equiv 0$ and $\pa_u\pa_v Y\equiv 0$.} \eqref{EqUEMSEin03} is the $du^2$ component, simplified by plugging in the expression for $\Lambda$ from \eqref{EqUEMSEin02}.)

Now, the linearized version of \eqref{EqUEMSEin01}, linearized around Schwarzschild--de~Sitter, so $X=1$, $Y=v$, with dotted variables denoting the linearization, is
\begin{equation}
\label{EqUEMSEin01L}
  -\pa_v \dot X+\pa_v^2 \dot Y=0,
\end{equation}
so
\[
  \dot X-\pa_v\dot Y=\xi(u),
\]
with $\xi$ independent of $v$. With this in mind it is convenient to further arrange that in $g$, one has $Y=v$, as one always may do by a diffeomorphism, and thus infinitesimally at Schwarzschild--de~Sitter by a $\delta_{g_{0,\pm}}^*$ term. Indeed, we are assuming that $Y$ is near $v$ by virtue of considering deformations of Schwarzschild--de~Sitter, so $Y$ can be used as a coordinate in place of $v$, and thus the inverse function theorem is applicable at least locally. Thus $v=V(u,Y)$ and the form of the metric (with no $dv^2$ term) is preserved. Arranging this directly on the linearized level can in fact be done globally (except in $r$, since our Schwarzschild--de~Sitter metric is only local in $r$): one can remove the $\dot Y$ component (appearing as a coefficient of $\slg$ in the linearized metric) by subtracting $\delta_{g_{0,\pm}}^* g_{0,\pm}(2 r^{-1}\dot Y\pa_r)$, which preserves mode expansions. Having arranged $\dot Y=0$, we conclude that $\dot X=\xi(u)$.

In fact, we can use this additional information to further simplify the metric by a simple change of variables to fix the coefficient of $du\otimes_s dv$ as $2$. Indeed, the $du\otimes_s dv$ term is $2X\,du\otimes_s dv=2\xi(u)\,du\otimes_s dv$, so changing $u$ appropriately, namely by doing a change of variables $\wt u=\wt U(u)$ with $\wt U'=\xi$, arranges this. In the linearized version, we note that solving $\pa_u f=2\dot X$ \emph{with $f$ independent of $r$}, and subtracting $\delta_{g_{0,\pm}}^*g_{0,\pm}(f\,\pa_u)$ from the symmetric 2-tensor removes its $du\otimes_s dr$ component as well. Note that this uses strongly that $\dot X=\xi(u)$: this is what ensures that the $\pa_r$ derivative in $\delta_{g_{0,\pm}}^*g_{0,\pm}$ does not give a $dr^2$ component. Again, notice that this gauge change keeps the $u$-modes unchanged except the $0$-mode, in which case a linear in $u$ term is generated.

But now, with $\dot X=0$, $\dot Y=0$, the linearization of \eqref{EqUEMSEin02} is
\begin{equation}\label{EqUEMSEin02L}
  v\pa_v \dot F+\dot F=0.
\end{equation}
This says $\pa_v(v\dot F)=0$, and hence
\[
  \dot F=M(u)v^{-1}.
\]
Finally \eqref{EqUEMSEin03} becomes
\begin{equation}
\label{EqUEMSEin03L}
  \pa_u\dot F=0,
\end{equation}
and thus $M(u)=M$ is independent of $u$. Comparing with \eqref{EqSdSMetric}, this is exactly the infinitesimal Schwarz\-schild--de~Sitter deformation corresponding to changing the mass. Thus, one concludes that locally in $r$, but globally in $t_*$, the only solution, \emph{without fixing a frequency $\sigma$ in $t_*$}, i.e.\ not working on the Fourier-transformed (in $t_*$) picture, is linearized Schwarzschild--de~Sitter---which is a 0-mode---up to gauge terms, which are mode gauge terms, with possibly linear growth in $t_*$, as described above.

Now, this argument applies separately in two regions of the form $r_- -\delta<r<r_2$ and $r_1<r<r_+ +\delta$, with $r_\pm$ as in \eqref{EqOpNoB0}, where $r_-<r_1<r_2<r_+$. Thus, in the overlap $r_1<r<r_2$, we can write the linearized solution as
\[
  \dot g=\delta_g^*\dot\theta_1+\dot g_{0,-}=\delta_g^*\dot\theta_2+\dot g_{0,+},
\]
where $\dot g_{0,\pm}$ are the linearized Schwarzschild--de~Sitter solutions in the two regions, and where we are using the global Schwarzschild--de~Sitter metric $g$, so the $\theta_\pm$ differ from the ones constructed above by pull-back by the diffeomorphism that puts $g$ into the form considered above, with vanishing $dr^2$ term; notice that this again preserves the $u$-mode expansions. This in particular implies that the $\dot g_{0,\pm}$ have the same mass parameter $M$, thus they are equal, that is $\dot g_{0,\pm}=\dot g_0$. Then one concludes that $\delta_g^*(\dot\theta_--\dot\theta_+)=0$, i.e.\ $\dot\theta_--\dot\theta_+$ is the one-form corresponding to a local Killing vector field which is spherically symmetric. This is necessarily a constant multiple of $\pa_t=\pa_{t_*}=\pa_u$, and is thus globally defined. Adding this multiple of the 1-form $g(\pa_t)$ to $\dot\theta_+$ to get $\dot\theta_+'$ we then have $\dot\theta_-=\dot\theta_+'$ in the overlap, and thus the one-form $\dot\theta$, equal to $\dot\theta_-$ and $\dot\theta_+'$ in the domains of definition of these two one-forms, is a well-defined global one-form. This proves the $l=0$ case of UEMS.

\subsection{Scalar perturbations with \texorpdfstring{$l=1$}{l equal to 1}}

In the scalar $l=1$ case one proceeds differently. Namely in this case the quantities \eqref{EqUEMSScalarGaugeInv} discussed for $l\geq 2$ are not gauge-invariant anymore as $\sldelta^*\sld\scal=-\scal\slg$ for $\scal$ as in \eqref{EqUEMSScalarEigenfunction}, that is $(\frac{1}{k^2}\sldelta^*\sld+\frac{1}{2}\slg)\scal\equiv 0$ (recalling $k^2=l(l+1)=2$), so the $H_T$ component in \eqref{EqUEMSScalar} is no longer defined; one simply puts $H_T:=0$. Moreover, they do not solve the full set of equations from $l\geq 2$ a priori, but it is shown in \cite{KodamaIshibashiMaster} that the additional equations can be regarded as gauge conditions. The extra gauge equations one gets are linear wave equations, thus solvable, with remaining gauge freedom given by the corresponding initial data, and thus these extra gauge equations can be assumed to hold. In fact, since we are working in $(2+1)$ spatial dimensions, the extra gauge equation is {\em exactly} the scalar wave equation, see \cite[Equation~(B$\cdot$3)]{KodamaIshibashiMaster},
\begin{equation}
\label{EqUEMSScalar1GaugeWave}
  \Box_g(L/r)\equiv r^{-2}\delta_{g_{\AS}}r^2 d_{\AS}(L/r)=-\frac{1}{k}\dot E_T,
\end{equation}
where $L$ is the coefficient of the gauge 1-form $r L\,\sld\scal$ (i.e.\ the special case $T=0$ of the expression \eqref{EqUEMSScalarGauge}),\footnote{The notation here is that of \cite[Equation~(47)]{KodamaIshibashiSetoBranes}, not the slightly different one used in \cite[Appendix~B]{KodamaIshibashiMaster}.} and where $\dot E_T$ is the (non-existing) component of the linearized Einstein tensor that \emph{would} correspond to $H_T$ in \eqref{EqUEMSScalar} \emph{if} $l\geq 2$ were the case;\footnote{In this case, the linearized Einstein tensor can be decomposed in the same manner as the linearized metric perturbation, i.e.\ in the present context as \eqref{EqUEMSScalar}. For $l=1$ however, as discussed above, the coefficient of $H_T$, or $\dot E_T$ for the linearized Einstein tensor, vanishes.} here it is directly defined in terms of $F$ and $\wt F$ from \eqref{EqUEMSScalarGaugeInv} as $\frac{2r^2}{k^2}\dot E_T=-\tr\wt F$, see \cite[Equation~(A$\cdot$1d)]{KodamaIshibashiMaster}. Now, for fixed temporal frequency $\sigma$ one still has an ODE system, \cite[Equation~(2$\cdot$24)]{KodamaIshibashiMaster}, which has a 2-dimensional space of solutions as a linear ODE in the exterior of the black hole/cosmological horizon, {\em without} imposing any conditions at the horizons. On the other hand, for fixed temporal frequency, again without imposing any conditions at the horizon, the ODE corresponding to the gauge wave equation \eqref{EqUEMSScalar1GaugeWave} also has a 2-dimensional space of solutions $\wh L$. Now the map from $\wh L$ to changes in the scalar quantities $(X,Y,\wt Z)$ of \cite[Equation~(2$\cdot$24)]{KodamaIshibashiMaster} can be thought of, via lowering an index by the metric and multiplying by $k$ (recall $k^2=2$), as the map
\begin{equation}
\label{EqUEMSScalar1GaugeMap}
  \wh L\mapsto-2\wh{\delta_{g_{\AS}}^*}r^2 \wh{d_{\AS}}(\wh L/r)+2r G(dr,d(\wh L/r))g_{\AS} + \frac{k^2}{r}\wh L g_{\AS},
\end{equation}
with $\wh{\phantom{a}}$ denoting the Fourier transform in $-t$. (Here, we regard $(X,Y,Z)$, with $\wt Z$ defined in terms of $Z$ by \cite[Equation~(2$\cdot$23)]{KodamaIshibashiMaster}, as the components of the aspherical tensor $\wt F-2 F g_{\AS}$; see \cite[Equation~(2$\cdot$20)]{KodamaIshibashiMaster}.) We thus need that on the 2-dimensional space of time-harmonic solutions of the aspherical $\Box_{g_{\AS}}$, i.e.\ on the space of $\wh L$ satisfying
\begin{equation}
\label{EqUEMSScalar1GaugeWave2}
  (\pa_r q^2 r^2\pa_r+\sigma^2 q^{-2}r^2)(\wh L/r)=0,
\end{equation}
this map is injective. This is straightforward to check: on the kernel of the Fourier-transformed version of the map \eqref{EqUEMSScalar1GaugeWave} (with right hand side $\dot E_T=0$), the trace of the first term in \eqref{EqUEMSScalar1GaugeMap} vanishes, so the pure trace component of the map \eqref{EqUEMSScalar1GaugeMap}, with $\wt L=\wh L/r$, is $2g_{\AS}$ times
\[
  \wt L \mapsto r q^2\pa_r\wt L + \frac{k^2}{2}\wt L.
\]
Assuming $\wt L$ is in the kernel of this map and solves \eqref{EqUEMSScalar1GaugeWave2}, we can rewrite $\pa_r\wt L$ in terms of $\wt L$ itself and obtain
\[
  \pa_r r\Bigl(-\frac{k^2}{2}\wt L\Bigr) + \sigma^2 q^{-2}r^2\wt L=0;
\]
expanding the derivative and substituting again, we get (using $k^2=2$)
\[
  (-q^2+1+\sigma^2 r^2)\wt L=\Bigl(\frac{2\bhm}{r}+\frac{\Lambda r^2}{3}+\sigma^2 r^2\Bigr)\wt L=0.
\]
For any value of $\sigma\in\C$, the factor here does not vanish for all but at most $3$ values of $r$, so we obtain $\wt L=0$, hence $\wh L=0$, as desired.

Since such a gauge change by $\wh L$ still gives a solution of the ODE system satisfied by $(X,Y,\wt Z)$, we conclude that {\em all solutions} of the ODE system are given by a gauge change from the $0$ solution; see the related discussion in \cite[Appendix~E]{KodamaIshibashiSetoBranes}. Thus, one may assume that the no-longer gauge invariant quantities constructed in fact vanish. In this case, with\footnote{This is the 1-form, denoted $X_a$ in \cite[Equation~(2$\cdot$8)]{KodamaIshibashiMaster}, not the scalar quantity $X$ from the previous paragraph.}
\[
  \bfX=\frac{r}{k}f,
\]
these quantities take the form
\[
  F=H_L+\frac{1}{r}G(dr,\bfX),\quad \wt F=\wt f+2\delta_{g_{\AS}}^*\bfX,
\]
and the vanishing of $F$ and $\wt F$ (which as we said can be assumed up to gauge terms) implies
\[
  \delta_g^*(-2\bfX\scal)=\wt f\scal-2\frac{r}{k}f\otimes_s \sld\scal+2r^2H_L\scal\slg=h,
\]
as desired. Note that as the ODE analysis in \cite{KodamaIshibashiMaster} is done in the exterior of the black hole/cosmological horizon, this is not quite yet the statement we want: we want to express the metric perturbation as a Lie derivative also \emph{beyond the horizons}. The only potential problem is that the gauge transformation $\wh L$ above could be non-smooth at the horizons. But note that the mode (which we set out to show is pure gauge) is smooth at the horizons, and the result of adding to it the pure gauge solution according to \eqref{EqUEMSScalar1GaugeMap} is identically $0$ between the horizons. However, by direct substitution one can check that the non-smooth asymptotics which could a priori be present for the solution of \eqref{EqUEMSScalar1GaugeWave} are not annihilated by the map \eqref{EqUEMSScalar1GaugeMap}, and therefore $\wh L$ \emph{cannot} have non-smooth asymptotics, as desired. This establishes UEMS in this case for mode solutions; for the case of generalized modes, see Lemma~\ref{LemmaUEMSWithLogs} below.

\subsection{Vector perturbations with \texorpdfstring{$l=1$}{l equal to 1}}

We now consider the $l=1$ vector case, i.e.\ $k=1$ in \eqref{EqUEMSVectorEigenfunction}. One can again proceed as in the scalar $l=1$ case; now the $H_T$ component in \eqref{EqUEMSVector} does not exist, since for $k=1$, we have $\sldelta^*\vect=0$, i.e.\ $\vect$ is the one-form corresponding to a Killing vector field on $\scal^2$. If one sets $H_T=0$ in the formula \eqref{EqUEMSVectorGaugeInv}, one gets the non-gauge invariant quantity $F=f$. However, $rd_{\AS}(f/r)$ becomes a gauge-invariant quantity in this case, see \cite[Equation~(34)]{KodamaIshibashiSetoBranes}. The linearized Einstein equation, with $g_{\AS}$ the aspherical part of the metric, then becomes
\[
  r^{-3}\delta_{g_{\AS}} r^3 \bigl(r d_{\AS}(f/r)\bigr)=0,
\]
see \cite[Equation~(37)]{KodamaIshibashiSetoBranes}, which gives (as $d_{\AS}(f/r)$ is a top degree, i.e.\ degree 2, aspherical form) that
\begin{equation}
\label{EqUEMSVectorL1}
  r d_{\AS}(f/r)=C r^{-3}(\star_{g_{\AS}}1)
\end{equation}
for a constant $C$, where $\star_{g_{\AS}}$ is the Hodge-star operator. (Note that this does {\em not} require the actual solution to be a mode solution, just like in the $l=0$ case there was no such requirement in our argument!) In particular, this gauge-invariant quantity is static, and in fact is the image of an infinitesimal Kerr--de~Sitter solution: indeed, recall from \eqref{EqUEMSVector} that the actual solution is $h=2 r f\otimes_s\vect$, which for a rotation vector field $\vect$ on $\Sph^2$ gives an infinitesimal Kerr--de~Sitter solution rotating around the same axis as $\vect$ if we take $f=(\frac{\Lambda r}{3}+\frac{2\bhm}{r^2}\bigr)dt$ (which indeed satisfies \eqref{EqUEMSVectorL1} with $C=6\bhm$); this follows from differentiating the metric \eqref{EqKdSSlowStationary} in the parameter $a$ at $a=0$.

Therefore, subtracting the corresponding infinitesimal Kerr--de~Sitter solution, we may assume that $rd_{\AS}(f/r)=0$. Now, $\delta_g^*(r^2\psi\vect)=r^2 d_{\AS}\psi\otimes_s\vect$; thus we want to have $2rf=r^2 d_{\AS}\psi$, i.e.\ $d_{\AS}\psi=f/2r$. But this can be arranged as $d_{\AS}(f/r)=0$ and the cohomology of the $(t,r)$ region is trivial; indeed it is easy to write down such a $\psi$ explicitly, simply integrating $f/r$. Notice that this gives, if $f$ had an expansion in powers of $t_*$, one higher power in terms of $t_*$. This establishes UEMS, and completes the proof of Theorem~\ref{ThmKeyUEMS}.

\subsection{Generalized modes}

We finally bridge the gap between the modes considered in Theorem~\ref{ThmKeyUEMS} \eqref{ItKeyUEMSNonzero}, which do not contain powers of $t_*$, and the modes appearing in the expansion \eqref{EqKeyESGExpansion}:

\begin{lemma}
\label{LemmaUEMSWithLogs}
  Let $\sigma\in\C$, $\Im\sigma\geq 0$, $\sigma\neq 0$, and suppose
  \begin{equation}
  \label{EqKdSLStabLUEMSMode}
    h(t_*,x)=\sum_{j=0}^k e^{-i\sigma t_*}t_*^j h_j(x),
  \end{equation}
  with $h_j\in\CI(Y,S^2T^*_Y\Omega^\circ)$, is a generalized mode solution of $D_{g_{b_0}}(\Ric+\Lambda)(h)=0$. Then there exists a 1-form $\omega\in\CI(\Omega^\circ,S^2T^*\Omega^\circ)$ such that $h=\delta_{g_{b_0}}^*\omega$. In fact, $\omega$ is a generalized mode as well, $\omega(t_*,x)=\sum_{j=0}^k e^{-i\sigma t_*}t_*^j\omega_j(x)$, with the same $k$, and with $\omega_j\in\CI(Y,S^2 T^*_Y\Omega^\circ)$.
\end{lemma}
\begin{proof}
  For $k=0$ and $l=0$ scalar or $l=1$ vector perturbations, this was already shown above. In the remaining cases, we have already shown this result for $k=0$. Let us now assume that the lemma holds for $k-1\geq 0$ in place of $k$, and let $h$, of the form \eqref{EqKdSLStabLUEMSMode}, be a solution of $D_{g_{b_0}}(\Ric+\Lambda)(h)=0$. Since $g_{b_0}$ is a stationary metric, the operator $D_{g_{b_0}}(\Ric+\Lambda)$ has $t_*$-independent coefficients; therefore $D_{g_{b_0}}(\Ric+\Lambda)(h)=0$ implies that
  \[
    0 = t_*^k D_{g_{b_0}}(\Ric+\Lambda)(e^{-i\sigma t_*}h_k) + \cO(e^{(\Im\sigma)t_*}t_*^{k-1}),
  \]
  hence $D_{g_{b_0}}(\Ric+\Lambda)(e^{-i\sigma t_*}h_k)=0$, and UEMS gives $h_k=\delta_{g_{b_0}}^*\omega$ with $\omega(t_*,x)=e^{-i\sigma t_*}\omega_0(x)$, where $\omega_0\in\CI(Y,T^*_Y\Omega^\circ)$. Let now $h'=h-\delta_{g_{b_0}}^*(t_*^k\omega)$, then we have
  \[
    h' = \sum_{j=0}^{k-1} e^{-i\sigma t_*}t_*^j h'_j(x),
  \]
  i.e.\ the leading term of $h$ involving $t_*^k$ is eliminated, and moreover $D_{g_{b_0}}(\Ric+\Lambda)(h')=0$; by the inductive hypothesis, $h'=\delta_{g_{b_0}}^*\omega'$, with $\omega'$ a generalized mode only including powers of $t_*$ up to $t_*^{k-1}$, and we conclude
  \[
    h = \delta_{g_{b_0}}(t_*^k\omega+\omega'),
  \]
  as claimed.
\end{proof}

\section{Stable constraint propagation (SCP)}
\label{SecSCP}

\emph{We continue to drop the subscript `$b_0$' as in \eqref{EqOpNoB0} and work only with the fixed Schwarzschild--de~Sitter metric $g=g_{b_0}$.}

We now aim to modify $\delta_g^*$ by suitable stationary $0$-th order terms, producing the operator $\tdel^*$, so as to move the resonances of $\delta_g \sfG_g\tdel^*$ into the lower half plane, thus establishing SCP. (Note that the principal symbol of $\delta_g^*$, $\sigma_1(\delta_g^*)(\zeta)=i\zeta\otimes_s(\cdot)$, is independent of the metric; we will later on use the \emph{same} modification $\tdel^*$, irrespective of the metrics we will be dealing with when solving linearized or non-linear gauged Einstein-type equations.)

Using the stationary structure of the spacetime, two natural modifications of $\delta_g^*$ that leave the principal symbol unchanged are the conjugated version $e^{-\gamma t_*}\delta_g^* e^{\gamma t_*}$ and the conformally weighted version $\delta_{e^{-2\gamma t_*}g}^*$, where $\gamma$ is a real parameter. The general form of a linear combination of these two for which the principal symbol is $i\zeta\otimes_s(\cdot)$ takes the form
\begin{equation}
\label{EqSCPDelStarTilde}
  \tdel^*u := \delta_g^*u + \gamma_1\,dt_*\cdot u - \gamma_2(i_{\nabla t_*}u)g,\quad \gamma_1,\gamma_2\in\R.
\end{equation}

We will prove:
\begin{thm}
\label{ThmSCP}
  Let $t_*$ be the timelike function on Schwarzschild--de~Sitter space $(M,g_{b_0})$ constructed in Lemma~\ref{LemmaSdSExt}, and define $\tdel^*$ by \eqref{EqSCPDelStarTilde}. Then there exist parameters $\gamma_1,\gamma_2>0$ and a constant $\alpha>0$ such that all resonances $\sigma$ of the constraint propagation operator $\wtBoxCP_g = 2\delta_g \sfG_g\tdel^*$ satisfy $\Im\sigma<-\alpha$.
\end{thm}

Rather than excluding the presence of resonances in $\Im\sigma\geq-\alpha$ by direct means, for example integration by parts and boundary pairings as in \cite{HintzVasyKdsFormResonances}, we will take $\gamma_1,\gamma_2\sim\semi^{-1}$, with the semiclassical parameter $\semi>0$ small, and study $\semi^2\wtBoxCP_g$ as a semiclassical b-differential operator on the spacetime domain $\Omega$. In principle, a direct computation in the spirit of \S\ref{SecUEMS} would show this directly and provide concrete values of $\gamma_1,\gamma_2$ for which the conclusion holds; we proceed in a more systematic (and less computational) manner, at the marginal cost of not obtaining such concrete bounds. In the simpler setting of de~Sitter space however, we can exactly compute the values of $\gamma_1,\gamma_2$ for which SCP holds; see \S\ref{SubsecdSSCP}.

In \S\ref{SubsecSCPSemi}, we discuss the semiclassical reformulation in some detail, preparing the high frequency analysis of \S\ref{SubsecSCPHi} and the low frequency analysis of \S\ref{SubsecSCPLo}. For simplicity, we use standard energy estimates beyond the horizons, proved in \S\ref{SubsecSCPEn}, to cap off the global estimates; we will obtain the latter in \S\ref{SubsecSCPGl}, and use them to finish the proof of Theorem~\ref{ThmSCP}.

We give a brief overview of the proof of Theorem~\ref{ThmSCP}. A good model operator to think about (in place of the more complicated operator of interest, $\semi^2\wtBoxCP_g$, which is discussed in \S\ref{SubsecSCPSemi}) is the scalar operator $P_\semi=\semi^2\Box_g-i L_\semi$, where $g$ is a Schwarzschild--de~Sitter metric and $-i L_\semi=-\semi\,i_{\nabla t_*}d u$ is a first order semiclassical differential operator.
\begin{itemize}
  \item For (semiclassical) microlocal propagation estimates (away from the artificial boundaries beyond the horizons), we need to determine the semiclassical characteristic set. Lemma~\ref{LemmaSCPSemiEll} shows that this is the union of the zero section and the intersection of fiber infinity with the light cone.
  \item \S\ref{SubsecSCPHi} deals with fiber infinity. The analysis there is similar to that of \cite{VasyMicroKerrdS}; the operator $-i L_\semi$ has a favorable sign and can thus be treated as a complex absorbing operator. (While $L_\semi$ is lower order in the sense of differential orders, it is a principal term in the semiclassical sense, i.e.\ it does affect the overall principal symbol at finite semiclassical frequencies.) Concretely, Proposition~\ref{PropSCPHiRealPrType} gives real principal type propagation of singularities, Proposition~\ref{PropSCPHiRad} provides the radial point estimate, and Proposition~\ref{PropSCPHiTrap} gives the estimate at the trapped set. Because of the extra damping $-i L_\semi$, these estimates (in particular the one at the trapping) work in \emph{all} weighted b-Sobolev spaces, including those with very fast exponential decay. This is different from the operator $\semi^2\Box_g$ where estimates only work on spaces with exponential decay rate $\rho<\alpha$ where $\alpha>0$ is the size of the essential spectral gap.
  \item \S\ref{SubsecSCPLo} deals with the zero section of the b-cotangent bundle. Here we use the operator $i P_\semi$ to deduce propagation estimates: $L_\semi$ is its real part, and $i\semi^2\Box_g$ becomes the complex absorption. The Hamiltonian flow of the principal symbol of $L_\semi$ is the symplectic lift of the gradient flow of $t_*$. There is a critical set for this flow at the zero section intersected with the boundary at future infinity, located at $r=r_c=\sqrt[3]{3\bhm/\Lambda}$ (see Lemma~\ref{LemmaSdSExt}) and thus between the two horizons. All Hamiltonian trajectories either tend to this critical set or escape to $t_*=0$ in the backward direction. This is used to set up the propagation of singularities: by Proposition~\ref{PropSCPLoPropRad}, one can reduce to a neighborhood of the critical set; in this neighborhood then, we use a radial point type estimate, Proposition~\ref{PropSCPLoPropTime}. The latter is the only part of the proof which restricts the amount of exponential decay to $\rho<\alpha$, where $\alpha>0$ is explicitly given in Remark~\ref{RmkSCPLoPropTimeWeight}.
  \item Beyond the horizons, we use standard energy estimates, proved in \S\ref{SubsecSCPEn}, which allow us to propagate our estimates all the way up to the artificial hypersurfaces located there.
  \item In \S\ref{SubsecSCPGl}, we show how to combine the results of the preceding sections to obtain global semiclassical estimates on $P_\semi$ on spaces of exponentially decaying functions. This implies the invertibility of $P_\semi$ on such spaces when one fixes $\semi>0$ to be small enough; thus, $P_\semi$ has no resonances in $\Im\sigma>-\alpha$.
\end{itemize}

The actual operator is a \emph{tensorial} operator acting on 1-forms. A considerable technical complication in \S\ref{SubsecSCPLo} is then that the principal symbol of $L_\semi$, defined below, is not a scalar matrix.

\begin{rmk}
  Using definition \eqref{EqSCPDelStarTilde}, the form of $\wtBoxCP_g$
  is the same as
  \cite[Equation~(13)]{GundlachCalabreseHinderMartinConstraintDamping}. (See
  also the paper by Pretorius \cite{PretoriusBinaryBlackHole} for
  impressive numerical results obtained using such techniques.) Thus,
  Theorem~\ref{ThmSCP} rigorously proves that
  \eqref{EqSCPDelStarTilde} leads to constraint damping, justifying in
  the setting of Schwarzschild--de~Sitter spacetimes and its
  (asymptotically) stationary perturbations the heuristic analysis of
  \cite[\S
  III]{GundlachCalabreseHinderMartinConstraintDamping}. Moreover, we
  point out the connection of the discussion of $\tdel^*$ vs.\
  $\delta_g^*$ around \eqref{EqIntroWhyTdel} to numerical
  investigations of Einstein's equations: introducing the modification
  $\tdel^*$ removes the otherwise very sensitive dependence of the
  solution of the gauged Einstein equation \eqref{EqIntroWhyTdel} on
  the constraint equations being satisfied.
\end{rmk}

\subsection{Semiclassical reformulation}
\label{SubsecSCPSemi}

With $\tdel^*$ defined in \eqref{EqSCPDelStarTilde}, we now let $e>0$ and take
\begin{equation}
\label{EqSCPGamma12}
  \gamma_1=\gamma, \quad \gamma_2=\frac{1}{2}e\gamma.
\end{equation}
The constraint propagation operator is thus
\begin{equation}
\label{EqSCPSemiFullOp}
  \wtBoxCP_g = 2\delta_g \sfG_g\delta_g^* + \gamma\bigl(-i_{\nabla t_*}du + (dt_*)\delta_g u + (\Box_g t_*)u - e\,d(i_{\nabla t_*}u)\bigr).
\end{equation}
We view $\semi=\gamma^{-1}$ as a semiclassical parameter; then
\begin{equation}
\label{EqSCPSemiRescaledOp}
  \cP_\semi := \semi^2\wtBoxCP_g
\end{equation}
is a semiclassical b-differential operator. Occasionally, we will indicate the parameter $e$ by writing $\cP_{e,\semi}$. We will show that Theorem~\ref{ThmSCP} follows from \emph{purely symbolic} microlocal arguments; the energy estimates we use for propagation in the $r$-direction beyond the horizons are rather crude, and also symbolic (albeit for differential operators) in character. (We will sketch an alternative, completely microlocal symbolic argument using complex absorption in Remark~\ref{RmkSCPSusp}.) The relevant operator algebra is the semiclassical b-algebra, which is described in Appendix~\ref{SubsecBSemi}.

In order to analyze $\cP_\semi\in\Diffbh^2(M)$, we first calculate the form of the second term in \eqref{EqSCPSemiFullOp}. \emph{We now work with the coordinate $t_*$ of Lemma~\ref{LemmaSdSExt}}, which simplifies computations significantly: the metric $g$ and the dual metric $G$ take the form
\begin{equation}
\label{EqSCPSemiMetric}
\begin{gathered}
  g = \mu\,dt_*^2 - 2\nu\,dt_*\,dr - c^2\,dr^2 - r^2\slg, \\
  G = c^2\,\pa_{t_*}^2 - 2\nu\,\pa_{t_*}\pa_r - \mu\,\pa_r^2 - r^{-2}\slG,
\end{gathered}
\end{equation}
where $c=c_{t_*}$ in the notation of Lemma~\ref{LemmaSdSExt}, and
\begin{equation}
\label{EqSCPSemiNu}
  \nu=\mp\sqrt{1-c^2\mu},\quad \pm(r-r_c)>0,
\end{equation}
with $r_c=\sqrt[3]{3\bhm/\Lambda}$ given in Lemma~\ref{LemmaSdSExt}. See Figure~\ref{FigSCPSemiGraphNu}.

\begin{figure}[!ht]
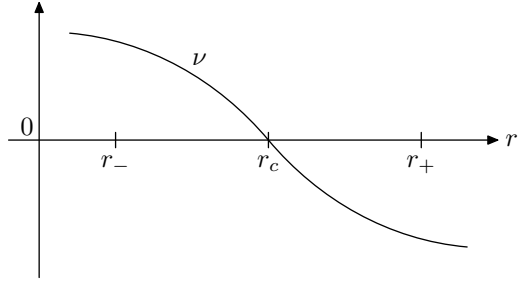

  \centering
  \inclfig{SCPGraphNu}
  \caption{The graph of the function $\nu$ defined in \eqref{EqSCPSemiNu}.}
  \label{FigSCPSemiGraphNu}
\end{figure}

We will also make use of the boundary defining function $\tau=e^{-t_*}$ of $M$.

Using the identities
\[
  \nu^2+c^2\mu = 1,\quad 2\nu\nu'+c^2\mu' = 0,
\]
one can then compute the connection coefficients and verify:
\begin{lemma}
\label{LemmaSCPSemiConn}
  For $w\in\CI(M^\circ,T^*\Sph^2)$ and $v\in\CI(M^\circ,T\Sph^2)$, we have
  \begin{gather*}
    \nabla_{\pa_{t_*}}dt_* = -\frac{1}{2}\nu\mu' dt_* - \frac{1}{2}c^2\mu' dr, \  \nabla_{\pa_{t_*}}dr = -\frac{1}{2}\mu\mu' dt_* + \frac{1}{2}\nu\mu' dr, \  \nabla_{\pa_{t_*}}w = \pa_{t_*}w, \\
    \nabla_{\pa_r} dt_* = -\frac{1}{2}c^2\mu' dt_* + c^2\nu' dr,               \  \nabla_{\pa_r}dr = \frac{1}{2}\nu\mu' dt_* - \nu\nu' dr,                 \  \nabla_{\pa_r}w = \pa_r w - r^{-1}w, \\
    \nabla_v dt_* = r\nu i_v\slg,                                              \  \nabla_v dr = r\mu i_v\slg,                                              \  \nabla_v w = \slnabla_v w - r^{-1}w(v) dr.
  \end{gather*}
\end{lemma}

Splitting the bundle
\begin{equation}
\label{EqSCPSemiBundle}
  T^*M^\circ = \la dt_* \ra \oplus \la dr \ra \oplus T^*\Sph^2
\end{equation}
and using $\nabla t_*=c^2\pa_{t_*}-\nu\pa_r$, we then calculate:

\begin{lemma}
  In the bundle decomposition \eqref{EqSCPSemiBundle}, we have
  \[
    -i_{\nabla t_*}d
      = \begin{pmatrix}
          \nu\pa_r & -\nu\pa_{t_*} & 0 \\
          c^2\pa_r & -c^2\pa_{t_*} & 0 \\
          c^2\sld  & -\nu\sld      & -c^2\pa_{t_*}+\nu\pa_r
        \end{pmatrix},
  \]
  further
  \[
    (dt_*)\delta_g
      = \openbigpmatrix{2pt}
          -c^2\pa_{t_*}+\nu\pa_r & \nu\pa_{t_*}+\mu\pa_r & -r^{-2}\sldelta \\
          0                      & 0                     & 0               \\
          0                      & 0                     & 0
        \closebigpmatrix
      + \openbigpmatrix{2pt}
          \nu'+2r^{-1}\nu & \mu'+2r^{-1}\mu & 0 \\
          0               & 0               & 0 \\
          0               & 0               & 0
        \closebigpmatrix,
  \]
  and
  \[
    d(i_{\nabla t_*}(\cdot))
     = \begin{pmatrix}
         c^2\pa_{t_*} & -\nu\pa_{t_*}  & 0 \\
         c^2\pa_r     & -\nu\pa_r-\nu' & 0 \\
         c^2\sld      & -\nu\sld       & 0
       \end{pmatrix}.
  \]
  Lastly, we have
  \[
    \Box_g t_* = \nu' + 2r^{-1}\nu.
  \]
\end{lemma}

The part of $\cP_\semi$ corresponding to the second term in \eqref{EqSCPSemiFullOp} is given by
\begin{align*}
  -i L_\semi u &:= -i_{\nabla t_*}\semi d u + (dt_*)\semi\delta_g u + \semi(\Box_g t_*)u - e\,\semi d(i_{\nabla t_*}u) \\
    &= \semi\begin{pmatrix}
         -(1+e)c^2\pa_{t_*}+2\nu\pa_r & e\nu\pa_{t_*}+\mu\pa_r  & -r^{-2}\sldelta \\
         (1-e)c^2\pa_r                & -c^2\pa_{t_*}+e\nu\pa_r & 0               \\
         (1-e)c^2\sld                 & (e-1)\nu\sld            & -c^2\pa_{t_*}+\nu\pa_r
       \end{pmatrix} \\
    &\qquad
      + \semi\begin{pmatrix}
          2\nu'+4r^{-1}\nu & \mu'+2r^{-1}\mu      & 0 \\
          0                & (1+e)\nu'+2r^{-1}\nu & 0 \\
          0                & 0                    & \nu'+2r^{-1}\nu
        \end{pmatrix}.
\end{align*}
Thus, $L_\semi\in\Diffbh^1(M;\Tb^*M)$; the first term here is semiclassically principal, the second one subprincipal due to the extra factor of $\semi$. We will sometimes indicate the parameter $e$ by writing $L_{e,\semi}$. We denote the principal symbol of $L_{e,\semi}$ by
\[
  \ell_e=\sigmabh(L_{e,\semi}).
\]
We decompose $L_\semi$ in the coordinates $(t_*,r,\omega)$ as
\begin{equation}
\label{EqSCPSemiLSplit}
  L_\semi = M_{t_*} \semi D_{t_*} + M_r \semi D_r + M_\omega + i \semi S,
\end{equation}
with $M_\omega$ capturing the $\sld$ and $\sldelta$ components, and  $S$ the subprincipal term (not containing differentiations); the bundle endomorphisms $M_{t_*}$, $M_r$ and $S$ of $\Tb^*M$ have real coefficients.

Since
\[
  \cP_\semi=\BoxCP_{g,\semi} - i L_\semi,\quad \BoxCP_{g,\semi}:=\semi^2\BoxCP_g = \semi^2\delta_g \sfG_g\delta_g^*,
\]
see \eqref{EqHypDTCP}, we see directly that $\cP_\semi$ is \emph{not} principally scalar due to the non-scalar nature of $L_\semi$; of course, the principal part in the sense of differentiable order is scalar and equal to $G\Id$, with $G$ the dual metric function.

\begin{lemma}
\label{LemmaSCPSemiEll}
  For $e\leq 1$ sufficiently close to $1$, the semiclassical principal symbol of $\cP_{e,\semi}$ is elliptic in $\ol{\Tb^*}M\setminus\bigl(o\cup(\Sigma\cap\Sb^*M)\bigr)$, where $\Sigma=G^{-1}(0)\subset\ol{\Tb^*}M$ is the characteristic set, and $\Sb^*M$ is the boundary of $\ol{\Tb^*}M$ at fiber infinity.
\end{lemma}
\begin{proof}
  At a point $\zeta\in\Sb^*M$, the ellipticity of $\cP_\semi$ is equivalent to $G(\zeta)\neq 0$, proving $\Ellbh(\cP_\semi)\cap\Sb^*M=\Sb^*M\setminus\Sigma$.

  Next, working away from fiber infinity, we note that $\sigmabh(\BoxCP_{g,\semi})=G$ is real, while $\sigmabh(i L_\semi)$ has purely imaginary eigenvalues; this follows from~\eqref{EqSCPSemiEllSymbol} below for $e=1$, and either by direct computation or from Lemma~\ref{LemmaSCPLoSymm} below for $0<e<1$. We thus only need to show that $\sigmabh(L_\semi)$ is elliptic on the part of the light cone $\Sigma\cap(\Tb^*M\setminus o)$ away from fiber infinity and the zero section. Moreover, $\ell_e$, as a smooth section of $\pi^*\End(\Tb^*M)$ over $\Tb^*M$ (with $\pi\colon\Tb^*M\to M$), is homogeneous of degree $1$ in the fibers of $\Tb^*M$; since $\Sigma\cap\Sb^*M$ is precompact (the non-compactness only being due to the boundaries $r=r_\pm\pm 3\eps_M$ being excluded in the definition of $M$), the set of $e$ for which $\ell_e$ is elliptic on $\Sigma\cap(\Tb^*M\setminus o)$ is therefore open. Thus, it suffices to prove the lemma for $e=1$. Writing b-covectors as
  \begin{equation}
  \label{EqSCPSemiEllCoord}
    \zeta = -\sigma\,dt_* + \xi\,dr + \eta,\quad \eta\in T^*\Sph^2,
  \end{equation}
  we have
  \begin{equation}
  \label{EqSCPSemiEllSymbol}
    \ell_1
     = \begin{pmatrix}
          -2(c^2\sigma+\nu\xi) & \nu\sigma-\mu\xi    & -r^{-2}i_\eta \\
          0                    & -(c^2\sigma+\nu\xi) & 0             \\
          0                    & 0                   & -(c^2\sigma+\nu\xi)
        \end{pmatrix}.
  \end{equation}
  But $-(c^2\sigma+\nu\xi)=G(d t_*,\zeta)$, which by the timelike nature of $dt_*$ is non-zero for $\zeta\in\Sigma\cap(\Tb^*M\setminus o)$.
\end{proof}

The last part of the proof shows:

\begin{cor}
\label{CorSCPSemiEllSymbol}
  For $e$ close to $1$, the symbol $\sigmabh(L_{e,\semi})$ is elliptic in the causal double cone $\{\zeta\in\Tb^*M\setminus o\colon G(\zeta)\geq 0\}$.
\end{cor}

Schematically, the principal symbol of $\cP_{1,\semi}$ is
\[
  G - i\begin{pmatrix}2\nabla t_* & * & * \\ 0 & \nabla t_* & 0 \\ 0 & 0 & \nabla t_* \end{pmatrix},
\]
viewing $\nabla t_*$ as a linear function on the fibers of $\Tb^*M$. Since in a conical neighborhood of the two components $\Sigma^\pm$ of the light cone, see \eqref{EqKdSGeoCharPM}, we have $\pm\nabla t_*>0$, the imaginary part has the required sign for real principal type propagation of regularity along the (rescaled) null-geodesic flow generated by $\rham_G$ in the forward, resp.\ backward, direction within $\Sigma^+\cap\Sb^*M$, resp.\ $\Sigma^-\cap\Sb^*M$; see \S\ref{SubsecSCPHi} for details.

Near the zero section, one would like to think of $L_{1,\semi}$ as a coupled system of first order ordinary differential operators, transporting energy along the orbits of the vector field $\nabla t_*$. See Figure~\ref{FigSCPNablaT}.

\begin{figure}[!ht]
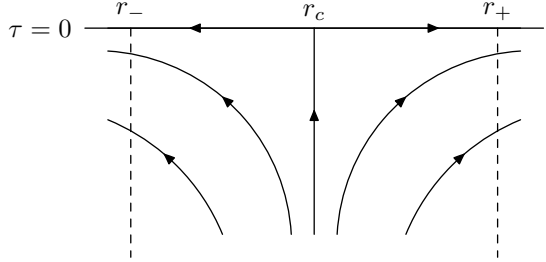

  \centering
  \inclfig{SCPNablaT}
  \caption{The flow of the vector field $\nabla t_*$, including at the boundary at infinity.}
  \label{FigSCPNablaT}
\end{figure}

\begin{rmk}
\label{RmkSCPLCP}
  We explain the structure of $L_{e,\semi}$ in a bit more detail. We first note that one can check that the conclusion of Lemma~\ref{LemmaSCPSemiEll} in fact holds for all $e>0$, but fails for $e=0$. Indeed, for $e=0$, one can easily compute the eigenvalues of $\ell_0(\zeta)$, $\zeta\in\Tb^*M\setminus o$ as in \eqref{EqSCPSemiEllCoord}, to be $-(c^2\sigma+\xi\nu)$ with 2-dimensional eigenspace,\footnote{This is independent of $e$ and can be computed explicitly: for $\eta\neq 0$ and $\xi\neq 0$, the eigenspace is $\eta^\perp\oplus\la \nu|\eta|^2\,dt_*+c^2|\eta|^2\,dr-r^2\xi\eta\ra$; if $\eta\neq 0$ and $\xi=0$, it is $\eta^\perp\oplus\la\nu\,dt_*+c^2\,dr\ra$; and if $\eta=0$, it is $\eta^\perp=T^*\Sph^2$.} and $c^2\sigma+\xi\nu\pm\sqrt{\xi^2+c^2 r^{-2}|\eta|^2}$ with 1-dimensional eigenspaces, respectively; the vanishing of any one of the latter two eigenvalues implies $G=0$. Thus, for $e=0$, $L_\semi$ can roughly be thought of as transporting energy in phase space along the flow of $\nabla t_*$ for a corank $2$ part of $L_\semi$, and along the flows of two other vector fields, whose projections to $M$ are \emph{null}, on two rank $1$ parts. As soon as $0<e<1$ however, the projections to $M$ of these two vector fields become future timelike, as we will discuss in Lemma~\ref{LemmaSCPLoTimelike} and Remark~\ref{RmkSCPLo2nd} below, and $\ell_e$ is still diagonalizable due to Lemma~\ref{LemmaSCPLoSymm} below; for $e=1$ on the other hand, all \emph{three} vector fields coincide (up to positive scalars), with projection to $M$ equal to a positive multiple of $\nabla t_*$, but $\ell_1$ is no longer diagonalizable as it is nilpotent with non-trivial Jordan block structure when evaluated on $(dt_*)^\perp\setminus o\subset\Tb^*M\setminus o$.

  In summary, ellipticity considerations force us to use $e>0$, while the structure of $L_{e,\semi}$ is simplest for $e=1$, with the technical caveat of non-diagonalizability, which disappears for $e<1$; this is the reason for us to work with $e<1$ close to $1$.
\end{rmk}

\subsection{High frequency analysis}
\label{SubsecSCPHi}

We now analyze $\cP_\semi$ near fiber infinity $\pa\ol{\Tb^*}M$. By Lemma~\ref{LemmaSCPSemiEll}, we only need to study the propagation of regularity within $\pa\Sigma=G^{-1}(0)\cap\Sb^*M$ along the flow of the rescaled Hamilton vector field
\[
  \rham_G=\wh\rho\ham_G\in\Vb(\Sb^*M),
\]
see \eqref{EqKdSGeoHamRescaled}; here $\wh\rho$ is a boundary defining function of fiber infinity $\Sb^*M\subset\ol{\Tb^*}M$. Note that $\ham_G$ is equal to the Hamilton vector field of the real part of $\sigmabh(\cP_\semi)$.

Let $\ell=\sigmabh(L_{1,\semi})$, so the eigenvalues of $\pm\wh\rho\ell$ are positive near $\pa\Sigma^\pm$. In order to prove the propagation of regularity (i.e.\ of estimates) in the future direction, as explained after the proof of Lemma~\ref{LemmaSCPSemiEll}, we need to choose a positive definite inner product on $\pi^*\Tb^*M$ (with $\pi\colon\pa\Sigma\to M$ the projection) so that $\pm\frac{\wh\rho}{2}(\ell+\ell^*)$ is positive at (and hence near) $\pa\Sigma^\pm$ in the sense of self-adjoint endomorphisms; see also \cite[Proposition~3.12]{HintzPsdoInner}. We phrase this in a more direct way in Lemma~\ref{LemmaSCPHiInner} below. For real principal type propagation estimates, we only need to arrange this locally in $\pa\Sigma$, but we can in fact arrange it globally, which will make the estimates at radial points and at the trapped set straightforward.

\begin{lemma}
\label{LemmaSCPHiInner}
  Fix the positive definite inner product
  \begin{equation}
  \label{EqSCPHiInnerRiem}
    g_R = dt_*^2 + dr^2 + \slg
  \end{equation}
  on $\Tb^*M$, used to define adjoints below. Then there exists a pseudodifferential operator $Q\in\Psibh^0(M;\Tb^*M)$ which is elliptic near $\pa\Sigma$, with microlocal inverse $Q^-$, so that
  \begin{equation}
  \label{EqSCPHiInnerEst}
    \pm\wh\rho \sigmabh\Bigl(\frac{1}{2i}\bigl(Q\cP_{e,\semi}Q^- - (Q\cP_{e,\semi}Q^-)^*\bigr)\Bigr) > 0
  \end{equation}
  holds near $\pa\Sigma^\pm$ for all $e$ close to $1$.
\end{lemma}
\begin{proof}
  Since $\BoxCP_{g,\semi}$ has a real scalar principal symbol, it does not contribute to \eqref{EqSCPHiInnerEst}; thus, we need to arrange \eqref{EqSCPHiInnerEst} for $-i L_{e,\semi}$ in place of $\cP_{e,\semi}$. Moreover, if \eqref{EqSCPHiInnerEst} holds for $e=1$ for some choice of $Q$, then one can take the same $Q$ for $e$ close to $1$ by compactness considerations. Explicitly then, we may diagonalize the symbol of $L_{1,\semi}$, given in \eqref{EqSCPSemiEllSymbol}, near $\pa\Sigma^\pm$ by conjugating it by an endomorphism-valued zeroth order symbol $q$; quantizing $q$ gives an operator $Q$ with the desired properties. Concretely, near $\pa\Sigma$, we may take $\wh\rho=|c^2\sigma+\nu\xi|^{-1}$, and then
  \[
    q = \begin{pmatrix}
          1 & \mp\wh\rho(\nu\sigma-\mu\xi) & \pm\wh\rho r^{-2}i_\eta \\
          0 & 1                            & 0                       \\
          0 & 0                            & 1
        \end{pmatrix}.
  \]
  The proof is complete.
\end{proof}

This immediately gives the propagation of regularity/estimates:

\begin{prop}
\label{PropSCPHiRealPrType}
  Suppose $B_1,B_2,S\in\Psibh^0(M;\Tb^*M)$ are operators with wave front set disjoint from $\pa\Sigma^-$, resp.\ $\pa\Sigma^+$, as well as from the zero section $o\subset\ol{\Tb^*}M$. Suppose also that all backward, resp.\ forward, null-bicharacteristics of $\rham_G$ from $\WFbh'(B_2)\cap\Sb^*M$ reach $\Ellbh(B_1)$ in finite time while remaining in $\Ellbh(S)$, and with $S$ elliptic on $\WFbh'(B_2)$. Then for all $s,\rho\in\R$ and $N\in\R$, there exists $\semi_0>0$ such that one has
  \begin{equation}
  \label{EqSCPHiRealPrTypeEst}
    \|B_2 u\|_{\Hbh^{s,\rho}} \lesssim \|B_1 u\|_{\Hbh^{s,\rho}} + \semi^{-1}\|S\cP_\semi u\|_{\Hbh^{s-1,\rho}} + \semi^N\|u\|_{\Hbh^{-N,\rho}}, \quad 0<\semi<\semi_0,
  \end{equation}
  in the strong sense that if the norms on the right are finite, then so is the norm on the left, and the estimate holds. The same holds if one replaces $\cP_\semi$ by its adjoint $\cP_\semi^*$ (with respect to any non-degenerate fiber inner product) and interchanges `backward' and `forward.' See Figure~\ref{FigSCPHiRealPrType}. 
\end{prop}

This propagates $\Hbh^{s,\rho}$-control of $u$ in the forward, resp.\ backwards, direction along the $\rham_G$-flow in $\pa\Sigma^+$, resp.\ $\pa\Sigma^-$.

\begin{figure}[!ht]
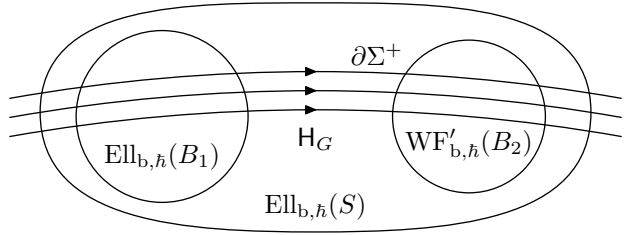

  \centering
  \inclfig{SCPHiRealPrType}
  \caption{Propagation of regularity in the future direction within $\pa\Sigma^+$: we can propagate microlocal control from $\Ellbh(B_1)$ to $\Ellbh(B_2)$. In $\pa\Sigma^-$, the arrows are reversed, corresponding to propagation in the backwards direction along $\rham_G$ (which is still the future direction in $\pa\Sigma$).}
  \label{FigSCPHiRealPrType}
\end{figure}

\begin{proof}[Proof of Proposition~\ref{PropSCPHiRealPrType}]
  First of all, with $Q$ as in Lemma~\ref{LemmaSCPHiInner}, we can write $Q\cP_\semi u=Q\cP_\semi Q^-(Q u)+R u$ with $\WFbh'(R)\cap\WFbh'(S)=\emptyset$; from this one easily sees that it suffices to prove the proposition for
  \begin{equation}
  \label{EqSCPHiRealPrTypePprime}
    \cP'_\semi:=Q\cP_\semi Q^-
  \end{equation}
  in place of $\cP_\semi$. This then follows from a standard positive commutator argument, considering
  \begin{equation}
  \label{EqSCPHiRealPrTypeComm}
  \begin{split}
    i\semi^{-1}(\la A u, A\cP'_\semi u\ra - \la A\cP'_\semi u,A u\ra) &= \la i\semi^{-1}[\cP'_\semi,A^2]u,u\ra \\
     &\quad - \la (i\semi)^{-1}(\cP'_\semi-(\cP'_\semi)^*)A^2 u,u\ra
  \end{split}
  \end{equation}
  for a suitable principally scalar commutant $A=A^*\in\tau^{-2r}\Psibh^{s-1/2}(M;\Tb^*M)$ with wave front set contained in a small neighborhood of $\pa\Sigma^\pm$, quantizing a non-negative symbol $a$ whose $\ham_G$-derivative has a sign on $\Ellbh(B_2)$, with an error term with the opposite sign on $\Ellbh(B_1)$. For example, in the case of $\pa\Sigma^+$, the term coming from the commutator can be written $\sigmabh(i\semi^{-1}[\cP'_\semi,A^2])=\ham_G a^2=-b_2^2+b_1^2$, while by the previous lemma, the skew-adjoint part gives a contribution (in fact of order $\semi^{-1}$ rather than $1$ due to the strict positivity in~\eqref{EqSCPHiInnerEst}) with the \emph{same sign} as the term $-b_2^2$; thus, this contribution can be dropped from the estimate.
  
  More concretely, quantizing $b_1$ and $b_2$, one can evaluate the right hand side of~\eqref{EqSCPHiRealPrTypeComm}; the a priori control term $\|B_1 u\|_{\Hbh^{s,\rho}}$, resp.\ conclusion term $\|B_2 u\|_{\Hbh^{s,\rho}}$, in the estimate~\eqref{EqSCPHiRealPrTypeEst} arises from the quantization of $b_1^2$, resp.\ $b_2^2$. The left hand of~\eqref{EqSCPHiRealPrTypeComm} is estimated by $\|A u\|^2+\semi^{-2}\|A\cP'_\semi u\|$, and the first term can be absorbed into the $\|B_2 u\|_{\Hbh^{s,\rho}}$ if one chooses $a$ carefully. We refer the reader to \cite[\S2]{DeHoopUhlmannVasyDiffraction} and \cite[\S2.5]{VasyMicroKerrdS} for a more detailed discussion and further references.
\end{proof}

\begin{rmk}
  Using the strict positivity of the contribution of the skew-adjoint part in~\eqref{EqSCPHiRealPrTypeComm}, the estimate~\eqref{EqSCPHiRealPrTypeEst} can be improved, to wit
  \begin{equation}
  \label{EqSCPHiRealPrTypeEstImproved}
    \|B_2 u\|_{\Hbh^{s,\rho}} \lesssim \semi^{1/2}\|B_1 u\|_{\Hbh^{s,\rho}} + \|S\cP_\semi u\|_{\Hbh^{s-1,\rho}} + \semi^N \|u\|_{\Hbh^{-N,\rho}},\quad 0<\semi<\semi_0.
  \end{equation}
  Indeed, the right hand side of~\eqref{EqSCPHiRealPrTypeComm} controls $\semi^{-1}\|A u\|^2$ (rather than merely $\|B_2 u\|^2$) due to the presence of the second term, provided one has a priori control of $\|B_1 u\|^2$; the left hand side on the other hand can be estimated by $\eps\semi^{-1}\|A u\|^2 + C_\eps\semi^{-1}\|A\cP'_\semi u\|^2$, the first term of which can be absorbed, giving~\eqref{EqSCPHiRealPrTypeEstImproved}. The same improvement can be made in the statements of Propositions~\ref{PropSCPHiRad} and \ref{PropSCPHiTrap} below.
\end{rmk}

Likewise, we have microlocal estimates for the propagation near the radial set $\pa\cR$:

\begin{prop}
\label{PropSCPHiRad}
  Suppose the wave front sets of $B_1,B_2,S\in\Psibh^0(M;\Tb^*M)$ are contained in a small neighborhood of $\pa\Sigma^+$, resp.\ $\pa\Sigma^-$, with $B_2$ elliptic at $\pa\cR^+_\pm$, resp.\ $\pa\cR^-_\pm$ (see \S\ref{SubsecKdSGeo}), and so that all backward, resp.\ forward, null-bicharacteristics from $\WFbh'(B_2)\cap\Sb^*M$ either tend to $\pa\cR^+_\pm$, resp.\ $\pa\cR^-_\pm$, or enter $\Ellbh(B_1)$ in finite time, while remaining in $\Ellbh(S)$; assume further that $S$ is elliptic on $\WFbh'(B_2)$. Then for all $s,\rho\in\R$ and $N\in\R$, there exists $\semi_0>0$ such that the estimate \eqref{EqSCPHiRealPrTypeEst} holds for $0<\semi<\semi_0$. See Figure~\ref{FigSCPHiRad}. 

  The same estimate holds if one replaces $\cP_\semi$ by its adjoint and interchanges `backward' and `forward.'
\end{prop}

\begin{figure}[!ht]
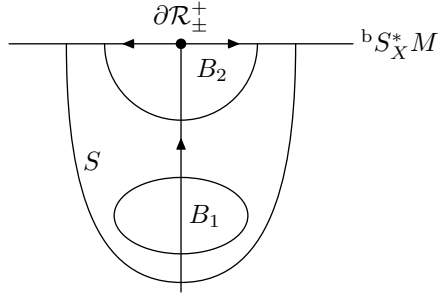

  \centering
  \inclfig{SCPHiRad}
  \caption{Propagation of singularities near a component $\pa\cR^+_\pm$ of the radial set within $\pa\Sigma^+$: we can propagate microlocal control from $\Ellbh(B_1)$ to $\Ellbh(B_2)$, in particular into the boundary, along the forward direction of the $\rham_G$-flow. In $\pa\Sigma^-$, we can propagate into the boundary as well by propagating backwards along the $\rham_G$-flow.}
  \label{FigSCPHiRad}
\end{figure}

\begin{proof}[Proof of Proposition~\ref{PropSCPHiRad}]
  This follows again from a positive commutator argument, see \cite[Proposition~2.1]{HintzVasySemilinear}. We may replace $\cP_\semi$ by $\cP'_\semi$ defined in \eqref{EqSCPHiRealPrTypePprime}. Using the non-negativity of $(i\semi)^{-1}(\cP'_\semi-(\cP'_\semi)^*)$ near $\pa\Sigma^+$, we get an estimate under a threshold condition on the regularity $s$ relative to the weight $\rho$ of the form $s-\beta\rho-1/2>0$, with $\beta=\beta_\pm$; but the strict positivity of $(i\semi)^{-1}(\cP'_\semi-(\cP'_\semi)^*)$ gives an extra contribution of size $\semi^{-1}$, so the main term in the positive commutator argument is (up to an overall sign) $\geq s-\beta \rho-1/2+\delta \semi^{-1}$ for some (small) $\delta>0$, which is $\geq\delta/2$ for small $\semi$. This explains why the radial point estimate does \emph{not} require any above-threshold assumptions which appear in the estimate given in the reference.
\end{proof}

Finally, at the trapping, we have:

\begin{prop}
\label{PropSCPHiTrap}
  There exist operators $B_1,B_2,S\in\Psibh^0(M;\Tb^*M)$, with $B_2$ and $S$ elliptic near $\pa\Gamma^+$, resp.\ $\pa\Gamma^-$, and $\WFbh'(B_1)\cap\pa\Gamma^+_\bw=\emptyset$, resp.\ $\WFbh'(B_1)\cap\pa\Gamma^-_\fw=\emptyset$, where $\Gamma^+_\bw\subset\Tb^*_X M$ is the backward trapped  set (with respect to the $\ham_G$ flow) in $\Sigma^+$ and $\Gamma^-_\fw\subset\Tb^*_X M$ is the forward trapped set in $\Sigma^-$, so that for all $s,\rho\in\R$ and $N\in\R$, there exists $\semi_0>0$ such that the estimate \eqref{EqSCPHiRealPrTypeEst} holds for $0<\semi<\semi_0$. See Figure~\ref{FigSCPHiTrap}.

  The same estimate holds if one replaces $\cP_\semi$ by its adjoint, now with $\WFbh'(B_1)\cap\pa\Gamma^+_\fw=\emptyset$, resp.\ $\WFbh'(B_1)\cap\pa\Gamma^-_\bw=\emptyset$, where $\Gamma^+_\fw\subset\Tb^*M$ is the forward trapped set in $\Sigma^+$ and $\Gamma^-_\bw\in\Tb^*M$ the backward trapped set in $\Sigma^-$.
\end{prop}

\begin{figure}[!ht]
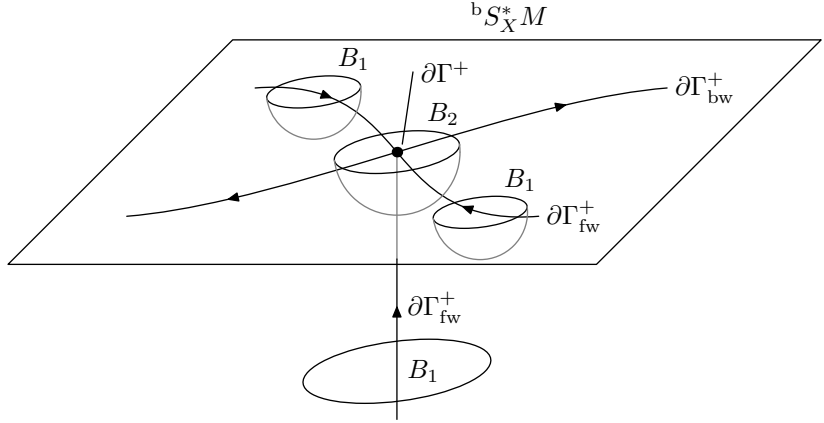

  \centering
  \inclfig{SCPHiTrap}
  \caption{Propagation of singularities near the trapped set $\pa\Gamma^+$: we can propagate microlocal control from $\Ellbh(B_1)$ to $\Ellbh(B_2)$, hence from the forward trapped set $\pa\Gamma^+_\fw$ into the trapped set $\pa\Gamma^+$ (and from there into the backward trapped set $\pa\Gamma^+_\bw$), along the forward direction of the $\rham_G$-flow. For the corresponding propagation result in $\pa\Gamma^-$, the subscripts `bw' and `fw' are interchanged.}
  \label{FigSCPHiTrap}
\end{figure}

\begin{proof}[Proof of Proposition~\ref{PropSCPHiTrap}]
  This follows from a straightforward adaption of the proof of the last part of \cite[Theorem~3.2]{HintzVasyNormHyp}; again the strict positivity of $(i\semi)^{-1}(\cP'_\semi-(\cP'_\semi)^*)$, and the fact that it has size $\semi^{-1}$ rather than $\cO(1)$ (which is the size of the main term of the commutator argument, where the main term of the commutant is $\tau^{-2 r}$, whose $\rham_G$-derivative has a sign near $\Gamma^\pm$), shifts the threshold weight ($r=0$ in the reference) to $C$ for any fixed $C>0$ provided $\semi>0$ is sufficiently small. This is the reason for the trapping estimate holding for \emph{all} weighted spaces, unlike in the reference.
\end{proof}

Combining these propagation results with elliptic regularity in $\Sb^*M\setminus\pa\Sigma$ and using the global structure of the null-geodesic flow, see \cite[\S6]{VasyMicroKerrdS} or \cite[\S2]{HintzVasyCauchyHorizon}, we can thus control $u$ microlocally near fiber infinity $\Sb^*_U M$ over any open set $U\Subset\Omega$, provided we have a priori control of $u$ at $\pa\Sigma$ near the Cauchy surface $\Sigma_0$ and the two spacelike hypersurfaces $[0,1]_\tau\times\pa Y$. (That is, $U$ is disjoint from $\Sigma_0$ and $[0,1]_\tau\times\pa Y$.)

\subsection{Low frequency analysis}
\label{SubsecSCPLo}

We next study the operator $\cP_\semi$ microlocally near the zero section $o\subset\ol{\Tb^*}M$. Since $\sigmabh(\BoxCP_{g,\semi})$ vanishes quadratically at the zero section, the propagation of semiclassical regularity for solutions of $\cP_\semi u=f$ there is driven by $L_\semi$. Let us consider the situation on the symbolic level in a model case first: namely, if $L$ is $i^{-1}$ times a future timelike vector field, then the principal symbol of the scalar operator $i(\Box_g-i L)$ is given by $\ell+i G$, with $\ell=\sigma(L)$. Propagation in the forward direction along $\ham_\ell$ requires a sign condition on $G$ which is not met here. The key is to rewrite
\begin{equation}
\label{EqSCPLoToy}
  \ell+iG = \ell+i C\ell^2 - i(C\ell^2-G)
\end{equation}
for a constant $C$; by the timelike nature of $iL$, we have $\ell>0$ in the future causal cone $\{\zeta\in\Tb^*_p M\colon G(\zeta)\geq 0,\ G(\zeta,dt_*)>0\}$, and therefore, for large $C$, we have $\sigma_2(C\ell^2-G)(\zeta)\geq 0$, $\zeta\in\Tb^*_p M$. (In fact for large $C$, the left hand side is a sum of squares, with positive definite Hessian at $\zeta=0$.) Then we can factor out $(1+i C\ell)$ from \eqref{EqSCPLoToy}, giving $\ell-i(C\ell^2-G)$ up to terms which vanish cubically at the zero section. Algebraically more simply, we can multiply $\sigma_2(L+i\Box_g)$ by $(1-iC\ell)$, obtaining
\begin{equation}
\label{EqSCPLoToy2}
  (1-i C\ell)(\ell+i G) = (\ell+C\ell G) - i(C\ell^2-G),
\end{equation}
for which we have propagation  of singularities in the zero section along the flow of $i L$.

The key lemma allowing us to make this work in a conceptually straightforward manner for the operator at hand, $\cP_\semi$, is the following:

\begin{lemma}
\label{LemmaSCPLoSymm}
  Denoting by $g_R$ the inner product \eqref{EqSCPHiInnerRiem}, the principal symbol $\ell_e=\sigmabh(L_{e,\semi})$ is symmetric with respect to the inner product
  \[
    b_e:=g_R(B_e\cdot,\cdot), \quad
    B_e=\begin{pmatrix}
             c^2  & -\nu                                  & 0 \\
             -\nu & c^{-2}\bigl(\nu^2+\frac{1}{1-e}\bigr) & 0 \\
             0    & 0                                     & \frac{1}{(1-e)r^2}
           \end{pmatrix}.
  \]
\end{lemma}

As before, we shall often drop the subscript `$e$' from the notation, thus simply writing $B\equiv B_e$.

\begin{proof}[Proof of Lemma~\ref{LemmaSCPLoSymm}]
  First of all, $B$ is symmetric and positive definite with respect to $g_R$ for $e\in(0,1)$. Next, and also to prepare subsequent arguments, we simply compute the precise form of $B L_\semi$. Namely, in the decomposition~\eqref{EqSCPSemiLSplit}, we find
  \begin{equation}
  \label{EqSCPLoSymmOp}
    B L_\semi = \wt M_{t_*} \semi D_{t_*} + \wt M_r \semi D_r + \wt M_\omega + i \semi\wt S,
  \end{equation}
  where $\wt M_{t_*}=B M_{t_*}$, $\wt M_r=B M_r$, $\wt M_\omega=B M_\omega$ and $\wt S=B S$; concretely,
  \begin{gather*}
    \wt M_{t_*}=\begin{pmatrix}
                  (1+e)c^4     & -(1+e)c^2\nu             & 0 \\
                  -(1+e)c^2\nu & (1+e)\nu^2+\frac{1}{1-e} & 0 \\
                  0            & 0                        & \frac{c^2}{(1-e)r^2}
                \end{pmatrix}, \\
    \wt M_r=\begin{pmatrix}
              -(1+e)c^2\nu  & -1+(1+e)\nu^2                                     & 0 \\
              -1+(1+e)\nu^2 & -c^{-2}\nu\bigl(\frac{2e-1}{1-e}+(1+e)\nu^2\bigr) & 0 \\
              0             & 0                                                 & -\frac{\nu}{(1-e)r^2}
            \end{pmatrix}
  \end{gather*}
  and
  \[
    \wt M_\omega=i \semi r^{-2}
                 \begin{pmatrix}
                   0       & 0        & -c^2\sldelta \\
                   0       & 0        & \nu\sldelta  \\
                   c^2\sld & -\nu\sld & 0
                 \end{pmatrix}.
  \]
  The lemma is now obvious.
\end{proof}

We occasionally indicate `$e$' explicitly as a subscript for the operators appearing in \eqref{EqSCPLoSymmOp}.

\begin{rmk}
\label{RmkSCPLo2nd}
  While we choose a slightly different route below, the conceptually cleanest and most precise way to proceed would be to use second microlocalization at the zero section in semiclassical b-phase space $\ol{\Tb^*}M$. For $0<e<1$, one can smoothly diagonalize $\ell_e$ on the front face $\ff$ of the blow-up $[\ol{\Tb^*}M;o]$ near the $2$-microlocal characteristic set of $L_{e,\semi}$ (which due to the homogeneity of $\ell_e$ is equivalent to the smooth diagonalizability of $\ell_e$ in a conic neighborhood of the characteristic set of $L_{e,\semi}$ in $\Tb^*M\setminus o$). Indeed, the pointwise diagonalizability follows directly from Lemma~\ref{LemmaSCPLoSymm}; the condition $0<e<1$ then ensures that in $\Tb^*M\setminus o$, the eigenvalue $-(c^2\sigma+\nu\xi)$ of $\ell_e$ discussed in Remark~\ref{RmkSCPLCP} is always different from the two remaining eigenvalues (which must then themselves be distinct, as their average is equal to $-(c^2\sigma+\nu\xi)$ by trace considerations), and the smooth diagonalizability follows easily. Moreover, for $e>0$, one has $G<0$ in the second microlocal characteristic set of $L_\semi$; this is just the statement that $\ell_e$ is elliptic in the causal double cone, away from the zero section. Therefore, one has $2$-microlocal propagation of regularity in the future direction along $\ff$, giving semiclassical Lagrangian regularity for solutions of $\cP_\semi u\in\CIdot$; the propagation into the boundary $X$ at future infinity uses a $2$-microlocal radial point estimate at the critical points of the $2$-microlocal null-bicharacteristics of $L_\semi$, which places a restriction on the weight of the function space at $X$. Finally, to deal with the Lagrangian error, one would use a direct microlocal energy estimate, which we explain below, near the zero section.
\end{rmk}

Following the observation in Lemma~\ref{LemmaSCPLoSymm}, we obtain the following sharpening of Corollary~\ref{CorSCPSemiEllSymbol}:

\begin{lemma}
\label{LemmaSCPLoTimelike}
  For $e<1$ close to $1$, the first order differential operator $L_{e,\semi}$ is future timelike in the sense that $\sigmabh(L_{e,\semi})(\zeta)$ is positive definite with respect to the inner product $b_e$ for all $\zeta\in\Tb^*M\setminus o$ which are future causal, i.e.\ $G(\zeta)\geq 0$, $G(\zeta,dt_*)>0$.
\end{lemma}
\begin{proof}
  The eigenvalues of $\sigmabh(L_{1,\semi})(\zeta)$ are positive for such $\zeta$, and the claim follows by continuity.
\end{proof}

\emph{For the rest of this section, symmetry and positivity will always be understood with respect to $b_e$.}

\begin{cor}
\label{CorSCPLoImagEll}
  There exists a constant $C>0$ such that $C\ell_e^2 - G\Id$ is positive definite on $\Tb^*M\setminus o$; that is, $(C\ell_e^2-G\Id)(\zeta)\in\End(\Tb^*_z M)$ is positive definite for $\zeta\in\Tb^*_z M\setminus o$, $z\in M$.
\end{cor}
\begin{proof}
  Note that $\ell_e^2\geq 0$, and indeed by Lemma~\ref{LemmaSCPLoTimelike} and using a continuity and compactness argument, there exists $\delta>0$ such that $\ell_e(\zeta)^2\geq\delta>0$ holds for all $\zeta$ with $|\zeta|_{g_R}=1$ and $G(\zeta)\geq -\delta$, while on the other hand $G(\zeta)\leq\delta^{-1}$ for all $\zeta$ with $|\zeta|_{g_R}=1$. But then $C\ell_e(\zeta)^2-G(\zeta)\Id\geq\min(C\delta - \delta^{-1},\delta)$ is bounded away from zero for sufficiently large $C>0$.
\end{proof}

The principally symmetric nature of $L_{e,\semi}$ will give rise to a (microlocal) energy estimate with respect to the inner product
\[
  \la u,v\ra = \int b_e(u,v) \,|dg_R|.
\]
Denoting formal adjoints with respect to this inner product by a superscript `$b$' (suppressing the dependence on $e$ here), so $Q^{*b}=B^{-1}Q^*B$ for an operator $Q$ on $\CIdotc(M;\Tb^*M)$, where $Q^*$ is the adjoint with respect to $g_R$. \emph{From now on, all adjoints are taken with respect to the fiber inner product $b$.} The energy estimate is then based on the commutator
\[
  \frac{i}{\semi}\bigl(\la A u, A L_\semi u \ra - \la A L_\semi u, A u \ra\bigr) = \Big\la \frac{i}{\semi}[L_\semi,A^2]u,u\Big\ra + \Big\la \frac{1}{i\semi}(L_\semi-L_\semi^*)A^2 u,u\Big\ra
\]
with a principally scalar commutant $A=A^*\in\Psibh(M;\Tb^*M)$, microlocalized near $o\subset\Tb^*M$, which we construct below. Roughly, near the critical set $\{\tau=0,r=r_c\}$ of the vector field $\nabla t_*$, we will take $A=\psi_0(t_*)\psi_1(r)$, composed with pseudodifferential cutoffs localizing near the zero section; the contribution of the main term, which is the commutator, then has a sign according to the following lemma.

\begin{lemma}
\label{LemmaSCPLoDiffR}
  For $e\in(0,1)$, we have $M_{t_*,e}>0$. Moreover, given $\delta_r>0$, there exists $\delta>0$ such that $\pm M_{r,e}>\delta\Id$ in $\pm(r-r_c)\geq\delta_r$ for all $e\in(1-\delta,1)$.
\end{lemma}

This, together with Lemma~\ref{LemmaSCPLoTimelike}, is consistent with the idea of $L_\semi$ being roughly differentiation along $\nabla t_*$. Indeed, $\nabla t_*(dt_*)=c^2>0$, and $\pm(\nabla t_*)(dr)=\mp\nu>0$ for $\pm(r-r_c)>0$.

\begin{proof}[Proof of Lemma~\ref{LemmaSCPLoDiffR}]
  We obtain equivalent statements by replacing $M_{t_*,e}$ and $M_{r,e}$ by $\wt M_{t_*,e}$ and $\wt M_{r,e}$, respectively, and proving the corresponding estimates in the sense of symmetric operators with respect to $g_R$.

  The first claim then follows immediately from the observation that the $(3,3)$ entry of $\wt M_{t_*}$ (acting on $T^*\Sph^2$) is positive, and the $2\times 2$ minor of $\wt M_{t_*}$ has positive trace and determinant for $e\in(0,1)$. For the second claim, we must show that $-\nu\wt M_{r,e}>0$ in $\nu\neq 0$. The $(3,3)$ entry of $-\nu\wt M_{r,e}$ is scalar multiplication by $r^{-2}\nu^2/(1-e)>0$, and for the $2\times 2$ minor of $-\nu\wt M_{r,e}$, one easily checks that its trace is positive for $e\in(1/2,1)$, while its determinant is
  \[
    \nu^2\Bigl(\frac{1+e}{1-e}\nu^2-1\Bigr).
  \]
  Given $\delta_r>0$, $\nu^2$ has a positive lower bound on $|r-r_c|\geq\delta_r$, and thus this determinant is positive for $e\in(1-\delta,1)$ if $\delta>0$ is sufficiently small.
\end{proof}

The skew-adjoint part of $L_\semi$ is
\begin{equation}
\label{EqSCPLoLSkewAdj}
  \ell' := \frac{1}{2i\semi}(L_\semi-L_\semi^*) = \frac{1}{2}\bigl(B^{-1}[\pa_r,B M_r] + S+S^*\bigr),
\end{equation}
which is an element of $\CI(M;\End(\Tb^*M))$ and pointwise self-adjoint with respect to the fiber inner product $b$. Again, we shall display the parameter $e$ as a subscript of $\ell'$ whenever necessary. Near the critical set of $\nabla t_*$, it will be crucial that $\ell'$ is negative definite in order to show propagation into $o_Y\subset\ol{\Tb^*_Y}M$ on decaying function spaces (see below for details):

\begin{lemma}
\label{LemmaSCPLoRadialSubpr}
  There exist $\delta_r,\delta>0$ such that for all $e\in(1-\delta,1)$, we have $\ell'_e\leq-\delta\Id$ in the neighborhood $|r-r_c|<\delta_r$ of the critical point of $\nabla t_*$.
\end{lemma}
\begin{proof}
  A simple calculation shows that $\ell'_e$ depends smoothly on $e$ in a full neighborhood of $e=1$, and at $e=1$, one computes
  \[
    \ell'_1:=\lim_{e\to 1} \ell'_e
    =\frac{1}{2}
     \begin{pmatrix}
       2\nu'+8r^{-1}\nu & c^{-2}r^{-1}(2-4\nu^2+r\nu\nu') & 0 \\
       0                & 3\nu'+4r^{-1}\nu                & 0 \\
       0                & 0                               & \nu'+6r^{-1}\nu
     \end{pmatrix}.
  \]
  Since $\nu=0$ and $\nu'<0$ at $r=r_c$, this shows that all eigenvalues of $\ell'_1$, and hence of $\ell'_e$ for $e$ near $1$, are indeed negative in a neighborhood of $r=r_c$.
\end{proof}

Choose $\delta_r$ according to Lemma~\ref{LemmaSCPLoRadialSubpr}, and then take $\delta>0$ to be equal to the smaller value among the ones provided by Lemmas~\ref{LemmaSCPLoDiffR}, with $\delta_r/2$ in place of $\delta_r$, and Lemma~\ref{LemmaSCPLoRadialSubpr}; then, rescaling $\delta_r$, we have for $e\in(1-\delta,1)$
\begin{equation}
\label{EqSCPLoSigns}
  \ell'_e\leq-\delta\Id\ \tn{ for }|r-r_c|< 16\delta_r,\quad \pm M_{r,e}>\delta\Id\ \tn{ for }\pm(r-r_c)\geq 8\delta_r.
\end{equation}
Let us fix such an $e$ and the corresponding inner product $b\equiv b_e$. We also pick $C>0$ according to Corollary~\ref{CorSCPLoImagEll} and define an elliptic (near the zero section) multiple of $\cP_\semi$ in analogy to \eqref{EqSCPLoToy2} by
\begin{equation}
\label{EqSCPLoEllMultOp}
  \cP'_\semi := (1 - i C L_\semi^*)i \cP_\semi = (L_\semi + J_\semi) - i Q_\semi,
\end{equation}
where
\begin{equation}
\label{EqSCPLoEllMultPieces}
\begin{gathered}
  J_\semi = C\Re(L_\semi^*\BoxCP_{g,\semi}) - \Im(\BoxCP_{g,\semi}), \\
  Q_\semi = C L_\semi^*L_\semi - \Re(\BoxCP_{g,\semi}) - C\Im(L_\semi^*\BoxCP_{g,\semi});
\end{gathered}
\end{equation}
here we use the inner product $b$ to compute adjoints as well as symmetric ($\Re$) and skew-symmetric ($\Im$) parts. Note that $J_\semi,Q_\semi$ are semiclassical b-differential operators, and $J_\semi=J_\semi^*$, $Q_\semi=Q_\semi^*$. Denoting by $\cO(\zeta^k)$ a symbol vanishing at least to order $k$ at the zero section, their full symbols satisfy
\begin{equation}
\label{EqSCPLoEllMultSymb}
\begin{gathered}
  \sigmabhfull(J_\semi) = \cO(\zeta^3) + \semi\cO(1), \\
  \sigmabhfull(Q_\semi) = (C\ell^2-G\Id) + \semi\cO(\zeta) + \semi^2\cO(1).
\end{gathered}
\end{equation}

We now prove the propagation of regularity (with estimates) away from the critical set:
\begin{prop}
\label{PropSCPLoPropRad}
  Suppose that $r_c+14\delta_r<r_1<r_2<r_++3\eps_M$ (see \eqref{EqOpNoB0} and \eqref{EqSdSMfd} for the definitions of $r_+$ and $\eps_M$). Moreover, fix $0<\tau_0<\tau_1$. Let $\cV\subset\Tb^*M$ be a fixed neighborhood of the zero section $o\subset\Tb^*M$, in particular $\cV$ is disjoint from $\Sb^*M$. Let $B_1,B_2,S\in\Psibh(M;\Tb^*M)$ be operators such that $B_1$ is elliptic near the zero section over the set
  \[
    \bigl([0,\tau_1)_\tau\times(r_c+10\delta_r,r_c+14\delta_r)_r \cup (\tau_0,\tau_1)_\tau\times(r_c+10\delta_r,r_2)_r\bigr)\times\Sph^2,
  \]
  while $\WFbh'(B_2)\subset\cV\cap\{\tau\in[0,\tau_0),\ r\in(r_1,r_2)\}$, and $S$ is elliptic in $\cV\cap\{\tau\in[0,\tau_1),\ r\in(r_c+10\delta_r,r_2)\}$. Then for all $\rho\in\R$ and $N\in\R$, the estimate
  \begin{equation}
  \label{EqSCPLoPropRadEst}
    \| B_2 u \|_{\Hbh^{0,\rho}} \lesssim \| B_1 u \|_{\Hbh^{0,\rho}} + \semi^{-1}\| S \cP_\semi u \|_{\Hbh^{0,\rho}} + \semi^N\| u \|_{\Hbh^{0,\rho}}
  \end{equation}
  holds for sufficiently small $\semi>0$. See Figure~\ref{FigSCPLoPropRad}.

  The same holds if instead $r_--3\eps_M<r_1<r_2<r_c-14\delta_r$, now with $B_1$ elliptic near the zero section over
  \[
    \bigl([0,\tau_0)_\tau\times(r_c-14\delta_r,r_c-10\delta_r) \cup (\tau_0,\tau_1)_\tau\times(r_1,r_c-10\delta_r)_r\bigr)\times\Sph^2,
  \]
  and $S$ elliptic in $\cV\cap\{\tau\in[0,\tau_1),\ r\in(r_1,r_c-10\delta_r)\}$.
\end{prop}

\begin{figure}[!ht]
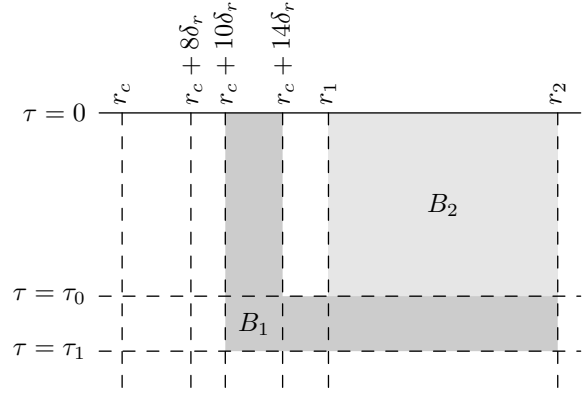

  \centering
  \inclfig{SCPLoPropRad}
  \caption{Propagation near the zero section away from the critical set of $\nabla t_*$: control on the elliptic set of $B_1$ propagates to the elliptic set of $B_2$.}
  \label{FigSCPLoPropRad}
\end{figure}

\begin{rmk}
\label{RmkSCPLo2ndProp}
  Continuing Remark~\ref{RmkSCPLo2nd}, there is a cleaner and more precise $2$-microlocal statement. Namely, $2$-microlocal regularity propagates along the Hamilton vector fields of the eigenvalues of $\ell$.
\end{rmk}

\begin{proof}[Proof of Proposition~\ref{PropSCPLoPropRad}]
  Since we are working near the zero section, the differentiability order is irrelevant. Furthermore, by Lemma~\ref{LemmaSCPSemiEll}, we may localize as close to the zero section as we wish, and by conjugating microlocally near $o\subset\Tb^*M$ by the elliptic operator $(1-i C L_\semi^*)$, it suffices to prove the estimate \eqref{EqSCPLoPropRadEst} for $\cP'_\semi$ in place of $\cP_\semi$. We then consider the scalar commutant
  \[
    a(z,\zeta) = e^{\rho t_*}\chi_0(t_*)\chi_1(r)\chi_*(\zeta),
  \]
  where $z=(t_*,r,\omega)\in M$ and $\zeta\in\Tb^*_z M$, where we suppress a factor of $\Id$, the identity map on $(\pi^*\Tb^*M)_{(z,\zeta)}$, $\pi\colon\Tb^*M\to M$ the projection. Here, writing $t_{*,j}=-\log\tau_j$, we take $\chi_0$ with $\chi_0\equiv 0$ for $t_*\leq t_{*,1}+\frac{1}{3}(t_{*,0}-t_{*,1})$ and $\chi_0\equiv 1$ for $t_*\geq t_{*,1}+\frac{2}{3}(t_{*,0}-t_{*,1})$, while
  \[
    \chi_1(r) = \psi_1(r)\psi_2(\digamma^{-1}(r_2-r)),\quad \psi_2(x)=H(x)e^{-1/x},
  \]
  and $\psi_1\equiv 0$ for $r\leq r_c+11\delta_r$ and $\psi_1\equiv 1$ for $r\geq r_c+13\delta_r$, see Figure~\ref{FigSCPLoPropRadChi1}; moreover,
  \[
    \chi_*(\zeta)=\chi_{*,0}(\digamma_*\zeta),
  \]
  with $\chi_{*,0}\in\CIc(\Tb^*M)$ identically $1$ for $|\zeta|\leq 1/2$ and identically $0$ for $|\zeta|\geq 1$. Further, we assume that $\sqrt{\chi_0}$, $\sqrt{\chi_1}$ and $\sqrt{\chi_*}$ are smooth; finally, $\digamma,\digamma_*>0$ are large, to be chosen later.

  \begin{figure}[!ht]
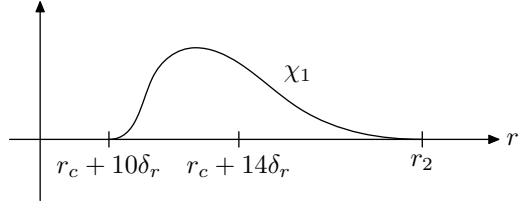

    \centering
    \inclfig{SCPLoPropRadChi1}
    \caption{Graph of the function $\chi_1$, with the negative derivative between $r_c+14\delta_r$ and $r_2$ giving positivity in the positive commutator argument, and the region $r_c+10\delta_r\leq r\leq r_c+14\delta_r$ being the a priori control region.}
    \label{FigSCPLoPropRadChi1}
  \end{figure}

  Let $\ell$ and $j$ denote the principal symbols of $L_\semi$ and $J_\semi$, respectively; then $H_\ell t_*=M_{t_*}$ and $H_\ell r=M_r$. Therefore, we compute
  \begin{equation}
  \label{EqSCPLoPropRadComm}
    \ham_\ell a = e^{\rho t_*}\bigl( \chi_1' M_r\chi_0\chi_* + (\rho\chi_0+\chi_0')M_{t_*}\chi_1\chi_* + (\ham_\ell\chi_*)\chi_0\chi_1 \bigr);
  \end{equation}
  using $\psi_2'=x^{-2}\psi_2$, we further have
  \[
    \chi_1' = \psi_1'\psi_2 - \digamma(r_2-r)^{-2}\chi_1,
  \]
  with the second term being the main term, and the first term supported in the a priori control region in $r$. On the other hand,
  \[
    \ham_j a = e^{\rho t_*}\bigl(\chi_1'(\ham_j r)\chi_0\chi_* + (\rho\chi_0+\chi_0')(\ham_j t_*)\chi_1\chi_* + (\ham_j\chi_*)\chi_0\chi_1\bigr);
  \]
  now in view of the cubic vanishing of $j$ at $\zeta=0$ by \eqref{EqSCPLoEllMultSymb}, the first two terms in the parenthesis---which are stationary ($t_*$-independent) smooth functions---vanish quadratically there; for large $\digamma_*$, we will thus be able to absorb them into the main term $\chi_1'M_r$ in \eqref{EqSCPLoPropRadComm} up to a priori controlled terms. Concretely, given a constant $\sfM>0$ (used to absorb further error terms), chosen below, we can write
  \begin{equation}
  \label{EqSCPLoPropRadComm2}
    \ham_{\ell+j}a = -(b'_2)^2 - \sfM a + e_0 + e_1,
  \end{equation}
  where
  \begin{align*}
    b'_2 &= e^{\rho t_*/2}\sqrt{\chi_0\chi_1\chi_*}\bigl( \digamma(r_2-r)^{-2}(M_r+\ham_j r) - \rho(M_{t_*}+\ham_j t_*) - \sfM\Id \bigr)^{1/2}, \\
    e_0 &= e^{\rho t_*}\chi_0\chi_1(\ham_{\ell+j}\chi_*), \\
    e_1 &= e^{\rho t_*}\chi_*\bigl(\psi_1'\psi_2\chi_0(M_r+\ham_j r) + \chi_0'\chi_1(M_{t_*}+\ham_j t_*)\bigr).
  \end{align*}
  Here, the parenthesis in the definition of $b'_2$ is a strictly positive self-adjoint endomorphism on $\supp a$ provided $\digamma_*$ and $\digamma$ are large, and we take $b'_2$ to be the unique positive definite square root, which is smooth; the term $e_1$ is supported in the region where we assume a priori control, corresponding to the elliptic set of $B_1$, and $e_0$ is supported in the elliptic set of $\cP'_\semi$.

  Let now $A=A^*\in\Psibh(M;\Tb^*M)$, $B'_2=(B'_2)^*\in\Psibh(M;\Tb^*M)$, $E_0$ and $E_1\in\Psibh(M;\Tb^*M)$ denote quantizations of $a$, $b'_2$, $e_0$ and $e_1$, respectively, with operator wave front set contained in the respective supports of the symbols, and Schwartz kernels supported in fixed small neighborhoods of the projection of the wave front set to the base. We then evaluate the $L^2=\Hb^0$ pairing (using the fiber inner product $b$)
  \begin{align}
    I := \semi^{-1}\Im&\la A\cP'_\semi u,A u\ra + \semi^{-1}\Re\la A Q_\semi u, A u\ra \nonumber\\
    &= \frac{i}{2\semi}\bigl(\la A u,A(L_\semi+J_\semi)u\ra - \la A(L_\semi+J_\semi)u,A u\ra\bigr) \nonumber\\
  \label{EqSCPLoPropRadComm3}
    &= \Big\la\frac{i}{2\semi}[L_\semi+J_\semi,A^2]u,u\Big\ra + \la \ell' A^2 u,u\ra,
  \end{align}
  recalling \eqref{EqSCPLoLSkewAdj}. Using the description \eqref{EqSCPLoPropRadComm2} on the symbolic level and integrating by parts, there exists $E_3\in\Psibh(M;\Tb^*M)$ with $\WFbh'(E_3)\subseteq\WFbh'(A)$ so that the right hand side is equal to
  \[
    I = \la E_0 u, A u\ra + \la E_1 u, A u\ra - \sfM\|A u\|^2 - \|B'_2 A u\|^2 + \semi\la E_3 u,A u\ra  + \la \ell' A^2 u,u\ra.
  \]
  We estimate
  \[
    |\la E_0 u,A u\ra| \leq C'' \semi^{-2}\|S\cP'_\semi u\|_{\Hbh^{0,\rho}}^2 + \semi^2\|A u\|^2 + C_N \semi^N\|u\|_{\Hbh^{0,\rho}}^2
  \]
  using the ellipticity of $\cP'_\semi$ on $\WFbh'(E_0)$; further, for $\eta>0$ arbitrary,
  \[
    |\la E_1 u,A u\ra| \leq C_\eta\|B_1 u\|_{\Hbh^{0,\rho}}^2 + \eta\|A u\|^2 + C_N \semi^N\|u\|_{\Hbh^{0,\rho}}^2
  \]
  by virtue of the ellipticity of $B_1$ on $\WFbh'(E_1)$; in both cases, the constants $C_\eta$ and $C_N$ depend implicitly also on the parameters $\sfM$, $\digamma$ and $\digamma_*$. Fixing $\wt A\in\Psibh(M;\Tb^*M)$, elliptic in a small neighborhood of $\WFbh'(A)$ (for $\digamma_*=1$, say), we further have
  \[
    |\semi\la E_3 u,A u\ra| \leq \eta\|A u\|^2 + C_\eta\|\semi\wt A u\|_{\Hbh^{0,\rho}}^2 + C_N \semi^N\|u\|_{\Hbh^{0,\rho}}^2.
  \]
  We can furthermore estimate
  \begin{align*}
    |\la \ell' A^2 u,u\ra| &\leq |\la \ell' A u,A u\ra| + |\la A u,[A,\ell']u\ra| \\
      &\leq (C'+\eta)\|A u\|^2 + C_\eta\|\semi\wt A u\|_{\Hbh^{0,\rho}}^2 + C_N \semi^N\|u\|_{\Hbh^{0,\rho}}^2.
  \end{align*}
  Lastly, we may estimate $\|B_2 A u\|^2\leq C_2'\|B_2' A u\|^2 + C_N \semi^N\|u\|_{\Hbh^{0,\rho}}^2$ by elliptic regularity, and therefore obtain
  \begin{equation}
  \label{EqSCPLoPropRadComm3RHS}
  \begin{split}
    I \leq C''&\semi^{-2}\|S\cP'_\semi u\|_{\Hbh^{0,\rho}}^2 + C_\eta\|B_1 u\|_{\Hbh^{0,\rho}}^2 - C_2'\|B_2 A u\|^2 \\
      &- (\sfM - C' - 3\eta - \semi^2)\|A u\|^2 + 2 C_\eta\|\semi\wt A u\|_{\Hbh^{0,\rho}}^2 + C_N \semi^N\|u\|_{\Hbh^{0,\rho}}^2.
  \end{split}
  \end{equation}
  We henceforth fix $\sfM>C'$.

  Turning to the left hand side of \eqref{EqSCPLoPropRadComm3}, we bound the first term by
  \[
    \semi^{-1}\Im\la A\cP'_\semi u,A u\ra \geq -\eta\|A u\|^2 - C_\eta \semi^{-2}\|A\cP'_\semi u\|^2.
  \]
  To treat the first term of \eqref{EqSCPLoPropRadComm3}, we write
  \begin{equation}
  \label{EqSCPLoPropRadSkew}
    \semi^{-1}\Re\la A Q_\semi u, A u\ra = \Re\la \semi^{-1}[A,Q_\semi]u,A u\ra + \semi^{-1}\la Q_\semi A u,A u\ra
  \end{equation}
  and bound the first term on the right by writing it as
  \[
    \semi \Re \Big\la \semi^{-1}\bigl[A,\semi^{-1}[A,Q_\semi]\bigr]u,u\Big\ra \geq -\wt C \semi\|\wt A u\|_{\Hbh^{0,\rho}}^2 - C_N \semi^N\|u\|_{\Hbh^{0,\rho}}^2
  \]
  For the second term of \eqref{EqSCPLoPropRadSkew} on the other hand, we recall the form \eqref{EqSCPLoEllMultSymb} of the full symbol of $Q_\semi$. We begin by estimating the contribution of the lower order terms: any $\semi^2\cO(1)$ term gives a contribution which is bounded by $C_2 \semi\|A u\|^2$, i.e.\ has an additional factor of $\semi$ relative to the term $\sfM\|A u\|^2$ in \eqref{EqSCPLoPropRadComm3RHS} and is thus small for small $\semi>0$. The contribution of an $\semi\cO(\zeta)$ term is small for $\digamma_*\gg 0$ (i.e.\ strong localization near the zero section):\footnote{A fixed bound would be sufficient for present, real principal type, purposes, but the smallness will be essential at the critical set in the proof of Proposition~\ref{PropSCPLoPropTime} below.} Indeed, if $\semi V$, with $V\in\Vb(M)$, is a semiclassical b-vector field (so $\sigmabh(\semi V)=\cO(\zeta)$), then we have
  \[
    \|V A u\| \leq \|\Xi V A u\| + C_{N,\digamma_*} \semi^N\|u\|_{\Hbh^{0,\rho}}
  \]
  for $\Xi\in\Psibh(M;\Tb^*M)$ a quantization of $\chi_{*,0}(\digamma^*\zeta/2)\Id$; but the supremum of the principal symbol of $\Xi V$ is then bounded by a constant times $\digamma_*^{-1}$. By \cite[Theorem~5.1]{ZworskiSemiclassical}, the operator norm on $L^2$ of a semiclassical operator is given by the $L^\infty$ norm of its symbol up to $\cO(\semi^{1/2})$ errors, hence we obtain
  \[
    \|V A u\| \leq (C'\digamma_*^{-1}+C_{\digamma_*} \semi^{1/2})\|A u\| + C_{N,\digamma_*} \semi^N\|u\|_{\Hbh^{0,\rho}};
  \]
  fixing $\digamma_*$, we get the desired bound for the contribution of $\semi\cO(\zeta)$ to the second term in \eqref{EqSCPLoPropRadSkew},
  \[
    \semi^{-1}|\la \semi V A u,A u\ra| \leq 2C'\digamma_*^{-1}\|A u\|^2+C_N \semi^N\|u\|^2_{\Hbh^{0,\rho}},
  \]
  for $\semi>0$ sufficiently small.

  In order to treat the main term of $Q_\semi$, define $\cQ\in\CI(M;\End(\Tb^*M\otimes\Tb^*M))$ as the pointwise self-adjoint section with quadratic form
  \[
    (\cQ(\zeta\otimes v),\zeta\otimes v)=C|\ell(\zeta)v|_b^2-G(\zeta)|v|_b^2, \quad \zeta,v\in\Tb^*_z M,\ z\in M,
  \]
  then $Q_\semi=\semi^2\nabla^*\cQ\nabla$ modulo lower order terms, i.e.\ semiclassical b-differential operators with full symbols of the form $\semi\cO(\zeta)+\semi^2\cO(1)$. By Corollary~\ref{CorSCPLoImagEll}, the \emph{classical} b-differential operator $Q':=\nabla^*\cQ\nabla$ is an elliptic element of $\Psib^2(M;\Tb^*M)$ and formally self-adjoint with respect to the fiber inner product $b$, therefore we can construct its square root using the symbol calculus, giving $R\in\Psib^1(M;\Tb^*M)$ such that $Q'-R^*R=R'\in\Psib^{-\infty}(M;\Tb^*M)$; but then the boundedness of $R'$ on $\Hbh^{0,0}$ gives
  \begin{equation}
  \label{EqSCPLoPropRadSkew2}
    \semi^{-1}\la \semi^2\nabla^*\cQ\nabla A u,A u\ra \geq -C' \semi\|A u\|^2.
  \end{equation}
  We conclude that the left hand side of \eqref{EqSCPLoPropRadComm3} satisfies the bound
  \begin{equation}
  \label{EqSCPLoPropRadComm3LHS}
  \begin{split}
    I &\geq -C_\eta \semi^{-2}\|S\cP'_\semi u\|^2 - (\eta+C_2 \semi+2 C'\digamma_*^{-1}+C' \semi)\|A u\|^2 \\
      &\qquad - \wt C\|\semi^{1/2}\wt A u\|_{\Hbh^{0,\rho}}^2 - C_N \semi^N\|u\|_{\Hbh^{0,\rho}}^2
  \end{split}
  \end{equation}
  for $\semi>0$ sufficiently small. Note that the constant in front of $\|A u\|^2$ can be made arbitrarily small by choosing $\eta>0$ small, $\digamma_*$ large and then $\semi>0$ small.

  Combining this estimate with \eqref{EqSCPLoPropRadComm3RHS}, we can absorb the $\|A u\|^2$ term from \eqref{EqSCPLoPropRadComm3LHS} into the corresponding term in \eqref{EqSCPLoPropRadComm3RHS}, and then drop it as it has the same sign as the main term $\|B_2 A u\|^2$. We thus obtain the desired estimate \eqref{EqSCPLoPropRadEst}, but with an additional control term $\|\semi^{1/2}\wt A u\|^2$ on the right hand side; this term however can be removed iteratively by applying the estimate to this control term itself, which weakens the required control by $\semi^{1/2}$ at each step, until after finitely many iterations we reach the desired power of $\semi^N$.
\end{proof}

Not using the flexibility in choosing $\digamma_*$, and simply using the semiclassical sharp G\aa{}rding inequality to estimate $Q_\semi$ in \eqref{EqSCPLoPropRadSkew}, we could take the support of $\chi_*$ in the above proof to be fixed (e.g.\ by choosing $\digamma_*=1$): the positivity of $\sigmabh(Q_\semi)$ would give a lower bound $-C'\|A u\|^2$ in \eqref{EqSCPLoPropRadSkew2} with some constant $C'$, which could be absorbed in the term in \eqref{EqSCPLoPropRadComm3RHS} involving $\|A u\|^2$ by taking $\sfM>0$ large. On the other hand, our argument presented above in particular shows that, using the structure of $Q_\semi$, the constant $C'$ can be made arbitrarily small, which is crucial in the more delicate radial point type estimate near the critical set:

\begin{prop}
\label{PropSCPLoPropTime}
  Fix $0<\tau_0<\tau_1$, and let $\cV\subset\Tb^*M$ be a fixed neighborhood of $o\subset\Tb^*M$. There exists $\alpha>0$ such that the following holds: let $B_1,B_2,S\in\Psibh(M;\Tb^*M)$ be operators such that $B_1$ is elliptic near the zero section over the set $(\tau_0,\tau_1)_\tau \times (r_c-16\delta_r,r_c+16\delta_r)_r \times \Sph^2$, while
  \[
    \WFbh'(B_2)\subset\cV\cap\{\tau\in[0,\tau_1),\ r\in[r_c-15\delta_r,r_c+15\delta_r]\},
  \]
  and $S$ is elliptic in $\cV\cap\{\tau\in[0,\tau_1),\ r\in(r_c-16\delta_r,r_c+16\delta_r)\}$. Then for all $\rho\leq\alpha$, the estimate
  \[
    \|B_2 u\|_{\Hbh^{0,\rho}} \lesssim \|B_1 u\|_{\Hbh^{0,\rho}} + \semi^{-1}\|S\cP_\semi u\|_{\Hbh^{0,\rho}} + \semi^N\|u\|_{\Hbh^{0,\rho}}
  \]
  holds. See Figure~\ref{FigSCPLoPropTime}.
\end{prop}

\begin{figure}[!ht]
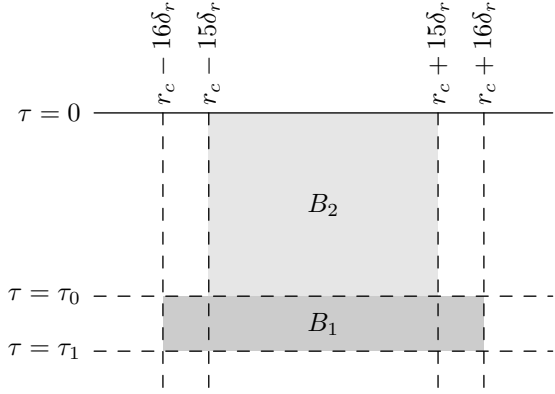

  \centering
  \inclfig{SCPLoPropTime}
  \caption{Propagation near the zero section in the vicinity of the critical set of $\nabla t_*$: control on the elliptic set of $B_1$ propagates to the elliptic set of $B_2$.}
  \label{FigSCPLoPropTime}
\end{figure}

The decisive gain here is that by taking $B_2$ to be elliptic near the zero section over $[0,\tau_0]_\tau\times(r_c-15\delta_r,r_c+15\delta_r)\times\Sph^2$, we can propagate estimates \emph{on decaying function spaces} $e^{-\alpha t_*}\Lb^2$ into the boundary $X$ at future infinity, as well as to the region $|r-r_c|\in(10\delta_r,14\delta_r)$ from where we can use Proposition~\ref{PropSCPLoPropRad} to propagate estimates outwards, i.e.\ in the radial direction across the event and the cosmological horizons.

\begin{proof}[Proof of Proposition~\ref{PropSCPLoPropTime}]
  By microlocal elliptic regularity, we may again prove the estimate for $\cP'_\semi$ in place of $\cP_\semi$. We consider the commutant
  \[
    a(z,\zeta) = e^{\rho t_*}\chi_0(t_*)\chi_1(r)\chi_*(\zeta),
  \]
  with the cutoffs $\chi_0$ (identically $1$ for $t_*\geq-\log\tau_0$ and $0$ for $t_*\leq-\log\tau_1$) and $\chi_*$ (localizing $\digamma_*$-close to the zero section) similar to the ones used in the proof of the previous proposition; furthermore, the radial cutoff is $\chi_1\equiv 1$ for $|r-r_c|\leq 9\delta_r$, $\chi_1\equiv 0$ for $|r-r_c|\geq 15\delta_r$, and $\nu\chi_1'\geq 0$ for all $r$; we moreover require that $\sqrt{\chi_1}$ and $\sqrt{\nu\chi_1'}$ are smooth, and
  \begin{equation}
  \label{EqSCPLoPropTimeRadPos}
    \chi_1' M_r\leq-c_1\delta \quad \tn{for }\pm(r-r_c)\geq 8\delta_r.
  \end{equation}
  We again have the commutator calculation \eqref{EqSCPLoPropRadComm3}, but since we can only get a limited amount positivity near $r=r_c$ from the weight in $t_*$, the sign of $\ell'$ on $\supp a$, guaranteed in \eqref{EqSCPLoSigns} plays a key role; for $\sfm>0$ sufficiently small, it allows us to write
  \[
    \ham_{\ell+j}a + \ell' a = -(b_2')^2 - (b_2'')^2 - \sfm a + e_0 + e_1,
  \]
  where
  \begin{align*}
    b_2' &= e^{\rho t_*/2}\sqrt{\chi_0\chi_1\chi_*}\bigl(-\ell'-\rho(M_{t_*}+\ham_j t_*) - \sfm\Id\bigr)^{1/2}, \\
    b_2'' &= e^{\rho t_*/2}\sqrt{\chi_0\chi_*}(-\chi_1'(M_r+\ham_j r))^{1/2}, \\
    e_0 &= e^{\rho t_*}\chi_0\chi_1(\ham_{\ell+j}\chi_*), \\
    e_1 &= e^{\rho t_*}\chi_0'\chi_1\chi_*(M_{t_*}+\ham_j t_*).
  \end{align*}
  Now for $\rho\leq\alpha$, with $\alpha>0$ sufficiently small, and all small $\sfm>0$, the square root defining $b_2'$ exists on $\supp(\chi_0\chi_1\chi_*)$ within the space of positive definite self-adjoint endomorphisms of $\Tb^*M$ due to $\ell'\leq-\delta\Id$, provided $\digamma_*$ is sufficiently large so that the contribution of $\ham_j$ is small. Let us fix such $\sfm>0$ and $\alpha>0$. Observe that, similarly, $b_2''$ is well-defined and smooth; in fact it is an elliptic symbol whenever $\chi_1'\neq 0$, but this does not provide any further control beyond what $b'_2$ gives. The error term $e_0$ is controlled by elliptic regularity for $\cP'_\semi$, and the term $e_1$ is the a priori control term.

  We can now quantize these symbols and follow the proof of Proposition~\ref{PropSCPLoPropRad} \emph{mutatis mutandis}. Note that the constant $C'$ in \eqref{EqSCPLoPropRadComm3RHS} does not appear anymore in the present context, since it came from $\ell'$, which, however, we incorporated into the symbolic calculation above; thus, the quantity in the parenthesis multiplying $\|A u\|^2$ in \eqref{EqSCPLoPropRadComm3RHS} can be made positive ($\geq\sfm/2$, say), and the prefactor of $\|A u\|^2$ in \eqref{EqSCPLoPropRadComm3LHS} can be made arbitrarily small, by choosing $\eta>0$ small, $\digamma_*>0$ large and then $\semi>0$ small. This completes the proof.
\end{proof}

\begin{rmk}
\label{RmkSCPLoPropTimeWeight}
  An inspection of the proof shows that for any fixed
  \[
    \alpha < \inf\bigl\{\rho\colon \spec(-\ell_1'-\rho M_{t_*,1})\subset\R_+\tn{ at }r=r_c \bigr\} = \frac{|\nu'(r_c)|}{2c^2},
  \]
  one can take $e<1$ so close to $1$ that the above microlocal propagation estimates hold.
\end{rmk}

\subsection{Energy estimates near spacelike surfaces}
\label{SubsecSCPEn}

\emph{We continue dropping the subscript `$e$' and using the positive definite fiber metric $b$ on $\Tb^*M$.} In order to `cap off' global estimates for $\cP_\semi$, we will use standard energy estimates near the Cauchy surface $\Sigma_0$, as well as near the two connected components of $[0,1]_\tau\times\pa Y$ (which are spacelike hypersurfaces, located beyond the horizons). Since we need \emph{semiclassical} estimates for the principally non-real operator $\cP_\semi$, this does not quite parallel the analysis of\footnote{The assumption in the reference that $|\Im\sigma|$ is bounded renders the skew-adjoint part of the operator in \cite[Proposition~3.8]{VasyMicroKerrdS} semiclassically subprincipal.} \cite[\S3.3]{VasyMicroKerrdS} if one extended it to estimates on b-Sobolev spaces as in \cite[\S2.1]{HintzVasySemilinear}; moreover, we need to take the non-scalar nature of the `damping' term $L_\semi$ into account.

The key to proving energy estimates in the principally scalar setting is the coercivity of the energy-momentum tensor when evaluated on two future timelike vector fields, one of which is a suitably chosen `multiplier,' the other often coming from boundary unit normals, see \cite[\S2.7]{TaylorPDE} and \cite[Appendix~D]{DafermosRodnianskiLectureNotes}. Since $\BoxCP_{g,\semi}-i L_\semi$ is not principally scalar in the semiclassical sense due to the presence of $L_\semi$, which however is future timelike as explained in Lemma~\ref{LemmaSCPLoTimelike} and should thus be considered a (strong) non-scalar damping term, it is natural to use
\[
  L_\semi = \ell^\mu\,\semi D_\mu + i \semi S
\]
(so $\ell^\mu=\ell(dx^\mu)\in\End(T^*M^\circ)$) as a non-scalar multiplier. For the resulting bundle-valued `energy-momentum tensor,' we have the following positivity property:

\begin{lemma}
\label{LemmaSCPEnEnMom}
  Given a future timelike covector $w\in T^*_z M^\circ$, $z\in M^\circ$, define
  \begin{equation}
  \label{EqSCPEnEnMom}
    T^{\mu\nu} = (G^{\mu\kappa}\ell^\nu + G^{\nu\kappa}\ell^\mu - G^{\mu\nu}\ell^\kappa)w_\kappa \in \End(T^*_z M^\circ).
  \end{equation}
  Then:
  \begin{enumerate}
  \item \label{ItSCPEnEnMom1Weak} For all $\zeta\in T^*_z M^\circ\setminus o$, we have $T^{\mu\nu}\zeta_\mu\zeta_\nu>0$, that is, there exists a constant $C_T>0$ such that
    \[
      \la T^{\mu\nu}\zeta_\mu\zeta_\nu v,v\ra \geq C_T|\zeta|_{g_R}^2|v|_b^2,\quad \zeta,\,v\in T^*_z M^\circ.
    \]
  \item \label{ItSCPEnEnMom2Strong} Write $w=d z^0$ near $z$, with $z^0=0$ at $z$, and denote by $z^1,\cdots,z^3$ local coordinates on $\{z^0=0\}$ near $z$ with $G(d z^0,d z^i)=0$, $i=1,\ldots,3$.\footnote{Given any local coordinates $z^i$ on $\{z_0=0\}$, the coordinates $z^i-\frac{G(dz^0,dz^i)}{G(dz^0,dz^0)}z^0$ have the desired property.} Then, with indices $i,j$ running from $1$ to $3$, we have
    \begin{equation}
    \label{EqSCPEnEnMom2Strong}
    \begin{split}
      \la T^{00}&v_0,v_0\ra + \la T^{0j}v_0,\zeta'_j v\ra + \la T^{i0}\zeta'_i v,v_0\ra + \la T^{ij}\zeta'_i v,\zeta'_j v\ra \\
        &\geq C_T\bigl(|v_0|_b^2 + |\zeta'|_{g_R}^2|v|_b^2\bigr),
         \qquad v_0,v\in T^*_z M^\circ,\ \zeta'=(\zeta'_1,\ldots,\zeta'_3)\in\R^3,
    \end{split}
    \end{equation}
    for some constant $C_T>0$, where we identify $\zeta'$ with $\zeta_j'\,dz^j\in T^*_z M^\circ$.
  \end{enumerate}
\end{lemma}
\begin{proof}
  Choosing local coordinates $z^\mu$ as in \eqref{ItSCPEnEnMom2Strong}, we may assume that $\{d z^\mu\}$ is an orthonormal frame at $z$; thus, $w_\kappa=\delta_{\kappa 0}$, $G^{00}=1$, $G^{ii}=-1$ and $G^{\mu\nu}=0$ for $\mu\neq\nu$. Moreover, by picking a basis of $T^*_z M^\circ$ which is orthonormal with respect to $\ell^0$ (which is positive definite by the timelike nature of $w$ and $L_\semi$), we may assume $\ell^0=\Id$ and
  \[
    (T^{\mu\nu}) = \begin{pmatrix}
                     \Id    & \ell^1 & \ell^2 & \ell^3 \\
                     \ell^1 & \Id    & 0      & 0      \\
                     \ell^2 & 0      & \Id    & 0      \\
                     \ell^3 & 0      & 0       & \Id
                   \end{pmatrix}.
  \]
  The future timelike nature of $L_\semi$ is then equivalent to $\ell^i\zeta'_i<\Id$ for all $\zeta'_1,\ldots,\zeta'_3\in\R$ with $|\zeta'|_\eucl^2:=\sum(\zeta'_j)^2\leq 1$. Given $v_0,v,\zeta'_1,\ldots,\zeta'_3$, we then compute the left hand side of \eqref{EqSCPEnEnMom2Strong} to be
  \[
    |v_0|^2 + |\zeta'|_\eucl^2|v|^2 + 2\la v_0,\ell^j\zeta'_j v\ra = |v_0+\ell^j\zeta'_j v|^2 + |\zeta'|_\eucl^2|v|^2 - |\ell^j\zeta'_j v|^2;
  \]
  the sum of the last two terms is equal to $\la T'v,v\ra$ for $T'=\sum(\zeta'_j)^2\Id - (\ell^j\zeta'_j)^2$. Now, factoring out $|\zeta'|_{\eucl}^2$ and letting $\xi'_j=\zeta'_j/|\zeta'|_\eucl$, which has $|\xi'|_\eucl^2\leq 1$, we can factor
  \[
    |\zeta'|_\eucl^{-2}T = (\Id - \ell^j\xi'_j)(\Id + \ell^j\xi'_j).
  \]
  Both factors on the right are positive definite, and they commute; thus, we can write $T'=R^*R$ for
  \[
    R = \bigl((\Id - \ell^j\xi'_j)(\Id + \ell^j\xi'_j)\bigr)^{1/2}
  \]
  the symmetric square root. Therefore,
  \[
    |v_0|^2 + |\zeta'|_\eucl^2|v|^2 + 2\la v_0,\ell^j\zeta'_j v\ra = |v_0+\ell^j\zeta'_j v|^2 + |\zeta'|^2 |R v|^2
  \]
  is non-negative. Suppose it vanishes, then we first deduce that $\zeta'=0$ or $R v=0$; in the first case, we find $|v_0|^2=0$, hence $v_0=0$, while in the second case, the non-degeneracy of $R$ gives $v=0$, and it again follows that $v_0=0$. Therefore, for $(v_0,\zeta'\otimes v)\neq 0$, the left hand side of \eqref{EqSCPEnEnMom2Strong} is in fact strictly positive. The estimate \eqref{EqSCPEnEnMom2Strong} then follows by homogeneity.

  This proves \eqref{ItSCPEnEnMom2Strong}; part \eqref{ItSCPEnEnMom1Weak} is the special case in which $v_0=\zeta_0 v$.
\end{proof}

This of course continues to hold, \emph{mutatis mutandis}, uniformly in $t_*$, or more precisely on the b-cotangent bundle. We leave the necessary (notational) modifications to the reader.

For the sake of simplicity, we will first prove semiclassical energy estimates near $\Sigma_0$ in order to illustrate the necessary arguments, see Proposition~\ref{PropSCPEn} below; the energy estimates of more central importance for our purposes concern estimates in semiclassical weighted (in $t_*$) b-Sobolev spaces, propagating control in the $r$-direction beyond the horizons; see Proposition~\ref{PropSCPEnR}.

For convenience, we construct a suitable time function near $\Sigma_0$:
\begin{lemma}
\label{LemmaSCPEnTime}
  Let $\phi=\phi(r)$ be a smooth function of $r\in(r_--3\eps_M,r_++3\eps_M)$, identically $0$ for $r_--\eps_M/2<r<r_++\eps_M/2$, with $\phi(r)\to+\infty$ as $r\to r_\pm \pm 3\eps_M$, and $\pm\phi'(r)\geq 0$ for $\pm(r-r_c)\geq 0$. Define
  \[
    \frakt := t_* + \phi(r);
  \]
  see Figure~\ref{FigSCPEnTimeLvl}. Then $d\frakt$ is future timelike.
\end{lemma}
\begin{proof}
  Using \eqref{EqSCPSemiMetric}, we compute $|d\frakt|_G^2 = c^2 - 2\nu\phi' - \mu(\phi')^2$; now $\nu\phi'\leq 0$ by assumption, while $\mu<0$ on $\supp\phi'$, hence $|d\frakt|_G\geq c^2$ everywhere.
\end{proof}

\begin{figure}[!ht]
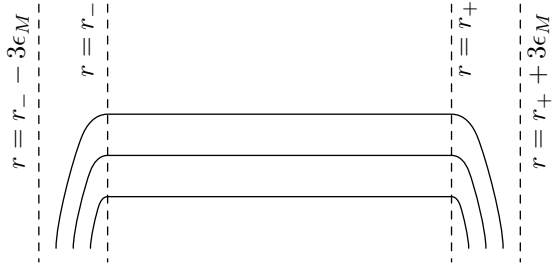

  \centering
  \inclfig{SCPEnTimeLvl}
  \caption{Level sets of the function $\frakt$.}
  \label{FigSCPEnTimeLvl}
\end{figure}

The main term of the energy estimate will have a sign up to low energy errors due to the following result.
\begin{lemma}
\label{LemmaSCPEnPos}
  Let $\frakt_0\in\R$, let $S:=\{\frakt=\frakt_0\}$, and fix $K\Subset S$. For the conormal $w=d\frakt$ of $S$, define the $\End(T^*_S M^\circ)$-valued tensor $T^{\mu\nu}$ by \eqref{EqSCPEnEnMom}, and let $C_T$ denote the constant in \eqref{EqSCPEnEnMom2Strong}. Then for all $\eta>0$, there exists a constant $C_\eta$ such that for all $u\in\CIc(S;T^*_S M^\circ)$ with $\supp u\subset K$, we have
  \[
    \int_S \la T^{\mu\nu}\nabla_\mu u,\nabla_\nu u\ra_b \geq (C_T-\eta)\int_S|\nabla u|_{g_R\otimes b}^2\,dx - C_\eta\int_S|u|_b^2\,dx,
  \]
  where $dx$ is the hypersurface measure on $S$ induced by the metric $g$.
\end{lemma}
\begin{proof}
  We first consider a special case: let $x^0=\frakt$, and denote by $x=(x^1,\ldots,x^3)$ local coordinates on $S$; further, choose a local trivialization of the bundle $T^*_S M^\circ$, identifying the fibers with $\R^4$. Suppose then that $b=b(x)$ and $T^{\mu\nu}$ are constant matrices, and that $\nabla_\mu u=\pa_\mu u$ is component-wise differentiation; suppose moreover that $G^{\mu\nu}$ is diagonal (with diagonal entries $1,-1,\ldots,-1$), and that $dx$ is the Lebesgue measure in the local coordinate system on $S$. For $u$ with support in the local coordinate patch, we can then use the Fourier transform on $S$, that is, $\wh u(x^0,\xi)=\int e^{i x\xi}u(x^0,x)\,dx$, and compute (with $i,j=1,\ldots,3$)
  \begin{align*}
    (2\pi)^4&\int_S \la T^{\mu\nu}\nabla_\mu u,\nabla_\nu u\ra_b\,dx \\
      &= \int_S \la T^{00}\wh{\nabla_0 u},\wh{\nabla_0 u}\ra_b\,d\xi + \int_S \la T^{0j}\wh{\nabla_0 u},\xi_j\wh u\ra_b\,d\xi \\
      &\hspace{20ex} + \int_S \la T^{i0}\xi_i\wh u,\wh{\nabla_0 u}\ra_b\,d\xi + \int_S \la T^{ij}\xi_i\wh u,\xi_j\wh u\ra_b\,d\xi \\
      &\geq C_T\int_S |\wh{\nabla_0 u}|_b^2+|\xi|^2|\wh u|^2\,d\xi = C_T (2\pi)^4\int_S |\nabla u|_{g_R\otimes b}^2\,dx,
  \end{align*}
  proving the desired estimate in this case (in fact with $\eta=0$ and $C_\eta=0$).

  To prove the estimate in general for any fixed $\eta>0$, we take a partition of unity $\{\phi_k^2\}$ on $S$, subordinate to coordinate patches, so that $T^{\mu\nu}$, $b$ and $dx$ vary by a small amount over each $\supp\phi_k$, where the required smallness depends on the given $\eta$, as will become clear momentarily. Fixing points $z_k\in\supp\phi_k$, we then have
  \begin{align*}
    \la \phi_k^2 T^{\mu\nu}\nabla_\mu u,\nabla_\nu u\ra &= \big\la T^{\mu\nu}(z_k)\nabla_\mu(\phi_k u),\nabla_\nu(\phi_k u)\big\ra \\
    & \quad + \big\la (T^{\mu\nu}-T^{\mu\nu}(z_k))\nabla_\mu(\phi_k u),\nabla_\nu(\phi_k u)\big\ra \\
    & \quad + \big\la T^{\mu\nu}[\nabla_\mu,\phi_k]u,\nabla_\nu(\phi_k u)\big\ra + \big\la \phi_k T^{\mu\nu}\nabla_\mu u,[\nabla_\nu,\phi_k]u\big\ra.
  \end{align*}
          The integral of the leading part of the first term, i.e.\ without the zeroth order terms in the local coordinate expressions of $\nabla_\mu u$, can then be estimated from below by $C_T\|\nabla(\phi_k u)\|^2$ as above, while the second term yields a small multiple of $\|\nabla(\phi_k u)\|^2$ upon integration provided $\supp\phi_k$ is sufficiently small (since then $T^{\mu\nu}-T^{\mu\nu}(z_k)$ is small on $\supp\phi_k$), and all remaining terms involve at most one derivative of $u$, hence by Cauchy--Schwarz can be estimated by a small constant times $\|\phi_k\nabla u\|^2$ or $\|\nabla(\phi_k u)\|^2$ plus a large constant times $\|\psi_k u\|^2$, where we fix smooth functions $\psi_k$ with $\psi_k\equiv 1$ on $\supp\phi_k$ and so that $\{\supp\psi_k\}$ is locally finite. Summing all these estimates, we obtain
  \begin{align*}
    \int_S&\la T^{\mu\nu}\nabla_\mu u,\nabla_\nu u\ra_b\,dx \\
      &\geq C_T\sum_k\|\nabla(\phi_k u)\|^2 - \eta\sum_k\bigl(\|\nabla(\phi_k u)\|^2+\|\phi_k\nabla u\|^2\bigr) - C_\eta\sum_k\|\psi_k u\|^2 \\
      &\geq (C_T-2\eta)\|\nabla u\|^2 - C_\eta'\|u\|^2
  \end{align*}
  by another application of the Cauchy--Schwarz inequality, finishing the proof.
\end{proof}

We can now prove:
\begin{prop}
\label{PropSCPEn0}
  Fix $0<\frakt_0<\frakt_1$. Then for $u\in\CIc(M^\circ;T^*M^\circ)$ with support in $t_*\geq 0$, we have the energy estimate
  \begin{equation}
  \label{EqSCPEn0}
    \|u\|_{H_\semi^1(\frakt^{-1}((-\infty,\frakt_0]))} \lesssim \semi^{-1}\|\cP_\semi u\|_{L^2(\frakt^{-1}((-\infty,\frakt_1]))},
  \end{equation}
  see Figure~\ref{FigSCPEn0}.
\end{prop}

\begin{figure}[!ht]
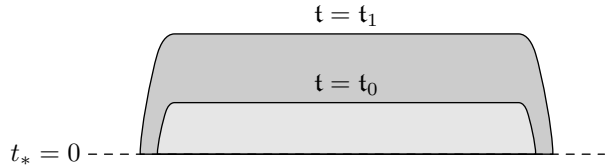

  \centering
  \inclfig{SCPEn0}
  \caption{Illustration of the energy estimate \eqref{EqSCPEn0}, with $u$ vanishing in $t_*<0$. The norm of $\cP_\semi u$ is taken on the union of the shaded regions, and we obtain control of $u$ on the lightly shaded region.}
  \label{FigSCPEn0}
\end{figure}

This follows from:
\begin{prop}
\label{PropSCPEn}
  Fix $0<t_{*,0}<\frakt_0<\frakt_1$. Then for all $u\in\CIc(M^\circ;T^*M^\circ)$, we have the energy estimate
  \begin{equation}
  \label{EqSCPEn}
  \begin{split}
    \|u\|_{H_\semi^1(\frakt^{-1}((-\infty,\frakt_0])\cap t_*^{-1}([t_{*,0},\infty)))} & \lesssim \semi^{-1}\|\cP_\semi u\|_{L^2(\frakt^{-1}((-\infty,\frakt_1])\cap t_*^{-1}([0,\infty)))} \\
      &\qquad\qquad + \|u\|_{H^1_\semi(\frakt^{-1}((-\infty,\frakt_1])\cap t_*^{-1}([0,t_{*,0}]))},
  \end{split}
  \end{equation}
  see Figure~\ref{FigSCPEn}.
\end{prop}

\begin{figure}[!ht]
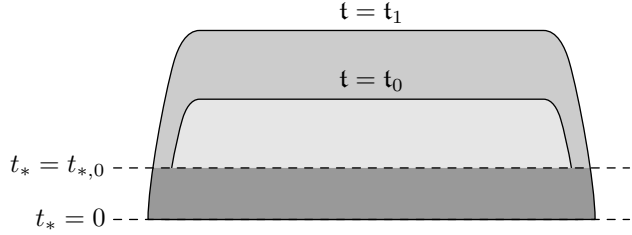

  \centering
  \inclfig{SCPEn}
  \caption{Illustration of the energy estimate \eqref{EqSCPEn}. The norm of $\cP_\semi u$ is taken on the union of the shaded regions, and we obtain control of $u$ on the lightly shaded region, assuming a priori control on the dark region.}
  \label{FigSCPEn}
\end{figure}

Indeed, after shifting $t_*$ and $\frakt$ by $t_{*,0}$, the a priori control term in \eqref{EqSCPEn} vanishes under the support assumption in Proposition~\ref{PropSCPEn0}.

\begin{proof}[Proof of Proposition~\ref{PropSCPEn}]
  We prove \eqref{EqSCPEn} by means of a positive commutator argument, using the compactly supported commutant (or `multiplier')
  \[
    V_\semi = \chi L_\semi,\quad \chi=\chi_1(t_*)\chi_2(\frakt),
  \]
  where $\chi_1$ is a non-negative function with $\chi_1(t_*)\equiv 0$ for $t_*\leq 0$ and $\chi_1(t_*)\equiv 1$ for $t_*\geq t_{*,0}$, while
  \[
    \chi_2(\frakt)=\psi_2\bigl(\digamma^{-1}(\frakt_1-\frakt)\bigr)
  \]
  with $\psi_2(x)=e^{-1/x}H(x)$. Write $\Box_\semi\equiv\BoxCP_{g,\semi}$ for brevity. Then, we consider
  \begin{equation}
  \label{EqSCPEnComm1}
    2 \semi^{-1}\Im\la\cP_\semi u,V_\semi u\ra = 2 \semi^{-1}\Im\la\Box_\semi u,V_\semi u\ra - 2 \semi^{-1}\|\sqrt{\chi}L_\semi u\|^2.
  \end{equation}
  In the first term on the right, we can integrate by parts, obtaining the operator $\frac{i}{\semi}(\Box_\semi^*\chi L_\semi - L_\semi^*\chi\Box_\semi)$, whose form we proceed to describe: first, recall that $\Box_\semi\in \semi^2\Diff^2$, hence $\Box_\semi^*-\Box_\semi\in \semi^2\Diff^1$; similarly, $L_\semi^*-L_\semi\in \semi\CI$. We then note that in local coordinates, we can rewrite $\Box_\semi\chi L_\semi-L_\semi\chi\Box_\semi$ by expanding $\Box_\semi$ into its zeroth and first order terms, plus the scalar principal terms which involve two derivatives; for the latter, we can then write (dropping the zeroth order term $S$ in $L_\semi$ and factoring out $\semi^3$)
  \begin{align*}
    D_\mu G^{\mu\nu}&D_\nu \chi \ell^\kappa D_\kappa - D_\mu\ell^\mu\chi D_\nu G^{\nu\kappa}D_\kappa \\
      &=D_\mu\bigl(\chi(G^{\mu\nu}D_\nu\ell^\kappa - \ell^\mu (D_\nu G^{\nu\kappa})) + G^{\mu\nu}\ell^\kappa(D_\nu\chi)\bigr)D_\kappa,
  \end{align*}
  the point being that one can write this as $\nabla^*A\nabla$, where the components $A^{\mu\nu}$ of $A$ only involve $\chi$ and $d\chi$, but no second derivatives. The upshot is that we have
  \begin{equation}
  \label{EqSCPEnComm2}
  \begin{split}
    \frac{i}{\semi}&(\Box_\semi^*\chi L_\semi-L_\semi^*\chi\Box_\semi) \\
      &= \semi^2\nabla^*(\chi_1\chi_2' T+\chi B)\nabla + \semi^2\nabla^*(\chi_1\chi_2'A_1^\sharp+\chi A_2^\sharp) \\
      &\hspace{20ex}+ \semi^2(\chi_1\chi_2' A_1^\flat+\chi A_2^\flat)\nabla + \semi^2(\chi_1\chi_2' A_3+\chi A_4) \\
      &\quad + \semi^2\nabla^*\chi_1'\chi_2\wt A\nabla + \semi^2\nabla^*\chi_1'\chi_2\wt A_1^\sharp + \semi^2\chi_1'\chi_2\wt A_1^\flat\nabla + \semi^2\chi_1'\chi_2\wt A_3,
  \end{split}
  \end{equation}
  where $T,\wt A,B$ are sections of $\End(T^*M^\circ\otimes T^*M^\circ)$, while $A_1^\sharp,\wt A_1^\sharp,A_2^\sharp$ are sections of $\Hom(T^*M^\circ,\bigotimes^2 T^*M^\circ)$, $A_1^\flat,\wt A_1^\flat,A_2^\flat$ sections of $\Hom(\bigotimes^2 T^*M^\circ,T^*M^\circ)$, and $A_3,\wt A_3$ and $A_4$ sections of $\End(T^*M^\circ)$. Crucially then, we can compute $T$ by a principal symbol calculation. Namely, the principal symbol of $\frac{i}{\semi}(\Box_\semi^*\chi L_\semi-L_\semi^*\chi\Box_\semi)$ is equal to that of
  \[
    \frac{i}{\semi}[\Box_\semi,\chi]L_\semi - \frac{i}{\semi}[L_\semi,\chi]\Box_\semi + \frac{i}{\semi}\chi[\Box_\semi,L_\semi];
  \]
  the sum of those terms in the principal symbol which contain derivatives of $\chi_2$ is equal to $\chi_1\bigl((\ham_G\chi_2)\ell-(\ham_\ell\chi_2)G\bigr)$, and we therefore find\footnote{This is the unique \emph{symmetric} choice of $T$, i.e.\ for which $T^{\mu\nu}=T^{\nu\mu}$.}
  \[
    T^{\mu\nu} = (d\frakt)_\kappa\bigl(G^{\mu\kappa}\ell^\nu + G^{\nu\kappa}\ell^\mu - G^{\mu\nu}\ell^\kappa\bigr),
  \]
  which is an `energy-momentum tensor' of the form \eqref{EqSCPEnEnMom}. Now
  \begin{equation}
  \label{EqSCPEnDeriv}
    \chi_2' = -\digamma(\frakt_1-\frakt)^{-2}\chi_2
  \end{equation}
  is a smooth function of $\frakt$ only; but when estimating, for $0\leq\frakt'\leq\frakt_1$, the integral
  \[
    \semi^2\int_{\frakt=\frakt'} \la\nabla^*\chi_1\chi_2' T\nabla u,u\ra\,dx,
  \]
  we are not quite in the setting of Lemma~\ref{LemmaSCPEnPos} due to the presence of $\chi_1$. Note however that $\chi_1|_{\frakt=\frakt'}$ is uniformly bounded in $\CI$ for $\frakt'\in[0,\frakt_1]$, hence, arranging as we may that $\sqrt{\chi_1}\in\CI$, we can commute $\sqrt{\chi_1}$ past $\nabla$, generating error terms in $\supp d\chi_1$, where we have a priori control (after integrating in $\frakt$), given by the second term on the right hand side of \eqref{EqSCPEn}. Due to \eqref{EqSCPEnDeriv}, we can choose $\digamma>0$ so large that the spacetime integral of $\semi^2\la\nabla^*\chi B\nabla u,u\ra$, estimated simply by Cauchy--Schwarz, can be absorbed by the main term $\nabla^*\chi_1\chi_2' T\nabla$, while the lower order terms in \eqref{EqSCPEnComm2} can be estimated directly by Cauchy--Schwarz, again estimating $\chi_2$ by $\chi_2'$ for the terms involving $\chi$ directly. Since $\chi_2'\leq 0$, we conclude that for any $\eta>0$, we have, for all $\digamma>0$ large enough,
  \begin{equation}
  \label{EqSCPEnRHS}
  \begin{split}
    \frac{i}{\semi}&\big\la(\Box_\semi^*\chi L_\semi-L_\semi^*\chi\Box_\semi)u,u\big\ra \\
      & \leq -(C_T-\eta)\bigl\|(-\chi_1\chi'_2)^{1/2} \semi\nabla u\bigr\|^2 + C_\eta \semi^2\|(-\chi_1\chi_2')^{1/2}u\|^2 \\
      & \qquad + C_\digamma\|u\|_{H^1_\semi(\frakt^{-1}((-\infty,\frakt_1])\cap t_*^{-1}([0,t_{*,0}]))}^2,
  \end{split}
  \end{equation}
  where we estimated the terms involving $\chi_1'$ rather crudely, resulting in the last term on the right, which is the a priori control term in \eqref{EqSCPEn}. We estimate the second term using the a priori control term and the fundamental theorem of calculus: in the present setting, this is conveniently done by considering a principally scalar operator $W=W^*\in\Diffh^1(M^\circ;T^*M^\circ)$ with principal symbol equal to $\sigma_1(\semi D_{t_*})$, and computing the pairing
  \[
    2 \semi^{-1}\Im\la \sqrt{\chi}u,\sqrt{\chi}W u\ra = \frac{1}{i \semi}\la[W,\chi]u,u\ra = -\la(\chi_1'\chi_2+\chi_1\chi_2')u,u\ra,
  \]
  where we used $\pa_{t_*}\frakt=1$ to differentiate $\chi_2$; again using \eqref{EqSCPEnDeriv} and taking $\digamma>0$ large, this implies after applying the Cauchy--Schwarz inequality that
  \begin{equation}
  \label{EqSCPEnFundamental}
    \semi^2\|(-\chi_1\chi_2')^{1/2}u\|^2 \leq C\bigl(\semi^2\|u\|^2_{L^2(\frakt^{-1}((-\infty,\frakt_1])\cap t_*^{-1}([0,t_{*,0}]))} + \|\sqrt{\chi_1\chi_2} \semi\nabla u\|^2\bigr).
  \end{equation}
  Plugging this into the estimate \eqref{EqSCPEnRHS}, we see that we can drop the second term on the right in \eqref{EqSCPEnRHS}, up to changing the constant $C_\digamma$ and increasing $\eta$ by an arbitrarily small but fixed amount if we choose $\digamma>0$ large, thus obtaining
  \begin{equation}
  \label{EqSCPEnRHS2}
  \begin{split}
    \frac{i}{\semi}&\big\la(\Box_\semi^*\chi L_\semi-L_\semi^*\chi\Box_\semi)u,u\big\ra \\
       &\leq -(C_T-\eta)\bigl\|(-\chi_1\chi'_2)^{1/2} \semi\nabla u\bigr\|^2 + C_{\digamma,\eta}\|u\|_{H^1_\semi(\frakt^{-1}((-\infty,\frakt_1])\cap t_*^{-1}([0,t_{*,0}]))}^2
  \end{split}
  \end{equation}
  for any $\eta>0$.

  On the other hand, using the fact that $L_\semi$ in local coordinates is equal to semiclassical derivatives plus zeroth order terms of size $\cO(\semi)$, we can bound the left hand side of \eqref{EqSCPEnComm1} by
  \begin{equation}
  \label{EqSCPEnLHS}
  \begin{split}
    2 \semi^{-1}&\Im\la\cP_\semi u,V_\semi u\ra \geq -\eta\|\sqrt{\chi}L_\semi u\|^2 - C_\eta \semi^{-2}\|\sqrt{\chi}\cP_\semi u\|^2 \\
      &\geq -C\eta\|\sqrt{\chi}\semi\nabla u\|^2-C_\eta \semi^2\|\sqrt{\chi}u\|^2-C_\eta \semi^{-2}\|\sqrt{\chi}\cP_\semi u\|^2;
  \end{split}
  \end{equation}
  estimating the second term by using \eqref{EqSCPEnFundamental} again and absorbing the $\|\sqrt{\chi}\semi\nabla u\|^2$ term into the main term of \eqref{EqSCPEnRHS2}, we obtain the desired estimate \eqref{EqSCPEn} by combining \eqref{EqSCPEnRHS2} and \eqref{EqSCPEnLHS}, noting that for the now fixed value $\digamma>0$, $-\chi_2'$ is bounded from below by a positive constant in $\frakt^{-1}((-\infty,\frakt_0])\cap t_*^{-1}([t_{*,0},\infty))$.
\end{proof}

The arguments presented here extend directly to give a proof of energy estimates beyond the event, resp.\ cosmological horizons, propagating `outwards,' i.e.\ in the direction of decreasing $r$, resp.\ increasing $r$:

\begin{prop}
\label{PropSCPEnR}
  Fix $\tau_0<1$ and $r_+<r_0<r_1<r_2<r_3<r_+ +3\eps_M$, and let $\rho\in\R$. Then for all $u\in\CIc(M^\circ;T^*M^\circ)$, we have the energy estimate
  \begin{equation}
  \label{EqSCPEnR}
  \begin{split}
    \|u\|&_{\Hbh^{1,\rho}(r^{-1}([r_1,r_2])\cap\tau^{-1}([0,\tau_0]))} \\
     &\lesssim \semi^{-1}\|\cP_\semi u\|_{\Hbh^{0,\rho}(r^{-1}([r_0,r_3])\cap\tau^{-1}([0,1]))} + \|u\|_{\Hbh^{1,\rho}(r^{-1}([r_0,r_1])\cap\tau^{-1}([0,1]))} \\
     & \qquad + \|u\|_{H_\semi^1(r^{-1}([r_0,r_3])\cap\tau^{-1}([\tau_0,1]))}
  \end{split}
  \end{equation}
  beyond the cosmological horizon, see Figure~\ref{FigSCPEnR}.
  
  Likewise, we have an estimate beyond the event horizon: if $r_- -3\eps_M<r_0<r_1<r_2<r_3<r_-$, we have the estimate \eqref{EqSCPEnR}, replacing the second term on the right by $\|u\|_{\Hbh^{1,\rho}(r^{-1}([r_2,r_3])\cap\tau^{-1}([0,1]))}$.
\end{prop}

\begin{figure}[!ht]
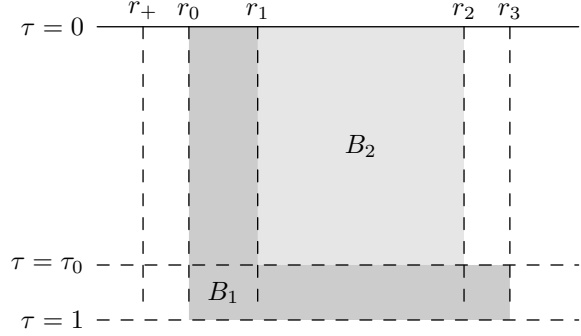

  \centering
  \inclfig{SCPEnR}
  \caption{Illustration of the energy estimate \eqref{EqSCPEnR} beyond the cosmological horizon $r=r_+$. The norm of $\cP_\semi u$ is taken on $r^{-1}([r_0,r_3])\cap\tau^{-1}([0,1])$, and we obtain control of $u$ on the lightly shaded region, assuming a priori control on the dark region.}
  \label{FigSCPEnR}
\end{figure}

\begin{proof}
  In order to eliminate the weight $\rho$ and work on unweighted spaces, one proves these estimates for the conjugated operator $\tau^\rho\cP_\semi\tau^{-\rho}$. Then, for the proof of \eqref{EqSCPEnR}, one uses the commutant $\chi L_\semi$, $\chi=\chi_1\chi_2\chi_3$, where now $\chi_1(\tau)\equiv 1$ for $\tau\leq\tau_0$ and $\chi_1(\tau)\equiv 0$ for $\tau\geq 1$, while $\chi_2(r)=\psi_2(\digamma^{-1}(r_3-r))$, with $\psi_2(x)=e^{-1/x}H(x)$ as before, and $\chi_3(r)\equiv 0$ for $r\leq r_0$ and $\chi_3(r)\equiv 1$ for $r\geq r_1$. The regions $\supp(\chi_1\chi_3')$ and $\supp(\chi_1'\chi_3)$ are where we assume a priori control, corresponding to the second and third term in \eqref{EqSCPEnR}, while the future timelike nature of $dr$ on $\supp\chi$, combined with the b-version of Lemma~\ref{LemmaSCPEnEnMom}, gives the conclusion on $u$ in $r^{-1}([r_1,r_2])\cap\tau^{-1}([0,\tau_0])$ for $\digamma>0$ large and fixed (used to dominate $\chi_2$ by a small constant times $-\chi_2'$), using the positive lower bound on $-\chi_2'$ in this region.
\end{proof}

Using the propagation of singularities, Propositions~\ref{PropSCPHiRealPrType} and \ref{PropSCPLoPropRad}, one can improve these estimates to allow for arbitrary real regularity, as we will indicate in the next section.

\subsection{Global estimates}
\label{SubsecSCPGl}

We now piece together the estimates obtained in the previous sections to establish an a priori estimate for $u$ solving $\cP_\semi u=f$ on \emph{decaying} b-Sobolev spaces on the domain $\Omega$ defined in \eqref{EqKdSWaveDomain}.

To do so, fix a weight $\alpha$ as in Proposition~\ref{PropSCPLoPropTime}. Suppose we are given $u\in\CIc(\Omega;\Tb^*_\Omega M)$ vanishing to infinite order at $\Sigma_0$, and let $f=\cP_\semi u$. For $\delta\in\R$ small, let
\[
  \Omega_\delta = [0,1]_\tau\times[r_--\eps_M-\delta,r_++\eps_M+\delta]_r\times\Sph^2 \subset M
\]
be a small modification of $\Omega$, so $\Omega_0=\Omega$, and $\Omega_{\delta_1}\subseteq\Omega_{\delta_2}$ if $\delta_1\leq\delta_2$. Let $\wt\Omega:=\Omega_{\eps_M/2}$. Denote by $\wt f\in\CIc(\wt\Omega;\Tb^*_{\wt\Omega}M)$ an extension of $f$, vanishing for $t_*\leq 0$, and with
\begin{equation}
\label{EqSCPGlExtF}
  \|\wt f\|_{\Hb^{0,\alpha}(\wt\Omega)^{\bullet,-}}\leq 2\|f\|_{\Hb^{0,\alpha}(\Omega)^{\bullet,-}}.
\end{equation}
We can then uniquely solve the forward problem
\[
  \cP_\semi\wt u=\wt f
\]
in $\wt\Omega$. If $u'$ is any smooth compactly supported extension of $u$ to $\wt\Omega$, then $\cP_\semi(\wt u-u')$ is supported in $\{r\geq r_++\eps_M\}\cup\{r\leq r_--\eps_M\}$, hence by the support properties of forward solutions of $\cP_\semi$, we find that $\wt u-u'$ has compact support, and moreover $\wt u\equiv u$ in $\Omega$.

Now, by the energy estimate near $\Sigma_0$, Proposition~\ref{PropSCPEn0}, we have
\[
  \|u\|_{H^1_\semi(\Omega_\delta \cap t_*^{-1}([0,1]))} \leq C \semi^{-1}\|\wt f\|_{L^2(t_*^{-1}([0,2]))}
\]
for fixed $\delta\in(0,\eps_M/2)$. But then we can use this information to propagate $\Hbh^{1,\alpha}$ regularity of $u$: at fiber infinity, this uses Propositions~\ref{PropSCPHiRealPrType}, \ref{PropSCPHiRad} and \ref{PropSCPHiTrap}, while we can use Proposition~\ref{PropSCPLoPropTime} to obtain an estimate on $u$ near the critical set $r=r_c$ of $\nabla t_*$, and propagate this control outwards in the direction of increasing $r$ for  $r>r_c$ and decreasing $r$ for $r<r_c$ by means of Proposition~\ref{PropSCPLoPropRad}. Away from the semiclassical characteristic set of $\cP_\semi$, we simply use elliptic regularity. We obtain an estimate
\begin{equation}
\label{EqSCPGlEst0}
  \|u\|_{\Hbh^{1,\alpha}(\Omega_{-\delta})} + \|u\|_{H_\semi^1(\Omega_\delta\cap t _*^{-1}([0,1]))} \leq C\bigl(\semi^{-1}\|\wt f\|_{\Hb^{0,\alpha}(\wt\Omega)} + \semi\|u\|_{\Hbh^{1,\alpha}(\Omega)}\bigr)
\end{equation}
for small $\semi>0$. The error term in $u$ here is measured on a larger set than the conclusion on the left hand side, so we now use the energy estimate beyond the horizons, Proposition~\ref{PropSCPEnR}, in order to bound
\[
  \|u\|_{\Hbh^{1,\alpha}(\Omega)} \leq C\bigl(\semi^{-1}\|\wt f\|_{\Hb^{0,\alpha}(\wt\Omega)}+\|u\|_{\Hbh^{1,\alpha}(\Omega_{-\delta})} + \|u\|_{H_\semi^1(\Omega_\delta\cap t _*^{-1}([0,1])}\bigr).
\]
Plugging \eqref{EqSCPGlEst0} into this estimate and choosing $0<\semi<\semi_0$ small, we can thus absorb the term $\semi\|u\|_{\Hbh^{1,\alpha}(\Omega)}$ from \eqref{EqSCPGlEst0} into the left hand side, and in view of \eqref{EqSCPGlExtF}, we obtain the desired a priori estimate
\begin{equation}
\label{EqSCPGlEst}
  \|u\|_{\Hbh^{1,\alpha}(\Omega)^{\bullet,-}} \leq C \semi^{-1}\|\cP_\semi u\|_{\Hb^{0,\alpha}(\Omega)^{\bullet,-}}, \quad 0<\semi<\semi_0,
\end{equation}
which by a simple approximation argument continues to hold for all $u\in\Hb^{1,\alpha}(\Omega)^{\bullet,-}$ for which $\cP_\semi u\in\Hb^{0,\alpha}(\Omega)^{\bullet,-}$.

Let us fix $0<\semi<\semi_0$ and drop the subscript `$\semi$.' There are a number of ways in which \eqref{EqSCPGlEst} can be used to rule out resonances of $\cP$ in $\Im\sigma\geq-\alpha$. One way is to notice that the a priori estimate \eqref{EqSCPGlEst} for $\cP$ yields the solvability of the adjoint $\cP^*$ (the adjoint taken with respect to the fiber inner product $b$) on growing function spaces by a standard application of the Hahn--Banach theorem (see e.g.\ \cite[Proof of Theorem~26.1.7]{HormanderAnalysisPDE4}); concretely, there is a bounded inverse
\begin{equation}
\label{EqSCPGlPAdjInverse}
  (\cP^*)^{-1} \colon \Hb^{-1,-\alpha}(\Omega;\Tb^*_\Omega M)^{-,\bullet} \to \Hb^{0,-\alpha}(\Omega;\Tb^*_\Omega M)^{-,\bullet}
\end{equation}
for the backwards problem. Now if $\sigma$ with $\Im\sigma>-\alpha$ were a resonance of $\cP$, then there would exist a dual resonant state $\psi\in\sD'(Y;\Tb^*_Y M)^\bullet$ with $\cP^*(\tau^{i\ol\sigma}\psi)=0$, and in fact by the radial point arithmetic (see the proof of Proposition~\ref{PropSCPHiRad}), we have $\psi\in L^2(Y;\Tb^*_Y M)^\bullet$ if our fixed $\semi>0$ is small enough. Letting $\chi(x)$ denote a smooth cutoff, $\chi\equiv 0$ for $x\leq 0$ and $\chi\equiv 1$ for $x\geq 1$, we put
\[
  v_j := \chi(j-t_*)\tau^{i\ol\sigma}\psi,\quad g_j := \cP^*v_j=[\cP^*,\chi(j-\cdot)]\tau^{i\ol\sigma}\psi.
\]
Then $v_j$ is the unique backwards solution of $\cP^*v=g_j$, so $v_j=(\cP^*)^{-1}g_j$; however $g_j\to 0$ in $\Hb^{-1,-\alpha}(\Omega;\Tb^*\Omega)^{-,\bullet}$ as $j\to\infty$, while $\|v_j\|_{\Hb^{0,-\alpha}}$ converges to a non-zero number. This contradicts the boundedness of \eqref{EqSCPGlPAdjInverse}, and establishes Theorem~\ref{ThmSCP} after reducing $\alpha>0$ by an arbitrarily small positive amount.

Another, somewhat more direct way of proving Theorem~\ref{ThmSCP} proceeds as follows: using the same arguments as above for $\cP_\semi$, but in reverse, one can prove an estimate for $\cP_\semi^*$ of the form
\begin{equation}
\label{EqSCPGlEstAdj}
  \|u\|_{\Hbh^{1,-\alpha}(\Omega)^{-,\bullet}} \leq C \semi^{-1}\|\cP_\semi^* u\|_{\Hbh^{0,-\alpha}(\Omega)^{-,\bullet}},\quad 0<\semi<\semi_0;
\end{equation}
this relies on versions of the propagation estimates proved in \S\ref{SubsecSCPHi} and \S\ref{SubsecSCPLo} in which the direction of propagation is reversed; since passing from $\cP_\semi=\BoxCP_{g,\semi}-i L_\semi$ to $\cP_\semi^*=(\BoxCP_{g,\semi})^*+i L_\semi^*$ effects a change of sign in the skew-adjoint part, the adjoint version of the crucial Proposition~\ref{PropSCPLoPropTime} now requires that the weight satisfy $\rho\geq-\alpha$. The estimate \eqref{EqSCPGlEstAdj} then gives the solvability of the forward problem  for $\cP_\semi$ on the dual spaces, which are spaces of decaying functions. Concretely, we obtain a forward solution operator
\[
  \cP_\semi^{-1}\colon \Hb^{-1,\alpha}(\Omega;\Tb^*_\Omega M)^{\bullet,-} \to \Hb^{0,\alpha}(\Omega;\Tb^*_\Omega M)^{\bullet,-},\quad 0<\semi<\semi_0.
\]
(Using elliptic regularity and propagation estimates, one also has $\cP_\semi^{-1}\colon\Hb^{s-1,\alpha}\to\Hb^{s,\alpha}$ for $s\geq 0$, or in fact for any fixed $s\in\R$ provided $\semi>0$ is sufficiently small.) Since the forward problem for $\cP_\semi$ is \emph{uniquely} solvable, the operator $\cP_\semi$ cannot have resonances with $\Im\sigma>-\alpha$, since otherwise solutions of $\cP_\semi u=f$ for suitable (generic) $f\in\CIc$ would have an asymptotic expansion with such a resonant state appearing with a non-zero coefficient, contradicting the fact that the unique solution $u$ lies in the space $\Hb^{0,\alpha}$.

This concludes the proof of Theorem~\ref{ThmSCP}.

\begin{rmk}
\label{RmkSCPSusp}
  The argument we presented above in some sense does more than what is strictly necessary; after all we only want to rule out resonances of the normal operator of $\wtBoxCP_g$ in the closed upper half plane, not study the solvability properties of $\wtBoxCP_g$, though the two are closely related. Thus, the `right' framework would be to work fully on the Mellin transform side, where one would have two large parameters, $\sigma$ and $\gamma=\gamma_1$, which can be thought of as a joint parameter $(\sigma,\gamma)$ lying in a region of $\C\times\R$. (This is also related to the notion of a `suspended algebra' in the sense of Mazzeo and Melrose \cite{MazzeoMelroseFibred}.) In this case one could use complex absorption around $r=r_--\eps_M$ and $r=r_++\eps_M$, as was done in \cite{VasyMicroKerrdS}, without the need for the initial or final hypersurfaces for Cauchy problems. Energy estimates still would play a minor role, as in \cite[\S3.3]{VasyMicroKerrdS}, to ensure that in a neighborhood of the black hole exterior, the resonant states are independent of the particular `capping' used (complex absorption vs.\ Cauchy hypersurfaces). We remark that if the final Cauchy surface is still used (rather than complex absorption), the proof of the absence of resonances in the closed upper half plane for the Mellin transformed normal operator $\wh N(\wtBoxCP_g)$ can be obtained by taking $u=e^{-i\sigma t_*}v$ in the above arguments, where $v$ is a function on $Y=\Omega\cap X$, and only integrating in $X$ (not in $M$) in the various pairings, dropping any cutoffs or weights in $t_*$ ($\Im\sigma$ plays the role of weights, contributing to the skew-adjoint part of $\wh N(\wtBoxCP_g)$ in the arguments), with the b-Sobolev spaces thus being replaced by large-parameter versions of standard Sobolev spaces. In any case, we hope that not introducing further microlocal analysis machinery, but rather working on $\Omega$ directly, makes this section more accessible.
\end{rmk}

\section{Spectral gap for the linearized gauged Einstein equation (ESG)}
\label{SecESG}

As in \eqref{EqOpNoB0}, we drop the subscript `$b_0$' from the metric since we will be only considering the fixed Schwarzschild--de~Sitter metric $g_{b_0}$.

\subsection{Microlocal structure at the trapped set}
\label{SubsecESGTrap}

We now analyze the high frequency behavior of the linearized gauged Einstein operator, modified so as to arrange SCP. Thus, we consider
\[
  L = D_g(\Ric+\Lambda) + \tdel^*\delta_g \sfG_g = \frac{1}{2}\bigl(\Box_g + 2\Lambda + 2\sR_g + 2(\tdel^*-\delta_g^*)\delta_g \sfG_g\bigr),
\]
where $\tdel^*$ is the modified symmetric gradient \eqref{EqSCPDelStarTilde}, and $\sR_g$ the $0$-th order curvature term \eqref{EqHypDTCurvatureTerm}. In order for $L$ to satisfy polynomial high energy bounds on (and in a strip below) the real axis, we only need to check a condition on the subprincipal part of $L$ at the trapped set $\Gamma$, see \eqref{EqKdSGeoTrapped}--\eqref{EqKdSGeoTrappedPM} for the definition. Concretely, it suffices to show that for our choice of $\tdel^*$, the skew-adjoint part of the subprincipal operator $S_\sub(L)$, evaluated at a point $(t,r,\omega;\sigma,0,\eta)\in\Gamma$, with respect to a suitable inner product on $\pi^*S^2T^*M^\circ$, has all eigenvalues in the half space $\{i\sigma\lambda \colon \Re\lambda\geq 0\}$, where we use the coordinates \eqref{EqKdSGeoTrapCovec}. (Note that the skew-adjoint part of $S_\sub$ is a smooth section of $\pi^*\End(S^2T^*M^\circ)$, homogeneous of degree $1$ in $(\sigma,\eta)$.) Namely, in this case, one can choose a stationary inner product on $\pi^*S^2T^*M^\circ$ such that the skew-adjoint part of $|\sigma|^{-1}S_\sub(L)$ (with respect to this inner product) is bounded above on $\Gamma^+=\Gamma\cap\{\sigma<0\}$ (which is a subset of the \emph{forward} light cone) by an arbitrarily small positive multiple of the identity, and bounded below on $\Gamma^-=\Gamma\cap\{\sigma>0\}$ (which is a subset of the \emph{backward} light cone) by an arbitrarily small negative multiple of the identity; see also \eqref{EqAsySm3TrappedSubpr}. This gives the desired high energy estimates by combining Dyatlov's result \cite{DyatlovSpectralGaps} with the framework of pseudodifferential inner products \cite{HintzPsdoInner}, as explained in \S\ref{SubsecKeyESG} and \S\ref{SubsecAsySm}.

The calculation of the subprincipal symbol of the wave operator $\Box_g$ on symmetric 2-tensors in the partial trivialization \eqref{EqOpSpatBundlesFull} is straightforward using \cite[Proposition~4.1]{HintzPsdoInner}; this states that the subprincipal operator is equal to the covariant derivative $-i\nabla^{\pi^*S^2T^*M^\circ}_{\ham_G}$, defined using the pullback connection on $\pi^*S^2T^*M^\circ$. Since the latter is simply the restriction of the product connection $\nabla^{\pi^*T^*M^\circ}\otimes\nabla^{\pi^*T^*M^\circ}$ to $\pi^*S^2T^*M^\circ$, we conclude that
\[
  S_\sub(\Box_g)(w_1 w_2) = S_\sub(\Box_g^{(1)}) w_1\cdot w_2 + w_1\cdot S_\sub(\Box_g^{(1)}) w_2
\]
for $w_1,w_2\in\CI(T^*M^\circ,\pi^*T^*M^\circ)$; $\Box^{(1)}_g$ on the right denotes the wave operator on 1-forms, whose subprincipal symbol at the trapped set we computed in \eqref{EqOpSdSTrapSubpr0}. We can now calculate the $0$-th order part $S_{(2)}$ of $S_\sub(\Box_g)$ in the bundle splitting \eqref{EqOpSpatBundlesFull}; indeed, $S_\sub(\Box_g)$ in this splitting has a canonical first order part, induced by the canonical first order part of $S_\sub(\Box_g^{(1)})$ which is the first line in \eqref{EqOpSdSTrapSubpr0}, i.e.\ the first order part involves precisely $t$-derivatives (which uses the stationary nature of the spacetime metric and the relevant vector bundles) and covariant derivatives on $\Sph^2$. Thus, $S_{(2)}$ is equal to the second symmetric tensor power of the $0$-th order part of $S_\sub(\Box_g^{(1)})$, i.e.\ of the final term in \eqref{EqOpSdSTrapSubpr0}, so
\begin{equation}
\label{EqESGWaveSubpr}
  S_{(2)}
  =i\begin{pmatrix}
     0              & -4r^{-1}\sigma & 0                 & 0              & 0                 & 0                 \\
     -2r^{-1}\sigma & 0              & -2 q r^{-3}i_\eta & -2r^{-1}\sigma & 0                 & 0                 \\
     0              & 2 q r^{-1}\eta & 0                 & 0              & -2r^{-1}\sigma    & 0                 \\
     0              & -4r^{-1}\sigma & 0                 & 0              & -4 q r^{-3}i_\eta & 0                 \\
     0              & 0              & -2r^{-1}\sigma    & 2 q r^{-1}\eta & 0                 & -2 q r^{-3}i_\eta \\
     0              & 0              & 0                 & 0              & 4 q r^{-1}\eta    & 0
   \end{pmatrix},
\end{equation}
where we recall $q=q(r)=(1-2\bhm/r-\Lambda r^2/3)^{1/2}$ from~\S\ref{SubsecOpSdS}.

In order to calculate the spectrum of the $0$-th order term of $S_\sub(2 L)$ at $\Gamma$ efficiently, we split
\begin{equation}
\label{EqESGTSSplit}
  T^*\Sph^2=\la q\sigma^{-1}\eta\ra \oplus \eta^\perp,
\end{equation}
and also use the induced splitting
\begin{equation}
\label{EqESGS2TSSplit}
  S^2T^*\Sph^2 = \la (q\sigma^{-1}\eta)^2 \ra \oplus 2 q\sigma^{-1}\eta\cdot\eta^\perp \oplus S^2 \eta^\perp
\end{equation}
of $S^2T^*\Sph^2$. Thus, recalling $e^0=q\,dt$ and $e^1=q^{-1}\,dr$, we refine \eqref{EqOpSpatBundlesFull} to the splitting
\begin{equation}
\label{EqESGS2TMMicrolocalSplit}
\begin{split}
  S^2T^*M^\circ &= \la e^0e^0\ra \oplus \la 2e^0e^1 \ra \oplus \Bigl( \la 2e^0 q\sigma^{-1}\eta \ra \oplus 2e^0\cdot\eta^\perp \Bigr) \\
   & \quad \oplus \la e^1e^1\ra \oplus \Bigl(\la 2e^1 q\sigma^{-1}\eta \ra \oplus 2e^1\cdot\eta^\perp \Bigr) \\
   & \quad \oplus \Bigl( \la (q\sigma^{-1}\eta)^2 \ra \oplus 2 q\sigma^{-1}\eta\cdot\eta^\perp \oplus S^2 \eta^\perp \Bigr).
\end{split}
\end{equation}
Now, in the decomposition \eqref{EqESGTSSplit} and using $|\eta|^2=q^{-2}r^2\sigma^2$ at $\Gamma$, one sees that $\eta\colon\R\to T^*\Sph^2$ and $i_\eta\colon T^*\Sph^2\to\R$ are given by
\[
  \eta = \begin{pmatrix}q^{-1}\sigma\\0\end{pmatrix},\quad i_\eta=\begin{pmatrix}q^{-1}r^2\sigma&0\end{pmatrix},
\]
which gives
\begin{gather*}
  \eta
  =\begin{pmatrix}
     q^{-1}\sigma & 0 \\
     0            & \frac{1}{2}q^{-1}\sigma \\
     0            & 0
   \end{pmatrix}
      \colon T^*\Sph^2\to S^2T^*\Sph^2, \\
  i_\eta
  =\begin{pmatrix}
     q^{-1}r^2\sigma & 0               & 0 \\
     0               & q^{-1}r^2\sigma & 0
   \end{pmatrix}
      \colon S^2T^*\Sph^2\to T^*\Sph^2
\end{gather*}
in the splittings \eqref{EqESGTSSplit} and \eqref{EqESGS2TSSplit}. Furthermore, we write
\begin{equation}
\label{EqESGTSMetric}
  \slg = |\eta|^{-2}\eta\cdot\eta + r^{-2}\slg_\perp = \begin{pmatrix} r^{-2} \\ 0 \\ r^{-2}\slg_\perp \end{pmatrix},
\end{equation}
where $r^{-2}\slg_\perp$ is defined as the orthogonal projection of $\slg$ to $S^2\eta^\perp$. Consequently,
\begin{equation}
\label{EqESGTSTrace}
  \sltr = \begin{pmatrix} r^2 & 0 & r^2\tr_{\slg_\perp} \end{pmatrix} \colon S^2T^*\Sph^2 \to \R.
\end{equation}
Since $\dim\Sph^2=2$, we can pick $\psi\in\eta^\perp\subset T^*\Sph^2$ such that $\slg_\perp=\psi\cdot\psi$; trivializing $S^2\eta^\perp$ via $\psi\cdot\psi$ therefore amounts to dropping $\slg_\perp$ and $\tr_{\slg_\perp}$ in \eqref{EqESGTSMetric}--\eqref{EqESGTSTrace}. We can be more explicit: on $\Sph^2$, we have the Hodge star operator $\star\colon T\Sph^2\to T\Sph^2$, which induces $\star\colon T^*\Sph^2\to T^*\Sph^2$, and we can then let $\psi=|\eta|^{-1} (\star\eta)$.

In terms of \eqref{EqESGS2TMMicrolocalSplit}, we therefore find
\[
  S_{(2)}=i r^{-1}\sigma
  \openbigpmatrix{2.3pt}
    0  & -4 & 0  & 0  & 0  & 0  & 0  & 0  & 0  & 0 \\
    -2 & 0  & -2 & 0  & -2 & 0  & 0  & 0  & 0  & 0 \\
    0  & 2  & 0  & 0  & 0  & -2 & 0  & 0  & 0  & 0 \\
    0  & 0  & 0  & 0  & 0  & 0  & -2 & 0  & 0  & 0 \\
    0  & -4 & 0  & 0  & 0  & -4 & 0  & 0  & 0  & 0 \\
    0  & 0  & -2 & 0  & 2  & 0  & 0  & -2 & 0  & 0 \\
    0  & 0  & 0  & -2 & 0  & 0  & 0  & 0  & -2 & 0 \\
    0  & 0  & 0  & 0  & 0  & 4  & 0  & 0  & 0  & 0 \\
    0  & 0  & 0  & 0  & 0  & 0  & 2  & 0  & 0  & 0 \\
    0  & 0  & 0  & 0  & 0  & 0  & 0  & 0  & 0  & 0
  \closebigpmatrix.
\]

Proceeding with the computation of $S_\sub(2 L)$, we next compute the form of the operator $\tdel^*-\delta_g^*$, with $\tdel^*$ as in \eqref{EqSCPDelStarTilde}. We write $t_*=t-F(r)$, then $dt_*=q^{-1}e^0-q F' e^1$ and $\nabla t_*=q^{-1}e_0+q F' e_1$, so we have
\[
  \tdel^* u = \delta_g^* u + \gamma_1(q^{-1}e^0\cdot u-q F' e^1\cdot u) - \gamma_2(q^{-1}u_0+q F' u_1)g.
\]
In  the decomposition~\eqref{EqOpWarpedBundles}, we find
\begin{align*}
  \tdel^*-\delta_g^*
    &= \gamma_1 q^{-1}
        \begin{pmatrix}
          1 & 0 \\
          0 & \frac{1}{2} \\
          0 & 0
        \end{pmatrix}
     -\gamma_1 q F'
        \begin{pmatrix}
          0              & 0 \\
          \frac{1}{2}e^1 & 0 \\
          0              & e^1
        \end{pmatrix} \\
     &\quad
    - \gamma_2 q^{-1}
      \begin{pmatrix}
        1  & 0 \\
        0  & 0 \\
        -h & 0
      \end{pmatrix}
    - \gamma_2 q F'
      \begin{pmatrix}
        0 & i_{e_1} \\
        0 & 0       \\
        0 & -h i_{e_1}
      \end{pmatrix}.
\end{align*}
Using \eqref{EqOpWarpedDelgGg}, this gives
\begin{equation}
\label{EqESGAdditional}
\begin{split}
  2&(\tdel^*-\delta_g^*)\delta_g \sfG_g
    = \gamma_1 q^{-1}
      \begin{pmatrix}
        -e_0                       & -2 q^{-2}\delta_h q^2 & -e_0\tr_h                               \\
        \frac{1}{2}q^{-2}d_\cX q^2 & -e_0                  & -q^{-1}\delta_h q-\frac{1}{2}d_\cX\tr_h \\
        0                          & 0                     & 0
      \end{pmatrix} \\
    &\quad -\gamma_1 q F'
       \begin{pmatrix}
         0                   & 0                         & 0                        \\
         -\frac{1}{2}e^1 e_0 & -e^1 q^{-2}\delta_h q^2   & -\frac{1}{2}e^1 e_0\tr_h \\
         e^1 q^{-2}d_\cX q^2 & -2 e^1 e_0                & -e^1(2 q^{-1}\delta_h q+d_\cX\tr_h)
       \end{pmatrix} \\
     &\qquad -\gamma_2 q^{-1}
        \begin{pmatrix}
          -e_0  & -2 q^{-2}\delta_h q^2  & -e_0\tr_h  \\
          0     & 0                      & 0          \\
          h e_0 & 2 h q^{-2}\delta_h q^2 & h e_0\tr_h
        \end{pmatrix} \\
     &\quad - \gamma_2 q F'
        \begin{pmatrix}
          i_{e_1}q^{-2}d_\cX q^2    & -2i_{e_1}e_0  & -i_{e_1}(2 q^{-1}\delta_h q+d_\cX\tr_h) \\
          0                         & 0             & 0                                       \\
          -h i_{e_1}q^{-2}d_\cX q^2 & 2h i_{e_1}e_0 & h i_{e_1}(2 q^{-1}\delta_h q+d_\cX\tr_h)
        \end{pmatrix},
\end{split}
\end{equation}
so the contribution of this term to the subprincipal operator of $2L$ at $\Gamma$ in the full splitting \eqref{EqOpSpatBundlesFull} is given by its principal symbol, which we evaluate with the help of Lemma~\ref{LemmaOpSpat}:
\begin{align*}
  \wt S&:=\sigma_1(2(\tdel^*-\delta_g^*)\delta_g \sfG_g) \\
    &= i\gamma_1 q^{-1}
       \begin{pmatrix}
         q^{-1}\sigma    & 0            & 2r^{-2}i_\eta & q^{-1}\sigma     & 0            & q^{-1}r^{-2}\sigma\sltr             \\
         0               & q^{-1}\sigma & 0             & 0                & r^{-2}i_\eta & 0                                   \\
         \frac{1}{2}\eta & 0            & q^{-1}\sigma  & -\frac{1}{2}\eta & 0            & r^{-2}(i_\eta-\frac{1}{2}\eta\sltr) \\
         0               & 0            & 0             & 0                & 0            & 0                                   \\
         0               & 0            & 0             & 0                & 0            & 0                                   \\
         0               & 0            & 0             & 0                & 0            & 0
       \end{pmatrix} \\
    &\quad - i\gamma_1 q F'
       \begin{pmatrix}
         0                       & 0              & 0            & 0                       & 0             & 0                                   \\
         \frac{1}{2}q^{-1}\sigma & 0              & r^{-2}i_\eta & \frac{1}{2}q^{-1}\sigma & 0             & \frac{1}{2}q^{-1}r^{-2}\sigma\sltr  \\
         0                       & 0              & 0            & 0                       & 0             & 0                                   \\
         0                       & 2 q^{-1}\sigma & 0            & 0                       & 2r^{-1}i_\eta & 0                                   \\
         \frac{1}{2}\eta         & 0              & q^{-1}\sigma & -\frac{1}{2}\eta        & 0             & r^{-2}(i_\eta-\frac{1}{2}\eta\sltr) \\
         0                       & 0              & 0            & 0                       & 0             & 0
       \end{pmatrix} \\
    &\quad - i\gamma_2 q^{-1}
       \begin{pmatrix}
         q^{-1}\sigma         & 0 & 2r^{-2}i_\eta  & q^{-1}\sigma         & 0 & q^{-1}r^{-2}\sigma\sltr  \\
         0                    & 0 & 0              & 0                    & 0 & 0                        \\
         0                    & 0 & 0              & 0                    & 0 & 0                        \\
         -q^{-1}\sigma        & 0 & -2r^{-2}i_\eta & -q^{-1}\sigma        & 0 & -q^{-1}r^{-2}\sigma\sltr \\
         0                    & 0 & 0              & 0                    & 0 & 0                        \\
         -q^{-1}r^2\sigma\slg & 0 & -2\slg i_\eta  & -q^{-1}r^2\sigma\slg & 0 & -q^{-1}\sigma \slg\sltr
       \end{pmatrix} \\
    &\quad - i\gamma_2 q F'
       \begin{pmatrix}
         0 & 2 q^{-1}\sigma     & 0 & 0 & 2r^{-2}i_\eta  & 0 \\
         0 & 0                  & 0 & 0 & 0              & 0 \\
         0 & 0                  & 0 & 0 & 0              & 0 \\
         0 & -2 q^{-1}\sigma    & 0 & 0 & -2r^{-2}i_\eta & 0 \\
         0 & 0                  & 0 & 0 & 0              & 0 \\
         0 & -2 q^{-1}r^2\sigma & 0 & 0 & -2\slg i_\eta  & 0
       \end{pmatrix}.
\end{align*}
In terms of \eqref{EqESGS2TMMicrolocalSplit} then, and defining $\gamma_j'=q^{-2}r\gamma_j$, $\gamma_j''=\gamma_j r F'$, for $j=1,2$, one finds that $\wt S$ equals $i r^{-1}\sigma$ times
\begin{align*}
  &\frac{\gamma_1'}{2}
          \openbigpmatrix{2pt}
            2 & 0 & 4 & 0 & 2  & 0 & 0 & 2 & 0 & 2  \\
            0 & 2 & 0 & 0 & 0  & 2 & 0 & 0 & 0 & 0  \\
            1 & 0 & 2 & 0 & -1 & 0 & 0 & 1 & 0 & -1 \\
            0 & 0 & 0 & 2 & 0  & 0 & 0 & 0 & 2 & 0  \\
            0 & 0 & 0 & 0 & 0  & 0 & 0 & 0 & 0 & 0  \\
            0 & 0 & 0 & 0 & 0  & 0 & 0 & 0 & 0 & 0  \\
            0 & 0 & 0 & 0 & 0  & 0 & 0 & 0 & 0 & 0  \\
            0 & 0 & 0 & 0 & 0  & 0 & 0 & 0 & 0 & 0  \\
            0 & 0 & 0 & 0 & 0  & 0 & 0 & 0 & 0 & 0  \\
            0 & 0 & 0 & 0 & 0  & 0 & 0 & 0 & 0 & 0
          \closebigpmatrix
       -\frac{\gamma_1''}{2}
          \openbigpmatrix{2pt}
            0 & 0 & 0 & 0 & 0  & 0 & 0 & 0 & 0 & 0  \\
            1 & 0 & 2 & 0 & 1  & 0 & 0 & 1 & 0 & 1  \\
            0 & 0 & 0 & 0 & 0  & 0 & 0 & 0 & 0 & 0  \\
            0 & 0 & 0 & 0 & 0  & 0 & 0 & 0 & 0 & 0  \\
            0 & 4 & 0 & 0 & 0  & 4 & 0 & 0 & 0 & 0  \\
            1 & 0 & 2 & 0 & -1 & 0 & 0 & 1 & 0 & -1 \\
            0 & 0 & 0 & 2 & 0  & 0 & 0 & 0 & 2 & 0  \\
            0 & 0 & 0 & 0 & 0  & 0 & 0 & 0 & 0 & 0  \\
            0 & 0 & 0 & 0 & 0  & 0 & 0 & 0 & 0 & 0  \\
            0 & 0 & 0 & 0 & 0  & 0 & 0 & 0 & 0 & 0
          \closebigpmatrix \\
        &+\gamma_2'
          \openbigpmatrix{2.7pt}
            -1 & 0 & -2 & 0 & -1 & 0 & 0 & -1 & 0 & -1 \\
            0  & 0 & 0  & 0 & 0  & 0 & 0 & 0  & 0 & 0  \\
            0  & 0 & 0  & 0 & 0  & 0 & 0 & 0  & 0 & 0  \\
            0  & 0 & 0  & 0 & 0  & 0 & 0 & 0  & 0 & 0  \\
            1  & 0 & 2  & 0 & 1  & 0 & 0 & 1  & 0 & 1  \\
            0  & 0 & 0  & 0 & 0  & 0 & 0 & 0  & 0 & 0  \\
            0  & 0 & 0  & 0 & 0  & 0 & 0 & 0  & 0 & 0  \\
            1  & 0 & 2  & 0 & 1  & 0 & 0 & 1  & 0 & 1  \\
            0  & 0 & 0  & 0 & 0  & 0 & 0 & 0  & 0 & 0  \\
            1  & 0 & 2  & 0 & 1  & 0 & 0 & 1  & 0 & 1 
          \closebigpmatrix
        +\gamma_2''
          \openbigpmatrix{2pt}
            0 & -2 & 0 & 0 & 0 & -2 & 0 & 0 & 0 & 0 \\
            0 & 0  & 0 & 0 & 0 & 0  & 0 & 0 & 0 & 0 \\
            0 & 0  & 0 & 0 & 0 & 0  & 0 & 0 & 0 & 0 \\
            0 & 0  & 0 & 0 & 0 & 0  & 0 & 0 & 0 & 0 \\
            0 & 2  & 0 & 0 & 0 & 2  & 0 & 0 & 0 & 0 \\
            0 & 0  & 0 & 0 & 0 & 0  & 0 & 0 & 0 & 0 \\
            0 & 0  & 0 & 0 & 0 & 0  & 0 & 0 & 0 & 0 \\
            0 & 2  & 0 & 0 & 0 & 2  & 0 & 0 & 0 & 0 \\
            0 & 0  & 0 & 0 & 0 & 0  & 0 & 0 & 0 & 0 \\
            0 & 2  & 0 & 0 & 0 & 2  & 0 & 0 & 0 & 0
          \closebigpmatrix.
\end{align*}
By direct computation,\footnote{This is not a complicated computation. Writing the basis vectors as $f_j$, $j=1,\ldots, 10$, $\la f_4,f_7,f_9\ra$ and $\la f_1,f_2,f_3,f_5,f_6,f_8,f_{10}\ra$ decouple. For the former space, the matrix $(ir^{-1}\sigma)^{-1}S$ is lower triangular in the basis $f_4,f_7,f_4-f_9$, with diagonal entries, thus eigenvalues, $\gamma_1',0,0$. For the second space, the matrix is lower triangular in the basis $f_1,f_2,f_5,2f_1-f_3,f_5-f_{10},f_2-f_6,f_1-f_3+f_8$, with diagonal entries $2\gamma_1',\gamma_1',2\gamma_2',0,0,0,0$.} one then verifies that the characteristic polynomial of $S:=S_{(2)}+\wt S$ is given by
\[
  (i r^{-1}\sigma)^{-10}\det(i r^{-1}\sigma\lambda-S) = \lambda^6(\lambda-\gamma_1')^2(\lambda-2\gamma_1')(\lambda-2\gamma_2')
\]
independently of $F'(r_P)$, hence the eigenvalues of $S$ are $0,i\gamma_1 q^{-2}\sigma,2 i\gamma_1 q^{-2}\sigma$, and $2 i\gamma_2 q^{-2}\sigma$. For $\gamma_1,\gamma_2\geq 0$---which is the case of interest due to our choice \eqref{EqSCPGamma12}---this implies that the size of the essential spectral gap of $L$ is positive, and more precisely, the operator $L$ satisfies high energy estimates \eqref{EqAsySmMeroEst} in a half space $\Im\sigma>-2\alpha$ for some $\alpha>0$, and \eqref{EqAsySmMeroEstPosIm} in any fixed half space $\Im\sigma\geq\eps>0$. (The high energy estimates on the real line that our arguments give are lossy as well, since the eigenvalues are merely non-negative, up to the factor $i\sigma$, rather than strictly positive.)

To see this, we proceed as in the discussion following \cite[Proposition~4.7]{HintzPsdoInner}: first, in the (non-microlocal!) bundle splitting \eqref{EqOpSpatBundlesFull}, the operator $S_\sub(2 L)$ is equal to $S$ minus $i$ times a diagonal matrix with $\nabla_{\ham_G}$ on the diagonal, where $\nabla$ is the pullback connection on the respective vector bundles in \eqref{EqOpSpatBundlesFull}; so by \eqref{EqKdSGeoTrappedHam}, $\nabla_{\ham_G}$ is equal to $-2\mu^{-1}\sigma\pa_t$ (which commutes with any $t$-independent operator) minus $r^{-2}\nabla_{\ham_{|\eta|^2}}$, the latter now being the pullback connection on the respective tensor bundles (i.e.\ zeroth, first or second tensor powers of $\pi_{\Sph^2}^*T^*\Sph^2$). We now note that in fact even in the `microlocal' splitting \eqref{EqESGS2TMMicrolocalSplit}, $\nabla_{\ham_{|\eta|^2}}$ is diagonal; indeed, we have:

\begin{lemma}
\label{LemmaESGHamMicro}
  Let $\pi=\pi_{\Sph^2}$. Away from the zero section $o\subset T^*\Sph^2$, consider the splitting
  \[
    \pi^*T^*\Sph^2 = E \oplus F, \quad E_{y,\eta}=\la\eta\ra,\quad F_{y,\eta}=\eta^\perp.
  \]
  Then $\nabla^{\pi^*T^*\Sph^2}_{\ham_{|\eta|^2}}$ is diagonal in this splitting, i.e.\ it preserves the space of sections of $E$ as well as the space of sections of $F$.
\end{lemma}

This holds for any Riemannian manifold $(S,\slg)$ if one replaces $|\eta|^2$ by the dual metric function of $\slg$.

\begin{proof}[Proof of Lemma~\ref{LemmaESGHamMicro}]
  Fix a point $p$ on $\Sph^2$, introduce geodesic normal coordinates $y^1,y^2$ vanishing at $p$, and denote the dual variables on the fibers of $T^*\Sph^2$ by $\eta_1,\eta_2$. Then $\ham_{|\eta|^2}=2\slg^{ij}\eta_i\pa_{y^j}$ at $p$, and therefore
  \[
    \nabla^{\pi^*T^*\Sph^2}_{\ham_{|\eta|^2}}(\eta_k\,dy^k) = 2\slg^{ij}\eta_i\eta_k \nabla_{\pa_{y^j}}dy^k = 0
  \]
  at $p$, proving the claim for $E$. On the other hand, the fact that the Levi-Civita connection $\nabla$ on $\Sph^2$ is a metric connection implies easily that for sections $\phi,\psi$ of $\pi^*T^*\Sph^2\to T^*\Sph^2$, we have
  \[
    \ham_{|\eta|^2}\bigl(\slG(\phi,\psi)\bigr) = \slG(\nabla^{\pi^*T^*\Sph^2}_{\ham_{|\eta|^2}}\phi,\psi) + \slG(\phi,\nabla^{\pi^*T^*\Sph^2}_{\ham_{|\eta|^2}}\psi)
  \]
  where by a slight abuse of notation we denote by $\slG$ the fiber inner product on $\pi^*T^*\Sph^2$ induced by the pullback via $\pi$. Specializing to the section $\phi=\eta$, we find that $\psi\perp\eta$ implies $\nabla^{\pi^*T^*\Sph^2}_{\ham_{|\eta|^2}}\psi\perp\eta$, which proves the claim for $F$. The proof is complete.
\end{proof}

Now, the matrix $S_0:=(ir^{-1}\sigma)^{-1}S$ has constant coefficients, hence we can choose a matrix $Q\in\R^{10\times 10}$ so that $Q S_0 Q^{-1}$ is in `Jordan block' form with small off-diagonal entries, i.e.\ so that it is upper triangular, with the eigenvalues of $S_0$ on the diagonal, with the entries immediately above the diagonal either equal to $0$ or equal to any small and fixed $\eps>0$, and all other entries $0$; see also \cite[\S3.4]{HintzPsdoInner}. Via \eqref{EqESGS2TMMicrolocalSplit}, $Q$ is the matrix of a bundle endomorphism of $\pi^*S^2T^*M^\circ$ where the splitting \eqref{EqESGS2TMMicrolocalSplit} is valid, so in particular near the trapped set $\Gamma$. By construction, if we equip $\pi^*S^2T^*M^\circ$ with the inner product which is given by the identity matrix in the splitting \eqref{EqESGS2TMMicrolocalSplit}, trivializing $\eta^\perp$ by means of the section $(\star\eta)\sigma^{-1}$ (which is homogeneous of degree $0$), then the symmetric part of $Q S_0 Q^{-1}$ (recall that we factored out an $i$ relative to $S$) relative to this is $\gtrsim-\eps$.

Now $Q$, having constant coefficients, commutes with the operator which, in the splitting \eqref{EqESGS2TMMicrolocalSplit}, is diagonal with diagonal entries $\nabla_{\ham_{|\eta|^2}}$; this implies that the skew-adjoint part of $Q S_\sub(L) Q^{-1}$ is $\gtrsim-\eps$. As discussed in Remark~\ref{RmkAsySmPsdoInner}, this guarantees that $L$ does satisfy a high energy estimate in a strip below the real axis. We conclude that $L$ has only finitely many resonances in this half space, and solutions of $L u=0$ have asymptotic expansions into resonant states, up to an exponentially decaying remainder term.

The perturbative results of \S\ref{SubsubsecAsySmPert} apply and show that perturbations of $L$, depending on a finite number of real parameters, with principal symbols given by the dual metric function of a Kerr--de~Sitter metric satisfy the same high energy estimates in a slightly smaller half space, say $\Im\sigma\geq -\frac{3}{2}\alpha$, with uniform constants; this uses the fact the trapping is normally hyperbolic with smooth (forward/backward) trapped set in a uniform manner for the Kerr--de~Sitter family, as explained in \cite[\S4.4]{HintzVasyQuasilinearKdS}. More generally, due to the $r$-normally hyperbolic (for every $r$) nature of the trapping, sufficiently small perturbations of the operator $L$, depending on a finite number of real parameters, within the class of principally scalar stationary second order operators with smooth coefficients satisfy the same high energy estimates (with uniform constants) in a half space $\Im\sigma\geq-\frac{3}{2}\alpha$; see the references given in \S\ref{SubsubsecAsySmPert}. We may assume, by changing $\alpha$ slightly if necessary, that $L$ has no resonances with imaginary part equal to $-\alpha$. By Proposition~\ref{PropAsySmPert}, the same then  holds for small perturbations of $L$.

This takes care of the part of the high energy estimate for $\wh L(\sigma)$ in \eqref{EqKeyESGHighEnergyEst} at the semiclassical trapped set.

\subsection{Threshold regularity at the radial set}
\label{SubsecESGRadial}

The regularity requirement $s>1/2$ in ESG is dictated by the threshold regularity at the radial sets. Thus, we need to compute
\[
  S_\sub(2 L) = -i\nabla_{\ham_G}^{\pi^*S^2T^*M^\circ}+\sigma_1\bigl(2(\tdel^*-\delta_g^*)\delta_g \sfG_g\bigr),
\]
at $\cL_\pm$, where $\pi\colon T^*M^\circ\to M^\circ$ is the projection. We computed the first term in \eqref{EqOpSdSSubprRadTensors}, so we only need to calculate the principal symbol of $2(\tdel^*-\delta_g^*)\delta_g \sfG_g$ at $\cL_\pm$ in the splitting \eqref{EqOpSdSSmoothS2TM}. We calculated the form of this operator in \eqref{EqESGAdditional} in the splitting \eqref{EqOpWarpedBundles}. Now $\pa_r$ in the coordinates $(t,r,\omega)$ is equal to $\pa_r\mp\mu^{-1}\pa_{t_0}$ in the coordinates $(t_0,r,\omega)$. Therefore, since at $\cL_\pm$, the principal symbols of $\pa_{t_0}$, $\sld$ and $\sldelta$ vanish, we can do our current calculation in two steps: first, we expand \eqref{EqESGAdditional} into the more refined splitting \eqref{EqOpSpatBundlesFull} and then discard all $e_0$-derivatives as well as differential operators on $\Sph^2$, while substituting $\sigma_1(e_1)=i q\xi$ for $e_1$, thus obtaining
\begin{align*}
  i&\gamma_1\xi
    \openbigpmatrix{2pt}
      0           & 2 & 0 & 0           & 0 & 0                       \\
      \frac{1}{2} & 0 & 0 & \frac{1}{2} & 0 & -\frac{1}{2}r^{-2}\sltr \\
      0           & 0 & 0 & 0           & 1 & 0                       \\
      0           & 0 & 0 & 0           & 0 & 0                       \\
      0           & 0 & 0 & 0           & 0 & 0                       \\
      0           & 0 & 0 & 0           & 0 & 0
    \closebigpmatrix
  -i\gamma_1\mu F'\xi
    \openbigpmatrix{2pt}
      0 & 0 & 0 & 0 & 0 & 0            \\
      0 & 1 & 0 & 0 & 0 & 0            \\
      0 & 0 & 0 & 0 & 0 & 0            \\
      1 & 0 & 0 & 1 & 0 & -r^{-2}\sltr \\
      0 & 0 & 0 & 0 & 1 & 0            \\
      0 & 0 & 0 & 0 & 0 & 0
    \closebigpmatrix \\
  &-i\gamma_2\xi
    \openbigpmatrix{2pt}
      0 & 2        & 0 & 0 & 0 & 0 \\
      0 & 0        & 0 & 0 & 0 & 0 \\
      0 & 0        & 0 & 0 & 0 & 0 \\
      0 & -2       & 0 & 0 & 0 & 0 \\
      0 & 0        & 0 & 0 & 0 & 0 \\
      0 & -2 r^2 \slg & 0 & 0 & 0 & 0
    \closebigpmatrix
   -i\gamma_2\mu F'\xi
    \openbigpmatrix{3pt}
      1         & 0 & 0 & 1         & 0 & -r^{-2}\sltr \\
      0         & 0 & 0 & 0         & 0 & 0            \\
      0         & 0 & 0 & 0         & 0 & 0            \\
      -1        & 0 & 0 & -1        & 0 & r^{-2}\sltr  \\
      0         & 0 & 0 & 0         & 0 & 0            \\
      -r^2 \slg & 0 & 0 & -r^2 \slg & 0 & \slg\sltr
    \closebigpmatrix.
\end{align*}
We split $S^2T^*\Sph^2=\la r^2 \slg\ra\oplus\slg^\perp$ (so $\slg^\perp=\ker\sltr$), and use the corresponding splitting of $\pi^*S^2T^*\Sph^2$; thus,
\[
  r^2\slg = \begin{pmatrix} 1 \\ 0 \end{pmatrix}, \quad r^{-2}\sltr=\begin{pmatrix} 2 & 0 \end{pmatrix}, \quad \slg\sltr=\begin{pmatrix}2&0\\0&0\end{pmatrix}.
\]
Second, we write $\mu F'=\pm(1+\mu c_\pm)$ as in \eqref{EqSdSTStar}, conjugate the above matrix by $(\sC^{(2)}_\pm)^{-1}$, see \eqref{EqOpSdSConjTensors}, and set $\mu=0$, which after a brief calculation gives, in the bundle decomposition \eqref{EqOpSdSSmoothS2TM} refined by the above splitting of $\pi^*S^2T^*\Sph^2$,
\[
  \mp i\gamma_1\xi
    \openbigpmatrix{2pt}
      2         & 0 & 0         & 0 & 0 & 0        & 0 \\
      \mp c_\pm & 0 & 0         & 0 & 0 & \pm 1    & 0 \\
      0         & 0 & 1         & 0 & 0 & 0        & 0 \\
      0         & 0 & 0         & 0 & 0 & -2 c_\pm & 0 \\
      0         & 0 & \mp c_\pm & 0 & 0 & 0        & 0 \\
      0         & 0 & 0         & 0 & 0 & 0        & 0 \\
      0         & 0 & 0         & 0 & 0 & 0        & 0
    \closebigpmatrix
  \mp i\gamma_2\xi
    \openbigpmatrix{2pt}
      0           & 0 & 0 & 0 & 0 & 0     & 0 \\
      \pm 2 c_\pm & 0 & 0 & 0 & 0 & \mp 2 & 0 \\
      0           & 0 & 0 & 0 & 0 & 0     & 0 \\
      0           & 0 & 0 & 0 & 0 & 0     & 0 \\
      0           & 0 & 0 & 0 & 0 & 0     & 0 \\
      -2 c_\pm    & 0 & 0 & 0 & 0 & 2     & 0 \\
      0           & 0 & 0 & 0 & 0 & 0     & 0
    \closebigpmatrix,
\]
where we factored out the `$\mp$' sign similarly to \eqref{EqOpSdSSubprRadTensors}; all singular terms cancel (as they should). This is thus equal to $\sigma_1(2(\tdel^*-\delta_g^*)\delta_g \sfG_g)$ at $\cL_\pm$, and we conclude that
\begin{equation}
\label{EqESGRadialSubprOp}
\begin{split}
  S_\sub(2&L) = \pm 2 \xi D_{t_0} \mp 2 \kappa_\pm \xi^2 D_\xi - 2\mu\xi D_r \\
     &\mp i\xi
       \openbigpmatrix{2pt}
         2\gamma_1-4\kappa_\pm        & 0 & 0                    & 0           & 0           & 0                       & 0 \\
         \mp(\gamma_1-2\gamma_2)c_\pm & 0 & 0                    & 0           & 0           & \pm(\gamma_1-2\gamma_2) & 0 \\
         0                            & 0 & \gamma_1-2\kappa_\pm & 0           & 0           & 0                       & 0 \\
         0                            & 0 & 0                    & 4\kappa_\pm & 0           & -2\gamma_1 c_\pm        & 0 \\
         0                            & 0 & \mp\gamma_1 c_\pm    & 0           & 2\kappa_\pm & 0                       & 0 \\
         -2\gamma_2 c_\pm             & 0 & 0                    & 0           & 0           & 2\gamma_2               & 0 \\
         0                            & 0 & 0                    & 0           & 0           & 0                       & 0 
       \closebigpmatrix
\end{split}
\end{equation}
at $\cL_\pm$. The first three terms are formally self-adjoint `at $\cL_\pm$` with respect to any $t_0$-invariant inner product on $\pi^*S^2T^*M^\circ$ which is homogeneous of degree $0$ with respect to dilations in the fibers of $T^*M^\circ$ (i.e.\ their skew-adjoint part, which is a \emph{function}, vanishes at $\cL_\pm$), while the eigenvalues of the last term---which is $\xi$ times a matrix which is constant along $\cL_\pm$---are equal to $\mp i \xi$ times\footnote{This is a simple calculation: writing the basis `vectors' as $f_j$, $j=1,\ldots,7$ (where really $f_3,f_5$ are two copies of the same basis of a fiber of $T^*\Sph^2$, and $f_7$ is a basis of a fiber of $\slg^\perp$---which is irrelevant here since all entries of the matrix in \eqref{EqESGRadialSubprOp} are scalar), the spaces $\la f_3,f_5\ra$ and $\la f_7\ra$ decouple, giving the eigenvalues $2\kappa_\pm+\gamma_1,6\kappa_\pm$ and $4\kappa_\pm$, respectively, and then restricted to $\la f_1,f_2,f_4,f_6\ra$, the matrix is lower triangular in the basis $f_1,f_6,f_2,f_4$, with diagonal entries $2\gamma_1,4\kappa_\pm+2\gamma_2,4\kappa_\pm,8\kappa_\pm$.}
\begin{equation}
\label{EqESGRadialSubprEigen}
  2\gamma_1-4\kappa_\pm,\ 
  0,\ 
  \gamma_1-2\kappa_\pm,\ 
  4\kappa_\pm,\ 
  2\kappa_\pm,\ 
  2\gamma_2,\ 
  0,
\end{equation}
in particular they are all non-negative when $\gamma_1\geq 2\kappa_\pm$ and $\gamma_2\geq 0$. Thus, we can find a stationary, homogeneous of degree $0$, positive definite inner product on $\pi^*S^2T^*M^\circ$ with respect to which $(\mp\xi)^{-1}\frac{1}{2i}(S_\sub(2 L)-S_\sub(2 L)^*)$ is positive semidefinite.

Let us first focus on the halves of the conormal bundles $\cL^+_\pm=\cL_\pm\cap\Sigma^+$ which lie in the future light cone, so the signs in \eqref{EqESGRadialSubprOp}, indicating the horizon $r_\pm$ we are working at, correspond to the subscript of $\cL^+_\pm$. (These sets were defined in \S\ref{SubsecKdSGeo} already.) In $\cL^+_\pm$, we have $\pm\xi>0$, and the quantity $\wh\beta_\pm$, see \eqref{EqAsySm2RadialSubpr}, is defined in terms of the quantity $\beta_{\pm,0}$, see \eqref{EqOpNoB0}, by
\[
  |\xi|^{-1}\frac{1}{2i}\bigl(S_\sub(2 L)-S_\sub(2 L)^*\bigr) = -\beta_{\pm,0}\wh\beta_\pm,
\]
so $\wh\beta_\pm$ is a (constant in $t_*$) self-adjoint endomorphism of the restriction of the bundle $\pi^*S^2T^*M^\circ$ to $T^*_{\{r=r_\pm\}}M^\circ$; hence $\wh\beta_\pm$ is bounded from below by $0$.

At $\cL_\pm^-=\cL_\pm\cap\Sigma^-$, where $\mp\xi>0$, we then have
\[
  -|\xi|^{-1}\frac{1}{2i}\bigl(S_\sub(2 L)-S_\sub(2 L)^*\bigr) = -\beta_{\pm,0}\wh\beta_\pm
\]
with the same $\wh\beta_\pm$, the overall sign switch being completely analogous to the one in \eqref{EqKdSGeoRadPoint}.

The criterion for the propagation of microlocal $\Hb^{s,r}$-regularity from $M^\circ$ into $\cR_\pm$ is then the inequality $s-1/2-\beta_\pm r>0$, see \cite[Proposition~2.1]{HintzVasySemilinear} (or Theorem~\ref{ThmAsySmMero} in a directly related context), where we dropped $\inf\wh\beta_\pm=\wh\beta\geq 0$ from the left hand side. For weights $r\geq-\alpha$, this condition holds if $s>1/2$, provided $\alpha>0$ is small enough; the largest possible weight $\alpha$ for the radial point propagation estimate is $\inf\beta_\pm^{-1}(s+1/2)-0$, hence gets larger as $s$ increases.

This proves Theorem~\ref{ThmKeyESG} for $\tdel^*$ defined in \eqref{EqSCPDelStarTilde}, in fact for any choice of parameters $\gamma_1,\gamma_2\geq 0$.

\section{Linear stability of the Kerr--de~Sitter family}
\label{SecKdSLStab}

We now use UEMS, SCP and ESG to establish the linear stability of slowly rotating Kerr--de~Sitter black holes. The proof of the linear stability of the linearized Kerr--de~Sitter family around Schwarz\-schild--de~Sitter space with parameters $b_0$ is straightforward; we remind the reader that the linearized initial value problem was discussed in~\S\ref{SubsecHypLin}. The form of the linearized gauged Einstein equation we consider uses the operator
\begin{equation}
\label{EqKdSLStabOp}
  L_b r := D_{g_b}(\Ric+\Lambda) - \tdel^*(D_{g_b}\Ups(r)),
\end{equation}
with the gauge 1-form $\Ups$ defined in \eqref{EqKdSInGauge}; recall the definition of $\tdel^*$ from \eqref{EqKeySCPDeltaTilde}.

\begin{thm}
\label{ThmKdSLStabNaive}
  Fix $s>1/2$, and let $\alpha>0$ be small. (See Remark~\ref{RmkKeyESGDecay}.) Let $(h_0',k_0')\in H^{s+1}(\Sigma_0;S^2 T^*\Sigma_0)\oplus H^s(\Sigma_0;S^2 T^*\Sigma_0)$, be solutions of the linearized constraint equations, linearized around the initial data $(h_{b_0},k_{b_0})$ of Schwarzschild--de~Sitter space $(\Omega,g_{b_0})$, and consider the initial value problem
  \begin{equation}
  \label{EqKdSLStabNaiveIVP}
    \begin{cases}
      L_{b_0}r = 0 & \tn{in }\Omega^\circ, \\
      \gamma_0(r) = D_{(h_{b_0},k_{b_0})}i_{b_0}(h'_0,k'_0) & \tn{on }\Sigma_0,
    \end{cases}
  \end{equation}
  with $i_b$ defined in \S\ref{SubsecKdSIn}. Then there exist $b'\in T_{b_0}B$ and a 1-form $\omega\in\CI(\Omega^\circ,T^*\Omega^\circ)$ such that
  \begin{equation}
  \label{EqKdSLStabNaiveIVPSol}
    r=g'_{b_0}(b') + \delta_{g_{b_0}}^*\omega + \wt r,
  \end{equation}
  with $\wt r\in\Hbext^{s,\alpha}(\Omega;S^2\,\Tb^*_\Omega M)$. In particular, for $s>2$, this implies the $L^\infty$ bound $|\wt r(t_*)|\lesssim e^{-\alpha t_*}$.
\end{thm}

Recall here that the map taking the initial data $(h'_0,k'_0)$ of the linearized Einstein equation to Cauchy data puts the initial data into the linearized wave map gauge, giving initial data for the linearized gauged Einstein equation; see Corollary~\ref{CorKdSInLin}.

\begin{proof}[Proof of Theorem~\ref{ThmKdSLStabNaive}]
  Since the initial data $(h_0',k_0')$ satisfy the linearized constraint equations, the solution $r$ of the initial value problem also solves the linearized Einstein equation $D_{g_{b_0}}(\Ric+\Lambda)(r)=0$. Now, ESG implies that $r$ has an asymptotic expansion \eqref{EqKeyESGExpansion}, and all resonances $\sigma_j$ in the expansion satisfy $\Im\sigma_j\geq 0$; but then Lemma~\ref{LemmaUEMSWithLogs} shows that the part of the expansion of $r$ coming from a non-zero resonance $\sigma_j$ lies in the range of $\delta_{g_{b_0}}^*$, since such a part by itself is annihilated by $D_{g_{b_0}}(\Ric+\Lambda)$ due to the stationary nature of the operator $D_{g_{b_0}}(\Ric+\Lambda)$, as used in the proof of Lemma~\ref{LemmaUEMSWithLogs}. On the other hand, the part of the asymptotic expansion of $r$ coming from resonances at $0$ is covered by Theorem~\ref{ThmKeyUEMS} \eqref{ItKeyUEMSZero}, which states that this part is equal to $g'_{b_0}(b')$ for some $b'\in T_{b_0}B$, plus an element in the range of $\delta_{g_{b_0}}^*$. This proves the theorem.
\end{proof}

This proof, \emph{which only uses UEMS and ESG}, has a major shortcoming: it is not robust; changing the metric $g_{b_0}$ around which we linearize to any nearby Kerr--de~Sitter metric $g_b$, $b\neq b_0$, makes the argument collapse immediately, since it is then no longer clear why the parts of the asymptotic expansion corresponding to non-decaying resonances should be pure gauge modes, and why the zero resonance should behave as in part~\eqref{ItKeyUEMSZero} of UEMS; recall here that we are only assuming UEMS for \emph{Schwarzschild}--de~Sitter parameters $b_0$. \emph{If} we had proved UEMS for slowly rotating Kerr--de~Sitter spacetimes as well, the proof of Theorem~\ref{ThmKdSLStabNaive} would extend directly, giving the linear stability of slowly rotating Kerr--de~Sitter black holes. We choose a different, much more robust, conceptually cleaner and computationally much simpler path, which will lead to the proof of non-linear stability later.

The additional input that has not been used above, which however allows for a robust proof, is the existence of a stable constraint propagation equation (SCP). Using the notation of Theorem~\ref{ThmKdSLStabNaive} and its proof, SCP ensures that all non-decaying modes in the asymptotic expansion of $r$, apart from those coming from the linearized Kerr--de~Sitter family, are pure gauge modes, i.e.\ lie in the range of $\delta_{g_{b_0}}^*$, \emph{regardless} of whether the Cauchy data of $r$ satisfy the linearized constraint equations.

Let us fix a cutoff function $\chi$ as in \eqref{EqKdSInCutoff}. Then, combining SCP with UEMS yields the following result, which is a concrete instance of the general results proved in \S\ref{SubsecAsySm}:
\begin{prop}
\label{PropKdSLStabComplement}
  For $b'\in T_{b_0}B$, let $\omega^\Ups_{b_0}(b')$ denote the solution of the Cauchy problem
  \[
    \begin{cases}
      D_{g_{b_0}}\Ups\bigl(\delta_{g_{b_0}}^*\omega^\Ups_{b_0}(b')\bigr) = -D_{g_{b_0}}\Ups(g'_{b_0}(b')) & \tn{in }\Omega^\circ, \\
      \gamma_0(\omega^\Ups_{b_0}(b')) = (0,0) & \tn{on }\Sigma_0,
    \end{cases}
  \]
  with $\gamma_0$ defined in \eqref{EqKdSWaveInitialData}; recall from \eqref{EqHypLinGauge}--\eqref{EqHypLinBoxGauge} that this is a wave equation. Define
  \[
    g^{\prime\Ups}_{b_0}(b') := g'_{b_0}(b') + \delta_{g_{b_0}}^*\omega^\Ups_{b_0}(b'),
  \]
  which thus solves the linearized gauged Einstein equation $L_{b_0}(g^{\prime\Ups}_{b_0}(b'))=0$. Fix $s>1/2$, and let $\alpha>0$ be sufficiently small. Then there exists a finite-dimensional linear subspace
  \[
    \Theta \subset \CIc(\Omega^\circ,T^*\Omega^\circ)
  \]
  such that the following holds: for any $(f,u_0,u_1)\in D^{s,\alpha}(\Omega;\Tb^*_\Omega M)$, there exist unique $b'\in T_{b_0}B$ and $\theta\in\Theta$ such that the solution of the forward problem
  \[
    L_{b_0} \wt r = f - \tdel^*\theta - L_{b_0}(\chi g^{\prime\Ups}_{b_0}(b')),\quad \gamma_0(\wt r)=(u_0,u_1),
  \]
  satisfies $\wt r\in\Hbext^{s,\alpha}(\Omega;S^2\,\Tb^*_\Omega M)$, and the map $(f,u_0,u_1)\mapsto(b',\theta)$ is linear and continuous.
\end{prop}

Thus, the equation $L_{b_0}r=f$, $\gamma_0(r)=(u_0,u_1)$, has an exponentially decaying solution if we modify the right hand side by an element in the finite-dimensional space $\tdel^*\Theta+L_{b_0}(\chi g^{\prime\Ups}_{b_0}(T_{b_0}B)) \subset \CIc(\Omega^\circ,S^2T^*\Omega^\circ)$; in other words, this space is a complement to the range of $(L_{b_0},\gamma_0)$ acting on symmetric 2-tensors in $\Hbext^{\infty,\alpha}$, i.e.\ which have decay least like $e^{-\alpha t_*}$.

\begin{rmk}
\label{RmkKdSLStabGaugedKdSMetrics}
  The pure gauge modification $\delta_{g_{b_0}}^*\omega^\Ups_{b_0}(b')$ of $g'_{b_0}(b')$ is necessary in view of the fact that the specific form of the Kerr--de~Sitter family of metrics given in \S\ref{SubsecKdSSlow} did not take any gauge considerations into account; hence $g'_{b_0}(b')$, while lying in the kernel of $D_{g_{b_0}}(\Ric+\Lambda)$, will in general not satisfy the linearized gauge condition, so $D_{g_{b_0}}\Ups(g'_{b_0}(b'))\neq 0$.
\end{rmk}

\begin{proof}[Proof of Proposition~\ref{PropKdSLStabComplement}]
  Let $\sigma_1=0$ and $\sigma_j\neq 0$, $j=2,\ldots,N_L$, denote the resonances of $L_{b_0}$ with non-negative imaginary part. For $j\geq 2$, fix a basis $\{r_{j1},\ldots,r_{j d_j}\}$ of $\Res(L_{b_0},\sigma_j)$. Then SCP and UEMS, in the sharper form given by Lemma~\ref{LemmaUEMSWithLogs}, imply (as explained in \S\ref{SubsecKeySCP}) the existence of $\omega_{j\ell}\in\CI(\Omega^\circ,T^*\Omega^\circ)$ such that $r_{j\ell}=\delta_{g_{b_0}}^*\omega_{j\ell}$. We define
  \[
    \theta_{j\ell} := -D_{g_{b_0}}\Ups\bigl(\delta_{g_{b_0}}^*(\chi\omega_{j\ell})\bigr),
  \]
  the point being that $L_{g_{b_0}}\bigl(\delta_{g_{b_0}}^*(\chi\omega_{j\ell})\bigr) = \tdel^*\theta_{j\ell}$, and then put
  \[
    \Theta_{\neq 0} := \mathspan\{ \theta_{j\ell} \colon j=2,\ldots,N_L,\ \ell=1,\ldots,d_j \}.
  \]
  We recall here that SCP implies that $D_{g_{b_0}}\Ups(r_{j\ell})=D_{g_{b_0}}\Ups(\delta_{g_{b_0}}^*\omega_{j\ell})\equiv 0$, hence the $\theta_{j\ell}$ are indeed compactly supported in $\Omega^\circ$.
  
  At the zero resonance, we first recall that $\omega^\Ups_{b_0}(b')$ has an asymptotic expansion up to an exponentially decaying remainder: indeed, $\BoxGauge_{g_{b_0}}=-2 D_{g_{b_0}}\Ups\circ\delta_{g_{b_0}}^*$ differs from $\Box_{g_{b_0}}$ by a term of order $0$, i.e.\ a sub-subprincipal term, by our definition of $\Ups$ (see also the discussion after \eqref{EqHypLinBoxGauge}), and hence the main theorem from \cite{HintzPsdoInner} applies. Thus, $g^{\prime\Ups}_{b_0}(b')$ has an asymptotic expansion up to an exponentially decaying remainder as well, hence its part $g^{\prime\Ups}_{b_0}(b')^{(0)}$ coming from the zero resonance is well-defined, and $g^{\prime\Ups}_{b_0}(b')^{(0)}\in\Res(L_{b_0},0)$. Define the linear subspace
  \[
    K := \{ g^{\prime\Ups}_{b_0}(b')^{(0)} \colon b'\in T_{b_0}B \} \subset \Res(L_{b_0},0),
  \]
  which has dimension $d_0\leq 4$. (One can in fact show that $d_0=4$, see Remark~\ref{RmkKdSSlowLinMetDegFreedom}, but this is irrelevant here.) Using Theorem~\ref{ThmKeyUEMS} \eqref{ItKeyUEMSZero}, we infer the existence of a complement of $K$ within $\Res(L_{b_0},0)$ which has a basis of the form $\{\delta_{g_{b_0}}^*\omega_\ell \colon 1\leq\ell\leq d_1\}$, $d_1=\dim\Res(L_{b_0},0)-d_0$, and we then define
  \[
    \theta_\ell := -D_{g_{b_0}}\Ups\bigl(\delta_{g_{b_0}}^*(\chi\omega_\ell)\bigr),\quad \Theta_0 := \mathspan\{\theta_\ell\colon 1\leq\ell\leq d_1\},
  \]
  and
  \[
    \Theta := \Theta_0 \oplus \Theta_{\neq 0};
  \]
  By construction, $g^{\prime\Ups}_{b_0}(b')-g^{\prime\Ups}_{b_0}(b')^{(0)}$ is a pure gauge 2-tensor annihilated by $L_{b_0}$ and hence a linear combination of $\delta_{g_{b_0}}^*\omega_{j\ell}$ and $\delta_{g_{b_0}}^*\omega_\ell$ up to an exponentially decaying remainder. Therefore, if we define the space
  \[
    \cZ := L_{b_0}\bigl(\chi g^{\prime\Ups}_{b_0}(T_{b_0}B)\bigr) + \tdel^*\Theta \subset \CIc(\Omega^\circ;S^2T^*\Omega^\circ) \hra D^{\infty,\alpha}(\Omega;S^2\,\Tb^*_\Omega M),
  \]
  the proposition follows from the bijective form of Corollary~\ref{CorAsySmResNoAsy}. Indeed, the bijectivity of the map $\lambda_\cZ$ in the statement of that corollary follows from dimension counting: on the one hand, $\lambda_\cZ$ is surjective by construction, as all non-decaying asymptotics can be eliminated by adding a solution of $L_{b_0}=z$ for some $z\in\cZ$. On the other hand, also by construction, the dimension of $\cZ$ is at most as large as the space of resonant states of $L_{b_0}$ which are not exponentially decaying; this proves the injectivity of $\lambda_\cZ$.
\end{proof}

The linear stability around $g_{b_0}$ in the initial value formulation \eqref{EqKdSLStabNaiveIVP} can now be re-proved as follows: there exist $b'\in T_{b_0}B$ and $\theta\in\Theta$ as in the proposition such that the solution of the Cauchy problem
\[
  L_{b_0}\wt r = -L_{b_0}(\chi g^{\prime\Ups}_{b_0}(b')) - \tdel^*\theta,
\]
with Cauchy data for $\wt r$ as in \eqref{EqKdSLStabNaiveIVP}, is exponentially decaying; rewriting this using the definition of $L_{b_0}$ shows that
\[
  r_1:=\chi g^{\prime\Ups}_{b_0}(b')+\wt r
\]
--- which has the same Cauchy data---solves
\begin{equation}
\label{EqKdSLStabBetterIVP}
  D_{g_{b_0}}(\Ric+\Lambda)(r_1) - \tdel^*(D_{g_{b_0}}\Ups(r_1) - \theta) = 0.
\end{equation}
In order to relate this to \eqref{EqKdSLStabNaiveIVP}, rewrite $r_1$ as
\[
  r_1 = g^{\prime\Ups}_{b_0}(b') + \wt r_1,
\]
where $\wt r_1=\wt r-(1-\chi)g^{\prime\Ups}_{b_0}(b')$ differs from $\wt r$ only near $\Sigma_0$; then, writing $\theta$ in terms of the 1-forms $\omega_{j\ell}$ and $\omega_\ell$ from the proof of Proposition~\ref{PropKdSLStabComplement}, recovers the solution $r$ of the unmodified equation \eqref{EqKdSLStabNaiveIVP} in the form \eqref{EqKdSLStabNaiveIVPSol}.

Finally, since the gauge condition $D_{g_{b_0}}\Ups(r_1)-\theta=0$ and the linearized constraints are satisfied at $\Sigma_0$ by construction of the problem \eqref{EqKdSLStabNaiveIVP} as $\theta$ is supported away from $\Sigma_0$, the constraint propagation equation implies that the gauge condition holds globally, and therefore $D_{g_{b_0}}(\Ric+\Lambda)(r_1)=0$ is indeed the solution of the equations of linearized gravity with the given initial data. This again proves linear stability. In this argument, $\theta$ is the change of gauge which ensures that the solution $r_1$ of the initial value problem for \eqref{EqKdSLStabBetterIVP} is equal to a gauged linearized Kerr--de~Sitter solution.

This argument eliminates the disadvantages of our earlier proof and allows the linear stability of $g_b$ with $b$ near $b_0$ to be proved by a perturbative argument. First, regarding the choice of gauge, we note that $g=g_b$ satisfies $(\Ric+\Lambda)(g)-\tdel^*(\Ups(g)-\Ups(g_b))=0$, which suggests the gauge condition $\Ups(g)-\Ups(g_b)=0$; the linearization of this equation in $g$ is precisely $L_b r=0$ with $L_b$ as in \eqref{EqKdSLStabOp}, and the linearized gauge condition at $\Sigma_0$ then reads $D_{g_b}\Ups(r)=0$. Returning to the perturbative argument then, and defining correctly gauged linearized Kerr--de~Sitter metrics $g^{\prime\Ups}_b(b')$ for $b\in\cU_B$ in a fashion similar to Proposition~\ref{PropKdSLStabComplement}, see also Lemma~\ref{LemmaKdSLStabComplementCont} below, we will show that the finite-dimensional vector space $L_b(\chi g^{\prime\Ups}_b(T_b B))+\tdel^*\Theta\subset\CIc(\Omega^\circ,S^2T^*\Omega^\circ)$---with $\Theta$ the \emph{same} space as the one constructed in Proposition~\ref{PropKdSLStabComplement}---can be arranged to depend continuously (even smoothly) on $b$. Then Corollary~\ref{CorAsySmPertNoAsy} applies, and therefore the above proof immediately carries over to show the linear stability of the Kerr--de~Sitter family linearized around $g_b$, $b\in\cU_B$. More precisely, in order to invoke Corollary~\ref{CorAsySmPertNoAsy}, one needs to parameterize the space $L_b(\chi g^{\prime\Ups}_b(T_b B))+\tdel^*\Theta$; this can be accomplished by choosing an isomorphism
\[
  \vartheta\colon\R^{N_\Theta}\to\Theta,\quad N_\Theta=\dim\Theta,
\]
which gives the parameterization $\R^{4+N_\Theta}\ni(b',\bfc)\mapsto L_b(\chi g^{\prime\Ups}_b(b'))+\tdel^*\vartheta(\bfc)$.

We proceed to establish the continuity claims involved in this argument.

\begin{lemma}
\label{LemmaKdSLStabComplementCont}
  For $b\in\cU_B$ and $b'\in T_b B$, let $\omega^\Ups_b(b')$ be the solution of the wave equation
  \begin{equation}
  \label{EqKdSLStabComplementContIVP}
    \begin{cases}
      D_{g_b}\Ups(\delta_{g_b}^*\omega^\Ups_b(b'))=-D_{g_b}\Ups(g'_b(b')) & \tn{in }\Omega^\circ, \\
      \gamma_0(\omega^\Ups_b(b')) = (0,0) & \tn{on }\Sigma_0,
    \end{cases}
  \end{equation}
  and define
  \begin{equation}
  \label{EqKdSLStabComplementContGBetaPrime}
    g^{\prime\Ups}_b(b') := g'_b(b') + \delta_{g_b}^*\omega^\Ups_b(b').
  \end{equation}
  Then the map
  \[
    \cU_B \times \R^{4+N_\Theta} \ni (b,b',\bfc) \mapsto L_b(\chi g^{\prime\Ups}_b(b'))+\tdel^*\vartheta(\bfc) \in \CIc(\Omega^\circ;S^2T^*\Omega^\circ)
  \]
  is continuous.
\end{lemma}

Proposition~\ref{PropKdSLStabComplement} then applies also for Kerr--de~Sitter parameters $b$ near $b_0$, i.e.\ we can replace $b_0$ by $b$ throughout its statement.

\begin{proof}[Proof of Lemma~\ref{LemmaKdSLStabComplementCont}]
  The only non-trivial part of this lemma is the continuous dependence of $L_b(\chi g^{\prime\Ups}_b(b'))$. However, note that
  \[
    L_b(\chi g^{\prime\Ups}_b(b')) = [L_b,\chi] g^{\prime\Ups}_b(b')
  \]
  by construction of $g^{\prime\Ups}_b(b')$, and the lemma follows from the continuous dependence of the solution of \eqref{EqKdSLStabComplementContIVP} on the initial data and the coefficients of the operator $\BoxGauge_{g_b}=-2 D_{g_b}\Ups\circ\delta_{g_b}^*$ \emph{in a fixed finite time interval}; see Proposition~\ref{PropAsyExpLocal} for such a result in a more general, non-smooth coefficient, setting.
\end{proof}

The linear stability result around a slowly rotating Kerr--de~Sitter metric can be formulated completely analogously to Theorem~\ref{ThmKdSLStabNaive}, but it is important to keep in mind that we obtain this by robust \emph{perturbative} methods, relying on the version of Proposition~\ref{PropKdSLStabComplement} for slowly rotating Kerr--de~Sitter spaces discussed above, and using the arguments presented around \eqref{EqKdSLStabBetterIVP} (which $b_0$ replaced by $b$). Thus:

\begin{thm}
\label{ThmKdSLStabKdS}
  Fix $s>1/2$, and let $\alpha>0$ be small. (See Remark~\ref{RmkKeyESGDecay}.) Let $(h_0',k_0')\in  H^{s+1}(\Sigma_0;S^2 T^*\Sigma_0)\oplus H^s(\Sigma_0;S^2 T^*\Sigma_0)$ be solutions of the linearized constraint equations, linearized around the initial data $(h_b,k_b)$ of a slowly rotating Kerr--de~Sitter space $(\Omega,g_b)$, and consider the initial value problem
  \[
    \begin{cases}
      L_b r = 0 & \tn{in }\Omega^\circ, \\
      \gamma_0(r) = D_{(h_b,k_b)}i_b(h'_0,k'_0) & \tn{on }\Sigma_0,
    \end{cases}
  \]
  with $i_{b_0}$ defined in \S\ref{SubsecKdSIn}. Then there exist $b'\in T_b B$ and a 1-form $\omega\in\CI(\Omega^\circ,T^*\Omega^\circ)$ such that
  \[
    r=g'_b(b') + \delta_{g_b}^*\omega + \wt r,
  \]
  with $\wt r\in\Hbext^{s,\alpha}(\Omega;S^2\,\Tb^*_\Omega M)$. In particular, for $s>2$, this implies the $L^\infty$ bound $|\wt r(t_*)|\lesssim e^{-\alpha t_*}$.
  
  More precisely, there exist $\theta\in\CIc(\Omega^\circ;T^*\Omega^\circ)$, lying in the fixed finite-dimensional space $\Theta$, and $\wt r'\in\Hbext^{s,\alpha}(\Omega;S^2\,\Tb^*_\Omega M)$ such that $r=\chi g^{\prime\Ups}_b(b') + \wt r'$ solves $D_{g_b}(\Ric+\Lambda)(r)=0$ (attaining the given initial data) in the gauge $D_{g_b}\Ups(r)-\theta=0$. The norms of $b',\theta$ and $\wt r'$ are bounded by the norm of the initial data.
\end{thm}

As explained in \S\ref{SubsecIntroCsq} and Remark~\ref{RmkKdSStabRingdown} below, one can in principle obtain a more precise asymptotic expansion of $r$; since no rigorous results on shallow resonances for linearized gravity are known, we do not state such results here.

\begin{rmk}
\label{RmkKdSLStabNoTheta}
  It is natural to ask whether the space $\Theta$ in Proposition~\ref{PropKdSLStabComplement} is in fact trivial for the chosen hyperbolic version (or a further modification) of the (linearized) Einstein equation; that is, whether in a suitable formulation of the gauged Einstein equation, linearized around Schwarzschild--de~Sitter, the only non-decaying resonances are precisely given by the linearized Kerr--de~Sitter family. (By a simple dimension counting and perturbation argument, this would continue to hold for slowly rotating Kerr--de~Sitter spaces, too.) If this were the case, one could easily prove linear and non-linear stability without any of the ingredients from \S\ref{SecKey} and \S\ref{SecAsy}, but only using the techniques of \cite{HintzVasyQuasilinearKdS}.
  
  On the static model of de~Sitter space, \emph{the answer to this question is negative} if one restricts to modifications of the gauge and the Einstein equation which are `natural' with respect to the conformal structure of global de~Sitter space, see Remark~\ref{RmkdSBAlwaysZero}.
  
  An additional obstacle is the incompatibility of the gauge with the given form of the Kerr--de~Sitter family, which necessitates the introduction of the correction terms $\omega^\Ups_b(b')$ above. While we cannot exclude the possibility that there is some formulation of Einstein's equations and the Kerr--de~Sitter family so that the answer to the above question is positive, this would presumably be rather delicate to arrange, and would very likely be difficult to generalize to other settings.

  On a related note, we point out that there is no need for the wave operator $\BoxGauge_{g_b}=-2 D_{g_b}\Ups\circ\delta_{g_b}^*$ to satisfy the analogue of SCP, i.e.\ to not have any resonances in the closed upper half plane. (Note that the $\delta_{g_b}^*$, i.e.\ Lie derivative, part of this operator is fixed, so this only depends on the choice of gauge $\Ups$.) This is closely related to the resonances of $L_b$, since the non-decaying resonances of $L_b$, other than the ones coming from the Kerr--de~Sitter family, are pure gauge resonances by SCP and UEMS, and thus they are resonances of $\BoxGauge_{g_b}$
\end{rmk}

We show that a forteriori, Theorem~\ref{ThmKdSLStabKdS} implies the mode stability for the linearized \emph{ungauged} Einstein equation around slowly rotating Kerr--de~Sitter black holes, albeit in a slightly weaker form:
\begin{thm}
\label{ThmKdSLStabUEMS}
  Let $b\in\cU_B$ be the parameters of a slowly rotating Kerr--de~Sitter black hole with parameters close to $b_0$.
  \begin{enumerate}
  \item Let $\sigma\in\C$, $\Im\sigma\geq 0$, $\sigma\neq 0$, and suppose $r(t_*,x)=e^{-i\sigma t_*}r_0(x)$, with $r_0\in\CI(Y,S^2T_Y^*\Omega^\circ)$, is a mode solution of the linearized Einstein equation $D_{g_b}(\Ric+\Lambda)(r) = 0$. Then there exists a 1-form $\omega\in\CI(\Omega^\circ,S^2T^*\Omega^\circ)$ such that $r = \delta_{g_b}^*\omega$.
  \item Suppose $k\in\N_0$, and
    \[
      r(t_*,x) = \sum_{j=0}^k t_*^j r_j(x),\quad r_j\in\CI(Y,S^2T_Y^*\Omega^\circ),\ j=0,\ldots,k,
    \]
    is a generalized mode solution of $D_{g_b}(\Ric+\Lambda)r=0$. Then there exist $b'\in T_b B$ and $\omega\in\CI(\Omega^\circ,S^2T^*\Omega^\circ)$ such that $r = g_b'(b') + \delta_{g_b}^*\omega.$
  \end{enumerate}
\end{thm}

The proof will produce a generalized mode $\omega$; however, for $\sigma\neq 0$, $\omega$ may not be a mode solution as in Theorem~\ref{ThmKeyUEMS} \eqref{ItKeyUEMSNonzero}, and in the context of Lemma~\ref{LemmaUEMSWithLogs} may not be a generalized mode solution with the same power of $t_*$ in its expansion, while for $\sigma=0$, the proof may produce a generalized mode $\omega$ which is more complicated than the one produced by our arguments in \S\ref{SecUEMS}.

\begin{proof}[Proof of Theorem~\ref{ThmKdSLStabUEMS}]
  We reduce this to the linear stability result for the linearized \emph{gauged} Einstein equation, i.e.\ the precise form stated at the end of Theorem~\ref{ThmKdSLStabKdS}. This may seem unnatural at first, but is a very robust way of obtaining the mode stability result we are after right now; see Remark~\ref{RmkKdSLStabWhyStable}.

  In both cases considered in the statement of the theorem, given a (generalized) mode solution $r$, we solve the equation $\frac{1}{2}\BoxGauge_{g_b}\omega=D_{g_b}\Ups(r)$ with arbitrary initial data for $\omega$. Then $\omega$ and thus $\delta_{g_b}^*\omega$ have an asymptotic expansion up to an exponentially decaying remainder term as explained in the proof of Proposition~\ref{PropKdSLStabComplement}, so upon replacing $r$ by the part of the asymptotic expansion of $r+\delta_{g_b}^*\omega$ which has frequency $\sigma$ in $t_*$, we may assume that $r$ is a generalized mode solution of both $D_{g_b}(\Ric+\Lambda)(r)=0$ and $L_b r=0$.

  We now put $f:=L_b(\chi r)\in\CIc(\Omega^\circ,S^2T^*\Omega^\circ)$. By the version of Proposition~\ref{PropKdSLStabComplement} for slowly rotating Kerr--de~Sitter black holes which we established above (see the discussion following the statement of Lemma~\ref{LemmaKdSLStabComplementCont}), we can find $b'\in T_b B$ and $\theta\in\Theta$ such that the forward solution $\wt r$ of
  \begin{equation}
  \label{EqKdSLStabUEMS}
    L_b(\chi g_b^{\prime\Ups}(b')+\wt r) = f + \tdel^*\theta
  \end{equation}
  satisfies $\wt r=\cO(e^{-\alpha t_*})$.

  In order to solve away the second term on the right hand side, we make the ansatz $L_b(\delta_{g_b}^*\varpi)=\tdel^*\theta$ and demand that $\varpi$ vanish near $\Sigma_0$. This equation is equivalent to $\tdel^*(D_{g_b}\Ups(\delta_{g_b}^*\varpi) + \theta)=0$ and hence is solved by solving the forward problem for $\frac{1}{2}\BoxGauge_{g_b}\varpi=\theta$. We can now rewrite \eqref{EqKdSLStabUEMS} as
  \[
    L_b\bigl(\chi g_b^{\prime\Ups}(b')+\wt r - \chi r - \delta_{g_b}^*\varpi\bigr)=0
  \]
  Since the argument of $L_b$ vanishes near $\Sigma_0$, it must vanish identically in $\Omega^\circ$. In its asymptotic expansion (up to exponentially decaying remainder terms), we can then take the part corresponding to the frequency $\sigma$ in $t_*$; upon doing so, $\wt r$ drops out, and we can take $\chi\equiv 1$ to conclude (using the definition \eqref{EqKdSLStabComplementContGBetaPrime} of $g_b^{\prime\Ups}(b')$) that $r$ is indeed a linearized Kerr--de~Sitter metric $g'_b(b')$ (only present if $\sigma=0$) plus a Lie derivative of $g_b$.
\end{proof}

\begin{rmk}
\label{RmkKdSLStabWhyStable}
  As alluded to in the discussion of the proof of Theorem~\ref{ThmKdSLStabNaive}, assuming merely UEMS does not provide any evidence for mode stability to hold for the linearized Einstein equation around Kerr--de~Sitter metrics $g_b$ with non-zero angular momentum. The reason is that $D_{g_{b_0}}(\Ric+\Lambda)$ by itself is a very ill-behaved partial differential operator. Making it hyperbolic by adding a linearized gauge term, i.e.\ considering the operator $L_{b_0}$, defined using any $0$-th order stationary modification $\tdel^*$ of $\delta_{g_{b_0}}^*$, already gives a much more well-behaved operator: for a large class of choices of $\tdel^*$, $L_{b_0}$ will have a positive essential spectral gap and satisfy high energy estimates in the closed upper half plane, i.e.\ ESG is valid, and thus the set of frequencies of those non-decaying modes for the linearized (around $g_{b_0}$) Einstein equation for which UEMS and Lemma~\ref{LemmaUEMSWithLogs} give non-trivial information is reduced from the entire half space $\{\Im\sigma\geq 0\}$ to the finite set $S_{b_0}=\Res(L_{b_0})\cap\{\Im\sigma\geq 0\}$; furthermore, since the set $S_b$ depends continuously on $b$, only frequencies lying in a neighborhood of $S_{b_0}$ can possibly be the frequencies of potential non-decaying non pure gauge modes of $D_{g_b}(\Ric+\Lambda)$. The final ingredient, SCP, then allows us to completely control the non-decaying modes of $L_{b_0}$ modulo pure gauge modes, which allowed us to finish the proof of linear stability of $g_b$ using rather simple perturbation arguments.

  Thus, the proof of Theorem~\ref{ThmKdSLStabNaive}, while very natural for the purpose of proving the linear stability of $g_{b_0}$, uses only a rather small amount of the structure of $(M,g_{b_0})$ available.
\end{rmk}

\section{Non-linear stability of the Kerr--de~Sitter family}
\label{SecKdSStab}

In \S\ref{SubsecKdSStab}, we will discuss the final ingredient of the non-linear stability argument---the `dynamic' change (in the sense that we update it at each step of our iteration scheme) of the asymptotic gauge condition---and conclude the proof of non-linear stability. In \S\ref{SubsecKdSStabIn} then, we discuss the construction of suitable initial data for Einstein's equations. We begin however by recalling the version of the Nash--Moser inverse function theorem which we will use later.

\subsection{Nash--Moser iteration}
\label{SubsecKdSStabNM}

We shall employ the Nash--Moser inverse function theorem given by Saint-Raymond \cite{SaintRaymondNashMoser}; we use his version because of its simplicity and immediate applicability, despite it being, according to the author, ``probably the worst that can be found in the literature [...]~with respect to the number of derivatives that are used.'' We refer to the introduction of that paper for references to more sophisticated versions, and to \cite{HamiltonNashMoser} for a detailed introduction.

\begin{thm}[Main theorem of \cite{SaintRaymondNashMoser}]
\label{ThmKdSStabNM}
  Let $(B^s,|\cdot|_s)$ and $(\bfB^s,\|\cdot\|_s)$ be Banach spaces for $s\geq 0$ with $B^s\subset B^t$ and indeed $|v|_t\leq|v|_s$ for $s\geq t$, likewise for $\bfB^*$ and $\|\cdot\|_*$; put $B^\infty=\bigcap_s B^s$ and similarly $\bfB^\infty=\bigcap_s\bfB^s$. Assume there are smoothing operators $(S_\theta)_{\theta>1}\colon B^\infty\to B^\infty$ satisfying for every $v\in B^\infty$, $\theta>1$ and $s,t\geq 0$:
  \begin{gather*}
    |S_\theta v|_s \leq C_{s,t}\theta^{s-t}|v|_t \tn{ if } s\geq t, \\
	|v-S_\theta v|_s \leq C_{s,t}\theta^{s-t}|v|_t \tn{ if } s\leq t.
  \end{gather*}

  Let $\phi\colon B^\infty\to\bfB^\infty$ be a $C^2$ map, and assume that there exist $u_0\in B^\infty$, $d\in\N$, $\delta>0$ and constants $C_1,C_2$ and $(C_s)_{s\geq d}$ such that for any $u,v,w\in B^\infty$,
  \begin{equation}
  \label{EqKdSStabNMMapCont}
    |u-u_0|_{3d}<\delta \Rightarrow
	  \begin{cases}
	    \forall s\geq d,\quad \|\phi(u)\|_s\leq C_s(1+|u|_{s+d}), \\
		\|\phi'(u)v\|_{2d}\leq C_1|v|_{3d}, \\
		\|\phi''(u)(v,w)\|_{2d}\leq C_2|v|_{3d}|w|_{3d}.
	  \end{cases}
  \end{equation}
  Moreover, assume that for every $u\in B^\infty$ with $|u-u_0|_{3d}<\delta$ there exists an operator $\psi(u)\colon\bfB^\infty\to B^\infty$ satisfying
  \[
    \phi'(u)\psi(u)h=h
  \]
  and the tame estimate
  \begin{equation}
  \label{EqKdSStabNMTameSol}
    |\psi(u)h|_s\leq C_s(\|h\|_{s+d}+|u|_{s+d}\|h\|_{2d}),\quad s\geq d,
  \end{equation}
  for all $h\in\bfB^\infty$. Then if $\|\phi(u_0)\|_{2d}$ is sufficiently small depending on $\delta,|u_0|_D$ and $(C_s)_{s\leq D}$, where $D=16d^2+43d+24$, there exists $u\in B^\infty$ such that $\phi(u)=0$.
\end{thm}

For our main theorem, we will take $\bfB^s=D^{s,\alpha}(\Omega;S^2\,\Tb^*_\Omega M)$, capturing initial data and inhomogeneous forcing terms, while
\[
  B^s = \R^N \oplus \Hbext^{s,\alpha}(\Omega;S^2\,\Tb^*_\Omega M)
\]
will also take a finite number of additional parameters into account, corresponding to the final black hole parameters and gauge modifications. As in \cite{HintzVasyQuasilinearKdS}, we then define the smoothing operators $S_\theta$ to be the identity on $\R^N$; the construction of $S_\theta$ acting on extendible b-Sobolev spaces $\Hbext^{s,\alpha}$ on the other hand is straightforward. Since the presence of the vector bundle $S^2\,\Tb^*_\Omega M$ is inconsequential for this discussion, we drop it from the notation. Locally near any point on $\Omega\setminus(([0,1]_\tau\times\pa Y)\cup\Sigma_0)$ then, i.e.\ away from the artificial boundaries of $\Omega$, the space $\Hbext^{s,\alpha}$ can be identified either with $H^s(\R^4)$ (if we are away from the boundary at infinity $Y$) or with $\Hb^{s,\alpha}(\Rhalfc{4})$, which in turn is isomorphic to $H^s(\R^n)$ after dividing by $\tau^\alpha$ and using a logarithmic change of coordinates; on the latter space, suitable smoothing operators were constructed in the \cite[Appendix]{SaintRaymondNashMoser}: using the Fourier transform on $\R^n$, they take the form
\[
  \wh{S_\theta u}(\xi) = \chi(\theta^{-1}\xi)\wh u(\xi),
\]
where $\chi\in\CIc(\R^n_\xi)$ is identically $1$ near $0$. In the neighborhood of a point $p\in\{\tau=0\}\cap\pa Y$, we similarly have an identification of $\Hbext^{s,\alpha}$ with $\Hext^s(\Rhalf{4})$, and we can then use bounded extension operators $\Hext^s(\Rhalf{4})\to H^s(\R^4)$, apply smoothing operators on the latter space, restrict back to $\Rhalf{4}$ and use the local identification to get an element of $\Hbext^{s,\alpha}$ near $p$. Near a point in $(0,1)_\tau\times\pa Y$ or $\Sigma_0\setminus\pa\Sigma_0$, where $\Hbext^{s,\alpha}$ can be identified with $\Hext^s(\Rhalf{4})$ again, a similar construction works. Lastly, near points in the corner $\pa\Sigma_0$ of $\Omega$, we can identify $\Hbext^{s,\alpha}$ with $\Hext^s((0,\infty)\times(0,\infty)\times\R^2)$, which embeds into $H^s(\R^4)$, thus we can again use extension and restriction operators as before. Patching together these local constructions via a partition of unity on $\Omega$ gives a smoothing operator $S_\theta$ on $\Hbext^{s,\alpha}$ with the desired properties.

\subsection{Proof of non-linear stability}
\label{SubsecKdSStab}

The precise form of the linear stability statement proved in \S\ref{SecKdSLStab} is not quite what we need for the proof of non-linear stability. Concretely, in order to realize the linearized Kerr--de~Sitter metric $g_b(b')$ as a $0$-resonant state of the linearized gauged Einstein operator $L_b$, we needed to add to it a pure gauge term $\delta_{g_b}^*\omega^\Ups_b(b')$ which in general has a non-trivial asymptotic part at $t_*$-frequency $0$, since this is the case for the right hand side in \eqref{EqKdSLStabComplementContIVP}---indeed, since our construction of the metrics $g_b$ only ensures the smooth dependence on $b$, but does not guarantee any gauge condition, the term $D_{g_b}\Ups(g'_b(b'))$ is in general non-zero and stationary, i.e.\ has $t_*$-frequency $0$. (In fact, without a precise analysis of the resonances of the operator family $\BoxGauge_{g_b}$ it is not even clear if $\omega^\Ups_b(b')$ can be arranged to both depend smoothly on $(b,b')\in T B$ and not be exponentially growing.)

Since the term $\delta_{g_b}^*\omega^\Ups_b(b')$, while pure gauge and therefore harmless for linear stability considerations, cannot be discarded in a non-linear iteration scheme, we need to treat it differently. The idea is very simple: since this is a gauge term, we take care of it by changing the gauge; the point is that changing the final black hole parameters from $b$ to $b+b'$ is incompatible with the gauge $\Ups(g)-\Ups(g_b)$ (with $g\approx g_{b+b'}$ the current approximation of the non-linear solution), but it \emph{is} compatible with an updated gauge $\Ups(g)-\Ups(g_{b+b'})$; updating the gauge in this manner will (almost) exactly account for the term $\delta_{g_b}^*\omega^\Ups_b(b')$.

To motivate the precise formulation, we follow the strategy outlined in \S\ref{SubsecIntroIdeas}: using the notation of \S\ref{SubsecKdSIn}, let us consider the non-linear differential operator
\[
  P_0(b,\theta,\wt g) := (\Ric+\Lambda)(g_{b_0,b}+\wt g) - \tdel^*\bigl(\Ups(g_{b_0,b}+\wt g)-\Ups(g_{b_0,b}) - \theta\bigr),
\]
with $b\in\cU_B$, $\theta$ a modification in some finite-dimensional space which we will determine, and $\wt g$ exponentially decaying. (We reserve the letter `$P$' for the actual non-linear operator used in the proof of non-linear stability below.) Note that the linearization of $P_0$ in $\wt g$ is given by the second order differential operator
\begin{equation}
\label{EqKdSStabLinP}
  L_{b,\wt g}r := (D_{\wt g}P_0(b,\theta,\cdot))(r) = D_{g_{b_0,b}+\wt g}(\Ric+\Lambda)(r) - \tdel^*\bigl(D_{g_{b_0,b}+\wt g}\Ups(r)\bigr),
\end{equation}
while a change in the asymptotic Kerr--de~Sitter parameter $b$ is infinitesimally given by
\begin{align}
\begin{split}
  (D_b P_0(\cdot,\theta,\wt g))(b') &= D_{g_{b_0,b}+\wt g}(\Ric+\Lambda)(\chi g'_b(b')) - \tdel^*\bigl(D_{g_{b_0,b}+\wt g}\Ups(\chi g'_b(b'))\bigr) \\
   &\quad + \tdel^*\bigl(D_{g_{b_0,b}}\Ups(\chi g'_b(b'))\bigr)
\end{split} \nonumber\\
\label{EqKdSStabDPb}
  &=L_{b,\wt g}(\chi g'_b(b')) + \tdel^*\bigl(D_{g_{b_0,b}}\Ups(\chi g'_b(b'))\bigr)
\end{align}
where we use $\frac{d}{ds}g_{b_0,b+sb'}|_{s=0}=\chi g'_b(b')$, see \eqref{EqKdSInPatchedMetric}.

Let us now reconsider the solvability result for $L_{b_0}$ described in Proposition~\ref{PropKdSLStabComplement} and use the specific structure of $g^{\prime\Ups}_{b_0}(b')=g'_{b_0}(b')+\delta_{g_{b_0}}^*\omega^\Ups_{b_0}(b')$ to arrive at a modification of the range which displays the change of the asymptotic gauge advertised above more clearly: namely, instead of $L_{b_0}(\chi g^{\prime\Ups}_{b_0}(b'))$ as in Proposition~\ref{PropKdSLStabComplement}, we use $L_{b_0}(\chi g'_{b_0}(b')+\delta_{g_{b_0}}^*(\chi\omega^\Ups_{b_0}(b')))$ as the modification, which is still compactly supported and thus can be used equally well to eliminate the asymptotic part $g^{\prime\Ups}_{b_0}(b')$ of linear waves. To see the benefit of this, we calculate
\begin{align}
  L_{b_0}&\bigl(\chi g'_{b_0}(b')+\delta_{g_{b_0}}^*(\chi\omega^\Ups_{b_0}(b'))\bigr) = L_{b_0}(\chi g'_{b_0}(b')) - \tdel^*\bigl(D_{g_{b_0}}\Ups(\delta_{g_{b_0}}^*(\chi\omega^\Ups_{b_0}(b')))\bigr) \nonumber\\
    &=L_{b_0}(\chi g'_{b_0}(b')) - \tdel^*\bigl(\bigl[D_{g_{b_0}}\Ups\circ\delta_{g_{b_0}}^*,\chi\bigr] \omega^\Ups_{b_0}(b')\bigr) + \tdel^*\bigl(\chi D_{g_{b_0}}\Ups(g'_{b_0}(b'))\bigr) \nonumber\\
\label{EqKdSStabGaugeRewrite}
    &=L_{b_0}(\chi g'_{b_0}(b')) + \tdel^*\bigl(D_{g_{b_0}}\Ups(\chi g'_{b_0}(b'))\bigr) - \tdel^*\theta_\chi(b'),
\end{align}
where we introduce the notation
\begin{equation}
\label{EqKdSStabGaugeCp}
  \theta_\chi(b') := \bigl[D_{g_{b_0}}\Ups\circ\delta_{g_{b_0}}^*,\chi\bigr]\omega^\Ups_{b_0}(b') + [D_{g_{b_0}}\Ups,\chi](g'_{b_0}(b')).
\end{equation}
The interpretation of the terms in \eqref{EqKdSStabGaugeRewrite} is clear: the first gives rise to linearized Kerr--de~Sitter asymptotics, corresponding to the first term in \eqref{EqKdSStabDPb}, the second corrects the gauge accordingly, corresponding to the second term in \eqref{EqKdSStabDPb} (note that $g_{b_0,b}\equiv g_b$ for $b=b_0$), and the final term patches up the gauge change in the transition region $\supp d\chi$; notice that $\theta_\chi(b')$ is \emph{compactly supported in $t_*$}. We moreover point out that the sum of the first two terms on the right hand side vanishes for large $t_*$ due to $D_{g_{b_0}}(\Ric+\Lambda)(g'_{b_0}(b'))=0$; exponential decay will be the appropriate and stable description when we discuss perturbations.

In order to put the non-linear stability problem into the framework developed in \S\ref{SubsecAsyExp}, we define the space
\begin{equation}
\label{EqKdSStabWtG}
  \wt G^s = \{ \wt g\in\Hbext^{s,\alpha}(\Omega;S^2\,\Tb^*_\Omega M) \colon \|\wt g\|_{\Hbext^{14,\alpha}} < \eps \},
\end{equation}
with $\eps>0$ sufficiently small for all our subsequent arguments---which rely mostly on the results of \S\ref{SubsecAsyExp}---to apply; moreover, we choose a trivialization $T_{\cU_B}B\cong B\times\R^4$. Then, we define the continuous map
\begin{equation}
\label{EqKdSStabZGauge}
\begin{split}
  z^\Ups \colon &\cU_B\times\wt G^{s+2}\times\R^4 \ni (b,\wt g,b') \\
    &\quad \mapsto L_{b,\wt g}(\chi g'_b(b')) + \tdel^*\bigl(D_{g_{b_0,b}}\Ups(\chi g'_b(b'))\bigr) \in \Hb^{s,\alpha}(\Omega;S^2\,\Tb^*_\Omega M)^{\bullet,-}
\end{split}
\end{equation}
for $s\geq 14$, which is just the linearization of $P_0$ in $b$ as in \eqref{EqKdSStabDPb}; the range of $z^\Ups$ consists of modifications which take care of changes of the asymptotic gauge. (The map $z^\Ups$ is certainly linear in $b'$, as $b'\mapsto g'_b(b')$ is linear.) Furthermore, we parameterize the space of compactly supported gauge modifications necessitated by these asymptotic gauge changes by
\begin{equation}
\label{EqKdSStabGaugeCpParam}
  \R^4 \ni b' \mapsto \theta_\chi(b') \in \CIc(\Omega^\circ;\Tb^*_\Omega M),
\end{equation}
with $\theta_\chi$ defined in \eqref{EqKdSStabGaugeCp}. (We could make this map depend on $b$ and $\wt g$, which may be more natural, though it makes no difference since our setup is stable under perturbations.)

We are now prepared to prove the main result of this paper: the non-linear stability of slowly rotating Kerr--de~Sitter spacetimes.

\begin{thm}
\label{ThmKdSStab}
  Let $h,k\in\CI(\Sigma_0;S^2T^*\Sigma_0)$ be initial data satisfying the constraint equations \eqref{EqHypDTConstraints}, and suppose $(h,k)$ is close to the Schwarzschild--de~Sitter initial data $(h_{b_0},k_{b_0})$ (see \eqref{EqKdSInKdSData}) in the topology of $H^{21}(\Sigma_0;S^2T^*\Sigma_0)\oplus H^{20}(\Sigma_0;S^2T^*\Sigma_0)$. Then there exist Kerr--de~Sitter black hole parameters $b\in B$, a compactly supported gauge modification $\theta$, lying in a fixed finite-dimensional space $\ol\Theta\subset\CIc(\Omega^\circ;T^*\Omega^\circ)$, and a section $\wt g\in\Hbext^{\infty,\alpha}(\Omega;S^2\,\Tb^*_\Omega M)$ such that the 2-tensor $g = g_{b_0,b} + \wt g$ is a solution of the Einstein vacuum equations
  \[
    \Ric(g)+\Lambda g = 0,
  \]
  attaining the given initial data $(h,k)$ at $\Sigma_0$, in the gauge $\Ups(g)-\Ups(g_{b_0,b})-\theta=0$ (see \eqref{EqKdSInGauge} for the definition of $\Ups$). More precisely, we obtain $(b,\theta,\wt g)$ and thus $g=g_{b_0,b}+\wt g$ as the solution of
  \begin{equation}
  \label{EqKdSStabIVP}
    \begin{cases}
      \Ric(g) + \Lambda g - \tdel^*\bigl(\Ups(g) - \Ups(g_{b_0,b}) - \theta\bigr) = 0 & \tn{in }\Omega^\circ, \\
      \gamma_0(g) = i_{b_0}(h,k) & \tn{on }\Sigma_0,
    \end{cases}
  \end{equation}
  where $i_{b_0}$ was defined in Proposition~\ref{PropKdSIn}.

  Moreover, the map
  \begin{equation}
  \label{EqKdSStabSmoothSolMap}
    \CI(\Sigma_0;S^2 T^*\Sigma_0)^2\ni(h,k)\mapsto(b,\theta,\wt g)\in\cU_B\times\ol\Theta\times\Hbext^{\infty,\alpha}(\Omega;S^2\,\Tb^*_\Omega M)
  \end{equation}
  is a smooth map of Fr\'echet spaces (in fact, a smooth tame map of tame Fr\'echet spaces) for $(h,k)$ in a neighborhood of $(h_{b_0},k_{b_0})$ in the topology of $H^{21}\oplus H^{20}$.
\end{thm}

Here, recall that $\alpha>0$ is a small fixed number, only depending on the spacetime $(M,g_{b_0})$ we are perturbing. Furthermore, we use any fixed Riemannian fiber metric on $S^2T^*\Sigma_0$, for instance the one induced by $h_{b_0}$, to define the $H^s$ norm of the initial data.

\begin{proof}[Proof of Theorem~\ref{ThmKdSStab}]
  Once we have solved \eqref{EqKdSStabIVP}, the fact that $g$ solves Einstein's equations in the stated gauge follows from the general discussion in \S\ref{SubsecHypDT}; we briefly recall the argument in the present setting: by definition of the map $i_{b_0}$, we have $\Ups(g)|_{\Sigma_0}=0$ (note that $\Ups(g_{b_0,b})=0$ near $\Sigma_0$ for all $b\in\cU_B$), hence $\Ups(g)-\theta=0$ at $\Sigma_0$ due to $\supp\theta\cap\Sigma_0=\emptyset$, and the constraint equations for $(h,k)$ imply that $\cL_{\pa_{t_*}}(\Ups(g)-\theta)=0$ at $\Sigma_0$ as well once we have solved \eqref{EqKdSStabIVP}; but then applying $\delta_g \sfG_g$ to \eqref{EqKdSStabIVP} implies the linear wave equation $\wtBoxCP_g(\Ups(g)-\Ups(g_{b_0,b})-\theta)=0$, hence $\Ups(g)-\Ups(g_{b_0,b})-\theta\equiv 0$ in $\Omega^\circ$ and therefore indeed $\Ric(g)+\Lambda g=0$.

  In order to solve \eqref{EqKdSStabIVP}, let $\Theta$ denote the finite-dimensional space constructed in Proposition~\ref{PropKdSLStabComplement}, and fix an isomorphism $\vartheta\colon\R^{N_\Theta}\xra{\cong}\Theta$, where $N_\Theta:=\dim\Theta$. We parameterize the modification space for the linear equations we will encounter by
  \begin{equation}
  \label{EqKdSStabParamZ}
  \begin{gathered}
    z \colon \cU_B\times\wt G^{s+2}\times\R^{4+4+N_\Theta} \to \Hb^{s,\alpha}(\Omega;S^2\,\Tb^*_\Omega M)^{\bullet,-} \hra D^{s,\alpha}(\Omega;S^2\,\Tb^*_\Omega M), \\
    z(b,\wt g,(b',b'_1,\bfc))=z^\Ups(b,\wt g,b')+\tdel^*\bigl(\theta_\chi(b'_1)+\vartheta(\bfc)\bigr),
  \end{gathered}
  \end{equation}
  using the maps \eqref{EqKdSStabZGauge} and \eqref{EqKdSStabGaugeCpParam}. Tensors in the range of $z^\Ups$ will be subsumed in changes of the asymptotic gauge condition. The non-linear differential operator we will consider is thus
  \begin{align*}
    P(b,b'_1,\bfc,\wt g)&:= (\Ric+\Lambda)(g_{b_0,b}+\wt g) \\
      &\qquad - \tdel^*\bigl(\Ups(g_{b_0,b}+\wt g)-\Ups(g_{b_0,b}) - \theta_\chi(b'_1) - \vartheta(\bfc)\bigr),
  \end{align*}
  with $(b,b'_1,\bfc,\wt g)\in\cU_B\times\R^{4+N_\Theta}\times\wt G^\infty$, and the non-linear equation we shall solve is
  \[
    \phi(b,b'_1,\bfc,\wt g) := \Bigl( P(b,b'_1,\bfc,\wt g),\ \gamma_0(\wt g)-\big(i_{b_0}(h,k)-\gamma_0(g_{b_0})\bigr)\Bigr) = 0.
  \]
  To relate this to the abstract Nash--Moser result, Theorem~\ref{ThmKdSStabNM}, we define the Banach spaces
  \[
    B^s := \R^4\times\R^{4+N_\theta} \times \Hbext^{s,\alpha}(\Omega;S^2\,\Tb^*_\Omega M), \quad \bfB^s = D^{s,\alpha}(\Omega;S^2\,\Tb^*_\Omega M);
  \]
  we will look for a solution near $u_0:=(b_0,0,\bfzero,0)$, for which $\phi(u_0)=(0,\gamma_0(g_{b_0})-i_{b_0}(h,k))$ is small in a Sobolev norm which we shall determine momentarily.

  The typical linearized equation we need to study in the Nash--Moser iteration is of the form
  \begin{equation}
  \label{EqKdSStabTypical}
    D_{(b,b'_1,\bfc,\wt g)}\phi(b',b_1'',\bfc',\wt r) = d = (f,r_0,r_1)\in \bfB^\infty;
  \end{equation}
  with $L_{b,\wt g}$, the linearization of $P$ in $\wt g$ around $(b,b'_1,\bfc,\wt g)$ (thus $L_{b,\wt g}$ does not depend on $b'_1$ and $\bfc$), given by the expression \eqref{EqKdSStabLinP}, this is equivalent to
  \[
    L_{b,\wt g}(\wt r) = f - z(b,\wt g,(b',b_1'',\bfc')),\quad \gamma_0(\wt r)=(r_0,r_1).
  \]
  Now the map $z$ satisfies the (surjective) assumptions of Theorem~\ref{ThmAsyExpMain}, in particular \eqref{EqAsyExpMainMod} with $b_0$ taking the place of $w_0$; surjectivity holds because of the $\vartheta$ term in the definition \eqref{EqKdSStabParamZ} of the map $z$, taking care of pure gauge modes, and the terms involving $b'$ and $b_1'$ which take care of the linearized Kerr--de~Sitter family in view of the computation \eqref{EqKdSStabGaugeRewrite}. Thus, we do obtain a solution
  \[
    \psi(b,b'_1,\bfc,\wt g)(f,r_0,r_1) := (b',b''_1,\bfc',\wt r)\in B^\infty
  \]
  of \eqref{EqKdSStabTypical} together with the estimates
  \[
    |b'|+|b''_1|+|\bfc'| \lesssim \|d\|_{13}, \quad \|\wt r\|_s \leq C_s\bigl(\|d\|_{s+3} + (1+\|\wt g\|_{s+6})\|d\|_{13}\bigr),
  \]
  for $s\geq 10$; this regularity requirement is the reason we need $2\cdot 10=20$ derivatives, see below. Two remarks are in order: first, the norm on $\wt g$ comes from the fact that for $\wt g\in\Hb^{s+6,\alpha}$, the non-smooth coefficients of the linearization of $\phi$ lie in $\Hb^{s+4,\alpha}$, corresponding to the norm on $\wt w$ in \eqref{EqAsyExpMainMod}; see also Remark~\ref{RmkAsyExpDataNumerology}. Second, the assumption on the skew-adjoint part of the linear operator at the radial set, $\wh\beta\geq -1$, in the statement of Theorem~\ref{ThmAsyExpMain} does hold; indeed, we showed $\wh\beta\geq 0$ in \S\ref{SubsecESGRadial}.
  
  Thus, we obtain \eqref{EqKdSStabNMTameSol} with $d=10$. One easily verifies that for this choice of $d$, the estimates \eqref{EqKdSStabNMMapCont} hold as well. (In fact, $d=4$ would suffice for the latter, see \cite[Proof of Theorem~5.10]{HintzVasyQuasilinearKdS}.) Theorem~\ref{ThmKdSStabNM} now says that we can solve $\phi(b,b'_1,\bfc,\wt g)=0$ provided $\|i_{b_0}(h,k)-\gamma_0(g_{b_0})\|_{H^{21}\oplus H^{20}}$ is small (here, $20=2d$), proving the existence of a solution of \eqref{EqKdSStabIVP} as claimed; the space $\ol\Theta$ in the statement of the theorem is equal to the sum of the ranges $\ol\Theta=\theta_\chi(\R^4)+\vartheta(\R^{N_\Theta})$.
  
  The smoothness of the solution map \eqref{EqKdSStabSmoothSolMap} (in fact with tame estimates), or indeed of
  \[
    (h,k)\mapsto (b,b'_1,\bfc,\wt g),
  \]
  follows from a general argument using the joint continuous dependence of the solution map for the linearized problem on the coefficients and the data, together with the fact that $\phi$ itself is a smooth tame map; see e.g.\ \cite[\S III.1.7]{HamiltonNashMoser} for details.
  
  The proof of non-linear stability is complete.
\end{proof}

\begin{rmk}
\label{RmkKdSStabRingdown}
  We explain in what sense one can see ringdown for the non-linear solution, at least in principle (since no rigorous results on shallow resonances for the linearized gauged Einstein equation are known): assume for the sake of argument that there is exactly one further resonance $\sigma$ in the strip $-\alpha_2<\Im\sigma<-\alpha<0$, where we assume to have high energy estimates \eqref{EqAsySmMeroEst} still, with 1-dimensional resonant space spanned by a resonant state $\varphi$; we assume that $\sigma$ is purely imaginary and $\varphi$ is real. The asymptotic expansion of the solution of the first linear equation that one solves in the Nash--Moser iteration then schematically is of the form $g'_{b_0}(b')+\eps g_{(1)}$, where $\eps$ is the size of the initial data, and $g_{(1)}=c\varphi+\wt g_{(1)}$, with $c\in\R$ and $\wt g_{(1)}\in\Hbext^{\infty,\alpha_2}$ of size $1$ (in $L^\infty$, say). Proceeding in the iteration scheme, we simply view $c\varphi+\wt g_{(1)}\in\Hbext^{\infty,\alpha}$, so the non-linear solution will be $g=g_{b_0,b}+\eps g_{(1)}+\eps^2 g_{(2)}$, with $g_{(2)}\in\Hbext^{\infty,\alpha}$ of size $1$. At the timescale $t_*=C^{-1}\log(\eps^{-1})$, for $C>0$ large only depending on $\alpha,\alpha_2$ and $\sigma$, the three components of $g$ are thus of size
  \[
    |\eps\varphi| \sim \eps^{-(\Im\sigma)/C+1}, \quad |\eps\wt g_{(1)}|\sim \eps^{\alpha_2/C+1},\quad |\eps^2 g_{(2)}|\sim \eps^{\alpha/C+2},
  \]
  so the term $\varphi$ coming from the refined partial expansion dominates by a factor $\eps^{-\delta}$ for some small $\delta>0$; in this sense, one can see the ringdown, embodied by $\varphi$ here, even in the non-linear solution. It would be very interesting to understand the asymptotic behavior of the non-linear solution more precisely, possibly obtaining a partial expansion using shallow resonances.
\end{rmk}

\subsection{Construction of initial data}
\label{SubsecKdSStabIn}

We briefly discuss three approaches to the construction of initial data sets in the context of Kerr--de~Sitter spacetimes. \emph{First}, Cortier \cite{CortierKdSGluing} described a gluing construction producing data sets with exact Kerr--de~Sitter ends, following work by Chru\'sciel--Pollack \cite{ChruscielPollackKottler} in the time-sym\-met\-ric (i.e. with vanishing second fundamental form) Schwarzschild--de~Sitter case. Such localized gluing methods for the constraint equations were first introduced by Corvino \cite{CorvinoScalar} for time-symmetric data; this restriction was subsequently removed by Corvino--Schoen \cite{CorvinoSchoenAsymptotics}, and Chru\'sciel--Delay generalized their analysis in \cite{ChruscielDelayMapping}.

\emph{Second}, by definition, one obtains initial data sets by selecting a spacelike hypersurface in a spacetime satisfying Einstein's equations. The point is that one may construct such spacetimes by solving the \emph{characteristic} initial value problem for Einstein's equations, the well-posedness of which was first proved by Rendall \cite{RendallCharacteristic}; the solution was later shown to exist in a full neighborhood of the in- and outgoing null cones by Luk \cite{LukCharacteristic}. For the characteristic problem, the constraint equations simplify dramatically, becoming simple transport equations rather than a non-linear coupled system of PDE (of elliptic type), see \cite[\S2.3]{LukCharacteristic}. In the case of interest for Theorem~\ref{ThmKdSStab} and adopting the notation of \cite{LukCharacteristic}, we can fix a 2-sphere $S_{0,0}$ at $t=t_0$, $r=r_0$ within Schwarzschild--de~Sitter space $(M,g_{b_0})$, with $t_0$ chosen so that $\Sigma_0$ lies entirely in the timelike future of $S_{0,0}$, and with $r_{b_0,-}<r_0<r_{b_0,+}$, so $r=r_0$ lies in the black hole exterior. We then consider the outgoing, resp.\ ingoing, future null cones $H_0$, resp.\ $\ul H_0$, which are swept out by the null-geodesics with initial velocities outgoing (increasing $r$), resp.\ ingoing (decreasing $r$), future null vectors orthogonal to $S_{0,0}$. Then, fixing the data of a Riemannian metric $\gamma_{AB}$, a 1-form $\zeta_A$ and functions $\tr\chi$ and $\tr\ul\chi$ on $S_{0,0}$, the constraint equations \cite[Equations~(8)--(11)]{LukCharacteristic} can be solved, at least locally near $S_{0,0}$, by solving suitable transport equations. If the data are equal to those induced by the metric $g_{b_0}$, the constraint equations of course do have a semi-global solution (namely the one induced by $g_{b_0}$), i.e.\ a solution defined on a portion $\Sigma_0^0$ of $H_0\cup\ul H_0$ extending past the horizons. See Figure~\ref{FigKdSStabInChar}.

\begin{figure}[!ht]
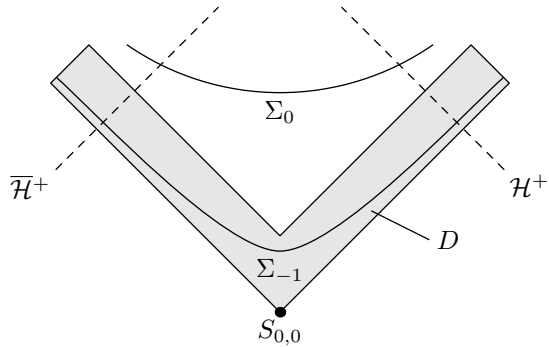

  \centering
  \inclfig{KdSStabInChar}
  \caption{Penrose diagram illustrating the construction of initial data for Theorem~\ref{ThmKdSStab} from characteristic initial data.}
  \label{FigKdSStabInChar}
\end{figure}

Thus, if one merely slightly perturbs the data, one still obtains a semi-global solution; one can then solve the characteristic initial value problem in a \emph{fixed} neighborhood $D$ of $\Sigma_0^0$ which contains a fixed spacelike hypersurface $\Sigma_{-1}$. If the characteristic data are close to those induced by $g_{b_0}$, the induced data on $\Sigma_{-1}$ are close to those induced by $g_{b_0}$. Given such data on $\Sigma_{-1}$, one can then either use a straightforward modification of Theorem~\ref{ThmKdSStab}, using $\Sigma_{-1}$ as the Cauchy hypersurface; alternatively, as depicted in Figure~\ref{FigKdSStabInChar}, one can solve the (non-characteristic) initial value problem with data on $\Sigma_{-1}$ in a domain which contains $\Sigma_0$ (provided the characteristic data were close to those induced by $g_{b_0}$), and the data on $\Sigma_0$ are close to $(h_{b_0},k_{b_0})$. (Theorem~\ref{ThmKdSStab} then applies directly.) These constructions can be performed for any desired level of regularity; recall here that Luk's result produces an $H^s$ solution for $H^{s+1}$ characteristic data for $s\geq 4$.

We finally discuss a \emph{third} approach, producing a sizeable set of solutions of the constraint equations directly, i.e.\ without using the above rather subtle tools. We will use a (slightly modified) conformal method, going back to Lichnerowicz \cite{LichnerowiczConstraints} and York \cite{YorkConstraints}; we refer the reader to the survey paper \cite{BartnikIsenbergConstraints} for further references. Our objective here is merely to construct initial data in the simplest manner possible. Thus, consider a compact hypersurface $\Sigma$, with smooth boundary, in the maximal analytic extension of a Schwarzschild--de~Sitter spacetime with parameters $b_0$, given by $t=0$ in static coordinates, which extends a bit past the bifurcation spheres of the future/past event horizon and the future/past cosmological horizon, see Figure~\ref{FigKdSStabInSurf}. Since $\Sigma$ is totally geodesic, the metric $g_{b_0}$ induces time-symmetric data $(h_0,0)$ on $\Sigma$. We shall construct initial data sets on $\Sigma$, which by a Cauchy stability argument as in the previous paragraph give rise to initial data sets on $\Sigma_0$.

\begin{figure}[!ht]
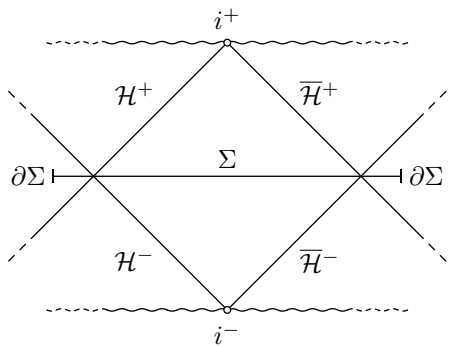

  \centering
  \inclfig{KdSStabInSurf}
  \caption{The totally geodesic hypersurface $\Sigma$, compact with $\CI$ boundary, within the maximal analytic extension of a Schwarzschild--de~Sitter spacetime.}
  \label{FigKdSStabInSurf}
\end{figure}

\begin{prop}
\label{PropKdSStabInConf}
  Let $s\geq s_0>3/2$. Then there exist $\eps=\eps(s_0)>0$ and $C=C(s)>0$ such that the following holds: for all constants $H\in\R$ and traceless, divergence-free (with respect to $h_0$) symmetric 2-tensors $\wt Q\in H^s(\Sigma;S^2T^*\Sigma)$ with $|H|+\|\wt Q\|_{H^{s_0}}<\eps$, there exists $\psi\in H^{s+2}(\Sigma)$ with the property that
  \[
    h = \varphi^4 h_0,\ k = \varphi^{-2}\wt Q + H\,h,\quad \varphi:=1+\psi,
  \]
  solve the constraint equations \eqref{EqHypDTConstraints} (with $n=3$), and $\|\psi\|_{H^{s+2}}\leq C(|H|+\|\wt Q\|_{H^s})$.
\end{prop}
\begin{proof}
  For general Riemannian metrics $h$ and symmetric 2-tensors $k=Q+H\,h$, with $Q$ tracefree, the constraint equations \eqref{EqHypDTConstraints} with $n=3$ read
  \[
    R_h-|Q|_h^2+6H^2+2\Lambda=0,\quad \delta_h Q+2\,dH = 0.
  \]
  Now, given $H\in\R$ and $\wt Q$ as in the statement of the proposition, define $h=\varphi^4 h_0$ and $Q=\varphi^{-2}\wt Q$, for $\varphi$ to be determined. Then we have $\delta_h Q=\varphi^{-6}\delta_{h_0}\wt Q=0=-2\,dH$, so the second constraint is always verified, while the first becomes
  \[
    P(\varphi;\wt Q,H):=\Delta\varphi+\frac{1}{8}R\varphi-\frac{1}{8}|\wt Q|^2\varphi^{-7}+\frac{1}{4}(3H^2+\Lambda)\varphi^5=0,
  \]
  where $\Delta\geq 0$ and $R$ are the Laplacian and the scalar curvature of $h_0$, respectively, and norms are taken with respect to $h_0$. This equation holds for $\wt Q=0$, $H=0$ and $\varphi=1$.
  
  Let us now extend $\Sigma$ to a closed 3-manifold $\wt\Sigma$ without boundary, and extend $h_0$ arbitrarily to a Riemannian metric on $\wt\Sigma$; we denote the extension by $h_0$ still. Extending $P$ by the same formula, we then have $P(1;0,0)=\frac{1}{8}R+\frac{1}{4}\Lambda=:f\in\CIdot(\ol{\Sigma^c})$, with the dot indicating infinite order of vanishing at $\pa\Sigma\subset\wt\Sigma$. Let us also extend $\wt Q$ to a symmetric 2-tensor on $\wt\Sigma$; we require neither the traceless nor the divergence-free condition to hold for the thus extended $\wt Q$ away from $\Sigma$. Applying the finite-codimension solvability idea in the present, elliptic, context, we now aim to solve the equation
  \begin{equation}
  \label{EqKdSStabInConfEqn}
    P(1+\psi;\wt Q,H)=f+z
  \end{equation}
  for $\psi$, where $z\in\CIdot(\ol{\Sigma^c})$ lies in a suitable fixed finite-dimensional space. Note that having solved \eqref{EqKdSStabInConfEqn}, we obtain a solution of the constraint equations in $\Sigma$ as in the statement of the proposition; what happens in $\ol{\Sigma^c}$ is irrelevant! In order to solve \eqref{EqKdSStabInConfEqn}, we rewrite the equation as
  \begin{equation}
  \label{EqKdSStabInConfEqn2}
  \begin{split}
    \wt L\psi+q(\psi)-z&=d,\quad d=\frac{1}{8}|\wt Q|^2-\frac{3}{4}H^2, \\
    &\quad \tn{where }\wt L=L+\frac{7}{8}|\wt Q|^2+\frac{15}{4}H^2,\ L=\Delta+\frac{1}{8}R+\frac{5}{4}\Lambda,
  \end{split}
  \end{equation}
  and $q(\psi)=\psi^2 q_0(\psi)$, with $q_0\colon H^s\to H^s$ continuous for all $s\geq s_0>3/2$: indeed, $q_0(\psi)$ is a rational function of $\psi$, with coefficients involving (powers of) $H$ and $\wt Q$; thus, $q$ depends on $H$ and $\wt Q$, though we drop this from the notation. The key observation is then we can use a unique continuation principle to determine a suitable space of $z$. Indeed, choosing a basis of the $L^2$-orthocomplement $\ran(L)^\perp=\mathspan\{v_1,\ldots,v_N\}\subset\CI(\wt\Sigma)$, unique continuation implies that $\{v_1,\ldots,v_N\}$ is linearly independent as a subset of $\CI(\wt\Sigma\setminus\Sigma)$, and we can therefore pick $\{z_1,\ldots,z_N\}\in\CIdot(\ol{\Sigma^c})$ such that the matrix $(\la v_i,z_j\ra)_{i,j=1,\ldots,N}$ is non-degenerate; letting $\cZ:=\mathspan\{z_1,\ldots,z_N\}$, this says that
  \[
    L' \colon H^{s+2}\oplus\cZ \to H^s,\quad (\psi,z)\mapsto L\psi-z,
  \]
  is an isomorphism for all $s\in\R$. Similarly defining $\wt L'(\psi,z)=L'\psi-z$, $\psi\in H^{s+2}$, $z\in\cZ$, for $s\geq s_0$, it follows that $\wt L'\colon H^{s_0+2}\oplus\cZ\to H^{s_0}$ is invertible if $|H|$ and $\|\wt Q\|_{H^{s_0}}$ are sufficiently small. But then, a contraction mapping argument using the map
  \[
    H^{s_0+2}\oplus\cZ \ni (\psi,z) \mapsto \bigl(\wt L'\bigr)^{-1}(d-q(\psi)) \in H^{s_0+2}\oplus\cZ,
  \]
  starting with  $(\psi,z)=(0,0)$, produces a solution of \eqref{EqKdSStabInConfEqn2}. A simple inductive argument using elliptic regularity for $L$ gives $\psi\in H^{s+2}$ if $\wt Q\in H^s$, $s\geq s_0$.
\end{proof}

\appendix

\section{b-geometry and b-analysis}
\label{SecB}

\subsection{b-geometry and b-differential operators}
\label{SubsecB}

In this appendix we recall the basics of b-geometry and b-analysis. As a general reference, we refer the reader to \cite{MelroseAPS}. Geometrically, b-analysis originates from the study of the Laplacian on manifolds with cylindrical ends (and this is the context of \cite{MelroseAPS}), but in fact analytically it arose in earlier work of Melrose on boundary problems for the wave equation, using b-, or totally characteristic, pseudodifferential operators to capture boundary regularity \cite{MelroseTransformation}. Recall that a (product) cylindrical metric on $M_\infty=\R_t\times X$ is one of the form $g_0=dt^2+h_0$, $h_0$ a Riemannian metric on $X$. In terms of the coordinate $\tau=e^{-t}$, which we consider for $t>0$ large (so $\tau$ is near $0$ and positive), thus $\frac{d\tau}{\tau}=-dt$, the cylindrical metric is of the form
\[
  g_0=\frac{d\tau^2}{\tau^2}+h_0.
\]
One then considers the compactification $M$ of $M_\infty$ by adding $\tau=0$ (similarly, at the end $t\to-\infty$, one would work with $\tau=e^{-|t|}$). Thus, locally, in the region where $\tau$ is small, the new manifold $M$ has a product structure $[0,\eps)_\tau\times X$. One advantage of this compactification is that working on compact spaces automatically ensures uniformity of many objects, such as estimates, though of course the latter can alternatively be encoded `by hand.' Smoothness of a function on $[0,\eps)_\tau\times X$ implies a Taylor series expansion at $\{0\}\times X$ in powers of $\tau$, i.e.\ $e^{-t}$. For instance, a metric of the form $g=a\frac{d\tau^2}{\tau^2}+h$, where $h$ is a smooth symmetric 2-cotensor and $a$ a smooth function on $M$ with $h|_{\tau=0}=h_0$ and $a|_{\tau=0}=1$, approaches $g_0$ exponentially fast in $t$.

In general then, we consider an $n$-dimensional manifold $M$ with boundary $X$, and denote by $\Vb(M)$ the space of \emph{b-vector fields}, which consists of all vector fields on $M$ which are tangent to $X$. In local coordinates $(\tau,x)\in[0,\infty)\times\R^{n-1}$ near the boundary, elements of $\Vb(M)$ are linear combinations, with $\CI(M)$ coefficients, of
\[
  \tau\pa_\tau,\pa_{x_1},\ldots,\pa_{x_{n-1}}.
\] 
(In terms of $t=-\log\tau$ as above, these are thus vector fields which are asymptotic to stationary vector fields at an exponential rate, and indeed they have an expansion in $e^{-t}$.) Correspondingly, elements of $\Vb(M)$ are sections of a natural vector bundle over $M$, the \emph{b-tangent bundle} $\Tb M$, the fibers of $\Tb M$ being spanned by $\tau\pa_\tau,\pa_{x_1},\ldots,\pa_{x_{n-1}}$, with $\tau\pa_\tau$ being a non-trivial b-vector field up to and including $\tau=0$ (even though it degenerates as an ordinary vector field).  The dual bundle, the \emph{b-cotangent bundle}, is denoted $\Tb^*M$. In local coordinates $(\tau,x)$ near the boundary as above, the fibers of $\Tb^*M$ are spanned by $\frac{d\tau}{\tau},dx_1,\ldots,dx_{n-1}$. A \emph{b-metric} $g$ on $M$ is then simply a non-degenerate section of the second symmetric tensor power of $\Tb^*M$, i.e.\ of the form
\[
  g = g_{00}(\tau,x)\frac{d\tau^2}{\tau^2} + \sum_{i=1}^{n-1} g_{0i}(\tau,x)\Bigl(\frac{d\tau}{\tau}\otimes dx_i + dx_i\otimes\frac{d\tau}{\tau}\Bigr) + \sum_{i,j=1}^{n-1} g_{ij}(\tau,x)dx_i\otimes dx_j,
\]
$g_{ij}=g_{ji}$, with smooth coefficients $g_{k\ell}$ such that the matrix $(g_{k\ell})_{k,\ell=0}^{n-1}$ is invertible. In terms of the coordinate $t=-\log\tau\in\R$, thus $\frac{d\tau}{\tau}=-dt$, the b-metric $g$ therefore approaches a stationary ($t$-independent in the local coordinate system) metric exponentially fast, as $\tau=e^{-t}$. A b-metric can have arbitrary signature, which corresponds to the signature of the matrix $(g_{k\ell})_{k,\ell=0}^{n-1}$; positive definite metrics (i.e.\ of signature $(n,0)$) are Riemannian, while those of signature $(1,n-1)$ (or $(n-1,1)$) are Lorentzian.

All natural tensorial constructions work equally well in the b-setting, such as the form bundles $\Lambdab^p M=\Lambda^p\,\Tb^*M$ and the symmetric tensor bundles $S^p\,\Tb^*M=\bigotimes^p_s\Tb^*M$; in particular, a b-metric is a smooth section of $S^2\,\Tb^*M$. Another important bundle is the b-density bundle $\Omegab M$, sections of which are smooth multiples of $|\frac{d\tau}{\tau}\wedge dx_1\ldots\wedge dx_{n-1}|$ in local coordinates; any b-metric of any signature gives rise to such a density via
\[
  |dg|=|\det (g_{k\ell})_{k,\ell=0}^{n-1}|^{1/2}\Bigl|\frac{d\tau}{\tau}\wedge dx_1\ldots\wedge dx_{n-1}\Bigr|.
\]
In particular, this gives rise to a positive definite inner product on $\CIc(M^\circ)$, or indeed $\CIdotc(M)$, the space of functions in $\CIc(M)$ which vanish at $X$ with all derivatives (i.e.\ to infinite order). The completion of $\CIdotc(M)$ in this inner product is $L^2_\bl(M)=L^2(M;|dg|)$.

The \emph{b-conormal bundle} $\Nb^*Y$ of a boundary submanifold $Y\subset X$ of $M$ is the subbundle of $\Tb^*_Y M$ whose fiber over $p\in Y$ is the annihilator of vector fields on $M$ tangent to $Y$ and $X$. In local coordinates $(\tau,x',x'')$, where $Y$ is defined by $x'=0$ in $X$, these vector fields are smooth linear combinations of $\tau\pa_\tau$, $\pa_{x''_j}$, $x'_i\pa_{x'_j}$, $\tau\pa_{x'_k}$, whose span in $\Tb_p M$ is that of $\tau\pa_\tau$ and $\pa_{x''_j}$, and thus the fiber of the b-conormal bundle is spanned by the $dx'_j$, i.e.\ has the same dimension as the codimension of $Y$ in $X$ (and {\em not} that in $M$, corresponding to $\frac{d\tau}{\tau}$ not annihilating $\tau\pa_\tau$).

We define the \emph{b-cosphere bundle} $\Sb^*M$ to be the quotient of $\Tb^*M\setminus o$ by the $\R^+$-action; here $o$ is the zero section. Likewise, we define the spherical b-conormal bundle of a boundary submanifold $Y\subset X$ as the quotient of $\Nb^*Y\setminus o$ by the $\R^+$-action; it is a submanifold of $\Sb^*M$. A better way to view $\Sb^*M$ is as the boundary at fiber infinity of the fiber-radial compactification $\ol{\Tb^*}M$ of $\Tb^*M$, where the fibers $\Tb^*_p M$ are replaced by their radial compactification
\begin{equation}
\label{EqBCompact}
  \ol{\Tb^*_p}M=\bigl(\Tb^*_p M\cup ([0,\infty)_x\times\Sph^{n-1})\bigr)/\sim,
\end{equation}
where the equivalence relation $\sim$ identifies $(x,\omega)$, $x>0$, with $x^{-1}\omega\in\Tb^*_p M$, upon choosing polar coordinates on $\Tb^*_p M\cong\R^n$; see also \cite[\S2]{VasyMicroKerrdS}. The b-cosphere bundle $\Sb^*M\subset\ol{\Tb^*}M$ still contains the boundary of the compactification of the `old' boundary $\ol{\Tb^*_X}M$, see Figure~\ref{FigBTbM}.

 \begin{figure}[!ht]
  \centering
  \inclfig{Bctgt}
  \caption{The radially compactified cotangent bundle $\ol{\Tb^*}M$ near $\ol{\Tb^*_X}M$. The horizontal axis is the base $M$, with its boundary $X$ on the right; the vertical axis is the typical fiber of $\ol{\Tb^*}M$ over a point. The cosphere bundle $\Sb^*M$, viewed as the boundary at fiber infinity of $\ol{\Tb^*}M$, is also shown, as well as the zero section $o_M\subset\ol{\Tb^*}M$ and the zero section over the boundary $o_X\subset\ol{\Tb^*_X}M$.}
  \label{FigBTbM}
\end{figure}

Next, the algebra $\Diffb(M)$ of \emph{b-differential operators} generated by $\Vb(M)$ consists of operators of the form
\[
  \cP = \sum_{|\alpha|+j\leq m} a_\alpha(\tau,x)(\tau D_\tau)^j D_x^\alpha,
\]
with $a_\alpha\in\CI(M)$, writing $D=\frac{1}{i}\pa$ as usual. (With $t=-\log\tau$ as above, the coefficients of $\cP$ are thus constant up to exponentially decaying remainders as $t\to\infty$.) Writing elements of $\Tb^*M$ as
\begin{equation}
\label{EqBDualCoords}
  \sigma\,\frac{d\tau}{\tau}+\sum_j\xi_j\,dx_j,
\end{equation}
we have the principal symbol
\[
  \sigma_{\bl,m}(\cP) = \sum_{|\alpha|+j=m} a_\alpha(\tau,x) \sigma^j\xi^\alpha,
\]
which is a homogeneous degree $m$ function in $\Tb^*M\setminus o$. (The subscripts on the notation $\sigma_{\bl,m}$ of the principal symbol distinguish it from the dual variable $\sigma$.) Principal symbols are multiplicative, i.e.\ $\sigma_{\bl,m+m'}(\cP\circ\cP')=\sigma_{\bl,m}(\cP)\sigma_{\bl,m'}(\cP')$, and one has a connection between operator commutators and Poisson brackets, to wit
\[
  \sigma_{\bl,m+m'-1}(i[\cP,\cP']) = \ham_p p', \quad p=\sigma_{\bl,m}(\cP),\ p'=\sigma_{\bl,m'}(\cP'),
\]
where $\ham_p$ is the extension of the Hamilton vector field from $T^*M^\circ\setminus o$ to $\Tb^*M\setminus o$, which is thus a homogeneous degree $m-1$ vector field on $\Tb^*M\setminus o$ tangent to the boundary $\Tb^*_X M$. In local coordinates $(\tau,x)$ on $M$ near $X$, with b-dual coordinates $(\sigma,\xi)$ as in \eqref{EqBDualCoords}, this has the form
\[
  \ham_p=(\pa_\sigma p)(\tau\pa_\tau)-(\tau\pa_\tau p)\pa_\sigma+\sum_j \big((\pa_{\xi_j}p)\pa_{x_j}-(\pa_{x_j}p)\pa_{\xi_j}\big),
\]
see \cite[Equation~(3.20)]{BaskinVasyWunschRadMink}, where a somewhat different notation is used, given by \cite[Equation~(3.19)]{BaskinVasyWunschRadMink}. 

We are also interested in b-differential operators acting on sections of vector bundles on $M$. If $E$, $F$ are vector bundles over $M$ of rank $N_E,N_F$, respectively, then in coordinate charts over which $E,F$ are trivialized, such operators $\cP\in\Diffb^m(M;E,F)$, so $P\colon\CI(M;E)\to\CI(M;F)$, are simply $N_F\times N_E$ matrices of (scalar) b-differential operators $\cP_{ij}\in\Diffb^m(M)$. An example is the b-version of the exterior differential $\bdiff\colon\CI(M;\Lambdab^pM)\to\CI(M;\Lambdab^{p+1}M)$, $\bdiff\in\Diffb^1(M;\Lambdab^p M,\Lambdab^{p+1}M)$, given for $p=0$ by
\[
  \bdiff u=(\tau\pa_\tau u)\,\frac{d\tau}{\tau}+\sum_{j=1}^{n-1} (\pa_{x_j}u)\,dx_j,
\]
and extended to the higher degree differential forms in the usual manner, so
\[
  \bdiff (u\,dx_{i_1}\wedge\ldots\wedge dx_{i_p})=(\bdiff u)\wedge dx_{i_1}\wedge\ldots\wedge dx_{i_p}
\]
and (note that $\frac{d\tau}{\tau}=d\log\tau$)
\[
  \bdiff \Big(u\,\frac{d\tau}{\tau}\wedge dx_{i_1}\wedge\ldots\wedge dx_{i_{p-1}}\Big)=(\bdiff u)\wedge \frac{d\tau}{\tau}\wedge dx_{i_1}\wedge\ldots\wedge dx_{i_{p-1}}.
\]
Thus, $\bdiff=d$ is the usual exterior differential away from $X=\pa M$ if one uses the natural identification of $\Tb M$ with $T M$ away from $X$, likewise for the associated bundles.

If $E,F$ are real vector bundles and $h_E,h_F$ are inner products {\em of any signature} (i.e.\ bilinear symmetric non-degenerate maps to the reals) on the fibers of $E,F$ respectively, and $\nu$ is a non-degenerate b-density (e.g.\ the density $|dg|$ of a b-metric) then $\cP\in\Diffb^m(M;E,F)$ has an adjoint $\cP^*\in\Diffb^m(M;F,E)$ characterized by
\[
  \la\cP u,v\ra_F=\la u,\cP^* v\ra_E,\quad u\in\CIdot(M;E),\ v\in\CIdot(M;F),
\]
where
\[
  \la u_1,u_2\ra_E=\int h_E(u_1,u_2)\,\nu,\quad u_1,u_2\in\CIdot(M;E),
\]
and similarly for $F$. We maintain the same notation for the complexified bundles to which $h_E,h_F$ extend as sesquilinear fiber inner products.  In particular, any non-degenerate b-metric $g$ induces inner products (of various signature!) on $\Lambdab M$ and $S^p\,\Tb^*M$; an example of an adjoint is $\bdiff^*=\delta\in\Diffb^1(M;\Lambdab^{p+1}M,\Lambdab^p M)$. Other important geometric operators include the covariant derivative with respect to a b-metric $g$, $\nabla\in\Diffb^1(M;S^p\,\Tb^*M,\Tb^*M\otimes S^p\,\Tb^*M)$, the symmetric gradient $\delta_g^*\in\Diffb^1(M;\Tb^*M, S^2\,\Tb^*M)$, and the divergence $\delta_g\in\Diffb^1(M,S^2\,\Tb^*M,\Tb^*M)$, besides bundle endomorphisms such as the Ricci curvature of a fixed b-metric $g$, $\Ric(g)\in\Diffb^0(M;\Tb^*M,\Tb^*M)$. For most analytic purposes the bundles are irrelevant, and thus we suppress them in the notation below.

While elements of $\Diffb(M)$ commute to leading order in the symbolic sense, they do not commute in the sense of the order of decay of their coefficients. (This is in contrast to the scattering algebra, see \cite{MelroseEuclideanSpectralTheory}.) The \emph{normal operator} captures the leading order part of $\cP\in\Diffb^m(M)$ in the latter sense, namely
\[
  N(\cP) = \sum_{j+|\alpha|\leq m} a_\alpha(0,x)(\tau D_\tau)^j D_x^\alpha.
\]
One can define $N(\cP)$ invariantly as an operator on the model space $M_I:=[0,\infty)_\tau\times X$ by fixing a boundary defining function of $M$, see \cite[\S3]{VasyMicroKerrdS}. Identifying a collar neighborhood of $X\subset M$ with a neighborhood of $\{0\}\times X$ in $M_I$, we then have $\cP-N(\cP)\in\tau\Diffb^m(M)$ (near $\pa M$). Since $N(\cP)$ is dilation-invariant (equivalently: translation-invariant in $t=-\log\tau$), it is naturally studied via the Mellin transform in $\tau$ (equivalently: Fourier transform in $-t$), which leads to the \emph{(Mellin transformed) normal operator family}
\[
  \wh N(\cP)(\sigma) \equiv \wh\cP(\sigma) = \sum_{j+|\alpha|\leq m} a_\alpha(0,x)\sigma^j D_x^\alpha,
\]
which is a holomorphic family of operators $\wh\cP(\sigma)\in\Diff^m(X)$.  Here the Mellin transform is the map
\begin{equation}
\label{EqBMellinTrafo}
  \cM\colon u\mapsto \wh u(\sigma,.)=\int_0^\infty \tau^{-\imath\sigma} u(\tau,.)\,\frac{d\tau}{\tau},
\end{equation}
with inverse transform
\[
  \cM^{-1}\colon v\mapsto \check v(\tau,.)=\frac{1}{2\pi}\int_{\R+\imath\alpha} \tau^{\imath\sigma} v(\sigma,.)\,d\sigma,
\]
with $\alpha$ chosen in the region of holomorphy of $v$. Note that for $u$ which are supported near $\tau=0$ and are polynomially bounded as $\tau\to 0$, with values in a space such as $\C$, $\CI(X)$, $L^2(X)$ or $\CmI(X)$, the Mellin transform $\cM u$ is holomorphic in $\Im\sigma>C$, $C>0$ sufficiently large, with values in the same space. The Mellin transform is described in detail in \cite[\S5]{MelroseAPS}, but the reader should keep in mind that it is a renormalized Fourier transform, corresponding to the exponential change of variables $\tau=e^{-t}$ mentioned above, so results for it are equivalent to related results for the Fourier transform.  The $L^2$-based result, Plancherel's theorem, states that if $\nu$ is a smooth non-degenerate density on $X$ and $r_c$ denotes restriction to the line $\Im\sigma=c$, then
\begin{equation}
\label{EqBMellinWeightL2}
  r_{-\alpha}\circ\cM:\tau^\alpha L^2\Bigl(X\times[0,\infty);\frac{|d\tau|}{\tau}\nu\Bigr)\to L^2(\R;L^2(X;\nu))
\end{equation}
is an isomorphism. We are interested in functions $u$ supported near $\tau=0$, in which case, with $r_{(c_1,c_2)}$ denoting restriction to the strip $c_1<\Im\sigma<c_2$, for $N>0$,
\begin{equation}
\begin{split}
\label{EqBMellinRangeL2}
  &r_{-\alpha,-\alpha+N}\circ\cM:\tau^\alpha(1+\tau)^{-N} L^2\Bigl(X\times[0,\infty);\frac{|d\tau|}{\tau}\nu\Bigr) \\
  &\qquad\to \Big\{v:\R\times\imath(-\alpha,-\alpha+N)\ni\sigma \to v(\sigma)\in L^2(X;\nu); \\
  &\qquad\qquad v\ \text{is holomorphic in}\ \sigma\ \text{and}\ \sup_{-\alpha<r<-\alpha+N}\|v(.+\imath r,.)\|_{L^2(\R;L^2(X;\nu))}<\infty\Big\},
\end{split}
\end{equation}
see \cite[Lemma~5.18]{MelroseAPS}. Note that in accordance with \eqref{EqBMellinWeightL2}, $v$ in \eqref{EqBMellinRangeL2} extends continuously to the boundary values, $r=-\alpha$ and $r=-\alpha-N$, with values in the same space as for holomorphy. Moreover, for functions supported in, say, $\tau<1$, one can take $N$ arbitrary.

\subsection{b-pseudodifferential operators and b-Sobolev spaces}
\label{SubsecBPsdo}

Passing from $\Diffb(M)$ to the algebra of \emph{b-pseudodifferential operators} $\Psib(M)$ amounts to allowing symbols to be more general functions than polynomials; apart from symbols being smooth functions on $\Tb^*M$ rather than on $T^*M$ if $M$ was boundaryless, this is entirely analogous to the way one passes from differential to pseudodifferential operators, with the technical details being a bit more involved. One can have a rather accurate picture of b-pseudodifferential operators, however, by considering the following: for $a\in\CI(\Tb^* M)$, we say $a\in S^m(\Tb^* M)$ if $a$ satisfies
\begin{equation}
\label{EqBSymbolEst}
  |\pa_z^\alpha\pa_\zeta^\beta a(z,\zeta)|\leq C_{\alpha\beta}\la\zeta\ra^{m-|\beta|}\tn{ for all multiindices }\alpha,\beta
\end{equation}
in any coordinate chart, where $z$ are coordinates in the base and $\zeta$ coordinates in the fiber; more precisely, in local coordinates $(\tau,x)$ near $X$, we take $\zeta=(\sigma,\xi)$ as above. We define the quantization $\Op(a)$ of $a$, acting on smooth functions $u$ supported in a coordinate chart, by
\begin{align}
\label{EqBOpMap}
  \Op(a)u(\tau,x)=(2\pi)^{-n}\int &e^{i(\tau-\tau')\wt\sigma+i(x-x')\xi}\phi\left(\frac{\tau-\tau'}{\tau}\right) \\
    &\quad\times a(\tau,x,\tau\wt\sigma,\xi)u(\tau',x')\,d\tau'\,dx'\,d\wt\sigma\,d\xi,
\end{align}
where the $\tau'$-integral is over $[0,\infty)$, and $\phi\in\CI_c((-1/2,1/2))$ is identically $1$ near $0$. The cutoff $\phi$ ensures that these operators lie in the `small b-calculus' of Melrose, in particular that such quantizations act on weighted b-Sobolev spaces, defined below; see also the explicit description of the Schwartz kernels below using blow-ups. For general $u$, we define $\Op(a)u$ using a partition of unity. We write $\Op(a)\in\Psib^m(M)$; every element of $\Psib^m(M)$ is of the form $\Op(a)$ for some $a\in S^m(\Tb^*M)$ modulo the set $\Psib^{-\infty}(M)$ of smoothing operators. We say that $a$ is a \emph{symbol} of $\Op(a)$. The equivalence class of $a$ in $S^m(\Tb^*M)/S^{m-1}(\Tb^*M)$ is invariantly defined on $\Tb^*M$ and is called the \emph{principal symbol} of $\Op(a)$.

A different way of looking at $\Psib(M)$ is in terms of H\"ormander's uniform algebra, namely pseudodifferential operators on $\R^n$ arising as, say, left quantizations of symbols $\wt a\in S^m_\infty(\R^n_{\wt z};\R^n_{\wt\zeta})$ satisfying estimates
\begin{equation}
\label{EqBNoBHormUnifSymbol}
  |\pa_{\wt z}^\alpha\pa_{\wt\zeta}^\beta \wt a(\wt z,\wt\zeta)|\leq C_{\alpha\beta}\la\wt\zeta\ra^{m-|\beta|}\tn{ for all multiindices }\alpha,\beta.
\end{equation}
To see the connection, consider local coordinates $(\tau,x)$ on $M$ near $\pa M$ as above, and write $\wt z=(t,x)$ with $t=-\log\tau$, with the region of interest being a cylindrical set $(C,\infty)_t\times\Omega_x$ corresponding to $(0,e^{-C})_\tau\times\Omega_x$. Then the uniform estimates \eqref{EqBNoBHormUnifSymbol} are equivalent to estimates (pulling back $\wt a$ via the map $\psi(\tau,x)=(-\log\tau,x)$)
\begin{equation}
\label{EqBHormUnifSymbol}
\begin{split}
  &|(\tau\pa_{\tau})^{\alpha_0}\pa_x^{\alpha'}\pa_{\sigma}^{\beta_0}\pa_{\xi}^{\beta'} \psi^*\wt a(\tau,x,\sigma,\xi)|\leq C_{\alpha\beta}\la(\sigma,\xi)\ra^{m-|\beta|} \\
  &\qquad\qquad\text{ for all multiindices }\alpha=(\alpha_0,\alpha'),\beta=(\beta_0,\beta'),
\end{split}
\end{equation}
and the quantization map becomes
\begin{align*}
  &\wt{\Op}(\wt a)u(\tau,x) \\
  &\qquad=(2\pi)^{-n}\int e^{i\log(\tau/\tau')\sigma+i(x-x')\xi}\psi^*\wt a(\tau,x,-\sigma,\xi)u(\tau',x')(\tau')^{-1}\,d\tau'\,dx'\,d\sigma\,d\xi,
\end{align*}
Letting $\wt\sigma=\tau^{-1}\sigma$, this reduces to an oscillatory integral of the form \eqref{EqBOpMap}, taking into account that in $\tau/\tau'\in (C^{-1},C)$, $C>0$, the function $\log(\tau/\tau')$ in the phase is equivalent to $\frac{\tau-\tau'}{\tau}$. (Notice that in terms of $t,t'$, the cutoff $\phi$ in \eqref{EqBOpMap} is a compactly supported function of $t-t'$, identically $1$ near $0$.) With $z=(\tau,x)$, $\zeta=(\sigma,\xi)$, these estimates \eqref{EqBHormUnifSymbol} would be exactly the estimates \eqref{EqBSymbolEst} if $(\tau\pa_{\tau})^{\alpha_0}$ were replaced by $\pa_\tau^{\alpha_0}$. Thus, \eqref{EqBHormUnifSymbol} gives rise to the space of b-ps.d.o's {\em conormal to the boundary}, i.e.\ in terms of b-differential operators, the coefficients are allowed to be merely conormal to $X=\pa M$ rather than smooth up to it. As in the setting of classical (one-step polyhomogeneous) symbols, for distributions smoothness up to the boundary is equivalent to conormality (symbolic estimates in the symbol case) plus an asymptotic expansion; thus, apart from the fact that we need to be careful in discussing supports, the b-ps.d.o.\ algebra is essentially locally a subalgebra of H\"ormander's uniform algebra. Most properties of b-ps.d.o's are true even in this larger, `conormal coefficients' class, and indeed this perspective is very important when the coefficients are generalized to have merely finite Sobolev regularity as was done in \cite{HintzQuasilinearDS,HintzVasyQuasilinearKdS}; indeed the `only' significant difference concerns the normal operator, which does not make sense in the conormal setting. We also refer to \cite[Chapter~6]{VasyPropagationNotes} for a full discussion, including an introduction of localizers, far from diagonal terms, etc.

If $A\in\Psib^{m_1}(M)$ and $B\in\Psib^{m_2}(M)$, then $AB,BA\in\Psib^{m_1+m_2}(M)$, while $[A,B]\in\Psib^{m_1+m_2-1}(M)$, and its principal symbol is $\frac{1}{i}\ham_a b\equiv\frac{1}{i}\{a,b\}$, with $\ham_a$ as above.

We also recall the notion of \emph{b-Sobolev spaces}: fixing a volume b-density $\nu$ on $M$, which locally is a positive multiple of $|\frac{d\tau}{\tau}\,dx|$, we define, for $s\in\N_0$,
\[
  \Hb^s(M) = \bigl\{ u\in L^2(M,\nu) \colon V_1\cdots V_j u\in L^2(M,\nu), V_i\in\Vb(M), 1\leq i\leq j\leq s \bigr\},
\]
which one can extend to $s\in\R$ by duality and interpolation. \emph{Weighted b-Sobolev spaces} are denoted
\begin{equation}
\label{EqBSobolevWeighted}
  \Hb^{s,\alpha}(M)=\tau^\alpha\Hb^s(M),
\end{equation}
i.e.\ its elements are of the form $\tau^\alpha u$ with $u\in\Hb^s(M)$. Any b-pseudodifferential operator $\cP\in\Psib^m(M)$ defines a bounded linear map $\cP\colon\Hb^{s,\alpha}(M)\to\Hb^{s-m,\alpha}(M)$ for all $s,\alpha\in\R$. Correspondingly, there is a notion of wave front set $\WFb^{s,\alpha}(u)\subset\Sb^*M$ for a distribution $u\in\Hb^{-\infty,\alpha}(M)=\bigcup_{s\in\R}\Hb^{s,\alpha}(M)$, defined analogously to the wave front set of distributions on $\R^n$ or closed manifolds: a point $\varpi\in\Sb^*M$ is \emph{not} in $\WFb^{s,\alpha}(u)$ if and only if there exists $\cP\in\Psib^0(M)$, elliptic at $\varpi$ (i.e.\ with principal symbol non-vanishing on the ray corresponding to $\varpi$), such that $\cP u\in\Hb^{s,\alpha}(M)$. Notice however that we \emph{do} need to have a priori control on the weight $\alpha$ (we are {\em assuming} $u\in\Hb^{-\infty,\alpha}(M)$), which again reflects the lack of commutativity of $\Psib(M)$ to leading order in the sense of decay of coefficients at $\pa M$.

The Mellin transform is also well-behaved on the b-Sobolev spaces $\Hb^s(X\times[0,\infty))$, and indeed gives a direct way of defining non-integer order Sobolev spaces.  For $s\geq 0$, cf.\ \cite[Equation~(5.41)]{MelroseAPS},
\begin{equation}
\label{EqBSobolevIso}
\begin{split}
  &r_{-\alpha}\circ\cM:\tau^\alpha \Hb^s(X\times[0,\infty);\frac{|d\tau|}{\tau}\nu)\\
  &\to \Big\{v\in L^2(\R;H^s(X;\nu)):\  (1+|\sigma|^2)^{s/2}v\in L^2(\R;L^2(X;\nu))\Big\}
\end{split}
\end{equation}
is an isomorphism, with the analogue of \eqref{EqBMellinRangeL2} also holding. Note that the right hand side of \eqref{EqBSobolevIso} is equivalent to
\begin{equation}
\label{EqBSobolevIsoMod}
  \la|\sigma|\ra^s v\in L^2\bigl(\R;H^s_{\la|\sigma|\ra^{-1}}(X;\nu)\bigr),
\end{equation}
where the space on the right hand side is the standard semiclassical Sobolev space and $\la|\sigma|\ra=(1+|\sigma|^2)^{1/2}$; indeed, for $s\geq 0$ integer both are equivalent to the statement that for all $\beta$ with $|\beta|\leq s$, $\la|\sigma|\ra^{s-|\beta|} D_x^\beta v\in L^2(\R;L^2(X;\nu))$. Here by equivalence we mean not only the membership in a set, but also that of the standard norms, such as
\[
  \Big(\sum_{|\beta|\leq s}\int_{\Im\sigma=-\alpha}\la|\sigma|\ra^{2(s-|\beta|)}\|D_x^\beta v\|^2_{L^2(X;\nu)}\,d\sigma\Big)^{1/2},
\]
corresponding to these spaces. Note that by dualization, \eqref{EqBSobolevIsoMod} characterizes the Mellin transform of $\Hb^{s,\alpha}$ for all $s\in\R$.

The basic microlocal results, such as elliptic regularity, propagation of singularities and radial point estimates, have versions in the b-setting; these are purely symbolic (i.e.\ do not involve normal operators), and thus by themselves are insufficient for a Fredholm analysis, since the latter requires estimates with relatively compact errors. It is usually convenient to state these results in terms of wave front set containments, but by the closed graph theorem such statements are automatically equivalent to microlocalized Sobolev estimates, and are indeed often proved by such.

Let us first discuss microlocal elliptic regularity for a classical operator $\cP\in\Psib^m(M)$. We recall that $\cP$ elliptic at $\varpi\in\Sb^*M$ if $\sigma_{\bl}(\cP)$ is invertible at $\varpi$; here one renormalizes the principal symbol by using any non-degenerate homogeneous degree $m$ section of $\Tb^*M$ so that the restriction to $\Sb^*M$, considered as fiber infinity, makes sense.

\begin{prop}
  Suppose $u\in\Hb^{s',\alpha}$ for some $s',\alpha$, and $\varpi\notin\WFb^{s-m,\alpha}(\cP u)$. Then $\varpi\notin\WFb^{s,\alpha}(u)$. Quantitatively, the estimate
  \[
    \|B_2 u\|_{\Hb^{s,\alpha}}\leq C\bigl(\|B_1\cP u\|_{\Hb^{s-m,\alpha}}+\|u\|_{\Hb^{s',\alpha}}\bigr),
  \]
  is valid whenever $B_1,B_2\in\Psib^0(M)$, with $B_1$ elliptic on $\WFb'(B_2)$ and $\cP$ elliptic on $\WFb'(B_2)$.
\end{prop}

Next, to describe propagation of singularities for a classical operator $\cP\in\Psib^m(M)$ with real principal symbol $p$ (scalar if $\cP$ is acting on vector bundles), we recall that the \emph{characteristic set} $\Char(P)$ is the complement of its elliptic set (the set of points where $\cP$ is elliptic). Then:
\begin{prop}
  Let $u\in\Hb^{s',\alpha}$ for some $s',\alpha$. Then $\WFb^{s,\alpha}(u)\setminus\WFb^{s-m+1,\alpha}(\cP u)$ is the union of maximally extended (null) bicharacteristics, i.e.\ integral curves of $\ham_p$ inside the characteristic set $\Char(\cP)$ of $\cP$ inside $\Char(\cP)\setminus \WFb^{s-m+1,\alpha}(\cP u)$.
\end{prop}
This statement is vacuous at points $\varpi$ where $\ham_p$ is radial, i.e.\ tangent to the dilation orbits in the fibers of $\Tb^*M\setminus o$. Elsewhere, it again amounts to an estimate, which now is of the form
\begin{equation}
\label{EqBProp}
  \|B_2 u\|_{\Hb^{s,\alpha}}\leq C\bigl(\|B_1 u\|_{\Hb^{s,\alpha}}+\|S\cP u\|_{\Hb^{s-m+1,\alpha}}+\|u\|_{\Hb^{s',\alpha}}\bigr),
\end{equation}
which is valid whenever $B_1,B_2,S\in\Psib^0(M)$ with $S$ elliptic on $\WFb'(B_2)$, and every bicharacteristic from $\WFb'(B_2)$ reaching the elliptic set of $B_1$, say in the backward direction along $\ham_p$, while remaining in the elliptic set in $S$; this estimate gives propagation in the forward direction along $\ham_p$.

The estimate~\eqref{EqBProp} remains valid, for propagation in the forward direction, if $p$ is no longer real, but $\Im p\leq 0$, and in the backward direction if $\Im p\geq 0$. Such an operator is called \emph{complex absorbing}. Notice that one has a better, elliptic, estimate where $\Im p>0$; the point is that the propagation of singularities estimate works at the boundary of this region.

Radial points of $\ham_p$ come in many flavors depending on the linearization of $\ham_p$. In the present context, at the b-conormal bundle of the boundary of the event or cosmological horizon at infinity, the important type is saddle points, or more precisely submanifolds $L$ of normally saddle points, introduced in \cite[\S2.1.1]{HintzVasySemilinear} in the b-setting. (See \cite{BaskinVasyWunschRadMink} for a source/sink case, which is relevant to the wave equation on Minkowski type spaces.) More concretely, the type of saddle point is that (within the characteristic set of $\cP$) one of the stable/unstable manifolds lies in $\Sb^*_X M$, and the other, call it $\cL$, is transversal to $\Sb^*_XM$, with the full assumptions stated in \cite[\S2.1.1]{HintzVasySemilinear}; see also the discussion of the dynamics in \S\ref{SubsecKdSGeo}. In fact, one should really consider at least the infinitesimal behavior of the linearization towards the interior of the cotangent bundle as well, i.e.\ work on $\ol{\Tb^*}M$, with $\Sb^*M$ its boundary at fiber infinity; then with $L$ a submanifold of $\Sb^*_XM$ still, we are interested in the setting in which the statement about stable/unstable manifolds still holds in $\ol{\Tb^*}M$. Then there is a critical regularity, $s=\frac{m-1}{2}+\beta\alpha$, where $\alpha$ is the Sobolev weight order as above, and $\beta$ arises from the subprincipal symbol of $\cP$ at $L$; in the case of event horizons it is the reciprocal surface gravity. (If $\cP\neq\cP^*$, there is a correction term to the critical regularity, see \S\ref{SubsecAsySm}.) Namely, the theorem, \cite[Proposition~2.1]{HintzVasySemilinear}, states:
\begin{itemize}
\item for $s>\frac{m-1}{2}+\beta\alpha$ one can propagate estimates from a punctured neighborhood of $\cL\cap\Sb^*_X M$ in $\cL$ to $X$;
\item if $s<\frac{m-1}{2}+\beta\alpha$, the opposite direction of propagation is possible.
\end{itemize}
\begin{rmk}
  This is the redshift effect in the direction of propagation into the boundary since one has then a source/sink within the boundary; that is, in the direction in which the estimates are propagated, the linearization at $L$ infinitesimally shifts the frequency (where one is in the fibers of $\Tb^*M$) away from fiber infinity, i.e.\ to lower frequencies. Dually, this gives a blue shift effect when one propagates the estimates out of the boundary.
\end{rmk}

When the $\ham_p$-flow has an appropriate global structure, e.g.\ when one has complex absorption in some regions, and the $\ham_p$-flow starts from and ends in these, potentially after `going through' radial saddle points, see \cite[\S2.1]{HintzVasySemilinear}, one gets global estimates
\begin{equation}
\label{EqBcPGlobalEst}
  \|u\|_{\Hb^{s,\alpha}}\leq C\bigl(\|\cP u\|_{\Hb^{s-m+1,\alpha}}+\|u\|_{\Hb^{s',\alpha}}\bigr),
\end{equation}
with $s'<s$, provided of course the threshold conditions are satisfied when radial points are present (depending on the direction of propagation). One also has dual estimates for $\cP^*$, propagating in the opposite direction.

Due to the lack of gain in $\alpha$, these estimates do not directly give rise to a Fredholm theory even if $M$ is compact, since the inclusion map $\Hb^{s,\alpha}\to\Hb^{s',\alpha}$ is not compact even if $s>s'$. This can be done via the analysis of the (Mellin transformed) normal operator $\wh N(\cP)(\sigma) \equiv \wh\cP(\sigma)$. Namely, when $\wh\cP(\sigma)$ has no poles in the region $-\Im\sigma\in [r',r]$, and when large $|\Re\sigma|$ estimates hold for $\wh\cP$, which is automatic when $\cP$ has the global structure allowing for the {\em global} estimates \eqref{EqBcPGlobalEst}, then one can obtain estimates like
\[
  \|u\|_{\Hb^{s,r}}\leq C\bigl(\|N(\cP) u\|_{\Hb^{s-m+1,r}}+\|u\|_{\Hb^{s',r'}}\bigr),
\]
which, when applied to the error term of \eqref{EqBcPGlobalEst} via the use of cutoff functions, gives
\[
  \|u\|_{\Hb^{s,r}}\leq C\bigl(\|\cP u\|_{\Hb^{s-m+1,r}}+\|u\|_{\Hb^{s',r'}}\bigr),
\]
with $s>s'$, $r>r'$, and dual estimates for $\cP^*$, which {\em does} give rise to a Fredholm problem for $\cP$.

We are also interested in domains in $M$, more precisely `product' or `p' submanifolds with corners $\Omega$ in $M$. Thus, $\Omega$ is given by inequalities of the form $\frakt_j\geq 0$, $j=1,2\ldots,m$, such that at any point $p$ the differentials of those of the $\frakt_j$, as well as of the boundary defining function $\tau$ of $M$, which vanish at $p$ must be linearly independent (as vectors in $T^*_p M$). For instance, if $m=2$, and $p\in X=\pa M$ with $\frakt_1(p)\neq 0$, $\frakt_2(p)=0$, then $d\tau(p)$ and $d\frakt_2(p)$ must be linearly independent. The main example of interest is the domain $\Omega$ defined in \S\ref{SubsecKdSWave}, see Figure~\ref{FigKdSWave}, in which case we can take $\frakt_1=\prod_\pm (r-(r_{b_0,\pm}\pm\eps_M))$ and $\frakt_2=1-\tau$.

On a manifold with corners, such as $\Omega$, one can consider supported and extendible distributions; see \cite[Appendix~B.2]{HormanderAnalysisPDE3} for the smooth boundary setting, with simple changes needed only for the corners setting, which is discussed e.g.\ in \cite[\S3]{VasyPropagationCorners}. Here we consider $\Omega$ as a domain in $M$, and thus its boundary face $X\cap\Omega$ is regarded as having a different character from the $H_j\cap\Omega$, $H_j=\frakt_j^{-1}(0)$, i.e.\ the support/extendibility considerations do not arise at $X$---all distributions are regarded as acting on a subspace of $\CI$ functions on $\Omega$ vanishing at $X$ to infinite order, i.e.\ they are automatically extendible distributions at $X$. On the other hand, at the $H_j$ we consider both extendible distributions, acting on $\CI$ functions vanishing to infinite order at $H_j$, and supported distributions, which act on all $\CI$ functions (as far as conditions at $H_j$ are concerned). For example, the space of supported distributions at $H_1$ extendible at $H_2$ (and at $X$, as we always tacitly assume) is the dual space of the subspace of $\CI(\Omega)$ consisting of functions vanishing to infinite order at $H_2$ and $X$ (but not necessarily at $H_1$). An equivalent way of characterizing this space of distributions is that they are restrictions of elements of the dual $\CmI(M)$ of $\CIdotc(M)$ with support in $\frakt_1\geq 0$ to $\CI$ functions on $\Omega$ which vanish to infinite order at $X$ and $H_2$; thus in the terminology of \cite{HormanderAnalysisPDE3}, they are restrictions of elements of $\CmI(M)$ with support in $\frakt_1\geq 0$ to $\Omega\setminus (H_2\cup X)$.

The main interest is in spaces induced by the Sobolev spaces $\Hb^{s,r}(M)$. Notice that the Sobolev norm is of completely different nature at $X$ than at the $H_j$, namely the derivatives are based on complete, rather than incomplete, vector fields: $\Vb(M)$ is being restricted to $\Omega$, so one obtains vector fields tangent to $X$ but not to the $H_j$. As for supported and extendible distributions corresponding to $\Hb^{s,r}(M)$, we have, for instance,
\[
  \Hb^{s,r}(\Omega)^{\bullet,-},
\]
with the first superscript on the right denoting whether supported ($\bullet$) or extendible ($-$) distributions are discussed at $H_1$, and the second the analogous property at $H_2$; thus $\Hb^{s,r}(\Omega)^{\bullet,-}$ consists of restrictions of elements of $\Hb^{s,r}(M)$ with support in $\frakt_1\geq 0$ to $\Omega\setminus (H_2\cup X)$.  Then elements of $\CI(\Omega)$ with the analogous vanishing conditions, so in the example vanishing to infinite order at $H_1$ and $X$, are dense in $\Hb^{s,r}(\Omega)^{\bullet,-}$; further the dual of $\Hb^{s,r}(\Omega)^{\bullet,-}$ is $\Hb^{-s,-r}(\Omega)^{-,\bullet}$ with respect to the $L^2$ (sesquilinear) pairing. For distributions extendible, resp.\ supported, at all boundary hypersurface, we shall write
\begin{equation}
\label{EqBSuppExt}
  \Hbext^{s,r}(\Omega) \equiv \Hb^{s,r}(\Omega)^{-,-}, \quad
  \Hbsupp^{s,r}(\Omega) \equiv \Hb^{s,r}(\Omega)^{\bullet,\bullet}.
\end{equation}

The main use of these spaces for the wave equation is that due to energy estimates, one can obtain a Fredholm theory using these spaces, with the supported distributions corresponding to vanishing Cauchy data (where one propagates estimates from in the complex absorption setting discussed above), while extendible distributions correspond to no control of Cauchy data (corresponding to the final spacelike hypersurfaces, i.e.\ with future timelike outward-pointing normal vector, to which one propagates estimates); note that dualization reverses these, i.e.\ one starts propagating for $\cP^*$ from the spacelike hypersurfaces towards which one propagated for $\cP$. We refer to \cite[\S2.1]{HintzVasySemilinear} for further details.

\subsection{Semiclassical analysis}
\label{SubsecBSemi}

In one part of the paper, namely the proof of SCP, we work with semiclassical b-pseudodifferential operators. First recall that the uniform semiclassical operator algebra, $\Psih(\R^n)$, is given by
\begin{align*}
  A_\semi=\Oph(a);\ &\Oph(a)u(\wt z)=(2\pi \semi)^{-n}\int_{\R^n\times\R^n} e^{i(\wt z-\wt z')\cdot\wt\zeta/\semi} a(\wt z,\wt\zeta,\semi)\,u(\wt z')\,d\wt\zeta\,d\wt z', \\
  &\qquad u\in\cS(\R^n),\ a\in \CI([0,1)_\semi; S^m_\infty(\R^n_{\wt z};\R^n_{\wt z'}));
\end{align*}
its classical subalgebra, $\Psihcl(\R^n)$ corresponds to $a\in \CI([0,1)_\semi; S_{\infty,\cl}^m(\R^n;\R^n))$, where $S_{\infty,\cl}^m$ denotes the space of symbols which are classical (one-step polyhomogeneous) in the fibers. The semiclassical principal symbol of such an operator is $\sigma_{\semi,m}(A)=a|_{\semi=0}\in S^m_\infty(T^*\R^n)$; the `standard' principal symbol is still the equivalence class of $a$ in $\CI([0,1)_\semi;S^{m}_\infty/S^{m-1}_\infty)$, or an element of $\CI([0,1)_\semi;S^{m}_{\infty,\hom})$ in the classical setting. There are natural extensions to manifolds without boundary $X$, for which the behavior of the symbols at infinity in $\R^n_{\wt z}$ is irrelevant since, as one transfers the operators to manifolds, one uses coordinate charts only whose compact subsets play a role. On the other hand, the `conormal coefficient' semiclassical b-pseudodifferential algebra can be defined via the identifications discussed above, namely locally using, with $z=(t,x)$, the quantization
\begin{align*}
  &A_\semi=\Oph(a); \\
  &\Oph(a)u(\wt z)=(2\pi \semi)^{-n}\int_{\R^n\times\R^n} e^{i(\wt z-\wt z')\cdot\wt\zeta/\semi} \wt\phi(t-t') a(\wt z,\wt\zeta,\semi)\,u(\wt z')\,d\wt\zeta\,d\wt z', \\
  &\qquad u\in\cS(\R^n),\ a\in \CI([0,1)_\semi; S^m_\infty(\R^n_{\wt z};\R^n_{\wt\zeta}));
\end{align*}
with $\wt\phi$ compactly supported, identically $1$ near $0$, requiring an expansion of $a$ in powers of $\tau=e^{-t}$ as $t\to\infty$. The Schwartz kernel of such an operator vanishes to infinite order at $\semi=0$ away from the diagonal (in the uniform sense that $t-t'$ is bounded away from $0$), thus working in a manifold setting is in fact almost the same as working locally.

The fully intrinsic version of this operator algebra can be obtained using Melrose's approach via blow-ups, as in \cite{MelroseAPS}. First, recall that the standard b-double space $M^2_\bl$ is constructed by taking $M^2=M\times M$, and blowing up the corner $(\pa M)^2$ in it: $M^2_\bl=[M^2;(\pa M)^2]$. The diagonal then lifts to a product submanifold of this resolved space, and the b-ps.d.o's on this space are simply distributions conormal to the diagonal which vanish to infinite order at the lift of the left and right boundaries $\pa M\times M$ and $M\times\pa M$. Indeed, \eqref{EqBOpMap} is an explicit way of writing such a parametrization of conormal distributions via oscillatory integrals taking into account that in $\tau/\tau'\in (C^{-1},C)$, $C>0$, regarding $(\tau/\tau',\tau)$ as valid coordinates on the blown-up space, $\log(\tau/\tau')$ in the phase is equivalent to $\frac{\tau-\tau'}{\tau}$, which together with $x_j-x'_j$ defines the lifted diagonal (or b-diagonal) $\diag_\bl$. In the semiclassical setting, one considers $M^2\times[0,1)_\semi$, blows up $(\pa M)^2\times[0,1)_\semi$ first to obtain a family (parametrized by $\semi$) of double spaces, $[M^2\times[0,1);(\pa M)^2\times[0,1)]=M^2_\bl\times[0,1)_\semi$. Then the b-diagonal {\em at} $\semi=0$ is a p-submanifold, and one blows this up to obtain the semiclassical b-double space,
\[
  M^2_{\bl,\semi}=[M^2_\bl\times[0,1)_\semi;\diag_\bl\times\{0\}].
\]
Elements of $\Psibh^m(M)$ are then given by Schwartz kernels which are conormal, of order $m$, to the diagonal, smooth up to the front faces of the last two blow-ups (b- and semiclassical), and vanishing to infinite order at the lifts of the 3 original faces: left (i.e.\ $\pa M\times M\times [0,1)$), right ($M\times\pa M\times[0,1)$) and semiclassical ($M^2\times\{0\}$). This is completely analogous to the construction of the semiclassical 0-double space in \cite{MelroseSaBarretoVasyResolvent}, but in that paper much more delicate semiclassical Fourier integral operators had to be considered. The algebraic properties of $\Psibh^m(M)$ can be derived directly, but they are even more transparent from the above discussion on $\R^n$.

The b-ps.d.o.\ results such as elliptic estimates, propagation of singularities, etc., have semiclassical b-analogues. First, the semiclassical b-Sobolev norms are defined (up to equivalence of norms on compact manifolds) for $s\in\N$, $u_\semi\in \Hb^{s,\alpha}$, $\semi\in(0,1)$, by 
\[
  \|u_\semi\|_{\Hbh^{s,\alpha}(M)}^2 = \sum \|(\semi V_1)\cdots (\semi V_j) (\tau^{-\alpha}u)\|^2_{L^2_\bl},
\]
where the finite sum is over all collections of up to $s$ (including $0$) vector fields $V_i\in\wt\cV$, $\wt\cV$ a finite subset of $\Vb(M)$, such that at each point $p$ in $M$, elements of $\wt\cV$ span $\Tb_pM$. In local coordinates near a point $p\in\pa M$ this is equivalent to the squared norm
\[
  \|\tau^{-\alpha}u\|^2_{L^2_\bl}+\|\tau^{-\alpha}(\semi\tau D_\tau)^s u\|^2_{L^2_\bl}+\sum_{j=1}^{n-1}\|\tau^{-\alpha}(\semi D_{x_j})^s u\|^2_{L^2_\bl},
\]
which one can again extend to $s\in\R$ by duality and interpolation. 

One then has a notion of semiclassical b-wave front set $\WFbh^{s,\alpha}(u)$, defined for families $u=(u_\semi)_{\semi\in(0,1)}$ which are bounded by $C\semi^N$ in $\Hbh^{s',\alpha}$ for some $s',\alpha,N,C$. (One says that $u$ is polynomially bounded in $\Hbh^{s',\alpha}$.) This is then a subset of
\[
  \pa(\ol{\Tb^*}M\times[0,1))=\Sb^*M\times[0,1)\cup\ol{\Tb^*}M\times\{0\},
\]
where the corner $\Sb^*M\times\{0\}=\pa\ol{\Tb^*}M\times\{0\}$ is part of both sets on the right, and is defined by $\varpi\notin\WFbh^{s,\alpha}(u)$ if there exists $A\in\Psibh^0(M)$, elliptic at $\varpi$, such that $\|Au\|_{\Hbh^{s,\alpha}}=\cO(\semi^\infty)$, i.e.\ bounded by $C_N\semi^N$ for all $N$. Then, for instance, elliptic regularity is the statement that if $\cP\in\Psibh^m(M)$ is elliptic at $\varpi$, then for $u$ polynomially bounded in $\Hbh^{s',\alpha}$,
\[
  \varpi\notin\WFbh^{s-m,\alpha}(\cP u)\Rightarrow\varpi\notin\WFb^{s,\alpha}(u).
\]
This corresponds to an estimate (by the uniform boundedness principle)
\[
  \|B_2 u\|_{\Hbh^{s,\alpha}}\leq C\bigl(\|B_1\cP u\|_{\Hbh^{s-m,\alpha}}+\semi^N\|u\|_{\Hb^{s',\alpha}}\bigr),
\]
which is valid whenever $B_1,B_2\in\Psib^0(M)$ with $B_1$ elliptic on $\WFb'(B_2)$ and $\cP$ elliptic on $\WFb'(B_2)$. There are analogues of propagation of singularities and radial point estimates. Thus, under global conditions on the $\ham_p$-flow, as above, one has estimates
\[
  \|u\|_{\Hbh^{s,\alpha}}\leq C\bigl(\semi^{-1}\|\cP u\|_{\Hbh^{s-m+1,\alpha}}+\semi^N\|u\|_{\Hb^{s',\alpha}}\bigr),
\]
with $s'<s$; the $\semi^{-1}$ corresponds to the loss of one derivative in the norm of $\cP u$ relative to the elliptic estimate due to the propagation of singularities estimate. Notice that these estimates give {\em small} remainders due to the factor $\semi^N$, which can thus be absorbed into the left hand side for $\semi$ sufficiently small. Therefore, one can obtain invertibility results for $\cP=\cP_\semi$ for sufficiently small $\semi$ directly, without having to analyze the normal operator.

\section{A general quasilinear existence theorem}
\label{SecQ}

Combining the results of \S\ref{SecAsy} with the Nash--Moser inverse function theorem described in \S\ref{SubsecKdSStabNM}, we now prove that one can solve rather general quasilinear wave equations with small data globally upon modifying the forcing or the initial data in a suitable finite-dimensional space, provided the linearization of the non-linear operator at $0$ fits into the framework of \S\ref{SubsecAsySm}. The purpose is to present a simple result that is powerful enough for interesting applications: we will be able to use it directly to prove the non-linear stability of the static model of de~Sitter space, see Theorem~\ref{ThmdSStab}. (Subsuming the black hole stability proof in \S\ref{SubsecKdSStab} into the general theorem below would complicate the setup only slightly.)

Thus, the simplest way (albeit not the most natural one geometrically) to describe our requirements for a non-linear differential operator $P$, acting on sections of a stationary vector bundle $E\to M$ of rank $k$ (see \S\ref{SubsecAsySm}), is to use coordinates $(t_*,x)=:(x_0,\ldots,x_3)$, where $x=(x_1,x_2,x_3)$ is a local coordinate system on $X$ in which $E$ is trivialized with fibers $\C^k$; we then require that for some \emph{fixed} $b\in\cU_B$,\footnote{The restriction to small angular momenta here is only due to the fact that we did not define the Kerr--de~Sitter family for larger angular momenta in \S\ref{SubsecKdSSlow}.} we have
\[
  P(u) = (g_b^{\mu\nu}+q^{\mu\nu}(x,u,D u))D_\mu D_\nu u + q^\mu(x,u,D u)D_\mu u + q(x,u,D u)u
\]
for $u$ with small $\cC^1$ norm, where the $q^{\mu\nu}\colon\R^3\times\C^k\times\C^{3k}\to\R$ are smooth with $q^{\mu\nu}(x,0,0)\equiv 0$, and $q,q^\mu\colon\R^3\times\C^k\times\C^{3k}\to\C^{k\times k}$ are smooth, valued in endomorphisms of $E$. For example, the non-linear operator $u\mapsto(\Ric+\Lambda)(g_b+u)-\tdel^*(\Ups(g_b+u)-\Ups(g_b))$ is of this form.

The main feature of such operators is that the linearization $L_u(r) := D_u P(r)$ is a principally scalar wave operator, and if $\wt u\in\Hb^{s+2,\alpha}(M;E)$, $\alpha>0$, $s>2$, is exponentially decaying, then $D_{\wt u} P$ is \emph{stationary} up to an operator in $\Hb^{s,\alpha}\Diffb^2(M;E)$.

We furthermore assume that the linearization $L_0=D_0 P$ satisfies the assumptions \eqref{ItAsySm1Sym}--\eqref{ItAsySm3Trapped}; for simplicity, we assume $\wh\beta\geq-1$ in \eqref{EqAsySm2RadialSubprInf} as in Theorem~\ref{ThmAsyExpMain}. Due to Theorem~\ref{ThmAsySmMero}, there exists $\alpha>0$ such that the operator $L_0$ has only finitely many resonances in a half space $\Im\sigma>-\alpha$ and satisfies high energy estimates in this half space; by shrinking $\alpha>0$ if necessary, we can assume $L_0$ has no resonances with $-\alpha\leq\Im\sigma<0$. (The latter assumption is unnecessary; we only make it for convenience.) Denote by $R=\Res(L_0,\{\Im\sigma>-\alpha\})$ the \emph{finite-dimensional} space of resonant states corresponding to the non-decaying resonances of $L_0$. We introduce a space of modifications of forcing terms removing the asymptotic behavior of elements of $R$ when solving linear initial value problems for $L_0$: fix a basis $\{\phi_1,\ldots,\phi_N\}$ of $R$, and a cutoff $\chi$, identically $0$ near $\Sigma_0$ and identically $1$ for large $t_*$; then, adopting the notation of Theorem~\ref{ThmAsyExpMain}, we define
\begin{gather*}
  z \colon \C^N \to \Hb^{\infty,\alpha}(\Omega;E)^{\bullet,-} \hra D^{\infty,\alpha}(\Omega;E), \\
  \bfc=(c_1,\ldots,c_N)\mapsto \sum_j L_0(\chi c_j\phi_j).
\end{gather*}
By construction, the assumptions of Theorem~\ref{ThmAsyExpMain} are satisfied if we take $W=\{0\}$, corresponding to the fact that we eliminate all non-decaying asymptotic behavior, thus the stationary parts of the linearized operators $L_u$, $u\in\Hb^{\infty,\alpha}$, we need to consider are fixed, i.e.\ do not depend on any parameters.

The general Nash--Moser iteration scheme, Theorem~\ref{ThmKdSStabNM}, then implies the following theorem:

\begin{thm}
\label{ThmQ}
  Suppose $P$ satisfies the above assumptions. Then there exist constants $\eps>0$ and $C$ such that the following holds: for data $d=(f,u_0,u_1)\in D^{\infty,\alpha}(\Omega;E)$ (recall Definition~\ref{DefAsySmDataSpace}) with $\|d\|_{20,\alpha}<\eps$, there exist $\bfc\in\C^N$ and $u\in\Hbext^{\infty,\alpha}(\Omega;E)$ solving the quasilinear wave equation
  \[
    \begin{cases}
      P(u) = f + z(\bfc) & \tn{in }\Omega^\circ, \\
      \gamma_0(u) = (u_0,u_1) & \tn{in }\Sigma_0,
    \end{cases}
  \]
  and $|\bfc|\leq C\|d\|_{13,\alpha}$.
\end{thm}

One also obtains an estimate for $\Hbext^{s,\alpha}$ norms of $u$, as follows from the proof of Theorem~\ref{ThmKdSStabNM} given in \cite{SaintRaymondNashMoser}.

Furthermore, one can show that the map $Z\colon D^{\infty,\alpha}(\Omega;E)\ni d\mapsto\bfc\in\C^N$ has surjective differential, as follows from the construction of the map $z$ and its relation to the linear operator $L_0$ (see \cite[\S III]{HamiltonNashMoser} for details), and then $Z_0:=Z^{-1}(0)\subset D^{\infty,\alpha}(\Omega;E)$ is, locally near $0$, an $N$-codimensional smooth Fr\'echet submanifold. Therefore, we can solve the quasilinear initial value problem $P(u)=f$, $\gamma_0(u)=(u_0,u_1)$, \emph{exactly} for $(f,u_0,u_1)\in Z_0$; that is, we have \emph{global existence} (and automatically uniqueness) \emph{in a space of decaying solutions for an $N$-codimensional submanifold of the space of data}.

\section{Non-linear stability of the static model of de~Sitter space}
\label{SecdS}

In this section, we prove the non-linear stability of the static model of de~Sitter space using the methods outlined in \S\ref{SubsecIntroIdeas}. We recall that the stability of \emph{global} de~Sitter space in $(3+1)$ dimensions was proved by Friedrich \cite{FriedrichStability} (with generalizations due to Anderson \cite{AndersonStabilityEvenDS} and Ringstr\"om \cite{RingstromEinsteinScalarStability}), which is thus a much stronger result because it shows stability on a larger spacetime; the point is thus only to illustrate the main ideas of the paper in a simpler context which however is very illuminating.

We recall that Graham--Lee \cite{GrahamLeeConformalEinstein} proved the existence of Poincar\'e--Einstein metrics on the ball, with prescribed conformal class of the metric induced on the conformal boundary, close to the hyperbolic metric. Growing indicial roots in the elliptic setting do not present a problem as they do in the hyperbolic setting; one solves an analogue of a boundary value problem (see in particular \cite[Theorem~3.10]{GrahamLeeConformalEinstein}) in which these are excluded from the considerations (somewhat analogously to scattering constructions from infinity in the hyperbolic setting). Our computations in the DeTurck gauge below parallel those of \cite{GrahamLeeConformalEinstein}; the difference in the signature affects the calculations only in a minor way.

Here, we will introduce de~Sitter space simply by using a local coordinate expression for its metric; we refer to \cite[\S8.1]{HintzQuasilinearDS} for a detailed discussion of de~Sitter space and the static model. We work in $(n+1)$ dimensions, use Greek letters for indices between $0$ and $n$, and Latin letters for indices between $1$ and $n$. Locally near a point of the future conformal boundary of (global) de~Sitter space $\sfM$, the de~Sitter metric $g_0$ takes the form
\[
  g_0 = \tau^{-2}\ol g_0,\quad \ol g_0 = d\tau^2 - \sum_i dw_i^2
\]
in a suitable coordinate system $\tau\geq 0$, $w_1,\ldots,w_n\in\R$, where $\tau=0$ defines the future conformal boundary $\sfX$ of $\sfM$ within the coordinate patch; see Figure~\ref{FigdS}. Thus, $g$ is a $0$-metric in the sense of Mazzeo--Melrose \cite{MazzeoMelroseHyp}, albeit with Lorentzian rather than Riemannian signature; more general Lorentzian manifolds, with a similar structure at infinity as de~Sitter space, were introduced and studied by Vasy \cite{VasyWaveOndS}, and we will make use of the results of that paper freely.

\begin{figure}[!ht]
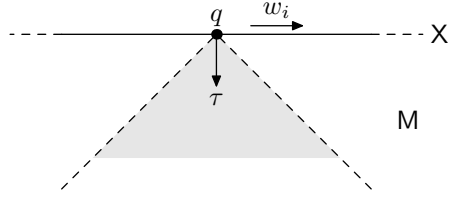

  \centering
  \inclfig{dS}
  \caption{A neighborhood of a point $q$ on the conformal boundary of de~Sitter space; also indicated (shaded) is the static model at $q$, which near $q$ is the interior of the backward light cone from $q$ with respect to $\ol g$ (and thus $g$).}
  \label{FigdS}
\end{figure}

It is natural to work with the frames
\[
  e_\mu := \tau\pa_\mu,\quad e^\mu := \frac{dw_\mu}{\tau}
\]
of the 0-tangent bundle $\Tzero\sfM$ and the 0-cotangent bundle $\Tzerodual\sfM$, respectively. \emph{From now on, indices refer to these frames, rather than the coordinate frame.} Thus, for instance, for a 1-form $\omega$, we write $\omega_\mu=\omega(e_\mu)$, so $\omega=\omega_\mu e^\mu$, and raising the index gives $\omega^0=\omega_0$ and $\omega^i=-\omega_i$, i.e.\ $G_0(\omega,-)=\omega^\mu e_\mu$ if $G_0$ denotes the dual metric. Similarly to \eqref{EqOpWarpedBundles}, we consider natural splittings of the vector bundles
\[
  \Tzerodual\sfM=W_N\oplus W_T,\quad S^2\,\Tzerodual\sfM=V_{NN}\oplus V_{NT}\oplus V_{TT},
\]
where
\begin{equation}
\label{EqdSSplit}
\begin{gathered}
  W_N = \la e^0 \ra,\quad W_T = \la e^i \ra, \\
  V_{NN}=\la e^0 e^0\ra,\quad V_{NT}=\la 2e^0e^i \ra,\quad V_{TT}=\la e^i e^j\ra,
\end{gathered}
\end{equation}
where we recall the notation $\xi\eta=\xi\otimes_s\eta=\frac{1}{2}(\xi\otimes\eta+\eta\otimes\xi)$. It will be useful to further split up $V_{TT}$ into its tracefree (`$0$') and pure trace (`$p$') parts,
\begin{equation}
\label{EqdSSplitVT}
  V_{TT} = V_{TT0} \oplus V_{TTp},\quad V_{TT0}=\{a_{ij}e^ie^j\colon a_i{}^i=0\},\quad V_{TTp}=\la h\ra,
\end{equation}
where we defined $h=\sum_i e^i e^i$ to be the restriction of $-g$ to $\tau=const.$ hypersurfaces. For a section $u$ of $V_{TT}$, we note that $\tr_g u=-\tr_h u$.

\subsection{Computation of the explicit form of geometric operators}
\label{SubsecdSOp}

One computes the connection coefficients
\[
  \nabla_0 e^0 = 0,\quad \nabla_0 e^i = 0, \quad \nabla_i e^0=e_i,\quad \nabla_i e^j=\delta_i^j e^0;
\]
this easily gives $R_{\mu\nu\kappa\lambda}=(g_0)_{\mu\lambda}(g_0)_{\nu\kappa}-(g_0)_{\nu\lambda}(g_0)_{\mu\kappa}$, so $\Ric(g_0)_{\nu\lambda}+n (g_0)_{\nu\lambda}=0$. Furthermore, the operator $\sR_{g_0}$ defined in \eqref{EqHypDTCurvatureTerm} is equal to $\sR_{g_0}(r) = \tr_{g_0}(r)g_0 - (n+1)r$, which in the splitting \eqref{EqdSSplit} is equal to the (block) matrix
\[
  \sR_{g_0}
        = \begin{pmatrix}
            -n & 0      & -\tr_h \\
            0  & -(n+1) & 0     \\
            -h & 0      & h\tr_h - (n+1)
          \end{pmatrix}.
\]
We next compute the wave operator on sections of the subbundles in \eqref{EqdSSplit} using the formula
\[
  -\tr\nabla^2 T = -\nabla_0\nabla_0 T + n\nabla_0 T + \sum_i \nabla_i\nabla_i T,
\]
valid for every tensor $T$ of any rank; thus, for an $NN$ tensor,
\[
  \Box_{g_0}(u e^0 e^0) = \bigl(-e_0^2+n e_0+\sum e_i^2+2n\bigr)u e^0 e^0 + 4(e_i u)e^0 e^i + 2 u e^i e^i,
\]
while for a $TN$ tensor,
\begin{align*}
  \Box_{g_0}&(2 u_k e^0 e^k) \\
    &= 4\sum_i(e_i u_i)e^0 e^0  + 2\bigl(-e_0^2+n e_0+\sum e_i^2 + n+3\bigr)u_k e^0 e^k + 4(e_i u_j)e^i e^j,
\end{align*}
and for a $T$ tensor,
\begin{align*}
  \Box_{g_0}&(u_{jk}e^j e^k) \\
   &=2 \sum_i u_{ii} e^0 e^0 + 4\sum_i(e_i u_{ij})e^0 e^j + \bigl(-e_0^2+n e_0+\sum e_i^2 + 2\bigr)u_{jk}e^j e^k.
\end{align*}

Let us reformulate these expressions in a more geometric manner: if $u$ is a function on $\sfX$, we have $\sum_i (e_i u)e^i = d_\sfX u$. Furthermore, for a tangential 1-form $u_k e^k$, we calculate its codifferential with respect to the metric $h$ to be $\delta_h(u_k e^k)=-\sum_k e_k u_k$; note that this is equal to $-\delta_{g_0}(u_k e^k)$. On symmetric 2-tensors, the divergence $(\delta_{g_0}(u_{\mu\nu}e^\mu e^\nu))_\lambda = -g_0^{\mu\kappa}\nabla_\mu u_{\kappa\lambda}$ acts via
\[
  \delta_h(u_{ij}e^i e^j)=-\sum_j(e_j u_{ij})e^i,
\]
which equals $-\delta_{g_0}(u_{ij}e^ie^j)$. On the other hand, the adjoint of $\delta_h$ acting on symmetric 2-tensors (relative to the inner products induced by $h$) is $(e_i u_j)e^i e^j=\delta_h^*(u_k e^k)=-\delta_{g_0}^*(u_k e^k)$. Therefore, the wave operator $\Box_{g_0}$ on symmetric 2-tensors, in the decomposition \eqref{EqdSSplit} and the trivializations described there, is given by
\begin{equation}
\label{EqdSOpWave}
  \Box_{g_0} = \bigl(-e_0^2+n e_0+\sum e_i^2\bigr)
    + \begin{pmatrix}
        2 n      & -4\delta_h  & 2\tr_h     \\
        2 d_\sfX & n+3         & -2\delta_h \\
        2 h      & 4\delta_h^* & 2
      \end{pmatrix}
\end{equation}

We furthermore compute
\[
  \delta_{g_0}^*
    =\begin{pmatrix}
       e_0               & 0                  \\
       \frac{1}{2}d_\sfX & \frac{1}{2}(e_0+1) \\
       h                 & \delta_h^*
     \end{pmatrix},
  \quad
  \delta_{g_0} =
    \begin{pmatrix}
      -e_0+n & -\delta_h  & \tr_h   \\
      0      & -e_0+(n+1) & -\delta_h
    \end{pmatrix},
\]
and, recalling \eqref{EqHypDTEinsteinify},
\[
  \sfG_{g_0} = \begin{pmatrix}
          \frac{1}{2}  & 0 & \frac{1}{2}\tr_h    \\
          0            & 1 & 0                   \\
          \frac{1}{2}h & 0 & 1-\frac{1}{2}h\tr_h
        \end{pmatrix}.
\]

\subsection{Unmodified DeTurck gauge}
\label{SubsecdSUnmod}

Let us now consider the hyperbolic formulation
\begin{equation}
\label{EqdSUnmodEin}
  \Ric(g) + n g - \delta_g^*\Ups(g) = 0,\quad \Ups(g)=g g_0^{-1}\delta_g \sfG_g g_0.
\end{equation}
of the Einstein equation. The linearized operator $L$ around $g=g_0$ is given by
\begin{equation}
\label{EqdSUnmodLin}
\begin{split}
  2 L r &= 2(D_{g_0}(\Ric+n) + \delta_{g_0}^*\delta_{g_0} \sfG_{g_0}) = \Box_{g_0} + 2 n + 2\sR_{g_0} \\
    &=-e_0^2 + ne_0 + \sum_i e_i^2
     +\begin{pmatrix}
        2n      & -4\delta_h  & 0          \\
        2d_\sfX & n+1         & -2\delta_h \\
        0       & 4\delta_h^* & 2h\tr_h
      \end{pmatrix}
\end{split}
\end{equation}
in the splitting \eqref{EqdSSplit}; see \cite[Equation~(2.4)]{GrahamLeeConformalEinstein}.

We note the exact commutation relation
\[
  [ L, \Delta ] = 0,\quad
  \Delta=
  \begin{pmatrix}\tau^{-2}\Delta_h&0&0\\0&\tau^{-2}\Delta_{h,(1)}&0\\0&0&\tau^{-2}\Delta_{h,(2)}\end{pmatrix},
\]
where the number in the subscript indicates the degree of the tensors the corresponding Laplace operator acts on. Indeed, $\tau^{-2}\Delta_h = \sum D_{w_j}^2$, with the same formula holding component-wise for $\Delta_{h,(1)}$ and $\Delta_{h,(2)}$ (trivializing the respective bundles via the frame $\{e_1,\ldots,e_n\}$ of $T\sfX$), clearly commutes with all summands of $L$ separately. (In fact, commutation up to leading order in $\tau$ suffices for present purposes.) Since $\Delta\in\Diffb^2$, the arguments of \cite[\S4]{VasyWaveOndS} apply to show that the asymptotic behavior of solutions of $L u=0$ is dictated by the indicial roots of $L$, and in fact the general form of all possible asymptotics can be deduced by purely formal calculations, which we proceed to discuss. In fact, we are only interested in indicial roots $\sigma$ with $\Im\sigma\geq 0$; roots with $\Im\sigma<0$ correspond to exponentially decaying (in $-\log\tau$) asymptotic behavior, hence we do not study them further here.

The formal calculations use the properties of the \emph{indicial operator} $I(L,\sigma)$; recall here that for a second order 0-differential operator such as $L$, the 2-tensor $\tau^{-i\sigma}L \tau^{i\sigma}r$, with $r\in\CI(\sfX;S^2\,\Tzerodual_\sfX\sfM)$ only depending on the spatial variables $w_i$, is equal to a quadratic polynomial in $\sigma$, valued in endomorphisms of $S^2\,\Tzero\sfM$, applied to $r$, plus terms in $\tau\CI(\sfM;S^2\,\Tzero\sfM)$, i.e.\ which vanish at the boundary. Thus, computing $I(L,\sigma)$ amounts to replacing $e_0$ by $i\sigma$ and dropping spatial derivatives (due to $\tau\pa_i$ acting on smooth functions gives a vanishing factor $\tau$), to wit
\[
  I(2 L,\sigma) = \sigma^2 + i n\sigma + \begin{pmatrix} 2n & 0 & 0 \\ 0 & n+1 & 0 \\ 0 & 0 & 2h\tr_h \end{pmatrix}.
\]
The indicial roots are those $\sigma\in\C$ for which $I(L,\sigma)$ is not invertible; they are the indicial roots for the regular-singular ODE obtained from $L$ by dropping spatial derivatives. (For the related b-problem which we discuss in \S\ref{SubsecdSB}, an indicial root $\sigma$ gives rise to resonances at $\sigma-i\N_0$.) Using the refined splitting \eqref{EqdSSplitVT}, we note that $2h\tr_h=0$ on $V_{TT0}$, while $2h\tr_h=2n$ on $V_{TTp}$. Thus, the indicial roots of $L$ are
\begin{gather*}
  \sigma_{NN}^\pm = \frac{i}{2}(-n\pm\sqrt{n^2+8n}), \\
  \sigma_{TN}^+ = i,\quad \sigma_{TN}^- = -i(n+1), \\
  \sigma_{TT0}^+ = 0,\quad \sigma_{TT0}^- = -i n, \\
  \sigma_{TTp}^\pm = \sigma_{NN}^\pm,
\end{gather*}
corresponding to $I(L,\sigma_*^\pm)$ having kernel $V_*$ for $*\in\{NN,TN,TT0,TTp\}$. This is completely analogous to the result of \cite[Lemma~2.9]{GrahamLeeConformalEinstein} in Riemannian signature; the differences of the expressions come from Graham and Lee using a different rescaling of the vector bundle $S^2\,\Tzerodual\sfM$. These indicial roots correspond to the fact that one can prescribe the coefficient $a_*^\pm(0)$ of $\tau^{i\sigma_*^\pm}$ at $\tau=0$ freely as a section of $\ker I(L,\sigma_*^\pm)=V_*$, and there exists a unique solution on $L r=0$ attaining this desired asymptotic behavior; conversely, any solution of $L r=0$ has an asymptotic expansion $\sum_* \tau^{i\sigma_*^\pm}a_*^\pm$, with $a_*^\pm\in\CI(\sfM;S^2\,\Tzerodual\sfM)$ and with $a_*^\pm(0)$ a section of $V_*$. (There may be terms $|\log \tau|^k$ present as well.)

To proceed, we note the indicial operators of $\delta_{g_0}$ and $\delta_{g_0}^*$,
\[
  I(\delta_{g_0}^*,\sigma) = \begin{pmatrix} i\sigma & 0 \\ 0 & \frac{1}{2}(i\sigma+1) \\ h & 0 \end{pmatrix},
  \quad
  I(\delta_{g_0},\sigma) = \begin{pmatrix} -i\sigma+n & 0 & \tr_h \\ 0 & -i\sigma+n+1 & 0 \end{pmatrix}.
\]
For brevity, we write
\begin{equation}
\label{EqdSUnmodPesky}
  \sigma_+ := \sigma_{NN}^+ = \frac{i}{2}(-n+\sqrt{n^2+8n}).
\end{equation}
Concretely then, for instance,
\begin{equation}
\label{EqdSUnmodPeskyPureGauge}
  r_1 = \tau^{i\sigma_+}\begin{pmatrix} i\sigma_+ \\ 0 \\ h \end{pmatrix}
\end{equation}
solves $L r_1 = \cO(\tau^{i\sigma_+ +1})$, i.e.\ $r_1$ solves the linearized gauged Einstein equation up to terms decaying one order better; and in fact $r_1$ is a pure gauge solution (up to faster decaying terms) in the sense that
\[
  r_1 = \delta_{g_0}^*\biggl[\tau^{i\sigma_+}\begin{pmatrix} 1 \\ 0 \end{pmatrix}\biggr] + \cO(\tau^{i\sigma_+ +1}).
\]
However, $r_2=\tau^{i\sigma_+}h$, say, which also solves $L r_2=0$ up to less growing error terms, is \emph{not} in the range of $\delta_{g_0}^*$ acting on $\tau^{i\sigma_+}$ times smooth sections of $\Tzerodual\sfM$. Taking the linear stability of de~Sitter space for granted, $r_2$ cannot appear as the asymptotic behavior of a gravitational wave on de~Sitter space; put differently, the asymptotic behavior $r_2$ is ruled out by the \emph{linearized} constraint equations.

As explained in \S\ref{SubsecIntroIdeas}, this argument, ruling out non-pure gauge growing asymptotics, is insufficient for the purpose of understanding the \emph{non-linear} stability problem, where one is given initial data satisfying the \emph{non-linear} constraint equations; we are therefore led to consider modifications of \eqref{EqdSUnmodLin} for which \emph{all} non-decaying modes are pure gauge modes. The way to arrange this is to study (modifications of) the constraint propagation equation, to which we turn next.

\subsection{Stable constraint propagation}
\label{SubsecdSSCP}

We recall that for the Einstein equation \eqref{EqdSUnmodEin} in the unmodified wave map/DeTurck gauge, the constraint propagation operator is
\[
  \BoxCP_{g_0} = 2\delta_{g_0} \sfG_{g_0}\delta_{g_0}^*.
\]
Again, we can compute the asymptotic behavior of solutions of $\BoxCP_{g_0} u=0$ (and thus resonances of $\BoxCP_{g_0}$ on the static patch) by finding indicial roots; we calculate
\[
  I(\BoxCP_{g_0},\sigma)
   = \begin{pmatrix}
       \sigma^2+i n\sigma+2n & 0 \\
       0                     & (\sigma-i)(\sigma+i(n+1))
     \end{pmatrix},
\]
and therefore find that $\BoxCP_{g_0}$ has the indicial roots $\sigma_{NN}^\pm$ and $\sigma_{TN}^\pm$, so in particular solutions of $\BoxCP_{g_0} u=0$ are generically exponentially growing (in $-\log \tau$).

As in \S\ref{SecSCP}, we therefore consider modifications of $\delta_{g_0}^*$; concretely, we consider
\begin{equation}
\label{EqdSSCPDelStarTilde}
  \tdel^*u := \delta_{g_0}^*u - \gamma_1\,e^0\cdot u + \gamma_2 u_0 g_0,\quad \gamma_1,\gamma_2\in\R,
\end{equation}
for $u$ a 1-form; this is the expression analogous to \eqref{EqSCPDelStarTilde} in the current setting. Defining
\[
  \wtBoxCP_{g_0}=2\delta_{g_0} \sfG_{g_0}\tdel^*,
\]
this gives an extra (first order) term $2\delta_{g_0} \sfG_{g_0}(\tdel^*-\delta_{g_0}^*)$ relative to $\BoxCP_{g_0}$. Using
\[
  \tdel^*-\delta_{g_0}^*=
    -\gamma_1
      \begin{pmatrix} 1 & 0 \\ 0 & \frac{1}{2} \\ 0 & 0 \end{pmatrix}
    +\gamma_2
      \begin{pmatrix} 1 & 0 \\ 0 & 0 \\ -h & 0 \end{pmatrix}
\]
and
\begin{equation}
\label{EqdSSCPDelgGg}
  2\delta_{g_0} \sfG_{g_0}=
    \begin{pmatrix}
      -e_0+2n & -2\delta_h  & (-e_0+2)\tr_h          \\
      d_\sfX  & 2(-e_0+n+1) & -2\delta_h-d_\sfX\tr_h
    \end{pmatrix},
\end{equation}
one computes
\[
  I(\wtBoxCP_{g_0},\sigma)
   = \begin{pmatrix}
       p_1 & 0 \\
       0   & p_2
     \end{pmatrix},
\]
where
\begin{gather*}
  p_1 = \sigma^2 + i(n+\gamma_1+(n-1)\gamma_2)\sigma - 2n(\gamma_1-1), \\
  p_2 = \sigma^2 + i(n+\gamma_1)\sigma - (n+1)(\gamma_1-1).
\end{gather*}
Therefore, if $n+\gamma_1>0$, $n+\gamma_1+(n-1)\gamma_2>0$ and $\gamma_1>1$, the roots of $p_1$ and $p_2$ have negative imaginary parts, giving SCP. For the sake of comparison with Theorem~\ref{ThmSCP}, if we take $\gamma_1=\gamma$, $\gamma_2=\frac{1}{2}\gamma$, i.e.\ taking $e=1$ in \eqref{EqSCPGamma12}, we obtain SCP for de~Sitter space (in any dimension) for all $\gamma>1$, or $h<1$ for $h=\gamma^{-1}$ in the notation of \eqref{EqSCPSemiRescaledOp}.

\subsection{Asymptotics for the linearized gauged Einstein equation}
\label{SubsecdSAsy}

For simplicity, we now fix $\gamma_1=2$ and $\gamma_2=1$ and the resulting operator $\tdel^*$ in \eqref{EqdSSCPDelStarTilde}; for these values, we do have SCP. We then consider the modified gauged Einstein equation $(\Ric+n)(g)-\tdel^*\Ups(g)=0$; we denote the linearized operator again by $L$ and calculate, using \eqref{EqdSSCPDelgGg},
\[
  2 L=-e_0^2 + n e_0 + \sum_i e_i^2
    +\begin{pmatrix}
       e_0       & -2\delta_h             & (e_0-2)\tr_h \\
       d_\sfX    & 2e_0-n-1               & d_\sfX\tr_h  \\
       h(e_0-2n) & 4\delta_h^*+2h\delta_h & h e_0\tr_h
     \end{pmatrix}.
\]
Since we have arranged SCP, a mode stability statement parallel to UEMS in the de~Sitter setting would now imply that all non-decaying modes of $L$ are pure gauge solutions. (This is the main difference to the black hole setting, in which one also has modes with frequency $0$ corresponding to the Kerr--de~Sitter family. Perturbations of de~Sitter space on the other hand decay exponentially fast to de~Sitter space, up to diffeomorphisms.) We prove this directly. Using the bundle splitting \eqref{EqdSSplit} and further splitting $V_{TT}$ according to \eqref{EqdSSplitVT}, we have
\[
  I(2 L,\sigma)
    = \sigma^2+i n\sigma
      + \begin{pmatrix}
          i\sigma       & 0            & 0 & (i\sigma-2)n \\
          0             & 2i\sigma-n-1 & 0 & 0            \\
          0             & 0            & 0 & 0            \\
          i\sigma-2n    & 0            & 0 & i n\sigma
        \end{pmatrix}.
\]
First, we note that $I(2 L,\sigma)$ preserves sections of the bundle $V_{TN}$ and equals scalar multiplication by $\sigma^2+i(n+2)\sigma-(n+1)$, which has roots $-i$ and $-i(n+1)$, both of which lie in the lower half plane.

Next, on $V_{TT0}$, the operator $I(2 L,\sigma)$ is scalar multiplication by $\sigma(\sigma+i n)$, whose only root in the closed upper half plane is $\sigma=0$; this corresponds to $L r=\cO(\tau)$ for any section $r\in\CI(\sfM,V_{TT0})$.

Lastly, on $V_{NN}\oplus V_{TTp}$, one finds
\[
  I(2 L,\sigma)\begin{pmatrix}1\\0\\h\end{pmatrix} = (\sigma^2+i(2n+1)\sigma-2n)\begin{pmatrix}1\\0\\h\end{pmatrix},
\]
which vanishes only for $\sigma=-i$ and $\sigma=-2 i n$ which are both in the lower half plane; on the other hand,
\[
  I(2 L,\sigma)\begin{pmatrix}i\sigma\\0\\h\end{pmatrix} = (\sigma^2+in\sigma+2n)\begin{pmatrix}i\sigma\\0\\h\end{pmatrix}.
\]
(Note that for $\Im\sigma\geq 0$, the two vectors above are linearly independent, hence span $V_{TT0}\oplus V_{TTp}$.) Corresponding to the unique zero of the quadratic polynomial appearing here with non-negative imaginary part $\sigma=\sigma_+$, see \eqref{EqdSUnmodPesky}, we have for any function $f\in\CI(\sfX)$
\[
  2 L \tau^{i\sigma_+}\begin{pmatrix}i\sigma_+ f \\ 0 \\ h f \end{pmatrix} = \tau^{i\sigma_+ +1}\begin{pmatrix} 0 \\ (i\sigma_++n)\sum_j \pa_{w_j}f e^j \\ 0 \end{pmatrix} + \cO(\tau^{i\sigma_+ +2}).
\]
Since $\sigma_+-i$ is not an indicial root of $L$, we can solve away the $\tau^{i\sigma_++1}$ error term, as we proceed to do; note that for all $n\geq 2$, one has $2<-n+\sqrt{n^2+8n}<4$, thus $\Im\sigma_+\in(1,2)$, and we ultimately find
\begin{equation}
\label{EqdSAsyNNTTp}
  2 L\begin{bmatrix} \tau^{i\sigma_+}\begin{pmatrix}i\sigma_+ f\\0\\h f\end{pmatrix} + \frac{n+i\sigma_+}{2n}\tau^{i\sigma_+ +1} \begin{pmatrix} 0 \\ \sum_j \pa_{w_j}f e^j \\ 0 \end{pmatrix} \end{bmatrix} = o(1).
\end{equation}
Note that the leading part is equal to $f r_1$, with $r_1$ given in \eqref{EqdSUnmodPeskyPureGauge}; this was shown to be a pure gauge solution up to lower order terms. In order to proceed, we now restrict to a static patch, where we can explicitly exhibit these non-decaying modes as pure gauge modes.

\subsection{Restriction to a static patch}
\label{SubsecdSB}

We fix a static patch of de~Sitter space by choosing a point $q\in\sfX$ as the origin of our coordinate system $(\tau,w_1,\ldots,w_n)$, and homogeneously blowing up $q$; coordinates on the static patch are then
\[
  \tau,\quad x_i := \frac{w_i}{\tau},
\]
with the front face given by $\tau=0$. Correspondingly, our frame takes the form
\[
  e_0 = \tau\pa_\tau - \sum x_j\pa_{x_j},\quad e_i = \pa_{x_i},
\]
and the coframe
\[
  e^0 = \frac{d\tau}{\tau},\quad e^i = dx_i + x_i\frac{d\tau}{\tau}.
\]
We work on a neighborhood $\Omega$ of the causal past of $q$; concretely, let us take
\[
  \Omega = \{ 0\leq\tau\leq 1,\ \sum x_j^2 \leq 1+\eps_M \}
\]
for any fixed $\eps_M>0$, so $\Omega$ is a domain with corners within
\[
  M := \{ 0\leq\tau<\infty,\ \sum x_j^2 < 1+3\eps_M \}.
\]
Suppose now $r\in\CI(\sfM)$ is a smooth function on global de~Sitter space $\sfM$. Then the pullback $\wt r$ of $r$ to $\Omega$ is
\begin{equation}
\label{EqdSB0ToB}
  \wt r(\tau,x_i) = r(\tau x_i) = r(0) + \tau\sum x_i\pa_{w_i}r(0) + \cO(\tau^2);
\end{equation}
continuing the Taylor expansion further, one finds that $\wt r$ is (asymptotically as $\tau\to 0$) a sum of terms of the form $\tau^j$ times a homogeneous polynomial of degree $j$ in the $x_i$. Thus, one can deduce the resonance expansion in the static patch from the calculations in \S\ref{SubsecdSAsy} by taking $r$ to be a homogeneous polynomial in the coordinates $w_i$ of degree $\leq 1$ (since higher order terms will give $o(1)$ contributions, the imaginary part of all resonances being $<2$) and reading off the terms in the resulting asymptotic expansions in $\tau$. Indeed, every resonant state on static de~Sitter space arises as a term in the asymptotic expansion of the solution of a wave equation with smooth forcing, compactly supported and with support disjoint from $\sfX$; and conversely, solutions of such equations on static de~Sitter spaces admit an expansion into resonances up to terms of any fixed, prescribed rate of decay $\tau^C$, $C\in\R$; but the asymptotic behavior of waves on static de~Sitter space, which are in this way equivalent to knowledge of resonances and resonant states, can simply be read off by restriction from the asymptotics on \emph{global} de~Sitter space.

We thus obtain:

\begin{prop}
\label{PropdSBEinStaticRes}
  For the operator $L\in\Diffb^2(\Omega;S^2\,\Tb^*_\Omega M)$, the following is a complete list of the resonances $\sigma$ of $L$ which satisfy $\Im\sigma>\frac{1}{2}(-n+\sqrt{n^2+8n})-2$, and the corresponding resonant states:
  \begin{enumerate}
  \item $\sigma=\sigma_+=\frac{i}{2}(-n+\sqrt{n^2+8n})$: resonance of order $1$ and rank $1$; basis of resonant states
    \[
      \tau^{i\sigma_+}\begin{pmatrix}i\sigma_+\\0\\h\end{pmatrix}.
    \]
  \item $\sigma=\sigma_+-i$: resonance of order $1$ and rank $n$; basis of resonant states
    \[
      \tau^{i\sigma_+ +1}\begin{pmatrix}-\sigma_+^2 x_j \\ (i\sigma_+ +1)e^j \\ i\sigma_+ x_j h \end{pmatrix}, \quad j=1,\ldots,n.
    \]
  \item \label{ItBEinStaticResZero} $\sigma=0$: resonance of order $1$ and rank $n(n+1)/2-1$; basis of resonant states
    \[
       \begin{pmatrix} 0 \\ 0 \\ a_{ij}e^i e^j\end{pmatrix},\quad \sum_i a_{ii}=0,\qquad
     \]
  \end{enumerate}
\end{prop}
\begin{proof}
  We prove the result for $\sigma=\sigma_+$ and $\sigma_+-i$: setting $f\equiv 1$ in \eqref{EqdSAsyNNTTp} yields the resonant state at $\sigma_+$, while setting $f=i\sigma_+ w_j$, $j=1,\ldots,n$, yields the resonant states at $\sigma_+-i$ due to \eqref{EqdSB0ToB}.
\end{proof}

We can now complete the proof that all non-decaying resonant states are pure gauge solutions:

\begin{prop}
\label{PropdSEinGaugeStates}
  All resonant states corresponding to resonances in $\Im\sigma\geq 0$, viewed as mode solutions of $L$ on the spacetime, lie in the range of $\delta_{g_0}^*$ acting on 1-forms on the spacetime.
\end{prop}
\begin{proof}
  At $\sigma_+$, we compute in the splittings \eqref{EqdSSplit}
  \[
    \delta_{g_0}^* \tau^{i\sigma_+}\begin{pmatrix}1\\0\end{pmatrix} = \tau^{i\sigma_+}\begin{pmatrix}i\sigma_+\\0\\h\end{pmatrix}.
  \]
  More generally, we may compute
  \[
    \delta_{g_0}^* \tau^{i\sigma_+}\begin{pmatrix}f\\0\end{pmatrix} = \tau^{i\sigma_+}\begin{pmatrix}i\sigma f\\0\\h f\end{pmatrix} + \tau^{i\sigma_+ +1}\begin{pmatrix}0\\ \frac{1}{2}\sum \pa_{w_j}f e^j \\0\end{pmatrix},
  \]
  and one can solve away the second term, with the result
  \[
    \delta_{g_0}^*\begin{bmatrix} \tau^{i\sigma_+}\begin{pmatrix}f\\0\end{pmatrix} + \frac{1}{i\sigma_+}\tau^{i\sigma_+ +1}\begin{pmatrix}0\\ \sum\pa_{w_j}f e^j \end{pmatrix} \end{bmatrix} = o(1).
  \]
  Putting $f=1$ gives the resonant state at $\sigma_+$, thus proving the result for the resonance $\sigma_+$, while putting $f=i\sigma_+ w_j$, $j=1,\ldots,n$, we find
  \[
    \delta_{g_0}^* \tau^{i\sigma_+ +1}\begin{pmatrix}i\sigma_+ x_j \\ e^j\end{pmatrix} = \tau^{i\sigma_+ +1}\begin{pmatrix}-\sigma_+^2 x_j\\(i\sigma_++1)e^j\\i\sigma_+ x_j h\end{pmatrix},
  \]
  proving the result for the resonance $\sigma_+-i$. Finally, for $\sigma=0$, we observe
  \[
    \delta_{g_0}^* \begin{pmatrix}0\\ \sum_i a_{ij}x_i e^j \end{pmatrix} = \delta_{g_0}^* \tau^{-1}\begin{pmatrix}0\\ \sum_i a_{ij}w_i e^j\end{pmatrix} = \begin{pmatrix} 0 \\ 0 \\ a_{ij}e^ie^j \end{pmatrix},
  \]
  finishing the proof.
\end{proof}

\begin{rmk}
\label{RmkdSBAlwaysZero}
  For any choice of parameters $\gamma_1,\gamma_2$, the space of resonances at $0$ is always non-trivial, and contains the resonant states given in Proposition~\ref{PropdSBEinStaticRes} \eqref{ItBEinStaticResZero}. There are further modifications one can consider, for instance using a conformally rescaled background metric $g^0=\tau^{\gamma_3}g_0$ and considering the gauge $\Ups(g)-\Ups(g_0)$, with $\Ups(g)=g(g^0)^{-1}\delta_g \sfG_g g^0$, but this does not affect the previous statement regarding the zero resonance. Thus, if we are restricting ourselves to modifications of Einstein's equations which are well-behaved from the perspective of global de~Sitter space, there seems to be no way to eliminate \emph{all} non-decaying resonances! Choosing $\gamma_1,\gamma_2,\gamma_3$ appropriately, one can remove all non-decaying resonances apart from $0$, but this is quite delicate.
\end{rmk}

Denote by $N_\Theta:=n+1+n(n+1)/2-1$ (so $N_\Theta=9$ for $n=3$) the total dimension of the space of non-decaying resonant states; then, paralleling the proof of Proposition~\ref{PropKdSLStabComplement}, we let $\Theta$ be the $N_\Theta$-dimensional space of 1-forms $\theta$ of the form
\[
  \theta=-D_{g_0}\Ups(\delta_{g_0}^*(\chi\omega))=\delta_{g_0} \sfG_{g_0}\delta_{g_0}^*(\chi\omega),
\]
where $\chi(\tau)$ is a fixed cutoff, identically $1$ near $\tau=0$ and identically $0$ for $\tau\geq 1/2$, say, and $\omega$ is one of the 1-forms used in the proof of Proposition~\ref{PropdSEinGaugeStates} exhibiting the non-decaying modes as pure gauge modes. Thus, we have
\[
  L(\delta_{g_0}^*(\chi\omega)) = \tdel^*\theta;
\]
and furthermore $\theta$ is compactly supported in $(0,1)_\tau$ due to
\[
  0=\delta_{g_0} \sfG_{g_0} L(\delta_{g_0}^*\omega)=-\wtBoxCP_{g_0}\bigl(D_{g_0}\Ups(\delta_{g_0}^*\omega)\bigr)
\]
and SCP, which gives $D_{g_0}\Ups(\delta_{g_0}^*\omega)=0$ as $\delta_{g_0}^*\omega$ is a non-decaying mode. One can of course also check directly that the resonant states described in Proposition~\ref{PropdSBEinStaticRes} are annihilated by $\delta_{g_0} \sfG_{g_0}$; for the zero resonant states, this is straightforward to check, while for the resonant states at $\sigma_+$ and $\sigma_+-i$, this follows from the fact that $\delta_{g_0} \sfG_{g_0}$ applied to the expression in square brackets in \eqref{EqdSAsyNNTTp} gives a result of order $\tau^{i\sigma_+ +2}$. The upshot is that
\[
  \Theta \subset \CIc(\Omega^\circ;T^*\Omega^\circ)
\]
can be used as the fixed, $N_\Theta$-dimensional space of gauge modifications, using which we can prove the non-linear stability of the static model. That is, modifying the forcing terms of the linearized equations which we need to solve in the course of a non-linear iteration scheme by $\tdel^*\theta$ for suitable $\theta\in\Theta$ (which are found at each step by the linear solution operators), we can solve the linear equations---and thus the non-linear gauged Einstein equation---in spaces of exponentially decaying 2-tensors:

\begin{thm}
\label{ThmdSStab}
  Let $\Sigma_0=\Omega\cap\{\tau=1\}$ be the Cauchy surface of $\Omega$. Let $h,k\in\CI(\Sigma_0;S^2 T^*\Sigma_0)$ be initial data satisfying the constraint equations \eqref{EqHypDTConstraints}, and suppose $(h,k)$ is close to the data induced by the static de~Sitter metric $g_0$ in the topology of $H^{21}(\Sigma_0;S^2 T^*\Sigma_0)\oplus H^{20}(\Sigma_0;S^2 T^*\Sigma_0)$. Then there exist a compactly supported gauge modification $\theta\in\Theta$ and a section $\wt g\in\Hb^{\infty,\alpha}(\Omega;S^2\,\Tb^*_\Omega M)$, with $\alpha>0$ small and fixed, such that $g=g_0+\wt g$ solves the Einstein equation
  \[
    \Ric(g) + n g = 0,
  \]
  attaining the given initial data at $\Sigma_0$, in the gauge $\Ups(g)-\theta=0$ (see \eqref{EqdSUnmodEin} for the definition of $\Ups$). More precisely, $g$ solves the initial value problem
  \[
    \begin{cases}
      \Ric(g)+n g-\tdel^*(\Ups(g)-\theta) = 0 & \tn{in }\Omega^\circ, \\
      \gamma_0(g) = i_0(h,k)                  & \tn{on }\Sigma_0,
    \end{cases}
  \]
  where $i_0$ constructs correctly gauged (relative to $\Ups(g)=0$) Cauchy data from the given initial data $(h,k)$, analogously to Proposition~\ref{PropKdSIn}.
\end{thm}
\begin{proof}
  Given what we have arranged above, this follows directly from Theorem~\ref{ThmQ} if we take $z\colon\R^{N_\Theta}\cong\Theta\to\CIc(\Omega^\circ;T^*\Omega^\circ)$ to be the map $\theta\mapsto -\tdel^*\theta$.
\end{proof}

The number of derivatives here is rather excessive: in fact, due to the lack of trapping, one does not lose derivatives beyond the usual loss of $1$ derivative for hyperbolic equations; thus one can prove this theorem using a Newton-type iteration method as in \cite[\S8]{HintzQuasilinearDS}. But since we state this result in order to present a simple analogue of Theorem~\ref{ThmKdSStab}, we refrain from optimizing it.

While the above arguments prove the stability of the static model of de~Sitter space, there is absolutely \emph{no} direct implication for the initial value problem near a Schwarzschild--de~Sitter spacetime: the limit $\bhm\to 0$ in which Schwarzschild--de~Sitter space becomes de~Sitter space is very singular.

\input{kdsstab-rr.bbl}

\end{document}

%% file: kdsstab-rr.bbl
\newcommand{\etalchar}[1]{$^{#1}$}